\documentclass[10pt,CMpaper]{CMbook}
\usepackage{ifpdf}
\ifpdf%
\PassOptionsToPackage{pdftex}{graphicx}
\PassOptionsToPackage{pdftex}{color}
\else%
\PassOptionsToPackage{dvips}{graphicx}%
\fi

\usepackage{graphicx,color}
\usepackage{xr}
\usepackage{eurosym}
\usepackage{
amsfonts,amssymb}
\usepackage{mathrsfs}
\usepackage{oldstyle}
\usepackage[T1]{fontenc}
\usepackage[utf8]{inputenc}
\usepackage[french]{babel}
\usepackage[autolanguage]{numprint}
\usepackage{datetime}
\usepackage[all]{xy}
\usepackage{pb-diagram}
\usepackage[french,nohints]{minitoc}
\usepackage{multicol}
\usepackage{CMmacros2}
\usepackage{makeidx}
\makeindex\makeglossary
\usepackage{xspace}

\usepackage{mathrsfs}
\usepackage{lmodern}
\usepackage{microtype}
\usepackage{array}
\usepackage{young}       
\usepackage{youngtab}    
\usepackage{stmaryrd}
\usepackage{float}
\usepackage{wasysym}

\usepackage{bm}

\DeclareFontFamily{U}{wasy}{}
\DeclareFontShape{U}{wasy}{m}{n}{<5> <6> <7> <8> <9> gen * wasy
      <10> <10.95> <12> <14.4> <17.28> <20.74> <24.88>wasy10  }{}
\DeclareFontShape{U}{wasy}{b}{n}{ <-10> ssub * wasy/m/n
 <10> <10.95> <12> <14.4> <17.28> <20.74> <24.88>wasyb10 }{}
\DeclareFontShape{U}{wasy}{bx}{n}{<5> <6> <7> <8> <9> gen * wasy
 <10> <10.95> <12> <14.4> <17.28> <20.74> <24.88>wasyb10}{}
\DeclareSymbolFont{wasy}{U}{wasy}{m}{n}
\SetSymbolFont{wasy}{bold}{U}{wasy}{bx}{n}
\DeclareFontFamily{U}{lasy}{}
\DeclareFontShape{U}{lasy}{m}{n}{ <5> <6> <7> <8> <9> gen * lasy
      <10> <10.95> <12> <14.4> <17.28> <20.74> <24.88>lasy10  }{}
\DeclareFontShape{U}{lasy}{b}{n}{ <5> <6> <7> <8> <9> gen * lasy
      <10> <10.95> <12> <14.4> <17.28> <20.74> <24.88>lasyb10  }{}
\DeclareFontFamily{U}{stmry}{}
\DeclareFontShape{U}{stmry}{m}{n}
   {  <5> <6> <7> <8> <9> <10> gen * stmary
      <10.95><12><14.4><17.28><20.74><24.88>stmary10%
   }{}
\DeclareFontShape{U}{stmry}{b}{n}
   {  <5> <6> <7> <8> <9> <10> gen * stmary
      <10.95><12><14.4><17.28><20.74><24.88>stmary10%
   }{}

\ifpdf
\makeatletter
\def\toclevel@chapternb{0}
\makeatother
\usepackage[bookmarksopen=false,pdftex=true,breaklinks=true,backref=page,%
    pagebackref=true,plainpages=false,%
    hyperindex=true,pdfstartview=FitH,
    pdfpagelabels=true,linkcolor=blue,citecolor=red,urlcolor=blue,
    colorlinks=true]%
    {hyperref}
\else

\fi

\title{R\'esolutions libres finies\\  M\'ethodes constructives 
}
\author{
Thierry Coquand, Henri Lombardi, Claude Quitté \& Claire Tête
}

\DeclareMathAlphabet{\mathpzc}{OT1}{pzc}{m}{it}

\begin{document}

\overfullrule=0cm
\hfuzz14pt

\let\oldref\ref
\renewcommand{\ref}[1]{\hbox{\oldref{#1}}}

\makeatletter
\def\U@@@ref#1\relax#2\relax{\let\showsection\relax#2}
\newcommand{\U@@ref}[5]{\U@@@ref#1}
\newcommand{\U@ref}[1]{\NR@setref{#1}\U@@ref{#1}}
\newcommand{\@iref}{\@ifstar\@refstar\U@ref}
\newcommand{\iref}[1]{\hbox{\@iref{#1}}}
\def\l@chapter#1#2{\ifnum\c@tocdepth>\m@ne\addpenalty{-\@highpenalty}\vskip1.0em\@plus\p@
\setlength\@tempdima{2.8em}\begingroup\parindent\z@\rightskip\@pnumwidth\parfillskip-\@pnumwidth\penalty-2000\leavevmode\bfseries
\advance\leftskip\@tempdima\hskip-\leftskip{\boldmath#1}\nobreak\hfil\nobreak\hb@xt@\@pnumwidth{\hss}\par\penalty\@highpenalty\endgroup\fi}
\def\l@chapterbis#1#2{\ifnum\c@tocdepth>\m@ne\addpenalty{-\@highpenalty}\vskip1.0em\@plus\p@
\setlength\@tempdima{2em}\begingroup\parindent\z@\rightskip\@pnumwidth\parfillskip-\@pnumwidth\penalty-2000\leavevmode\bfseries
\advance\leftskip\@tempdima\hskip-\leftskip{\boldmath#1}\nobreak\hfil\nobreak\hb@xt@\@pnumwidth{\hss#2}\par\penalty\@highpenalty\endgroup\fi}
\def\toclevel@chapterbis{0}
\def\l@section{\@dottedtocline{1}{1em}{0em}}
\def\l@subsection{\@dottedtocline{2}{1.75em}{0.25em}}
\def\l@subsubsection{\@dottedtocline{3}{2.5em}{0.25em}}
\def\numberline#1{\hb@xt@\@tempdima{\hss#1\ \hfil}}
\def\contentsline#1#2#3#4{\ifx\\#4\\\csname
l@#1\endcsname{#2}{#3}\else\ifHy@linktocpage\csname
l@#1\endcsname{{#2}}{\hyper@linkstart{link}{#4}{#3}\hyper@linkend }
\else\csname l@#1\endcsname{\hyper@linkstart{link}{#4}{#2}\hyper@linkend
}{#3}\fi\fi}
\makeatother


\setcounter{minitocdepth}{3}
\dominitoc

\ifx\mesMacrosDejaChargees\undefined\let\chargeMesMacros\relax
\else\let\chargeMesMacros \fi
\chargeMesMacros

\gdef\mesMacrosDejaChargees{}

\newcommand \dotdiv {\; {{\,.\,} \over {} } \;}



\FrenchFootnotes

\renewcommand{\labelenumii}{\textit{\theenumii.}}
\renewcommand{\labelenumi}{\textit{\theenumi.}}

\newcommand \entrenous[1] {\sibrouillon{\begin{flushleft}

\textsf{\textbf{Entre nous\,:}}
{\small\sf  #1\par}
\vspace{-.8mm}
\textsf{\textbf{Fin d'entre nous}}\end{flushleft}}}

\newcommand \hum[1]{\sibrouillon{\rdb\begin{flushleft}\tt\small hum:  #1\eoe

\end{flushleft}}}

\newcommand \ttt[1]{{\sibrouillon{\tt\small  #1
}}}

\newcommand\perso[1]{\sibrouillon{\marginpar{\hspace*{1em}
\begin{minipage}{10em}{\begin{flushleft}\footnotesize #1

\end{flushleft}}
\end{minipage}}}}

\newcounter{bidon}
\newcommand{\rdb}{\refstepcounter{bidon}}

\newcommand \ix[1] {\index{#1}\emph{#1}}
\newcommand \ixc[2] {\index{#1!#2}\emph{#1}}
\newcommand \ixx[2] {\index{#1!#2}\emph{#1 #2}}
\newcommand \ixy[2] {\index{#2!#1}\emph{#1 #2}}
\newcommand \ixd[2] {\index{#2!#1}\emph{#1}}
\newcommand \ixe[2] {\index{#2@#1}\emph{#1}}
\newcommand \ixf[3] {\index{#3@#2!#1}\emph{#1}}
\newcommand \ixg[3] {\index{#3@#1!#2}\emph{#1}}

\newcommand \CAdre[2]{%
\begin{center}
\begin{tabular}%
{|p{#1\textwidth}|}
\hline
\vspace{-1.5mm}
#2 
\vspace{1mm}\\ %
\hline
\end{tabular}
\end{center} 
}

\newcommand \Grandcadre[1]{%
\begin{center}
\begin{tabular}{|c|}
\hline
~\\[-3mm]
#1\\[-3mm]
~\\ 
\hline
\end{tabular}
\end{center}

}


\CMnewtheorem{theorem}{Théorème}{\itshape}   
\CMnewtheorem{thdef}{Théorème et définition}{\itshape}
\CMnewtheorem{plcc}{Principe local-global concret}{\itshape}
\CMnewtheorem{prcf}{Principe de recouvrement fermé}{\itshape}
\CMnewtheorem{prmf}{Principe de recollement fermé}{\itshape}
\CMnewtheorem{prvq}{Principe de recouvrement par quotients}{\itshape}
\CMnewtheorem{prvqp}{Principe de recouvrement par quotients principaux}{\itshape}
\CMnewtheorem{prmq}{Principe de recollement de quotients}{\itshape}
\CMnewtheorem{prmqp}{Principe de recollement de quotients principaux}{\itshape}
\CMnewtheorem{plca}{Principe local-global abstrait$\mathbf{^*}$}{\itshape}
\CMnewtheorem{plcd}{Principe local-global dynamique}{\itshape}
\CMnewtheorem{proposition}{Proposition}{\itshape}
\CMnewtheorem{propdef}{Proposition et définition}{\itshape}
\CMnewtheorem{propnot}{Proposition et notation}{\itshape}
\CMnewtheorem{lemma}{Lemme}{\itshape}
\CMnewtheorem{corollary}{Corolaire}{\itshape}
\CMnewtheorem{fact}{Fait}{\itshape}
\CMnewtheorem{theoremc}{Théorème\eto}{\itshape}
\CMnewtheorem{lemmac}{Lemme\etoz}{\itshape}
\CMnewtheorem{corollaryc}{Corolaire\eto}{\itshape}
\CMnewtheorem{proprietec}{Propriété\eto}{\itshape}
\CMnewtheorem{propositionc}{Proposition\eto}{\itshape}
\CMnewtheorem{factc}{Fait\eto}{\itshape}

\CMnewtheorem{remark}{Remarque}{}
\CMnewtheorem{remarks}{Remarques}{}
\CMnewtheorem{comment}{Commentaire}{}
\CMnewtheorem{comments}{Commentaires}{}
\CMnewtheorem{example}{Exemple}{}
\CMnewtheorem{exaf}{Exemple fondamental}{}
\CMnewtheorem{examples}{Exemples}{}
\CMnewtheorem{defic}{Définition\eto}{}
\CMnewtheorem{definition}{Définition}{}
\CMnewtheorem{definitions}{Définitions}{}
\CMnewtheorem{definota}{Définition et notation}{}
\CMnewtheorem{definotas}{Définitions et notations}{}
\CMnewtheorem{convention}{Convention}{}
\CMnewtheorem{notation}{Notation}{} 
\CMnewtheorem{context}{Contexte}{} 
\CMnewtheorem{notations}{Notations}{} 
\CMnewtheorem{question}{Question}{}
\CMnewtheorem{questions}{Questions}{}
\CMnewtheorem{algorithm}{Algorithme}{}

\newcounter{exercise}[chapter]
\newenvironment{exercise}{\ifhmode\par\fi
\vskip-\lastskip\vskip1.5ex\mou\penalty-300 \relax
\everypar{}\noindent
\refstepcounter{exercise}{\bfseries Exercice \theexercise.}\relax
\itshape\ignorespaces}{\par\vskip-\lastskip\vskip1em}

\newcounter{problem}[chapter]
\newenvironment{problem}{\ifhmode\par\fi
\vskip-\lastskip\vskip1em\mou\penalty-300 \relax
\everypar{}\noindent
\refstepcounter{problem}{\bfseries Probl\`eme \theproblem.}\relax
\itshape\ignorespaces}{\par\vskip-\lastskip\vskip1em}

\newcommand
\CHAP[1]{ 
\goodbreak\vskip4mm \mou \noindent  {\bf #1}\par\nobreak
\vskip1mm \mou \nobreak
}

\newcommand {\junk}[1]{}
\newcommand {\eop}{\hbox{$\square$}}

\newcommand \DebP{\raisebox{1pt}{\large\bf \Rightcircle}\,}
\newcommand \FinP{\raisebox{1pt}{\large\bf \Leftcircle$\!\!$}}

\def\thefootnote{\arabic{footnote}}

\newcommand \dsp{\displaystyle}
\newcommand \ndsp{\textstyle}

\newcommand \mapright[1]{\smash{\mathop{\longrightarrow}\limits^{#1}}} 
\newcommand \maprightto[1]{\smash{\mathop{\longmapsto}\limits^{#1}}} 
\def\mapdown#1{\downarrow\rlap{$\vcenter{\hbox{$\scriptstyle 
#1$}}$}}

\newcommand \bonbreak {\penalty-2500}

\newcommand \intro{\addcontentsline{toc}{section}{Introduction}
\subsection*{Introduction} }

\newcommand \Intro{\addcontentsline{toc}{section}{Introduction}
\markright{Introduction}%
\pagestyle{CMExercicesheadings}\subsection*{Introduction} }



\newcommand\Sclm[5] {#1=\big((#2_i)_{i\in #3},(#2_{ij}),(#2_{ijk}); (#4_{ij}),(#4_{ijk}); (#5_{ij})\big)}

\newcommand\sclm[5] {$\Sclm{#1}{#2}{#3}{#4}{#5}$}

\newcommand\Scsp[5] {#1=\big((#2_i)_{i\in #4},(#3_i),(#3_{ij}),(#3_{ijk}); (#5_{ij}),(#5_{ijk})\big)}

\newcommand\scsp[5] {$\Scsp{#1}{#2}{#3}{#4}{#5}$}

\renewcommand\matrix[1]{{\begin{array}{ccccccccccccccccccccccccc} #1 \end{array}}}  

\newcommand\Sady[1]{\mbox{\large$\mathpzc{#1}$}}
\newcommand\sa[1] {\Sady{#1}}

\newcommand \MA[1]{\mathop{#1}\nolimits}

\newcommand \tsbf[1]{\textbf{\textsf{#1}}}
\newcommand \lab[1]{\item[\tsbf{#1}]}

\newcommand{\vect}[1]{\mathchoice{\overrightarrow{\strut#1}}%
{\overrightarrow{\textstyle\strut#1}}{\overrightarrow{\scriptstyle#1}}{\overrightarrow{\scriptscriptstyle#1}}}

\newcommand \vab[2]{[\,#1\;#2\,]}

\newcommand \abs[1]{\left|{#1}\right|}
\newcommand \abS[1]{\big|{#1}\big|}
\newcommand \aqo[2]{#1\sur{\gen{#2}}\!}
\newcommand \Aqo[2]{#1\sur{\big\langle{#2}\big\rangle}\!}
\newcommand \bloc[4]{\left[\matrix{#1 & #2 \cr #3 & #4}\right]}

\newcommand \carray[2]{{\left[\begin{array}{#1} #2 \end{array}\right]}}
\newcommand \cmatrix[1]{\left[\matrix{#1}\right]}
\newcommand \clmatrix[1]{{\left[\begin{array}{lllllll} #1 \end{array}\right]}}
\newcommand \crmatrix[1]{{\left[\begin{array}{rrrrrrr} #1 \end{array}\right]}}
\newcommand \dmatrix[1]{\abs{\matrix{#1}}}
\newcommand\Cmatrix[2]{\setlength{\arraycolsep}{#1}\left[\matrix{#2}\right]}
\newcommand\Dmatrix[2]{\setlength{\arraycolsep}{#1}\left|\matrix{#2}\right|}

\newcommand{\Dpp}[2]{{{\partial #1}\over{\partial #2}}}

\newcommand \dt[1] {\bm{[}#1\bm{]}}
\newcommand \Hd{\tsbf{H}^{\llcorner}}
\newcommand \Hg{{\,^\lrcorner\tsbf{H}}}
\newcommand \DT[1] {\big{\bm{[}}#1\big{\bm{]}}}
\newcommand \Exl[1] {\mni {\it Exemple #1.} }
\newcommand \pex[1] {\left\llbracket #1 \right\rrbracket}
\newcommand \pexmat[1] {\pex{ \matrix{#1}} }

\newcommand \Dlu[2]{{\rm Dl}_{#1}(#2)}
\newcommand \dessus[2]{{\textstyle {#1} \atop \textstyle {#2}}}
\newcommand \eqdf[1]{\buildrel{#1}\over =}
\newcommand \formule[1]{{\left\{ {\arraycolsep2pt\begin{array}{lll} #1 \end{array}}\right.}}
\newcommand \formul[2]{{\left\{ \begin{array}{#1} #2 \end{array}\right.}}
\newcommand \gen[1]{\left\langle{#1}\right\rangle}
\newcommand \geN[1]{\big\langle{#1}\big\rangle}
\newcommand \impdef[1]{\buildrel{#1}\over \Longrightarrow}

\newcommand \eqdefi{\eqdf{\rm def}}
\newcommand \eqdef{\buildrel{{\rm def}}\over \Longleftrightarrow }

\newcommand\boite[2]{\begin{minipage}[c]{#1\cm}
     \centering {#2} \end{minipage}}    
\newcommand\Boite[3]{\parbox[t][#1\cm][c]{#2\cm}{\boite{#2}{#3}}}

\newcommand\emdf[1]{\textbf{\textit{#1}}}
\newcommand{\Kr}[2]{#1\lrb{#2}}

\newcommand \lra[1]{\langle{#1}\rangle}
\newcommand \lfb[1] {\llfloor #1 \rrfloor}
\newcommand \lrb[1] {\llbracket #1 \rrbracket}
\newcommand \lrbn {\lrb{1..n}}
\newcommand \lrbN {\lrb{1..N}}
\newcommand \lrbzn {\lrb{0..n}}
\newcommand \lrbl {\lrb{1..\ell}}
\newcommand \lrbm {\lrb{1..m}}
\newcommand \lrbk {\lrb{1..k}}
\newcommand \lrbh {\lrb{1..h}}
\newcommand \lrbp {\lrb{1..p}}
\newcommand \lrbq {\lrb{1..q}}
\newcommand \lrbr {\lrb{1..r}}
\newcommand \lrbs {\lrb{1..s}}

\newcommand \fraC[2] {{{#1}\over {#2}}}
\newcommand \meck[2] {\{#1, #2\}}
\newcommand \sat[1] {#1^{\rm sat}}
\newcommand \satu[2] {#1^{\rm sat_{#2}}}
\newcommand \scp[2] {\gen{#1\,|\,#2}\!}
\newcommand \scP[2] {\geN{#1\,|\,#2}\!}
\newcommand \sur[1]{\!\left/#1\right.}
\newcommand \so[1]{\left\{{#1}\right\}}
\newcommand \sO[1]{\big\{{#1}\big\}}
\newcommand \sotq[2]{\so{\,#1\,\vert\,#2\,}}
\newcommand \sotQ[2]{\sO{\,#1\;\big\vert\;#2\,}}
\newcommand \frt[1]{\!\left|_{#1}\right.\!}
\newcommand \Frt[2]{\left.#1\right|_{#2}\!}
\newcommand \sims[1]{\buildrel{#1}\over \sim}
\newcommand \tra[1]{{\,^{\rm t}\!#1}}
\newcommand \Al[1]{\Vi^{\!#1}}
\newcommand \Ae[1]{\gA^{\!#1}}

\newcommand \idg[1] {\hbox{$|\,#1\,|$}}
\newcommand \idG[1] {\hbox{$\big|\,#1\,\big|$}}

\newcommand \dex[1] {\hbox{$[\,#1\,]$}}
\newcommand \deX[1] {\hbox{$\big[\,#1\,\big]$}}

\newcommand \lst[1] {\hbox{$[\,#1\,]$}}
\newcommand \lsT[1] {\hbox{$\big[\,#1\,\big]$}}
\newcommand \brk[1] {[#1]}

\newcommand \Snic[1]{\sni{\centering$#1$\par}}
\newcommand \Snac[1]{\sni{\small{\centering$#1$\par}}}
\newcommand \Snuc[1]{\sni{\footnotesize{\centering$#1$\par}}}
\newcommand \snic[1]{

{\centering$#1$\par}

}
\newcommand \fnic[1]{

{\centering\fbox{$#1$}\par}

}

\newcommand \snac[1]{

{\small\centering$#1$\par}

}
\newcommand \snuc[1]{

{\footnotesize\centering$#1$\par}
}

\newcommand \snucc[1]{

{\footnotesize$$\preskip0pt\postskip0pt\textstyle#1$$}}

\newcommand \snicc[1]{
{$$\preskip0pt\postskip0pt\textstyle#1$$}}

\newcommand \snif[3]{\vspace{#1}\noindent\centerline{$#3$}
\vspace{#2}}

\newcommand \env[2]{{{#2}_{#1}^{\mathrm{e}}}} 
\newcommand \Om[2]{\Omega_{{#2}/{#1}}}
\newcommand \Der[3]{{\rm Der}_{{#1}}({#2},{#3})}

\newcommand \isA[1] {_{#1/\!\gA}}
\newcommand \OmA[1]{\Omega\isA{#1}}

\newcommand \bra[1] {\left[{#1}\right]}
\newcommand \bu[1] {{{#1}\bul}}
\newcommand \ci[1] {{{#1}^\circ}}
\newcommand \wi[1] {\widetilde{#1}}
\newcommand \wh[1]{{\widehat{#1}}}
\newcommand \ov[1] {\overline{#1}}
\newcommand \und[1] {\underline{#1}}

\newcommand \Bref[1] {\siBookdeux{\ref{#1}}}
\newcommand \Fref[1] {\siFFR{\ref{#1}}}
\newcommand \Dref[1] {\siDiviseurs{\ref{#1}}}
\newcommand \Cref[1] {[CACM, #1]}
\newcommand \cref[1] {\Cref{#1}}
\newcommand \thCref[1] {\Cref{\tho~#1}}
\newcommand \thcref[1] {\thCref{#1}}
\newcommand \aref[1] {\cref{#1}}
\newcommand \eref[1] {\emph{\ref{#1}}}
\newcommand \vref[1] {\ref{#1}}
\newcommand \vpageref[1] {\paref{#1}}
\newcommand{\pref}[1]{\textup{\hbox{\normalfont(\iref{#1})}}}
\newcommand \egrf[1] {\egt~\pref{#1}}
\newcommand \egref[1] {\egt~\pref{#1}  \paref{#1}}

\newcommand \VRT[1]{\rotatebox{90}{\hbox{$#1$}}}
\newcommand \VRTsubseteq{\VRT{\subseteq}}
\newcommand \VRTsupseteq{\VRT{\supseteq}}
\newcommand \VRTlongrightarrow{\VRT{\longrightarrow}}
\newcommand \VRTlongleftarrow{\VRT{\longleftarrow}}
\newcommand \VRTllar{\VRT{\longleftarrow}}
\newcommand \VRTlar{\VRT{\leftarrow}}
\newcommand \VRTlora{\VRT{\lora}}

\newcommand \rC[1]{\MA{{\rm C}_{#1}}}
\newcommand \rF[1]{\MA{{\rm F}_{\!#1}}}
\newcommand \rR[1]{\MA{{\rm R}_{#1}}}
\newcommand \rRs[1]{\MA{\Rs_{#1}}}
\newcommand \ep[1]{^{(#1)}}

\newcommand \bal[1] {^\rK_{#1}}
\newcommand \ul[1] {_\rK^{#1}}
\newcommand \SNw[1] {P_{#1}}
\newcommand \gBtst {\gB[[t]]^{\!\times}}


\newcommand{\dodo}{~~\cdots\cdots\cdots \!\!\!\!\!\!\!\!\!\!\!\!}

\newcommand{\llongrightarrow}{\relbar\joinrel\mkern-1mu\longrightarrow}
\newcommand{\lllongrightarrow}{\relbar\joinrel\mkern-1mu\llongrightarrow}
\newcommand{\llllongrightarrow}{\relbar\joinrel\mkern-1mu\lllongrightarrow}
\newcommand{\lllllongrightarrow}{\relbar\joinrel\mkern-1mu\llllongrightarrow}
\newcommand \lora {\longrightarrow}
\newcommand \llra {\llongrightarrow}
\newcommand \lllra {\lllongrightarrow}
\newcommand \llllra {\llllongrightarrow}
\newcommand \lllllra {\lllllongrightarrow}
\newcommand \simarrow{\vers{_\sim}}
\newcommand \isosim {\simarrow}
\newcommand \vers[1]{\buildrel{#1}\over \lora }
\newcommand \vvers[1]{\buildrel{#1}\over \llra }
\newcommand \vvvers[1]{\buildrel{#1}\over \lllra }
\newcommand \vvvvers[1]{\buildrel{#1}\over \llllra }
\newcommand \vvvvvers[1]{\buildrel{#1}\over \lllllra }
\newcommand{\llongleftarrow}{\longleftarrow\mkern-3mu\relbar\joinrel}

\newcommand \mt{\mapsto}
\newcommand \lmt{\longmapsto}

\renewcommand \le{\leqslant}
\renewcommand \leq{\leqslant}
\renewcommand \preceq{\preccurlyeq}
\renewcommand \ge{\geqslant}
\renewcommand \geq{\geqslant}
\renewcommand \succeq{\succurlyeq}

\newcommand \DeuxCol[2]{%
\sni\mbox{\parbox[t]{.475\textwidth}{#1}%
\hspace{.05\textwidth}%
\parbox[t]{.475\textwidth}{#2}}}

\newcommand \Deuxcol[4]{%
\sni\mbox{\parbox[t]{#1\textwidth}{#3}%
\hspace{.05\textwidth}%
\parbox[t]{#2\textwidth}{#4}}}

\newcommand \DeuxCols[2]{%
\sni\mbox{\hspace{.04\textwidth}%
\parbox[t]{.475\textwidth}{#1}%
\hspace{.03\textwidth}%
\parbox[t]{.475\textwidth}{#2}}}

\newcommand \UneCol[1]{%
\sni\mbox{\hspace{.04\textwidth}%
\parbox[t]{.96\textwidth}{#1}%
}}

\newcommand \UneRegle[2]{%
\sni\mbox{\hspace{.04\textwidth}%
\parbox[t]{.96\textwidth}{
\begin{enumerate}
\lab{#1}{#2} 
\end{enumerate}
}%
}}


\floatstyle{boxed}
\floatname{agc}{Algorithme}
\newfloat{agc}{ht}{lag}[section]

\floatstyle{boxed}
\floatname{agC}{Algorithme}
\newfloat{agC}{H}{lag}[section]

\newenvironment{algor}[1][]%
{\par\smallskip\begin{agc}
\vskip 1mm
\begin{algorithm}{\bfseries#1}
\upshape\sffamily
}
{\end{algorithm}
\end{agc}
}

\newenvironment{algoR}[1][]%
{\par\smallskip\begin{agC}
\vskip 1mm
\begin{algorithm}{\bfseries#1}
\upshape\sffamily
}
{\end{algorithm}
\end{agC}
}

\newcommand \Vrai{\mathsf{Vrai}}
\newcommand \Faux{\mathsf{Faux}}
\newcommand \ET{\mathsf{ et }}
\newcommand \OU{\mathsf{ ou }}
\newcommand \pour[3]{\textbf{pour } $#1$ \textbf{ de } $#2$
           \textbf{ \`a } $#3$ \textbf{ faire }}
\newcommand \pur[2]{\textbf{pour } $#1$ \textbf{ dans } $#2$
            \textbf{ faire }}
\newcommand \por[3]{\textbf{pour } $#1$ \textbf{ de } $#2$
           \textbf{ \`a } $#3$  }
\def\sialors#1{\textbf{si } $#1$ \textbf{ alors }}
\def\tantque#1{\textbf{tant que } $#1$ \textbf{ faire }}
\newcommand \finpour{\textbf{fin pour}}
\newcommand \sinon{\textbf{sinon }}
\newcommand \finsi{\textbf{fin si }}
\newcommand \fintantque{\textbf{fin tant que }}
\newcommand \aff{\leftarrow }
\newcommand \Debut{\\[1mm] \textbf{D\'ebut }}
\newcommand \Fin{\textbf{\\ Fin.}}
\newcommand \Entree{\\[1mm] \textbf{Entr\'ee : }}
\newcommand \Sortie{\\ \textbf{Sortie : }}
\newcommand \Varloc{\\ \textbf{Variables locales : }}
\newcommand \Repeter{\textbf{R\'ep\'eter }}
\newcommand \jusqua{\textbf{jusqu\`a ce que }}
\newcommand \hsz{\\ }
\newcommand \hsu{\\ \hspace*{4mm}}
\newcommand \hsd{\\ \hspace*{8mm}}
\newcommand \hst{\\ \hspace*{1,2cm}}
\newcommand \hsq{\\ \hspace*{1,6cm}}
\newcommand \hsc{\\ \hspace*{2cm}}
\newcommand \hsix{\\ \hspace*{2,4cm}}
\newcommand \hsept{\\ \hspace*{2,8cm}}

\newcommand \legendre {\overwithdelims()}
\newcommand \legendr[2] {\Big(\frac {#1}{#2} \Big)}
\newcommand \som {\sum\nolimits}
\newcommand \ds {\displaystyle}

\newcommand \et{\;\;\hbox{ et }\;\;}


\newcommand \divi {\mid}
\def \nedivi {\not\kern 2.5pt\mid}

\newcommand \vii {\hbox{\,{\footnotesize $\bm{\wedge}$}\,}}
\newcommand \vuu {\hbox{\,{\footnotesize $\bm{\vee}$}\,}}

\newcommand \vu {\vee} 
\newcommand \vi {\wedge} 
\newcommand \cvu {\curlyvee} 
\newcommand \cvi {\curlywedge} 
\newcommand \Vu {\bigvee\nolimits}
\newcommand \Vi {\bigwedge\nolimits}

\newcommand \vd[1] {\,\vdash_{#1}\,}
\newcommand \vda {\,\vdash\,}
\newcommand \vdg {\mathrel{\vrule height6.94444pt width1pt\vrule height3.97222pt width5.11111pt depth-2.97222pt}}  
\newcommand \Vd {\,\,\vdg\,\,}

\newcommand \Exists {\rotatebox[origin=c]{180}{\tsbf{E}}\hspace{.1em}} 
\newcommand \Bot {\rotatebox[origin=c]{180}{\tsbf{T}}\hspace{.1em}} 
\newcommand \Top {{\tsbf{T}}\hspace{.1em}}

\newcommand \im {\rightarrow} 
\newcommand \dar[1] {\MA{\downarrow \!#1}}
\newcommand \uar[1] {\MA{\uparrow \!#1}}

\newcommand \Un {\mathbf{1}}
\newcommand \Deux {\mathbf{2}}
\newcommand \Trois {\mathbf{3}}
\newcommand \Quatre {\mathbf{4}}
\newcommand \Cinq {\mathbf{5}}


\newcommand \Pf {{\rm P}_{{\rm f}}}
\newcommand \Pfe {{\rm P}_{{\rm fe}}}
\newcommand \Pfs {{\rm P}_{{\rm fe}}^*}

\newcommand \Ex {{\exists}}
\newcommand \Tt {{\forall}}

\newcommand \Lst{\mathsf{Lst}}
\newcommand \Irr{\mathsf{Irr}}
\newcommand \Prim{\mathsf{Prim}}
\newcommand \Rec{\mathsf{Rec}}

\newcommand \FAN{\tsbf{FAN}\xspace}
\newcommand \KL{\FAN }
\newcommand \KLp{\tsbf{KL}$_2^+$\xspace}
\newcommand \KLz{\tsbf{KL}$_2$}
\newcommand \kl{\tsbf{KL}$_1$\xspace}
\newcommand \HC{\tsbf{HC}\xspace}
\newcommand \LLPO{\tsbf{LLPO}\xspace}
\newcommand \LPO{\tsbf{LPO}\xspace}
\newcommand \MP{\tsbf{MP}\xspace}
\newcommand \TEM{\tsbf{PTE}\xspace}
\newcommand \UC{\tsbf{UC}\xspace}
\newcommand \UCp{\tsbf{UC}$^+$\xspace}
\newcommand \Mini{\tsbf{Min}\xspace}
\newcommand \Minip{\tsbf{Min}$^+$\xspace}
\newcommand \Minim{\tsbf{Min}$^-$\xspace}

\newcommand \bul{^\bullet}
\newcommand \eci{^\circ}
\newcommand \esh{^\sharp}
\newcommand \efl{^\flat}
\newcommand \epr{^\perp}
\newcommand \eti{^\times}
\newcommand \etl{^* }
\newcommand \eto{$^*\!$\xspace}
\newcommand \sta{^\star}
\newcommand \ist{_\star}
\newcommand \eo {^{\mathrm{op}}}

\newcommand \Abul {\gA\!\bul}
\newcommand \kbul {\gk\bul}
\newcommand \Ati {\gA^{\!\times}}
\newcommand \Asta {\gA^{\!\star}}
\newcommand \Atl {\gA^{\!*}}
\newcommand \Btl {\gB^{*}}
\newcommand \Ktl {\gK^{*}}
\newcommand \Bti {\gB^{\times}}
\newcommand \Bst {\Bti}
\newcommand{\KAt}{\Ktl\!\sur{\Ati}}
\newcommand{\AAt}{\Atl\!\sur{\Ati}}
\newcommand \te  {\otimes}

\newcommand \Arg {\gA^{\!\mathrm{rg}}}
\newcommand \Brg {\gB^{\mathrm{rg}}}
\newcommand \Krg {\gK^{\mathrm{rg}}}

\newcommand \iBA {_{\gB/\!\gA}}
\newcommand \iBk {_{\gB/\gk}}
\newcommand \iAk {_{\gA/\gk}}
\newcommand \iCk {_{\gC/\gk}}

\newcommand \tgaBG {\gB\{G\}}  
\newcommand \zcoho {Z^1(G, \Bti)}
\newcommand \bcoho {B^1(G, \Bti)}
\newcommand \hcoho {H^1(G, \Bti)}

\newcommand \vep{{\varepsilon}}

\newcommand \equidef{\buildrel{{\rm def}}\over{\;\Longleftrightarrow\;}}



\newcommand \noi {\noindent}
\newcommand \sni {\smallskip\noi}
\newcommand \snii {}
\newcommand \ms {\medskip}
\newcommand \mni {\ms\noi}
\newcommand \bs {\bigskip}
\newcommand \bni {\bs\noi}
\newcommand \ce{\centering}
\newcommand \alb{\allowbreak}

\newcommand \ua  {{\underline{a}}}
\newcommand \uap {{\underline{a'}}}
\newcommand \ual {{\underline{\alpha}}}
\newcommand \ub  {{\underline{b}}}
\newcommand \ube {{\underline{\beta}}}
\newcommand \uc  {{\underline{c}}}
\newcommand \ucx  {{\uc,\ux}}
\newcommand \uci[1]{{\buildrel{\circ}\over{#1}}}
\newcommand \ud  {{\underline{d}}}
\newcommand \udel{{\underline{\delta}}}
\newcommand \ue  {{\underline{e}}}
\newcommand \uf  {{\underline{f}}}
\newcommand \uF  {{\underline{F}}}
\newcommand \ug  {{\underline{g}}}
\newcommand \uh  {{\underline{h}}}
\newcommand \uH  {{\underline{H}}}
\newcommand \uga {{\underline{\gamma}}}
\newcommand \um  {{\underline{m}}}
\newcommand \ur  {{\underline{r}}}
\newcommand \us  {{\underline{s}}}
\newcommand \ut  {{\underline{t}}}
\newcommand \uu  {{\underline{u}}}
\newcommand \ux  {{\underline{x}}}
\newcommand \uxy  {{\ux,\uy}}
\newcommand \uyx  {{\uy,\ux}}
\newcommand \uxi {{\underline{\xi}}}
\newcommand \uy  {{\underline{y}}}
\newcommand \uP  {{\underline{P}}}
\newcommand \uS  {{\underline{S}}}
\newcommand \uT  {{\underline{T}}}
\newcommand \uU  {{\underline{U}}}
\newcommand \uv  {{\underline{v}}}
\newcommand \uw  {{\underline{w}}}
\newcommand \uX  {{\underline{X}}}
\newcommand \uY  {{\underline{Y}}}
\newcommand \uZ  {{\underline{Z}}}
\newcommand \uz  {{\underline{z}}}
\newcommand \uzeta  {{\underline{\zeta}}}
\newcommand \uze {{\underline{0}}}

\newcommand \ak {a_1,\ldots,a_k}
\newcommand \an {a_1,\ldots,a_n}
\newcommand \am {a_1,\ldots,a_m}
\newcommand \aq {a_1,\ldots,a_q}
\newcommand \bk {b_1,\ldots,b_k}
\newcommand \bn {b_1,\ldots,b_n}
\newcommand \br {b_1,\ldots,b_r}
\newcommand \bbm {b_1,\ldots,b_m}
\newcommand \azn {a_0,\ldots,a_n}
\newcommand \bzn {b_0,\ldots,b_n}
\newcommand \czn {c_0,\ldots,c_n}
\newcommand \ck {c_1,\ldots,c_k}
\newcommand \cn {c_1,\ldots,c_n}
\newcommand \cq {c_1,\ldots,c_q}
\newcommand \gq {g_1,\ldots,g_q}
\newcommand \hr {h_1,\ldots,h_r}
\newcommand \hn {h_1,\ldots,h_n}
\newcommand \sn {s_1,\ldots,s_n} 
\newcommand \uk {u_1,\ldots,u_k} 
\newcommand \vk {v_1,\ldots,v_k} 
\newcommand \un {u_1,\ldots,u_n} 
\newcommand \vn {v_1,\ldots,v_n} 
\newcommand \vp {v_1,\ldots,v_p} 
\newcommand \xk {x_1,\ldots,x_k}
\newcommand \Xk {X_1,\ldots,X_k}
\newcommand \xl {x_1,\ldots,x_\ell}
\newcommand \xm {x_1,\ldots,x_m}
\newcommand \xn {x_1,\ldots,x_n}
\newcommand \xp {x_1,\ldots,x_p}
\newcommand \xq {x_1,\ldots,x_q}
\newcommand \xzk {x_0,\ldots,x_k}
\newcommand \xzn {x_0,\ldots,x_n}
\newcommand \xhn {x_0:\ldots:x_n}
\newcommand \Xn {X_1,\ldots,X_n}
\newcommand \Xzn {X_0,\ldots,X_n}
\newcommand \Xm {X_1,\ldots,X_m}
\newcommand \Xq {X_1,\ldots,X_q}
\newcommand \Xr {X_1,\ldots,X_r}
\newcommand \xr {x_1,\ldots,x_r}

\newcommand \Yd {Y_1,\ldots,Y_d}
\newcommand \yk {y_1,\ldots,y_k}
\newcommand \ym {y_1,\ldots,y_m}
\newcommand \Ym {Y_1,\ldots,Y_m}
\newcommand \Yn {Y_1,\ldots,Y_n}
\newcommand \yn {y_1,\ldots,y_n}
\newcommand \yp {y_1,\ldots,y_p}
\newcommand \Yr {Y_1,\ldots,Y_r}
\newcommand \yq {y_1,\ldots,y_q}
\newcommand \yr {y_1,\ldots,y_r}
\newcommand \Ys {Y_1,\ldots,Y_s}

\newcommand \zn {z_1,\ldots,z_n}
\newcommand \Zn {Z_1,\ldots,Z_n}

\newcommand \Sun {$S_1$, $\dots$, $S_n$ }
\newcommand \Sunz {$S_1$, $\dots$, $S_n$}

\newcommand \xpn {x'_1,\ldots,x'_n}
\newcommand \uxp  {{\underline{x'}}}
\newcommand \ypm {y'_1,\ldots,y'_m}
\newcommand \uyp  {{\underline{y'}}}

\newcommand \aln {\alpha_1,\ldots,\alpha_n}
\newcommand \gan {\gamma_1,\ldots,\gamma_n}
\newcommand \xin {\xi_1,\ldots,\xi_n}
\newcommand \xihn {\xi_0:\ldots:\xi_n}

\newcommand \lar {a_1,\ldots,a_r}
\newcommand \lfc {f_1,\ldots,f_c}
\newcommand \lfk {f_1,\ldots,f_k}
\newcommand \lfm {f_1,\ldots,f_m}
\newcommand \lfn {f_1,\ldots,f_n}
\newcommand \lfp {f_1,\ldots,f_p}
\newcommand \lfq {f_1,\ldots,f_q}
\newcommand \lfs {f_1,\ldots,f_s}
\newcommand \lfr {f_1,\ldots,f_r}
\newcommand \lgm {g_1,\ldots,g_m}
\newcommand \lgn {g_1,\ldots,g_n}
\newcommand \lgp {g_1,\ldots,g_p}
\newcommand \lgq {g_1,\ldots,g_q}

\newcommand \Cin{C^{\infty}}
\newcommand \Ared {\gA\red}
\newcommand \Asep {\gA_{\rm sep}}
\newcommand \Aqim {\gA\qim}
\newcommand \Amin {\Aqim}
\newcommand \Aqi {\gA_\mathrm{qi}}

\newcommand \AT {{\gA[T]}}
\newcommand \AX {{\gA[X]}}
\newcommand \CT {{\gC[T]}}
\newcommand \CX {{\gC[X]}}
\newcommand \AY {{\gA[Y]}}
\newcommand \ArX {{\gA\lra X}}
\newcommand \ArY {{\gA\lra Y}}
\newcommand \BX {{\gB[X]}}
\newcommand \BY {{\gB[Y]}}
\newcommand \BT {{\gB[T]}}
\newcommand \kX {{\gk[X]}}
\newcommand \kT {{\gk[T]}}
\newcommand \KT {{\gK[T]}}
\newcommand \KX {{\gK[X]}}
\newcommand \KY {{\gK[Y]}}
\newcommand \QQX {{\QQ[X]}}
\newcommand \VX {{\gV[X]}}
\newcommand \ZZX {{\ZZ[X]}}
\newcommand \ZZx {{\ZZ[x]}}

\newcommand \AXY {\AuX \lra Y}

\newcommand \kYd {{\gk[\Yd]}}

\newcommand \kGa {{\gk[\Gamma]}}

\newcommand \kXn {{\gk[\Xn]}}
\newcommand \kYn {{\gk[\Yn]}}
\newcommand \lXn {{\gl[\Xn]}}
\newcommand \AXn {{\gA[\Xn]}}
\newcommand \BXn {{\gB[\Xn]}}
\newcommand \CXn {{\gC[\Xn]}}
\newcommand \KXn {{\gK[\Xn]}}
\newcommand \LXn {{\gL[\Xn]}}
\newcommand \RXn {{\gR[\Xn]}}
\newcommand \QQXn{{\QQ[\Xn]}}
\newcommand \VXn {{\gV[\Xn]}}
\newcommand \ZZXn{{\ZZ[\Xn]}}

\newcommand \AXq {{\gA[\Xq]}}

\newcommand \AXk  {{\gA[\Xk]}}
\newcommand \QQXk {{\QQ[\Xk]}}
\newcommand \ZZXk {{\ZZ[\Xk]}}

\newcommand \kXr {{\gk[\Xr]}}
\newcommand \lXr {{\gl[\Xr]}}
\newcommand \AXr {{\gA[\Xr]}}
\newcommand \KXr {{\gK[\Xr]}}
\newcommand \KYr {{\gK[\Yr]}}
\newcommand \LXq {{\gL[\Xq]}}
\newcommand \LXr {{\gL[\Xr]}}
\newcommand \RXr {{\gR[\Xr]}}
\newcommand \VXr {{\gV[\Xr]}}

\newcommand \AXm {{\gA[\Xm]}}
\newcommand \KKXm {{\KK[\Xm]}}
\newcommand \kkXm {{\kk[\Xm]}}
\newcommand \kXm {{\gk[\Xm]}}
\newcommand \KXm {{\gK[\Xm]}}
\newcommand \KYm {{\gK[\Ym]}}
\newcommand \kYm {{\gk[\Ym]}}
\newcommand \BYm {{\gB[\Ym]}}

\newcommand \kYr {{\gk[\Yr]}}
\newcommand \AYr {{\gA[\Yr]}}
\newcommand \lYr {{\gl[\Yr]}}
\newcommand \LYr {{\gL[\Yr]}}
\newcommand \RYr {{\gR[\Yr]}}

\newcommand \kyr {{\gk[\yr]}}
\newcommand \Ayr {{\gA[\yr]}}
\newcommand \Kyr {{\gK[\yr]}}

\newcommand \kym {{\gk[\ym]}}
\newcommand \kyn {{\gk[\yn]}}
\newcommand \Ayn {{\gA[\yn]}}
\newcommand \Ryn {{\gR[\yn]}}

\newcommand \kYs {{\gk[\Ys]}}

\newcommand \AYn {{\gA[\Yn]}}

\newcommand \Axr {{\gA[\xr]}}
\newcommand \Kxr {{\gK[\xr]}}
\newcommand \kxr {{\gk[\xr]}}

\newcommand \Kxzn {{\gK[\xzn]}}
\newcommand \RXzn {{\gR[\Xzn]}}
\newcommand \Rxzn {{\gR[\xzn]}}

\newcommand \kxm {{\gk[\xm]}}

\newcommand \kxn {{\gk[\xn]}}
\newcommand \lxn {{\gl[\xn]}}
\newcommand \Axn {{\gA[\xn]}}
\newcommand \Bxn {{\gB[\xn]}}
\newcommand \Cxn {{\gC[\xn]}}
\newcommand \Kxn {{\gK[\xn]}}
\newcommand \Kyn {{\gK[\yn]}}
\newcommand \Lxn {{\gL[\xn]}}
\newcommand \Rxn {{\gR[\xn]}}

\newcommand \Aux {{\gA[\ux]}}
\newcommand \Auy {{\gA[\uy]}}
\newcommand \Bux {{\gB[\ux]}}
\newcommand \Kuy {{\gK[\uy]}}
\newcommand \Kuu {{\gK[\uu]}}
\newcommand \Kux {{\gK[\ux]}}
\newcommand \Lux {{\gL[\ux]}}
\newcommand \kux {{\gk[\ux]}}
\newcommand \kuy {{\gk[\uy]}}
\newcommand \Rux {{\gR[\ux]}}
\newcommand \Ruy {{\gR[\uy]}}

\newcommand \ZZuS {{\ZZ[\uS]}}

\newcommand \AuX {{\gA[\uX]}}
\newcommand \BuX {{\gB[\uX]}}
\newcommand \KuX {{\gK[\uX]}}
\newcommand \kuX {{\gk[\uX]}}
\newcommand \LuX {{\gL[\uX]}}
\newcommand \RuX {{\gR[\uX]}}
\newcommand \KKuX {{\KK[\uX]}}
\newcommand \kkuX {{\kk[\uX]}}
\newcommand \ZZuX {{\ZZ[\uX]}}
\newcommand \QQuX {{\QQ[\uX]}}
\newcommand \AuY {{\gA[\uY]}}
\newcommand \BuY {{\gB[\uY]}}
\newcommand \KuY {{\gK[\uY]}}
\newcommand \kuY {{\gk[\uY]}}

\newcommand \Gn  {\gG_n}
\newcommand \Gnk {\gG_{n,k}}
\newcommand \Gnr {\gG_{n,r}}
\newcommand \cGn {{\cG_n}}
\newcommand \cGnk{{\cG_{n,k}}}

\newcommand \GGn {\GG_{n}}
\newcommand \GGnk{{\GG_{n,k}}} 
\newcommand \GGnr{{\GG_{n,r}}}
\newcommand \GA  {\mathbb{GA}}
\newcommand \GAn {\GA_{n}}  
\newcommand \GAq {\GA_{q}}
\newcommand \GAnk{\GA_{n,k}}
\newcommand \GAnr{\GA_{n,r}}

\newcommand \Mm {{\MM_{m}}}
\newcommand \Mn {{\MM_{n}}}
\newcommand \Mk {{\MM_{k}}}
\newcommand \Mq {{\MM_{q}}}
\newcommand \Mr {{\MM_{r}}}
\newcommand \MMn {{\MM_{n}}}

\newcommand \Bo{\BB\mathrm{o}}

\newcommand \GL {\mathbb{GL}}
\newcommand \GLn {{\GL_n}}
\newcommand \Gl {\mathbf{GL}}
\newcommand \Gln {{\Gl_n}}
\newcommand \PGL {\mathbb{PGL}}
\newcommand \SL {\mathbb{SL}}
\newcommand \SLn {{\SL_n}}
\newcommand \EE {\mathbb{E}}
\newcommand \En {\EE_n}
\newcommand \Pn {\PP^n}
\newcommand \Pm {\PP^m}
\newcommand \An {\AA^{\!n}}
\newcommand \Am {\AA^{\!m}}
\newcommand \Sl {\mathbf{SL}}
\newcommand \Sln {{\Sl_n}}

\newcommand \PnK {\Pn(\gK)}
\newcommand \PnA {\Pn(\gA)}
\newcommand \PnL {\Pn(\gL)}
\newcommand \PmL {\Pm(\gL)}

\newcommand \I  {\mathrm{I}}
\newcommand \G  {\mathrm{G}}

\newcommand \rA {\mathrm{A}}
\newcommand \rD {\mathrm{D}}
\newcommand \rE {\mathrm{E}}
\newcommand \rG {\mathrm{G}}
\newcommand \rH {\mathrm{H}}
\newcommand \rI {\mathrm{I}}
\newcommand \rJ {\mathrm{J}}
\newcommand \rK {\mathrm{K}}
\newcommand \rL {\mathrm{L}}
\newcommand \rM {\mathrm{M}}
\newcommand \rN {\mathrm{N}}
\newcommand \rO {\mathrm{O}}
\newcommand \rP {\mathrm{P}}
\newcommand \rQ {\mathrm{Q}}
\newcommand \rS {\mathrm{S}}
\newcommand \rT {\mathrm{T}}
\newcommand \rX {\mathrm{X}}
\newcommand \rY {\mathrm{Y}}
\newcommand \rZ {\mathrm{Z}}

\newcommand \rc {\mathrm{c}}
\newcommand \rd {\mathrm{d}}
\newcommand \rv {\mathrm{v}}

\renewcommand \AA{\mathbb{A}}
\newcommand \BB{\mathbb{B}}
\newcommand \CC{\mathbb{C}}
\newcommand \FF{\mathbb{F}}
\newcommand \FFp{\FF_p}
\newcommand \FFq{\FF_q}
\newcommand \GG{\mathbb{G}}
\newcommand \II{\mathbb{I}}
\newcommand \KK{\mathbb{K}}
\newcommand \kk{\mathbbmtt{k}} 
\newcommand \MM{\mathbb{M}}
\newcommand \NN{\mathbb{N}}
\newcommand \OO{\mathbb{O}}
\newcommand \PP{\mathbb{P}}
\newcommand \QQ{\mathbb{Q}}
\newcommand \UU{\mathbb{U}}
\newcommand \VV{\mathbb{V}}
\newcommand \ZZ{\mathbb{Z}}
\newcommand \ZB{\mathbb{ZB}}
\newcommand \RR{\mathbb{R}}
\newcommand \ASL {\mathbb{ASL}}
\newcommand \AGL {\mathbb{AGL}}

\newcommand \Z{\mathbb{Z}} 

\newcommand \cA {\mathcal{A}}
\newcommand \cB {\mathcal{B}}
\newcommand \cC {\mathcal{C}}
\newcommand \cD {\mathcal{D}}
\newcommand \cE {\mathcal{E}}
\newcommand \Diff {\mathcal{D}}
\newcommand \cG {\mathcal{G}}
\newcommand \cI {\mathcal{I}}
\newcommand \cJ {\mathcal{J}}
\newcommand \cF {\mathcal{F}}
\newcommand \cK {\mathcal{K}}
\newcommand \cL {\mathcal{L}}
\newcommand \cO {\mathcal{O}}
\newcommand \cP {\mathcal{P}}
\newcommand \cR {\mathcal{R}}
\newcommand \cM {\mathcal{M}}
\newcommand \cN {\mathcal{N}}
\newcommand \cS {\mathcal{S}}
\newcommand \cT {\mathcal{T}}
\newcommand \cV {\mathcal{V}}
\newcommand \cZ {\mathcal{Z}}

\newcommand \SK {\cS^\rK}
\newcommand \IK {\cI^\rK}
\newcommand \JK {\cJ^\rK}
\newcommand \IH {\cI^\rH}
\newcommand \JH {\cJ^\rH}

\newcommand \ga{\mathbf{a}}
\newcommand \gb{\mathbf{b}}
\newcommand \gc{\mathbf{c}}
\newcommand \gh{\mathbf{h}}
\newcommand \gk{\mathbf{k}}
\newcommand \gl{\mathbf{l}}
\newcommand \gs{\mathbf{s}}
\newcommand \gv{\mathbf{v}}
\newcommand \gw{\mathbf{w}}
\newcommand \gA{\mathbf{A}}
\newcommand \gB{\mathbf{B}}
\newcommand \gC{\mathbf{C}}
\newcommand \gD{\mathbf{D}}
\newcommand \gE{\mathbf{E}}
\newcommand \gF{\mathbf{F}}
\newcommand \gG{\mathbf{G}}
\newcommand \gK{\mathbf{K}}
\newcommand \gL{\mathbf{L}}
\newcommand \gM{\mathbf{M}}
\newcommand \gQ{\mathbf{Q}}
\newcommand \gR{\mathbf{R}}
\newcommand \gS{\mathbf{S}}
\newcommand \gT{\mathbf{T}}
\newcommand \gU{\mathbf{U}}
\newcommand \gu{\mathbf{u}} 
\newcommand \gV{\mathbf{V}}
\newcommand \gx{\mathbf{x}}
\newcommand \gy{\mathbf{y}}
\newcommand \gz{\mathbf{z}} 
\newcommand \gX{\mathbf{X}}
\newcommand \gW{\mathbf{W}}
\newcommand \gZ{\mathbf{Z}}

\newcommand \fa{\mathfrak{a}}
\newcommand \fb{\mathfrak{b}}
\newcommand \fc{\mathfrak{c}}
\newcommand \fd{\mathfrak{d}}
\newcommand \fA{\mathfrak{A}}
\newcommand \fB{\mathfrak{B}}
\newcommand \fC{\mathfrak{C}}
\newcommand \fD{\mathfrak{D}}
\newcommand \fG{\mathfrak{G}}
\newcommand \fI{\mathfrak{i}}
\newcommand \fII{\mathfrak{I}}
\newcommand \fj{\mathfrak{j}}
\newcommand \fJ{\mathfrak{J}}
\newcommand \ff{\mathfrak{f}}
\newcommand \ffg{\mathfrak{g}}
\newcommand \fF{\mathfrak{F}}
\newcommand \fh{\mathfrak{h}}
\newcommand \fl{\mathfrak{l}}
\newcommand \fm{\mathfrak{m}}
\newcommand \fn{\mathfrak{n}} 
\newcommand \fp{\mathfrak{p}}
\newcommand \fq{\mathfrak{q}}
\newcommand \fM{\mathfrak{M}}
\newcommand \fN{\mathfrak{N}}
\newcommand \fP{\mathfrak{P}}
\newcommand \fQ{\mathfrak{Q}}
\newcommand \fR{\mathfrak{R}}
\newcommand \fx{\mathfrak{x}}
\newcommand \fy{\mathfrak{y}}
\newcommand \fV{\mathfrak{V}}
\newcommand \fZ{\mathfrak{Z}}

\newcommand \scC{\mathscr{C}}
\newcommand \scH{\mathscr{H}}
\newcommand \scR{\mathscr{R}}

\newcommand{\bma}{\bm{a}}
\newcommand{\bmb}{\bm{b}}
\newcommand{\bmc}{\bm{c}}
\newcommand{\bmd}{\bm{d}}
\newcommand{\bme}{\bm{e}}
\newcommand{\bmf}{\bm{f}}
\newcommand{\bmu}{\bm{u}}
\newcommand{\bmv}{\bm{v}}
\newcommand{\bmw}{\bm{w}}
\newcommand{\bmy}{\bm{y}}
\newcommand{\bmx}{\bm{x}}
\newcommand{\bmz}{\bm{z}}

\newcommand \Kbu {{\rK\bul}}
\newcommand \ibu {{_\bullet}}
\newcommand \itbu {\item [$\bullet$]}
\newcommand \intd {\MA{\,\llcorner\,}}
\newcommand \intg {\MA{\,\lrcorner\,}}


\newcommand \Ig {\mathrm{Ig}}
\newcommand \Id {\mathrm{Id}}
\newcommand \In {{\rI_n}}
\newcommand \J {\MA{\mathrm{Jac}}}
\newcommand \Jf {\MA{\mathfrak{Jac}}}
\newcommand \JJ {\MA{\mathrm{JAC}}}
\newcommand \Sn {{\mathrm{S}_n}}

\newcommand \DA {\rD_{\!\gA}}
\newcommand \DB {\rD_\gB}
\newcommand \DV {\rD_\gV}
\newcommand \JA {\rJ_\gA}

\newcommand \red {_{\mathrm{red}}}
\newcommand \qim {_{\mathrm{min}}}
\newcommand \rja {\mathrm{Ja}}

\newcommand \ide {\mathrm{e}}

\newcommand \Adj {\MA{\mathrm{Adj}}}
\newcommand \adj {\MA{\mathrm{adj}}}
\newcommand \Adu {\MA{\mathrm{Adu}}}
\newcommand \Ann {\mathrm{Ann}}
\newcommand \Aut {\MA{\mathrm{Aut}}}
\newcommand \BZ {\mathrm{BZ}}
\newcommand \car {\MA{\mathrm{car}}}
\newcommand \Cl {\MA{\mathrm{Cl}}}
\newcommand \ClW {\MA{\mathrm{Cl}_{\mathrm{W}}}}
\newcommand \Coker {\MA{\mathrm{Coker}}}
\newcommand \Cont{\mathrm{Co}}
\newcommand \DDiv {\MA{\mathrm{Dv}}}
\renewcommand \det {\MA{\mathrm{det}}}
\renewcommand \deg {\MA{\mathrm{deg}}}
\newcommand \Diag {\MA{\mathrm{Diag}}}
\newcommand \di {\MA{\mathrm{di}}}
\newcommand \disc {\MA{\mathrm{disc}}}
\newcommand \Disc {\MA{\mathrm{Disc}}}
\newcommand \Div {\MA{\mathrm{Div}}}
\newcommand \DivA {\Div\gA }
\newcommand \DivAp {(\Div\gA)^{+} }
\newcommand \DivB {\Div\gB }
\newcommand \DivBp {(\Div\gB)^{+} }
\newcommand \DkM {\MA{\mathrm{DkM}}}
\newcommand \dv {\MA{\mathrm{div}}}
\newcommand \dvA {\dv_\gA }
\newcommand \dvB {\dv_\gB }
\newcommand \ev {{\mathrm{ev}}}
\newcommand \End {\MA{\mathrm{End}}}
\newcommand \Fix {\MA{\mathrm{Fix}}}
\newcommand \Frac {\MA{\mathrm{Frac}}}
\newcommand \Gal {\MA{\mathrm{Gal}}}
\newcommand \Gfr {\MA{\mathrm{Gfr}}}
\newcommand \gr {\MA{\mathrm{gr}}}
\newcommand \Gram {\MA{\mathrm{Gram}}}
\newcommand \gram {\MA{\mathrm{gram}}}
\newcommand \Grl {\MA{\mathrm{Grl}}}
\newcommand \GRL {\MA{\mathrm{GRL}}}
\newcommand \Grlg {\MA{\mathrm{Grlg}}}
\newcommand \GRLG {\MA{\mathrm{GRLG}}}
\newcommand \Grlgo {\MA{\mathrm{Grlgo}}}
\newcommand \GRLGO {\MA{\mathrm{GRLGO}}}
\newcommand \hauteur {\mathrm{hauteur}}
\newcommand \Ker {\MA{\mathrm{Ker}}}
\newcommand \Hom {\MA{\mathrm{Hom}}}
\newcommand \Idif {\MA{\mathrm{Idif}}}
\newcommand \Idv {\MA{\mathrm{Idv}}}
\newcommand \Ifr {\MA{\mathrm{Ifr}}}
\newcommand \Icl {\MA{\mathrm{Icl}}}
\renewcommand \Im {\MA{\mathrm{Im}}}
\newcommand \Lin {\mathrm{L}}
\newcommand \LIN {\mathrm{Lin}}
\newcommand \Lsf {\MA{\mathrm{Lsf}}}
\newcommand \Mat {\MA{\mathrm{Mat}}}
\newcommand \Mip {\mathrm{Min}}
\newcommand \md {\mathrm{md}}
\newcommand \Mgcd {\MA{\mathrm{Mgcd}}}
\newcommand \mod {\;\mathrm{mod}\;}
\newcommand \Mor {\MA{\mathrm{Mor}}}
\newcommand \poids {\mathrm{poids}}
\newcommand \poles {\hbox {\rm p\^oles}}
\newcommand \pgcd {\MA{\mathrm{pgcd}}}
\newcommand \ppcm {\MA{\mathrm{ppcm}}}
\newcommand \Rad {\MA{\mathrm{Rad}}}
\newcommand \Reg {\MA{\mathrm{Reg}}}
\newcommand \rg{\MA{\mathrm{rg}}}
\newcommand \rgst {\mathrm{rgst}}
\newcommand \Res {\mathrm{Res}}
\newcommand \Rs {\MA{\mathrm{Rs}}}
\newcommand \rPr{\MA{\mathrm{Pr}}}
\newcommand \Rv {\mathrm{Rv}}
\newcommand \Stp {\MA{\mathrm{Stp}}}
\newcommand \Set {\mathrm{Set}}
\newcommand \St {\mathrm{St}}
\newcommand \Tri {\MA{\mathrm{Tri}}}
\newcommand \Tor {\MA{\mathrm{Tor}}}
\newcommand \tr {\MA{\mathrm{tr}}}
\newcommand \Tr {\MA{\mathrm{Tr}}}
\newcommand \Tsc {\MA{\mathrm{Tsch}}}
\newcommand \Um {\MA{\mathrm{Um}}}
\newcommand \val {\MA{\mathrm{val}}}
\newcommand \Suslin{{\rm Suslin}}
\newcommand{\DBxk}{{\Der \gk\gB\xi}}%
\newcommand{\DAxk}{{\Der \gk\gA\xi}}%
\newcommand{\DkXxk}{{\Der \gk\kuX\xi}}%
\newcommand \DAbul {\rD_{\!\Abul}}

\newcommand \Syl {\mathrm{Syl}}
\newcommand \Syuf {\Syl_{(\uf)}}
\newcommand \Syfs {\Syl_{(\lfs)}}
\newcommand \Sylv{\mathrm{Sylv}} 
\newcommand \Esylv{\mathrm{Esylv}} 
\newcommand \Sylh[2]{\Sylv^{{\rm ho},#1}_{#2}} 
\newcommand \Esylh[2]{\Esylv^{{\rm ho},#1}_{#2}} 

\newcommand \rgl {\rg^\lambda}
\newcommand \rgg {\rg^\gamma}

\newcommand \HS {\MA{\mathit{HS}}}
\newcommand \HSA {\HS_{\!\gA}}


\newcommand \sfP {\mathsf{P}}
\newcommand \sfC {\mathsf{C}}

\renewcommand \Div {\MA{\mathsf{Div}}}
\newcommand \AnCo {\MA{\mathsf{AnCo}}}
\newcommand \Ens {\MA{\mathsf{Ens}}}
\newcommand \AMod {\MA{\mathsf{AMod}}}

\newcommand \GK {\MA{\mathsf{GK}}}
\newcommand \GKO {\MA{\mathsf{GK}_0}}
\newcommand \HO {\MA{\mathsf{H}_0}}
\newcommand \HOp {\MA{\mathsf{H}_0^+}}
\newcommand \HeA {{\Heit\gA}}
\newcommand \Heit {\MA{\mathsf{Heit}}}
\newcommand \Hspec {\MA{\mathsf{Hspec}}}
\newcommand \Jspec {\MA{\mathsf{Jspec}}}
\newcommand \jspec {\MA{\mathsf{jspec}}}
\newcommand \KO {\MA{\mathsf{K}_0}}
\newcommand \KOp {\MA{\mathsf{K}_0^+}}
\newcommand \KTO {\MA{\wi{\mathsf{K}}_0}}
\newcommand \Max {\MA{\mathsf{Max}}}
\newcommand \Min {\MA{\mathsf{Min}}}
\newcommand \OQC {\MA{\mathsf{Oqc}}}
\newcommand \Pic {\MA{\mathsf{Pic}}}
\newcommand \Spec {\MA{\mathsf{Spec}}}
\newcommand \SpecA {\Spec\gA}
\newcommand \SpecT {\Spec\gT}
\newcommand \Zar {\MA{\mathsf{Zar}}}
\newcommand \Proj {\MA{\mathsf{Proj}}}
\newcommand \Vpr {\MA{\mathsf{Vpr}}}
\newcommand \ZF {\MA{\mathsf{ZF}}}
\newcommand \ZarA {{\Zar\gA}}
\newcommand \STR {\MA{\rm SR}} 
\newcommand \Str {\MA{\mathsf{Sr}}}
\newcommand \Adim {\MA{\mathsf{Adim}}}
\newcommand \Bdim {\MA{\mathsf{Bdim}}}
\newcommand \Cdim {\MA{\mathsf{Cdim}}}
\newcommand \Edim {\MA{\mathsf{Edim}}}
\newcommand \Fdim {\MA{\mathsf{Fdim}}}
\newcommand \Fd {\MA{\mathsf{Fd}}}
\newcommand \Gdim {\MA{\mathsf{Gdim}}}
\newcommand \geodim {\MA{\mathsf{geodim}}}
\newcommand \Hdim {\MA{\mathsf{Hdim}}}
\newcommand \Jdim {\MA{\mathsf{Jdim}}}
\newcommand \jdim {\MA{\mathsf{jdim}}}
\newcommand \Kdim {\MA{\mathsf{Kdim}}}
\newcommand \Divdim {\MA{\mathsf{Divdim}}}
\newcommand \Ld {\MA{\mathsf{Ld}}}
\newcommand \pdim {\MA{\mathsf{pdim}}}
\newcommand \pd {\MA{\mathsf{pd}}}
\newcommand \Pdim {\MA{\mathsf{Pdim}}}
\newcommand \Pd {\MA{\mathsf{Pd}}}
\newcommand \Sdim {\MA{\mathsf{Sdim}}}
\newcommand \Vdim {\MA{\mathsf{Vdim}}}

\newcommand \Prf {\mathsf{Pr}}
\newcommand \Gr {\mathsf{Gr}}

\newcommand{\zar}{\hbox{{\large$\mathpzc{Zar}$}}}


\newcommand \fRes {\MA{\mathfrak{Res}}}

\newcommand \SPEC {\MA{\mathfrak{Spec}}}
\newcommand \SPECK {\SPEC_\gK}
\newcommand \SPECk {\SPEC_\gk}

\newcommand \Ap {{\gA_\fp}}
\newcommand \zg {\ZZ[G]}

\newcommand \eoe {\hbox{}\nobreak\hfill
\vrule width 1.4mm height 1.4mm depth 0mm \par \smallskip}

\newcommand \eoq{\hbox{}\nobreak
\vrule width 1.4mm height 1.4mm depth 0mm}

\newcommand \cm{cm}

\makeatletter
\DeclareRobustCommand\\{%
  \let \reserved@e \relax
  \let \reserved@f \relax
  \@ifstar{\let \reserved@e \vadjust \let \reserved@f \nobreak
             \@xnewline}%
          \@xnewline}
\makeatother

\newcommand{\blocs}[8]{%
{\setlength{\unitlength}{.0833\textwidth}
\tabcolsep0pt\renewcommand{\arraystretch}{0}%
\begin{tabular}{|c|c|}
\hline
\parbox[t][#3\cm][c]{#1\cm}{\begin{minipage}[c]{#1\cm}
\centering#5
\end{minipage}}&
\parbox[t][#3\cm][c]{#2\cm}{\begin{minipage}[c]{#2\cm}
\centering#6
\end{minipage}}\\
\hline
\parbox[t][#4\cm][c]{#1\cm}{\begin{minipage}[c]{#1\cm}
\centering#7
\end{minipage}}&
\parbox[t][#4\cm][c]{#2\cm}{\begin{minipage}[c]{#2\cm}
\centering#8
\end{minipage}}\\
\hline
\end{tabular}
}}


\newcommand \tri[7]{
$$\quad\quad\quad\quad
\vcenter{\xymatrix@C=1.5cm
{
#1 \ar[d]_{#2} \ar[dr]^{#3} \\
{#4} \ar[r]_{{#5}}   & {#6} \\
}}
\quad\quad \vcenter{\hbox{\small {#7}}\hbox{~\\[1mm] ~ }}
$$
}


\newcommand \carre[8]{
$$
\xymatrix @C=1.2cm{
#1\,\ar[d]^{#4}\ar[r]^{#2}   & \,#3\ar[d]^{#5}   \\
#6\,\ar[r]    ^{#7}    & \,#8  \\
}
$$
}

\newcommand\dcan[8]{
\xymatrix @C=1.2cm{
#1\,\ar[d]_{#4}\ar[r]^{#2}   & \,#3   \\
#6\,\ar[r]_{#7}    & #8\ar[u]_{#5}  
}
}

\newcommand \pun[7]{
$$\quad\quad\quad\quad
\vcenter{\xymatrix@C=1.5cm
{
#1 \ar[d]_{#2} \ar[dr]^{#3} \\
{#4} \ar@{-->}[r]_{{#5}\,!}   & {#6} \\
}}
\quad\quad \vcenter{\hbox{\small {#7}}\hbox{~\\[1mm] ~ }}
$$
}

\newcommand \puN[8]{
$$\hspace{#8}
\vcenter{\xymatrix@C=1.5cm
{
#1 \ar[d]_{#2} \ar[dr]^{#3} \\
{#4} \ar@{-->}[r]_{{#5}\,!}   & {#6} \\
}}
\quad\quad \vcenter{\hbox{\small {#7}}\hbox{~\\[1mm] ~ }}
$$
}

\newcommand \Pun[8]{
$$\quad\quad\quad\quad
\vcenter{\xymatrix@C=1.5cm
{
#1 \ar[d]_{#2} \ar[dr]^{#3} \\
{#4} \ar@{-->}[r]_{{#5}\,!}   & {#6} \\
}}
\quad\quad
\vcenter{\hbox{\small {#7}}
\hbox{~\\[3.5mm] ~ }
\hbox{\small {#8}}
\hbox{~\\[-3.5mm] ~ }}
$$
}


\newcommand \PUN[9]{
$$\quad\quad
\vcenter{\xymatrix@C=1.5cm
{
#1 \ar[d]_{#2} \ar[dr]^{#3} \\
{#4} \ar@{-->}[r]_{{#5}\,!}   & {#6} \\
}}
\quad\quad
\vcenter{
\hbox{\small {#7}}
\hbox{~\\[-3mm] ~}
\hbox{\small {#8}}
\hbox{~\\[-3mm] ~}
\hbox{\small {#9}}
\hbox{~\\[-3.5mm] ~ }}
$$
}

\newcommand \Pnv[9]{
$$\quad\quad\quad\quad
\vcenter{\xymatrix@C=1.5cm
{
#1 \ar[d]_{#2} \ar[dr]^{#3} \\
{#4} \ar@{-->}[r]_{{#5}\,!}   & {#6} \\
}}
\quad\quad
\vcenter{
\hbox{\small {#7}}
\hbox{~\\[1mm] ~}
\hbox{\small {#8}}
\hbox{~\\[-1mm] ~}
\hbox{\small {#9}}
\hbox{~\\[0mm] ~ }}
$$
}

\newcommand \pnv[9]{
$$
\vcenter{\xymatrix@C=1.5cm
{
#1 \ar[d]_{#2} \ar[dr]^{#3} \\
{#4} \ar@{-->}[r]_{{#5}\,!}   & {#6} \\
}}
\quad\quad
\vcenter{
\hbox{\small {#7}}
\hbox{~\\[1mm] ~}
\hbox{\small {#8}}
\hbox{~\\[-1mm] ~}
\hbox{\small {#9}}
\hbox{~\\[0mm] ~ }}
$$
}

\newcommand \PNV[9]{
$$\quad\quad\quad
\vcenter{\xymatrix@C=1.5cm
{
#1 \ar[d]_{#2} \ar[dr]^{#3} \\
{#4} \ar@{-->}[r]_{{#5}\,!}   & {#6} \\
}}
\quad\quad
\vcenter{
\vspace{4mm}
\hbox{\small {#7}}
\hbox{~\\[-1.7mm] ~}
\hbox{\small {#8}}
\hbox{~\\[-1.7mm] ~}
\hbox{\small {#9}}
\hbox{~\\[2mm] ~ }
}
$$
}


\newdimen\xyrowsp
\xyrowsp=3pt
\newcommand{\SCO}[6]{
\xymatrix @R = \xyrowsp {
                                  &1 \ar@{-}[dl] \ar@{-}[dr] \\
#3 \ar@{-}[ddr]                   &   & #6 \ar@{-}[ddl] \\
                                  &\bullet\ar@{-}[d] \\
                                  &\bullet   \\
#2 \ar@{-}[ddr] \ar@{-}[uur]      &   & #5 \ar@{-}[ddl] \ar@{-}[uul] \\
                                  &\bullet \ar@{-}[d] \\
                                  &\bullet  \\
#1 \ar@{-}[uur]                   &   & #4 \ar@{-}[uul] \\
                                  & 0 \ar@{-}[ul] \ar@{-}[ur] \\
}
}


\makeatletter
\newif\if@borderstar
\def\bordercmatrix{\@ifnextchar*{%
  \@borderstartrue\@bordercmatrix@i}{\@borderstarfalse\@bordercmatrix@i*}%
}
\def\@bordercmatrix@i*{\@ifnextchar[{%
  \@bordercmatrix@ii}{\@bordercmatrix@ii[()]}
}
\def\@bordercmatrix@ii[#1]#2{%
  \begingroup
    \m@th\@tempdima.875em\setbox\z@\vbox{%
      \def\cr{\crcr\noalign{\kern 2\p@\global\let\cr\endline}}%
      \ialign {$##$\hfil\kern.2em\kern\@tempdima&\thinspace%
      \hfil$##$\hfil&&\quad\hfil$##$\hfil\crcr\omit\strut%
      \hfil\crcr\noalign{\kern-\baselineskip}#2\crcr\omit%
      \strut\cr}}%
    \setbox\tw@\vbox{\unvcopy\z@\global\setbox\@ne\lastbox}%
    \setbox\tw@\hbox{\unhbox\@ne\unskip\global\setbox\@ne\lastbox}%
    \setbox\tw@\hbox{%
      $\kern\wd\@ne\kern-\@tempdima\left\@firstoftwo#1%
        \if@borderstar\kern.2em\else\kern -\wd\@ne\fi%
      \global\setbox\@ne\vbox{\box\@ne\if@borderstar\else\kern.2em\fi}%
      \vcenter{\if@borderstar\else\kern-\ht\@ne\fi%
        \unvbox\z@\kern-\if@borderstar2\fi\baselineskip}%
\if@borderstar\kern-2\@tempdima\kern.4em\else\,\fi\right\@secondoftwo#1 $%
    }\null\;\vbox{\kern\ht\@ne\box\tw@}%
  \endgroup
}
\makeatother



\renewcommand\paragraph[1]{

\rdb\addcontentsline{toc}{subsubsection}{#1} \medskip \noindent $\bullet$ \textbf{#1}}

\newcommand{\vou}{\MA{\tsbf{ ou }}}
\newcommand{\Vou}{\MA{\tsbf{OU}}}
\newcommand \EXists[1] {\tsbf{Introduire }{#1}\tsbf{ tel que }\,}
\newcommand \vet {\tsbf{,}\;}


\newcounter{MF}
\newcommand\stMF{\stepcounter{MF}}

\newcommand{\lec}{\stMF\ifodd\value{MF}lecteur\xspace \else 
lectrice\xspace \fi}

\newcommand{\lecs}{\stMF\ifodd\value{MF}lecteurs\xspace \else 
lectrices\xspace \fi}

\newcommand{\alec}{\stMF\ifodd\value{MF}au lecteur\xspace \else%
à la lectrice\xspace \fi}

\newcommand{\dlec}{\stMF\ifodd\value{MF}du lecteur\xspace \else%
de la lectrice\xspace \fi}

\newcommand{\llec}{\stMF\ifodd\value{MF}le lecteur\xspace \else la lectrice\xspace \fi}

\newcommand{\Llec}{\stMF\ifodd\value{MF}Le lecteur\xspace \else La lectrice\xspace \fi}

\newcommand{\lui}{\ifodd\value{MF}lui\xspace \else
elle\xspace \fi}

\newcommand{\celui}{\ifodd\value{MF}celui\xspace \else
celle\xspace \fi}

\newcommand{\ceux}{\ifodd\value{MF}ceux\xspace \else
celles\xspace \fi}

\newcommand{\er}{\ifodd\value{MF}er\xspace \else
ère\xspace \fi}

\newcommand{\eux}{\ifodd\value{MF}eux\xspace \else
elles\xspace \fi}

\newcommand{\eUx}{\ifodd\value{MF}eux\xspace \else
euse\xspace \fi}

\newcommand{\leux}{\ifodd\value{MF}leux \else
leuse \fi}

\newcommand{\il}{\ifodd\value{MF}il\xspace \else
elle\xspace \fi}

\newcommand{\ien}{\ifodd\value{MF}ien\xspace \else
ienne\xspace \fi}

\newcommand{\e}{\ifodd\value{MF}\xspace \else e\xspace \fi}

\newcommand{\n}{\ifodd\value{MF}n \else nne \fi}
\newcommand{\nz}{\ifodd\value{MF}n\else nne\fi}

\makeatletter
\newcommand{\la}{\@ifstar{\ifodd\value{MF}le\xspace\else
la\xspace\fi}{\stMF\ifodd\value{MF}le\xspace \else la\xspace \fi}}
\makeatother

\newcommand \rem{\rdb
\noi{\it Remarque. }}

\newcommand \REM[1]{\rdb
\noi{\it Remarque#1. }}

\newcommand \rems{\rdb
\noi{\it Remarques. }}

\newcommand \exl{\rdb
\noi{\bf Exemple. }}

\newcommand \EXL[1]{\rdb
\noi{\bf Exemple: #1. }}

\newcommand \exls{\rdb
\noi{\bf Exemples. }}

\newcommand \thref[1] {théorème~\ref{#1}}
\newcommand \thrf[1] {\thref{#1} \paref{#1}}
\newcommand \thrfs[2] {théorèmes~\ref{#1} et~\ref{#2}}
\renewcommand \rref[1] {\ref{#1} \paref{#1}}
\newcommand \paref[1] {page~\pageref{#1}}
\newcommand \plgrf[1] {\plg~\ref{#1} \paref{#1}}
\newcommand \plgref[1] {\plg~\ref{#1}}
\newcommand \plgrfs[2] {\plgs~\ref{#1} et~\ref{#2}}
\newcommand \pstfref[1] {Positivstel\-lensatz formel~\ref{#1}}
\newcommand \pstref[1] {Positivstel\-lensatz~\ref{#1}}
\newcommand \eqrf[1] {équation~\pref{#1}}
\newcommand \eqvrf[1] {équation~\pref{#1} \paref{#1}}
\newcommand \prirf[1] {principe~\ref{#1}}

\newcommand\oge{\leavevmode\raise.3ex\hbox{$\scriptscriptstyle\langle\!\langle\,$}}
\newcommand\feg{\leavevmode\raise.3ex\hbox{$\scriptscriptstyle\,\rangle\!\rangle$}}

\newcommand\gui[1]{\oge{#1}\feg}

\newcommand \facile{\begin{proof}
La démonstration est laissée \alec.
\end{proof}
}

\newenvironment{proof}{\ifhmode\par\fi\vskip-\lastskip\vskip0.5ex\global\insidedemotrue
\everypar{}\noindent{\setbox0=\hbox{$\!\!$\DebP }\global\wdTitreEnvir\wd0\box0}
\ignorespaces}{\enddemobox\par\vskip.5em}
\def\enddemobox{\ifinsidedemo
\ifmmode\hbox{$\square$}\else
\ifhmode\unskip\else\noindent\fi\nobreak\null\nobreak\hfill
\nobreak$\square$\fi\fi\global\insidedemofalse}

\newenvironment{Proof}[1]{\ifhmode\par\fi\vskip-\lastskip\vskip0.5ex\global\insidedemotrue
\everypar{}\noindent{\setbox0=\hbox{\it #1}\global\wdTitreEnvir\wd0\box0}\ignorespaces}{\enddemobox\par\vskip.5em}
\def\enddemobox{\ifinsidedemo
\ifmmode\hbox{$\square$}\else
\ifhmode\unskip\else\noindent\fi\nobreak\null\nobreak\hfill
\nobreak$\square$\fi\fi\global\insidedemofalse}

\newcommand \ihi {\index{Hilbert} }
\newcommand \imlg {\index{machinerie locale-globale \elr}\xspace}
\newcommand \imlb {\index{machinerie locale-globale de base (à \ideps)}\xspace}
\newcommand \imla {\index{machinerie locale-globale des \anarsz}\xspace}
\newcommand \imlma {\index{machinerie locale-globale à \idemasz}\xspace}
\newcommand \iplg {\index{principe local-global de base}\xspace}

\newcommand \num {{n$^{\mathrm{ o}}$}}

\newcommand\comm{\rdb
\noi{\it Commentaire. }}

\newcommand\COM[1]{\rdb
\noi{\it Commentaire #1. }}

\newcommand\comms{\rdb
\noi{\it Commentaires. }}

\newcommand\Pb{\rdb
\noi{\bf Problème. }}

\newcommand \Cad {C'est-à-dire\xspace}
\newcommand \recu {récur\-rence\xspace}
\newcommand \hdr {hypo\-thèse de \recu}
\newcommand \cad {c'est-à-dire\xspace}
\newcommand \cade {c'est-à-dire en\-co\-re\xspace}
\newcommand \ssi {si, et seu\-lement si, }
\newcommand \ssiz {si, et seu\-lement si,~}
\newcommand \cnes {con\-di\-tion néces\-sai\-re et suf\-fi\-san\-te\xspace}
\newcommand \spdg {sans per\-te de géné\-ra\-lité\xspace}
\newcommand \Spdg {Sans per\-te de géné\-ra\-lité\xspace}

\newcommand \Propeq {Les proprié\-tés suivan\-tes sont 
équiva\-lentes.}
\newcommand \propeq {les proprié\-tés suivan\-tes sont 
équiva\-lentes.}

\newcommand \THo[2]{\rdb
\mni{\bf Théorème {#1}~} {\it #2

}}

\newcommand \PLCC[2]{\rdb
\mni{\bf Principe \lgb concret \ref{#1} bis~} {\it #2

}}

\newcommand \THO[2]{\rdb
\mni{\bf Théor\`eme \ref{#1} bis~} {\it #2

}}

\newcommand \defa[2]{\mni\rdb\textbf{Définition alternative \ref{#1}
\it #2 
}}

\newcommand \KRA {\index{Kronecker!astuce de ---}Kronecker\xspace}
\newcommand \KRO {\index{Kronecker!\tho de ---  (1)}Kronecker\xspace}
\newcommand \KRN {\index{Kronecker!\tho de ---  (2)}Kronecker\xspace}

\newcommand \Kev {$\gK$-\evc}
\newcommand \Kevs {$\gK$-\evcs}

\newcommand \Lev {$\gL$-\evc}
\newcommand \Levs {$\gL$-\evcs}

\newcommand \Qev {$\QQ$-\evc}
\newcommand \Qevs {$\QQ$-\evcs}

\newcommand \kev {$\gk$-\evc}
\newcommand \kevs {$\gk$-\evcs}

\newcommand \lev {$\gl$-\evc}
\newcommand \levs {$\gl$-\evcs}

\newcommand \Alg {$\gA$-\alg}
\newcommand \Algs {$\gA$-\algs}

\newcommand \Blg {$\gB$-\alg}
\newcommand \Blgs {$\gB$-\algs}

\newcommand \Clg {$\gC$-\alg}
\newcommand \Clgs {$\gC$-\algs}

\newcommand \klg {$\gk$-\alg}
\newcommand \klgs {$\gk$-\algs}

\newcommand \llg {$\gl$-\alg}
\newcommand \llgs {$\gl$-\algs}

\newcommand \Klg {$\gK$-\alg}
\newcommand \Klgs {$\gK$-\algs}

\newcommand \Llg {$\gL$-\alg}
\newcommand \Llgs {$\gL$-\algs}

\newcommand \QQlg {$\QQ$-\alg}
\newcommand \QQlgs {$\QQ$-\algs}

\newcommand \Rlg {$\gR$-\alg}
\newcommand \Rlgs {$\gR$-\algs}

\newcommand \RRlg {$\RR$-\alg}
\newcommand \RRlgs {$\RR$-\algs}

\newcommand \ZZlg {$\ZZ$-\alg}
\newcommand \ZZlgs {$\ZZ$-\algs}

\newcommand \Amo {$\gA$-mo\-du\-le\xspace}
\newcommand \Amos {$\gA$-mo\-du\-les\xspace}

\newcommand \Bmo {$\gB$-mo\-du\-le\xspace}
\newcommand \Bmos {$\gB$-mo\-du\-les\xspace}

\newcommand \Cmo {$\gC$-mo\-du\-le\xspace}
\newcommand \Cmos {$\gC$-mo\-du\-les\xspace}

\newcommand \kmo {$\gk$-mo\-du\-le\xspace}
\newcommand \kmos {$\gk$-mo\-du\-les\xspace}

\newcommand \Kmo {$\gK$-mo\-du\-le\xspace}
\newcommand \Kmos {$\gK$-mo\-du\-les\xspace}

\newcommand \Lmo {$\gL$-mo\-du\-le\xspace}
\newcommand \Lmos {$\gL$-mo\-du\-les\xspace}

\newcommand \Ali {appli\-ca\-tion $\gA$-\lin}
\newcommand \Alis {appli\-ca\-tions $\gA$-\lins}

\newcommand \Kli {appli\-ca\-tion $\gK$-\lin}
\newcommand \Klis {appli\-ca\-tions $\gK$-\lins}

\newcommand \Bli {appli\-ca\-tion $\gB$-\lin}
\newcommand \Blis {appli\-ca\-tions $\gB$-\lins}

\newcommand \Cli {appli\-ca\-tion $\gC$-\lin}
\newcommand \Clis {appli\-ca\-tions $\gC$-\lins}

\newcommand \ABH {Auslander-Buchsbaum-Hochster\xspace}

\newcommand \ac{algé\-bri\-quement clos\xspace}  

\newcommand \acl {an\-neau \icl}
\newcommand \acls {an\-neaux \icl}

\newcommand \adp {an\-neau de Pr\"u\-fer\xspace}
\newcommand \adps {an\-neaux de Pr\"u\-fer\xspace}

\newcommand \adpc {\adp \coh}
\newcommand \adpcs {\adps \cohs}

\newcommand \adu {\alg de décom\-po\-sition univer\-selle\xspace}
\newcommand \adus {\algs de décom\-po\-sition univer\-selle\xspace}

\newcommand \adv {anneau de valuation\xspace}
\newcommand \advs {anneaux de valuation\xspace}

\newcommand \advl {anneau \dvla} 
\newcommand \advls {anneaux \dvlas} 

\newcommand \Afr {Anneau \frl}
\newcommand \Afrs {Anneaux \frls}
\newcommand \afr {anneau \frl}
\newcommand \aFr {\hyperref[theorieAfr]{anneau \frl}\xspace}
\newcommand \afrs {anneaux \frls}

\newcommand \afrr {\afr réduit\xspace}
\newcommand \afrrs {\afrs réduits\xspace}
\newcommand \Afrrs {\Afrs réduits\xspace}

\newcommand \afrvr {\afr avec \ravs}
\newcommand \aFrvr {\hyperref[theorieAfrrv]{\afrvr}\xspace}
\newcommand \afrvrs {\afrs avec \ravs}

\newcommand \aftr {anneau réticulé \ftm réel\xspace}
\newcommand \aftrs {anneaux réticulés \ftm réels\xspace}

\newcommand \aG {\alg galoisienne\xspace}
\newcommand \aGs {\algs galoisiennes\xspace}

\newcommand \agB {\alg de Boole\xspace}
\newcommand \agBs {\algs de Boole\xspace}

\newcommand \agH {\alg de Heyting\xspace}
\newcommand \agHs {\algs de Heyting\xspace}

\newcommand \agq{algé\-bri\-que\xspace}
\newcommand \agqs{algé\-bri\-ques\xspace}

\newcommand \agqt{algé\-bri\-que\-ment\xspace}

\newcommand \aKr {anneau de Krull\xspace}
\newcommand \aKrs {anneaux de Krull\xspace}

\newcommand \alg {algè\-bre\xspace}
\newcommand \algs {algè\-bres\xspace}

\newcommand \algo{algo\-rithme\xspace}
\newcommand \algos{algo\-rithmes\xspace}

\newcommand \algq{al\-go\-rith\-mi\-que\xspace}
\newcommand \algqs{al\-go\-rith\-mi\-ques\xspace}

\newcommand \ali {appli\-ca\-tion \lin}
\newcommand \alis {appli\-ca\-tions \lins}

\newcommand \alo {an\-neau lo\-cal\xspace}
\newcommand \alos {an\-neaux lo\-caux\xspace}

\newcommand \algb {an\-neau \lgb}
\newcommand \algbs {an\-neaux \lgbs}

\newcommand \alrd {\alo \dcd}
\newcommand \alrds {\alos \dcds}

\newcommand \anar {anneau \ari}
\newcommand \anars {anneaux \aris}

\newcommand \anor {an\-neau nor\-mal\xspace}
\newcommand \anors {an\-neaux nor\-maux\xspace}

\newcommand \apf {\alg \pf}
\newcommand \apfs {\algs \pf}

\newcommand \apG {\alg pré\-galoisienne\xspace}
\newcommand \apGs {\algs pré\-galoisiennes\xspace}

\newcommand \arc {anneau réel clos\xspace}
\newcommand \aRc {\hyperref[theorieArc]{\arc}\xspace}
\newcommand \arcs {anneaux réels clos\xspace}

\newcommand \ari{arith\-mé\-tique\xspace}  
\newcommand \aris{arith\-mé\-tiques\xspace}  

\newcommand \Asr {Anneau \str}
\newcommand \Asrs {Anneaux \strs}
\newcommand \asr {anneau \str}
\newcommand \asrs {anneaux \strs}

\newcommand \asrvr {\asr avec \ravs}
\newcommand \asrvrs {\asrs avec \ravs}

\newcommand \atf {\alg \tf}
\newcommand \atfs {\algs \tf}

\newcommand \auto {auto\-mor\-phisme\xspace}
\newcommand \autos {auto\-mor\-phismes\xspace}

\newcommand \azd {anneau \zed}
\newcommand \azds {anneaux \zeds}

\newcommand \azrd {anneau \zedr}
\newcommand \azrds {anneaux \zedrs}


\newcommand \bdg {base de Gr\"obner\xspace}
\newcommand \bdgs {bases de Gr\"obner\xspace}

\newcommand \bdp {base de \dcn partielle\xspace}
\newcommand \bdps {bases de \dcn partielle\xspace}

\newcommand \bdf {base de \fap\xspace}

\newcommand \Bif {Borne infé\-rieure\xspace} %
\newcommand \bif {borne infé\-rieure\xspace} %
\newcommand \bifs {bornes infé\-rieures\xspace} %

\newcommand \bsp {borne supé\-rieure\xspace} %
\newcommand \bsps {borne supé\-rieures\xspace} %


\newcommand \cac{corps \ac}  

\newcommand \calf{calcul formel\xspace}  

\newcommand \cara{carac\-té\-ris\-tique\xspace}  
\newcommand \caras{carac\-té\-ris\-tiques\xspace}  

\newcommand \care{carac\-té\-risé\xspace}
\newcommand \caree{carac\-té\-risée\xspace}
\newcommand \cares{carac\-té\-risés\xspace}
\newcommand \carees{carac\-té\-risées\xspace}

\newcommand \carn{carac\-té\-ri\-sation\xspace}  
\newcommand \carns{carac\-té\-ri\-sations\xspace}  
\newcommand \Carn{Carac\-té\-ri\-sation\xspace}  
\newcommand \Carns{Carac\-té\-ri\-sations\xspace}  

\newcommand \carar{carac\-té\-riser\xspace}

\newcommand \carf{de carac\-tère fini\xspace}  

\newcommand \cat{caté\-gorie\xspace}
\newcommand \cats{caté\-gories\xspace}
\newcommand \catb{\cat abé\-lienne\xspace}
\newcommand \catbs{\cats abé\-liennes\xspace}
\newcommand \catd{\cat addi\-tive\xspace}
\newcommand \catds{\cats addi\-tives\xspace}

\newcommand \cdi{corps discret\xspace}
\newcommand \cdis{corps discrets\xspace}
  
\newcommand \cdv{changement de variables\xspace}  
\newcommand \cdvs{changements de variables\xspace}  

\newcommand \Cech {\v Cech\xspace}

\newcommand \cEP {\cara d'Euler-Poincaré\xspace}

\newcommand \cEse {complè\-tement~\hbox{$E$-sécante}\xspace}
\newcommand \cEses {complè\-tement~\hbox{$E$-sécantes}\xspace}

\newcommand \cli {clô\-ture inté\-grale\xspace}

\newcommand \codi {corps ordonné discret\xspace}
\newcommand \codis {corps ordonnés discrets\xspace}

\newcommand \coe {coef\-fi\-cient\xspace}
\newcommand \coes {coef\-fi\-cients\xspace}

\newcommand \coh {co\-hé\-rent\xspace}
\newcommand \cohs {co\-hé\-rents\xspace}

\newcommand \cohc {cohé\-rence\xspace}

\newcommand \coli {combi\-naison \lin}
\newcommand \colis {combi\-naisons \lins}

\newcommand \com {coma\-ximaux\xspace}
\newcommand \come {coma\-xi\-males\xspace}

\newcommand \coo {coor\-donnée\xspace}
\newcommand \coos {coor\-données\xspace}

\newcommand \cop {complé\-men\-taire\xspace}
\newcommand \cops {complé\-men\-taires\xspace}

\newcommand \cor {coré\-guliers\xspace}
\newcommand \core {coré\-gulières\xspace}

\newcommand \corg{\coh régulier\xspace}
\newcommand \corgs{\cohs réguliers\xspace}

\newcommand \cosc {complè\-tement sécant\xspace}
\newcommand \csce {complè\-tement sécante\xspace}
\newcommand \cscs {complè\-tement sécants\xspace}
\newcommand \csces {complè\-tement sécantes\xspace}

\newcommand \cosv {conser\-vative\xspace}
\newcommand \cosvs {conser\-vatives\xspace}

\newcommand \cOsv {\hyperref[defithconserv]{conser\-vative}\xspace}
\newcommand \cOsvs {\hyperref[defithconserv]{conser\-vatives}\xspace}

\newcommand \covr {corps ordonné avec \ravs}
\newcommand \covrs {corps ordonnés avec \ravs}

\newcommand \cpb {compa\-tible\xspace} 
\newcommand \cpbs {compa\-tibles\xspace} 

\newcommand \cpbt {compa\-tibi\-lité\xspace} 
\newcommand \cpbtz {compa\-tibi\-lité} 

\newcommand \crc {corps réel clos\xspace}
\newcommand \crcs {corps réels clos\xspace}

\newcommand \crcd {corps réel clos discret\xspace}
\newcommand \crcds {corps réels clos discrets\xspace}


\newcommand \dcd {rési\-duel\-lement dis\-cret\xspace}
\newcommand \dcds {rési\-duel\-lement dis\-crets\xspace}

\newcommand \dcn {décom\-po\-sition\xspace}
\newcommand \dcns {décom\-po\-sitions\xspace}

\newcommand \dcnb {\dcn bornée\xspace}

\newcommand \dcnc {\dcn complète\xspace}

\newcommand \dcnp {\dcn partielle\xspace}

\newcommand \dcp {décom\-posa\-ble\xspace}
\newcommand \dcps {décom\-posa\-bles\xspace}

\newcommand \ddk {dimension de~Krull\xspace}
\newcommand \ddi {de dimension infé\-rieure ou égale à~}

\newcommand \ddp {domaine de Pr\"u\-fer\xspace}
\newcommand \ddps {domaines de Pr\"u\-fer\xspace}

\newcommand \Demo{Démon\-stra\-tion\xspace}     

\newcommand \dem{démon\-stra\-tion\xspace}     
\newcommand \dems{démons\-tra\-tions\xspace}

\newcommand \denb{dénom\-brable\xspace}
\newcommand \denbs{dénom\-brables\xspace}

\newcommand \deno{déno\-mi\-nateur\xspace}     
\newcommand \denos{déno\-mi\-nateurs\xspace}   

\newcommand \deter {déter\-mi\-nant\xspace}  
\newcommand \deters {déter\-mi\-nants\xspace}  
 
\newcommand{\dfd}{de \prof~$\geq 2$\xspace}

\newcommand \Dfn{Défi\-nition\xspace}  
\newcommand \Dfns{Défi\-nitions\xspace}  
\newcommand \dfn{défi\-nition\xspace}  
\newcommand \dfns{défi\-nitions\xspace}  

\newcommand \dftr {droite réticulée \ftm réelle\xspace}
\newcommand \dftrs {droites réticulées \ftm réelles\xspace}
  
\newcommand \dil{diffé\-rentiel\xspace}  
\newcommand \dils{diffé\-rentiels\xspace}  
\newcommand \dile{diffé\-ren\-tielle\xspace}  
\newcommand \diles{diffé\-ren\-tielles\xspace}  

\newcommand \dip{diviseur principal\xspace}
\newcommand \dips{diviseurs principaux\xspace}

\newcommand \discri{discri\-minant\xspace}  
\newcommand \discris{discri\-minants\xspace}  

\newcommand \divle {dimension divisorielle\xspace} 

\newcommand \iDKM {\index{Dedekind-Mertens}}
\newcommand \DKM {\iDKM Dedekind-Mertens\xspace}

\newcommand \dit{distri\-bu\-ti\-vité\xspace}

\newcommand \dlg{d'élar\-gis\-sement\xspace}  

\newcommand \dok {domaine de Dedekind\xspace}
\newcommand \doks {domaines de Dedekind\xspace}

\newcommand \dve {divi\-si\-bi\-lité\xspace}

\newcommand \dvee {à \dve explicite\xspace}

\newcommand \dvla {à diviseurs\xspace}
\newcommand \dvlas {à diviseurs\xspace}

\newcommand \dvld {\dvlt décom\-posé\xspace} %
\newcommand \dvlds {\dvlt décom\-posés\xspace} %

\newcommand \dvlg {diviso\-riel\xspace} 
\newcommand \dvlgs {diviso\-riels\xspace} 

\newcommand \dvli {\dvlt inver\-sible\xspace} 
\newcommand \dvlis {\dvlt inver\-sibles\xspace} 

\newcommand \dvlt {diviso\-riel\-lement\xspace} %

\newcommand \dvn {déri\-vation\xspace}
\newcommand \dvns {déri\-vations\xspace}

\newcommand \dvr {diviseur\xspace}
\newcommand \dvrs {diviseurs\xspace}

\newcommand \dvz {di\-viseur de zéro\xspace}
\newcommand \dvzs {di\-viseurs de zéro\xspace}


\newcommand \eco {\elts \com}
\newcommand \ecr {\elts \cor}

\newcommand \Eds {Exten\-sion des sca\-laires\xspace}
\newcommand \edss {exten\-sions des sca\-laires\xspace}
\newcommand \eds {exten\-sion des sca\-laires\xspace}

\newcommand \egmt {éga\-lement\xspace}

\newcommand \egt {éga\-li\-té\xspace}
\newcommand \egts {éga\-li\-tés\xspace}

\newcommand \eli{élimi\-nation\xspace}  

\newcommand \elr{élé\-men\-taire\xspace}  
\newcommand \elrs{élé\-men\-taires\xspace}  

\newcommand \elrt{élé\-men\-tai\-rement\xspace}  

\newcommand \elt{élé\-ment\xspace}  
\newcommand \elts{élé\-ments\xspace}  

\def \endo {endo\-mor\-phisme\xspace}
\def \endos {endo\-mor\-phismes\xspace}

\newcommand \entrel {rela\-tion impli\-ca\-tive\xspace}
\newcommand \entrels {rela\-tions impli\-ca\-tives\xspace}

\newcommand \eqn{équa\-tion\xspace}
\newcommand \eqns{équa\-tions\xspace}

\newcommand \eqv {équi\-valent\xspace}  
\newcommand \eqve {équi\-va\-lente\xspace}  
\newcommand \eqvs {équi\-valents\xspace}  
\newcommand \eqves {équi\-val\-entes\xspace}  

\newcommand \eqvc {équi\-va\-lence\xspace}  
\newcommand \eqvcs {équi\-va\-lences\xspace}  

\newcommand \Erg {$E$-\ndz}
\newcommand \Erge {$E$-\ndze}
\newcommand \Ergs {$E$-\ndzs}
\newcommand \Erges {$E$-\ndzes}

\newcommand \esid {essen\-tiel\-lement iden\-tique\xspace}  
\newcommand \esids {essen\-tiel\-lement iden\-tiques\xspace}  

\newcommand \Esid {\hyperref[defitdyesidentiques]{\esid}\xspace}  
\newcommand \Esids {\hyperref[defitdyesidentiques]{\esids}\xspace}  

\newcommand \eseq {essen\-tiel\-lement \eqve}  
\newcommand \eseqs {essen\-tiel\-lement \eqves}  

\newcommand \Eseq {\hyperref[defitheseq]{\eseq}\xspace}  
\newcommand \Eseqs {\hyperref[defitheseq]{\eseqs}\xspace}

\newcommand\evc{espace vec\-toriel\xspace} 
\newcommand\evcs{espaces vec\-toriels\xspace} 

\newcommand \evn{éva\-lua\-tion\xspace}
\newcommand \evns{éva\-lua\-tions\xspace}

\newcommand \fab {\fcn bornée\xspace}
\newcommand \fabs {\fcns bornées\xspace}

\newcommand \fat {\fcn totale\xspace}
\newcommand \fats {\fcn totales\xspace}

\newcommand \fap {\fcn partielle\xspace}
\newcommand \faps {\fcns partielles\xspace}

\newcommand \fcn {factorisation\xspace}
\newcommand \fcns {factorisations\xspace}

\newcommand \fcs {\ftm \cosc}
\newcommand \fcse {\ftm \csce}

\newcommand \fdi {for\-te\-ment dis\-cret\xspace}
\newcommand \fdis {for\-te\-ment dis\-crets\xspace}

\newcommand \fip {filtre pre\-mier\xspace}
\newcommand \fips {filtres pre\-miers\xspace}

\newcommand \fipma {\fip maxi\-mal\xspace}
\newcommand \fipmas {\fips maxi\-maux\xspace}

\newcommand \fit {fidè\-lement\xspace}

\newcommand \fpt {\fit plat\xspace}
\newcommand \fpte {\fit plate\xspace}
\newcommand \fpts {\fit plats\xspace}
\newcommand \fptes {\fit plates\xspace}

\newcommand \Frg {$F$-\ndz}
\newcommand \Frge {$F$-\ndze}
\newcommand \Frgs {$F$-\ndzs}
\newcommand \Frges {$F$-\ndzes}

\newcommand \fsa {fermé \sagq}
\newcommand \fsas {fermés \sagqs}

\newcommand \fsagc {fonction \sagc}
\newcommand \fsagcs {fonctions \sagcs}

\newcommand \fsagce {\fsagc entière\xspace}
\newcommand \fsagces {\fsagcs entières\xspace}

\newcommand \fmt {formel\-lement\xspace}

\newcommand \frl {for\-tement réticulé\xspace}
\newcommand \frle {for\-tement réticulée\xspace}
\newcommand \frls {for\-tement réticulés\xspace}

\newcommand \ftm {fortement\xspace}

\newcommand\gmt{géométrie\xspace}  
\newcommand\gmts{géométries\xspace}  

\newcommand\gaq{\gmt \agq}  

\newcommand\gmq{géomé\-trique\xspace}  
\newcommand\gmqs{géomé\-triques\xspace}  

\newcommand\gmqt{géomé\-tri\-quement\xspace}  

\newcommand\gne{géné\-ra\-lisé\xspace}  
\newcommand\gnee{géné\-ra\-lisée\xspace}  
\newcommand\gnes{géné\-ra\-lisés\xspace}  
\newcommand\gnees{géné\-ra\-lisées\xspace}  

\newcommand\gnl{géné\-ral\xspace}  
\newcommand\gnle{géné\-rale\xspace}  
\newcommand\gnls{géné\-raux\xspace}  
\newcommand\gnles{géné\-rales\xspace}  

\newcommand\gnlt{géné\-ra\-lement\xspace}  

\newcommand\gnn{géné\-ra\-li\-sa\-tion\xspace}  
\newcommand\gnns{géné\-ra\-li\-sa\-tions\xspace}  

\newcommand\gnq {géné\-rique\xspace}  
\newcommand\gnqs {géné\-riques\xspace}  

\newcommand \gnqt{géné\-ri\-quement\xspace}

\newcommand\gnr{géné\-ra\-liser\xspace}  

\newcommand \gns{géné\-ra\-lise\xspace}

\newcommand \gnt{géné\-ra\-lité\xspace}
\newcommand \gnts{géné\-ra\-lités\xspace}

\newcommand \grl{groupe \rtl}
\newcommand \grls{groupes \rtls}

\newcommand \Grg {$G$-\ndz}
\newcommand \Grge {$G$-\ndze}
\newcommand \Grgs {$G$-\ndzs}
\newcommand \Grges {$G$-\ndzes}

\newcommand \gRl {\hyperref[theorieGrl]{\grl}\xspace}
\newcommand \gRls {\hyperref[theorieGrl]{\grls}\xspace}

\newcommand\gtr{géné\-ra\-teur\xspace}  
\newcommand\gtrs{géné\-ra\-teurs\xspace}  


\newcommand \hml {homologie\xspace}
\newcommand \hmls {homologiem\xspace}

\newcommand \hmg {homo\-gène\xspace}
\newcommand \hmgs {homo\-gènes\xspace}

\newcommand \hmgn {homo\-gé\-né\-isation\xspace}

\newcommand \homo {ho\-mo\-mor\-phisme\xspace}
\newcommand \homos {ho\-mo\-mor\-phismes\xspace}

\newcommand \hmq {homologique\xspace}
\newcommand \hmqs {homologiques\xspace}

\newcommand \Hrg {$H$-\ndz}
\newcommand \Hrge {$H$-\ndze}
\newcommand \Hrgs {$H$-\ndzs}
\newcommand \Hrges {$H$-\ndzes}

\newcommand \icftr {intervalle compact réticulé \ftm réel\xspace}
\newcommand \icftrs {intervalles compacts réticulés \ftm réels\xspace}

\newcommand \icl {inté\-gra\-lement clos\xspace}
\newcommand \icle {inté\-gra\-lement close\xspace}

\newcommand \ico {inter\-section compl\`ete\xspace}
\newcommand \icos {inter\-sections compl\`etes\xspace}

\newcommand \icog {\ico globale\xspace}
\newcommand \icogs {\icos globales\xspace}

\newcommand \icol {\ico locale\xspace}
\newcommand \icols {\icos locales\xspace}

\newcommand \icsr {intervalle compact \stm réticulé\xspace}
\newcommand \icsrs {intervalles compacts \stm réticulés\xspace}

\newcommand \icrc {intervalle compact réel clos\xspace}
\newcommand \icrcs {intervalles compact réels clos\xspace}

\newcommand \id {idéal\xspace}
\newcommand \ids {idéaux\xspace}

\newcommand \ida {\idt \agq}
\newcommand \idas {\idts \agqs}

\newcommand \idc  {\idt de Cramer\xspace}
\newcommand \idcs {\idts de Cramer\xspace}

\newcommand \Idca  {Idéal \cara}
\newcommand \Idcas  {Idéaux \caras}
\newcommand \idca  {\id \cara}
\newcommand \idcas  {\ids \caras}

\newcommand \idd {idéal déter\-minan\-tiel\xspace}
\newcommand \idds {idéaux déter\-minan\-tiels\xspace}

\newcommand \ideli  {\id d'\eli}

\newcommand \idema {idéal maxi\-mal\xspace}
\newcommand \idemas {idéaux maxi\-maux\xspace}

\newcommand \idep {idéal pre\-mier\xspace}
\newcommand \ideps {idéaux pre\-miers\xspace}

\newcommand \idemi {\idep minimal\xspace}
\newcommand \idemis {\ideps minimaux\xspace}

\newcommand \idf {idéal de Fitting\xspace}
\newcommand \idfs {idéaux de Fitting\xspace}

\newcommand \idif {idéal \dvlg fini\xspace}
\newcommand \idifs {idéaux \dvlgs finis\xspace}

\newcommand \idli {idéal \dvli\xspace} 
\newcommand \idlis {idéaux \dvlis\xspace} 

\newcommand \idm {idem\-po\-tent\xspace}
\newcommand \idms {idem\-po\-tents\xspace}
\newcommand \idme {idem\-po\-tente\xspace}
\newcommand \idmes {idem\-po\-tentes\xspace}

\newcommand \idp {idéal prin\-ci\-pal\xspace}
\newcommand \idps {idé\-aux prin\-ci\-paux\xspace}

\newcommand \idt {iden\-ti\-té\xspace}
\newcommand \idts {iden\-ti\-tés\xspace}

\newcommand \idtr {in\-dé\-ter\-mi\-née\xspace}
\newcommand \idtrs {in\-dé\-ter\-mi\-nées\xspace}

\newcommand \ifr {idéal frac\-tion\-nai\-re\xspace}
\newcommand \ifrs {idéaux frac\-tion\-nai\-res\xspace}

\newcommand \imd {immé\-diat\xspace}
\newcommand \imde {immé\-diate\xspace}
\newcommand \imds {immé\-diats\xspace}
\newcommand \imdes {immé\-diates\xspace}

\newcommand \imdt {immé\-dia\-te\-ment\xspace}

\newcommand \indtr {inf-demi-treillis\xspace} 

\newcommand \inteq {intui\-ti\-vement \eqve}
\newcommand \inteqs {intui\-ti\-vement \eqves}

\newcommand \Inteq {\hyperref[defextintequiv]{\inteq}\xspace}
\newcommand \Inteqs {\hyperref[defextintequiv]{\inteqs}\xspace}

\newcommand \ing {in\-ver\-se \gne}
\newcommand \ings {in\-ver\-ses \gnes}

\newcommand \iMP {in\-ver\-se de Moo\-re-Pen\-ro\-se\xspace}
\newcommand \iMPs {in\-ver\-ses de Moo\-re-Pen\-ro\-se\xspace}

\newcommand \iMR {\id de MacRae\xspace}
\newcommand \iMRs {\ids de MacRae\xspace}

\newcommand \ipp {\idep poten\-tiel\xspace}
\newcommand \ipps {\ideps poten\-tiels\xspace}

\newcommand \ird {irré\-duc\-tible\xspace}
\newcommand \irds {irré\-duc\-tibles\xspace}

\newcommand \iso {iso\-mor\-phisme\xspace}
\newcommand \isos {iso\-mor\-phismes\xspace}

\newcommand \itf {idéal \tf}
\newcommand \itfs {idé\-aux \tf}

\newcommand \itid {intui\-ti\-vement iden\-tique\xspace}
\newcommand \itids {intui\-ti\-vement iden\-tiques\xspace}

\newcommand \iv {inversible\xspace}
\newcommand \ivs {inversibles\xspace}

\newcommand \ivdg {inverse divisoriel\xspace} 
\newcommand \ivdgs {inverses divisoriels\xspace} 

\newcommand \ivde {inverse divisorielle\xspace} 
\newcommand \ivdes {inverses divisorielles\xspace} 

\newcommand \ivda {inverse divisoriel\xspace} 
\newcommand \ivdas {inverses divisoriels\xspace} 

\newcommand \lcs {\lot \cosc}
\newcommand \lcse {\lot \csce}
\newcommand \lcss {\lot \cscs}
\newcommand \lcses {\lot \csces}

\newcommand \lgb {local-global\xspace}
\newcommand \lgbe {locale-globale\xspace}
\newcommand \lgbes {locale-globales\xspace}
\newcommand \lgbs {local-globals\xspace}

\newcommand \lin {liné\-aire\xspace}
\newcommand \lins {liné\-aires\xspace}

\newcommand \lint {liné\-ai\-rement\xspace}

\newcommand \lmo {\lot mono\-gène\xspace}
\newcommand \lmos {\lot mono\-gènes\xspace}

\newcommand \lMR {\lot \MR}

\newcommand \lnl {\lot \nl}
\newcommand \lnls {\lot \nls}

\newcommand \lot {loca\-lement\xspace}

\newcommand \lon {loca\-li\-sation\xspace}
\newcommand \lons {loca\-li\-sations\xspace}

\newcommand \lop {\lot prin\-cipal\xspace}
\newcommand \lops {\lot prin\-cipaux\xspace}

\newcommand \lrf {libre de rang fini\xspace}
\newcommand \lrfs {libres de rang fini\xspace}

\newcommand \lrsb {libre\-ment \rsb}
\newcommand \lrsbs {libre\-ment \rsbs}

\newcommand \lnrsb {libre\-ment $n$-\rsb}
\newcommand \lnrsbs {libre\-ment $n$-\rsbs}

\newcommand \lonrsb {\lot $n$-\rsb}
\newcommand \lonrsbs {\lot $n$-\rsbs}

\newcommand \lorsb {\lot \rsb}
\newcommand \lorsbs {\lot \rsbs}

\newcommand \lsdz {\lot \sdz}

\newcommand \mach {machinerie\xspace}
\newcommand \machs {machineries\xspace}

\newcommand \MHL {lemme de Hensel multivarié\xspace}

\newcommand \mdi {module des \diles}

\newcommand \mgpf {module gradué \pf}
\newcommand \mgpfs {modules gradué \pf}

\newcommand \mhml {module d'\hml}
\newcommand \mhmls {modules d'\hml}

\newcommand \mlm {mo\-dule \lmo}
\newcommand \mlms {mo\-dules \lmos}

\newcommand \mlmo {ma\-tri\-ce de \lon mono\-gène\xspace}
\newcommand \mlmos {ma\-tri\-ces de \lon mono\-gène\xspace}

\newcommand \mlMR {module \lMR}
\newcommand \mlMRs {modules \lMR}

\newcommand \mlr {mani\-pu\-lation \elr}
\newcommand \mlrs {mani\-pu\-lations \elrs}

\newcommand \mlrsb {module \lrsb} 
\newcommand \mlrsbs {modules \lrsbs}

\newcommand \mlnrsb {module \lnrsb} 
\newcommand \mlnrsbs {modules \lnrsbs}

\newcommand \mlonrsb {module \lonrsb}
\newcommand \mlonrsbs {modules \lonrsbs}

\newcommand \mlorsb {module \lorsb} 
\newcommand \mlorsbs {modules \lorsbs}

\newcommand \mlp {ma\-tri\-ce de \lon prin\-ci\-pale\xspace}
\newcommand \mlps {ma\-tri\-ces de \lon prin\-ci\-pale\xspace}

\newcommand \mlrf {module \lrf} 
\newcommand \mlrfs {modules \lrfs}
\newcommand \mlf {\mlrf} 
\newcommand \mlfs {\mlrfs} 

\newcommand \mMR {module de MacRae\xspace}
\newcommand \mMRs {modules de MacRae\xspace}

\newcommand \Mo {Mo\-noïde\xspace}
\newcommand \Mos {Mo\-noïdes\xspace}
\newcommand \mo {mo\-noïde\xspace}
\newcommand \mos {mo\-noïdes\xspace}

\newcommand \moco {\mos \com}

\newcommand \molo {morphisme de \lon\xspace}
\newcommand \molos {morphismes de \lon\xspace}

\newcommand \mom {mo\-nô\-me\xspace}
\newcommand \moms {mo\-nô\-mes\xspace}

\newcommand \moquo {morphisme de passage au quotient\xspace}
\newcommand \moquos {morphismes de passage au quotient\xspace}

\newcommand \mpf {mo\-dule \pf}
\newcommand \mpfs {mo\-dules \pf}

\newcommand \mpl {mo\-dule plat\xspace}
\newcommand \mpls {mo\-dules plats\xspace}

\newcommand \mpn {ma\-trice de \pn}
\newcommand \mpns {ma\-trices de \pn}

\newcommand \mprn {ma\-trice de \prn}
\newcommand \mprns {ma\-trices de \prn}

\newcommand \mptf {mo\-dule \ptf}
\newcommand \mptfs {mo\-dules \ptfs}

\newcommand \MR {de MacRae\xspace}

\newcommand \mrc {mo\-dule \prc}
\newcommand \mrcs {mo\-dules \prcs}

\newcommand \mtf {mo\-du\-le \tf}
\newcommand \mtfs {mo\-du\-les \tf}


\newcommand \ncr{néces\-saire\xspace}       
\newcommand \ncrs{néces\-saires\xspace}       

\newcommand \ncrt{néces\-sai\-rement\xspace}       

\newcommand \ndz {régu\-lier\xspace}
\newcommand \ndze {régu\-lière\xspace}
\newcommand \ndzs {régu\-liers\xspace}
\newcommand \ndzes {régu\-lières\xspace}

\newcommand \nl {simple\xspace}
\newcommand \nls {simples\xspace}

\newcommand \noco {\noe \coh}
\newcommand \nocos {\noes \cohs}

\newcommand \Noe {Noether\xspace}

\newcommand \noe {noethé\-rien\xspace}
\newcommand \noes {noethé\-riens\xspace}
\newcommand \noee {noethé\-rienne\xspace}
\newcommand \noees {noethé\-riennes\xspace}

\newcommand \noet {noethé\-ria\-nité\xspace}

\newcommand \nst {Null\-stellen\-satz\xspace}
\newcommand \nsts {Null\-stellen\-s\"atze\xspace}

\newcommand \op{opé\-ra\-tion\xspace}  
\newcommand \ops{opé\-ra\-tions\xspace}
\newcommand \opari{\op\ari}  
\newcommand \oparis{\ops\aris}  
\newcommand \oparisv{\ops\arisv}  

\newcommand \oqc {ouvert \qc}
\newcommand \oqcs {ouverts \qcs}

\newcommand \ort{or\-tho\-go\-nal\xspace}  
\newcommand \orte{or\-tho\-go\-na\-le\xspace}  
\newcommand \orts{or\-tho\-go\-naux\xspace}  
\newcommand \ortes{or\-tho\-go\-na\-les\xspace}  


\newcommand \pa {couple saturé\xspace}
\newcommand \pas {couples saturés\xspace}
 
\newcommand \paral{paral\-lèle\xspace}  
\newcommand \parals{paal\-lèles\xspace}  

\newcommand \paralm{paral\-lè\-lement\xspace}   

\newcommand \pb{pro\-blè\-me\xspace}  
\newcommand \pbs{pro\-blè\-mes\xspace}  

\newcommand \pdi {dimension \prov}
\newcommand \pdis {dimensions \provs}

\newcommand \peq {purement équa\-tion\-nelle\xspace}
\newcommand \peqs {purement équa\-tion\-nelles\xspace}

\newcommand \pf {de \pn finie\xspace}

\newcommand \pfb {présen\-table\xspace}
\newcommand \pfbs {présen\-tables\xspace}

\newcommand \pKr {\pol de Kronecker\xspace} 
\newcommand \pKrs {\pols de Kronecker\xspace} 

\newcommand \plc {rési\-duel\-lement \zed}
\newcommand \plcs {rési\-duel\-lement \zeds}

\newcommand \pHi {\pol de Hilbert\xspace} 
\newcommand \pHis {\pols de Hilbert\xspace} 

\newcommand \pit {\tho de l'\idp de Krull\xspace} 
\newcommand \Pit {\Tho de l'\idp de Krull\xspace} 

\newcommand \Plg {Prin\-cipe \lgb}
\newcommand \plg {prin\-cipe \lgb}
\newcommand \plgs {prin\-cipes \lgbs}

\newcommand \plga {\plg abstrait\xspace}
\newcommand \plgas {\plgs abstraits\xspace}

\newcommand \Plgc {\Plg con\-cret\xspace}
\newcommand \plgc {\plg con\-cret\xspace}
\newcommand \plgcs {\plgs con\-crets\xspace}

\newcommand \pn {présen\-ta\-tion\xspace}
\newcommand \pns {présen\-ta\-tions\xspace}

\newcommand \pog {\pol \hmg\xspace}
\newcommand \pogs {\pols \hmgs\xspace}

\newcommand \Pol {Poly\-nôme\xspace}
\newcommand \Pols {Poly\-nômes\xspace}

\newcommand \pol {poly\-nôme\xspace}
\newcommand \pols {poly\-nômes\xspace}

\newcommand \poll{poly\-nomial\xspace}  
\newcommand \polls{poly\-nomiaux\xspace}  
\newcommand \polle{poly\-no\-miale\xspace}  
\newcommand \polles{poly\-no\-miales\xspace}  

\newcommand \pollt{poly\-no\-mia\-lement\xspace}  

\newcommand \polfon {\pol fon\-da\-men\-tal\xspace}
\newcommand \polmu {\pol rang\xspace}
\newcommand \polmus {\pols rang\xspace}
\newcommand \polcar {\pol carac\-té\-ris\-tique\xspace}
\newcommand \polmin {\pol mini\-mal\xspace}

\newcommand \prc {\pro de rang constant\xspace}
\newcommand \prcs {\pros de rang constant\xspace}

\newcommand \prcc {prin\-cipe de \rcc}
\newcommand \prca {prin\-cipe de \rca}
\newcommand \prce {prin\-cipe de \rce}

\newcommand \prf {prin\-cipe de recou\-vrement fermé\xspace}

\newcommand \prmt {préci\-sé\-ment\xspace}
\newcommand \Prmt {Préci\-sé\-ment\xspace}

\newcommand \prn {pro\-jec\-tion\xspace}
\newcommand \prns {pro\-jec\-tions\xspace}

\newcommand \pro {pro\-jec\-tif\xspace}
\newcommand \pros {pro\-jec\-tifs\xspace}
\newcommand \prov {pro\-jec\-tive\xspace}
\newcommand \provs {pro\-jec\-tives\xspace}

\newcommand \Prof {Profondeur\xspace}
\newcommand \prof {profondeur\xspace}
\newcommand \profs {profondeurs\xspace}

\newcommand \Proh {\Prof \hmq}
\newcommand \proh {\prof \hmq}
\newcommand \prohs {\profs \hmqs}
\newcommand \prohz {\prof \hmqz}
\newcommand \prohsz {\profs \hmqsz}

\newcommand \prr {pro\-jec\-teur\xspace}
\newcommand \prrs {pro\-jec\-teurs\xspace}

\newcommand \Prt {Pro\-pri\-été\xspace}
\newcommand \Prts {Pro\-pri\-étés\xspace}
\newcommand \prt {pro\-pri\-été\xspace}
\newcommand \prts {pro\-pri\-étés\xspace}

\newcommand \ptf {\pro \tf}
\newcommand \ptfs {\pros \tf}

\newcommand \ptrd {prétreillis distributif\xspace}
\newcommand \ptrds {prétreillis distributifs\xspace}

\newcommand \qc {quasi-compact\xspace}
\newcommand \qcs {quasi-compacts\xspace}

\newcommand \qi {quasi in\-tè\-gre\xspace}
\newcommand \qis {quasi in\-tè\-gres\xspace}

\newcommand \qli {quasi libre\xspace}
\newcommand \qlis {quasi libres\xspace}

\newcommand \qiv {quasi inverse\xspace}
\newcommand \qivs {quasi inverses\xspace}

\newcommand \qnl {quasi-\nl}
\newcommand \qnls {quasi-\nls}

\newcommand \qreg {quasi régu\-lière\xspace}
\newcommand \qregs {quasi régu\-lières\xspace}

\newcommand \qsi {quasi singu\-li\`ere\xspace}
\newcommand \qsis {quasi singu\-li\`eres\xspace}

\newcommand \qtf {quanti\-fi\-cateur\xspace}
\newcommand \qtfs {quanti\-fi\-cateurs\xspace}

\newcommand \ralg {règle \agq}
\newcommand \ralgs {règles \agqs}

\newcommand \rav {racine virtuelle\xspace}
\newcommand \ravs {racines virtuelles\xspace}

\newcommand \rcc {\rcm con\-cret\xspace}
\newcommand \rca {\rcm abs\-trait\xspace}
\newcommand \rce {\rcc des é\-ga\-li\-tés\xspace}

\newcommand \rcm {recol\-lement\xspace}
\newcommand \rcms {recol\-lements\xspace}

\newcommand \rcv {recou\-vrement\xspace} 
\newcommand \rcvs {recou\-vrements\xspace} 

\newcommand \rcvq {recou\-vrement par quotients\xspace} 
\newcommand \rcvqs {recou\-vrements par quotients\xspace} 

\newcommand \rde {rela\-tion de dépen\-dance\xspace}
\newcommand \rdes {rela\-tions de dépen\-dance\xspace}

\newcommand \rdi {\rde inté\-grale\xspace}
\newcommand \rdis {\rdes inté\-grales\xspace}

\newcommand \rdl {\syzy}
\newcommand \rdls {\syzys}

\newcommand \rdt {rési\-duel\-lement\xspace}

\newcommand \rdy {règle dyna\-mique\xspace}
\newcommand \rdys {règles dyna\-miques\xspace}

\newcommand \rlf {\rsn libre finie\xspace}
\newcommand \rlfs {\rsns libres finies\xspace}

\newcommand \rmq {\rcm de quotients\xspace} 
\newcommand \rvq {\rcv par quotients\xspace} 

\newcommand \rmqs {\rcms de quotients\xspace} %
\newcommand \rvqs {\rcvs par quotients\xspace} %

\newcommand \rpf {réduite-de-présen\-tation-finie\xspace}
\newcommand \rpfs {réduites-de-présen\-tation-finie\xspace}

\newcommand \rsb {réso\-luble\xspace}
\newcommand \rsbs {réso\-lubles\xspace}

\newcommand \rsf {\rsp finie\xspace}
\newcommand \rsfs {\rsps  finies\xspace}

\newcommand \rsim {règle de simplification\xspace}
\newcommand \rsims {règles de simplification\xspace}

\newcommand \Rsn {Réso\-lution\xspace}
\newcommand \Rsns {Réso\-lutions\xspace}
\newcommand \rsn {réso\-lution\xspace}
\newcommand \rsns {réso\-lutions\xspace}

\newcommand \rsp {\rsn \prov}
\newcommand \rsps {\rsns \provs}

\newcommand \rtf {radi\-ca\-lement \tf}
\newcommand \rtfz {radi\-ca\-lement \tfz}

\newcommand \rtl {réti\-culé\xspace}
\newcommand \rtls {réti\-culés\xspace}


\newcommand \sad {\salg dynamique\xspace}
\newcommand \sads {\salgs dynamiques\xspace}

\newcommand \sagq {semi\agq}
\newcommand \sagqs {semi\agqs}

\newcommand \sagc {\sagq continue\xspace}
\newcommand \sagcs {\sagqs continues\xspace}

\newcommand \salg {structure \agq}
\newcommand \salgs {structures \agqs}

\newcommand \scentrel {relation semi-implicative\xspace}
\newcommand \scentrels {relations semi-implicatives\xspace}

\newcommand \scf {schéma spectral\xspace}
\newcommand \scfs {schémas spectraux\xspace}

\newcommand \scl {schéma \elr}
\newcommand \scls {schémas \elrs}

\newcommand \scm {schéma\xspace}
\newcommand \scms {schémas\xspace}

\newcommand \scEs {suite \cEse}
\newcommand \scEss {suites \cEses}

\newcommand \scs {suite \csce}
\newcommand \scss {suites \csces}

\newcommand \sdo {\sdr \orte}
\newcommand \sdos {\sdrs \ortes}

\newcommand \sdr {somme directe\xspace}
\newcommand \sdrs {sommes directes\xspace}

\newcommand \sdz {sans \dvz}

\newcommand \seco {\sex courte\xspace}
\newcommand \secos {\sexs courtes\xspace}

\newcommand \seqreg {\srg} 
\newcommand \seqregs {\srgs}

\newcommand \seqqreg {suite \qreg} 
\newcommand \seqqregs {suites \qregs}

\newcommand \sErg {suite \Erge} 
\newcommand \sErgs {suites \Erges}

\newcommand \sex {suite exacte\xspace}
\newcommand \sexs {suites exactes\xspace}

\newcommand \sFrg {suite \Frge }
\newcommand \sFrgs {suites \Frges }

\newcommand \sfio {sys\-tème fondamental d'\idms ortho\-gonaux\xspace}
\newcommand \sfios {sys\-tèmes fondamentaux d'\idms ortho\-gonaux\xspace}

\newcommand \sgr {\sys \gtr}
\newcommand \sgrs {\syss \gtrs}

\newcommand \sing {singu\-lière\xspace}
\newcommand \sings {singu\-lières\xspace}

\newcommand \sKr {suite de Kronecker\xspace}
\newcommand \sKrs {suites de Kronecker\xspace}

\newcommand \slc {suite \lcse\xspace} 
\newcommand \slcs {suites \lcses\xspace}

\newcommand \slgb {stricte\-ment \lgb}
\newcommand \slgbs {stricte\-ment \lgbs}

\newcommand \sli {\sys \lin}
\newcommand \slis {\syss \lins}

\newcommand \smq {symé\-trique\xspace}
\newcommand \smqs {symé\-triques\xspace}

\newcommand \spb {sépa\-rable\xspace}  
\newcommand \spbs {sépa\-rables\xspace}

\newcommand \spe {spéci\-fi\-cation\xspace}
\newcommand \spes {spéci\-fi\-cations\xspace}

\newcommand \spi {\spe incomplète\xspace}
\newcommand \spis {\spes incomplètes\xspace}

\newcommand \spl {sépa\-rable\xspace}  
\newcommand \spls {sépa\-rables\xspace}

\newcommand \spo {semipolynôme\xspace}
\newcommand \spos {semipolynômes\xspace}

\newcommand \spt{sépa\-ra\-bi\-lité\xspace}

\newcommand \srg {suite régu\-lière\xspace}
\newcommand \srgs {suites régu\-lières\xspace}

\newcommand \stf {strictement fini\xspace}
\newcommand \stfs {strictement finis\xspace}
\newcommand \stfe {strictement finie\xspace}
\newcommand \stfes {strictement finies\xspace}

\newcommand \stl {stablement libre\xspace}
\newcommand \stls {stablement libres\xspace}

\newcommand \stm {strictement\xspace}

\newcommand \str {\stm réticulé\xspace}
\newcommand \stre {\stm réticulée\xspace}
\newcommand \strs {\stm réticulés\xspace}
\newcommand \stres {\stm réticulées\xspace}

\newcommand \suc {suite $1$-sécante\xspace}
\newcommand \sucs {suites $1$-sécantes\xspace}

\newcommand \sul {supplé\-men\-taire\xspace}
\newcommand \suls {supplé\-men\-taires\xspace}

\newcommand \Sut {Support\xspace}
\newcommand \Suts {Supports\xspace}
\newcommand \sut {support\xspace}

\newcommand \syc {\sys de coordon\-nées\xspace}
\newcommand \sycs {\syss de coordon\-nées\xspace}

\newcommand \syp {\sys \poll}
\newcommand \syps {\syss \polls}

\newcommand \sys {sys\-tème\xspace}
\newcommand \syss {sys\-tèmes\xspace}

\newcommand \syzy {syzygie\xspace}
\newcommand \syzys {syzygies\xspace}

\newcommand \talg {théorie \agq}
\newcommand \talgs {théories \agqs}

\newcommand \tco {théorie cohé\-rente\xspace}
\newcommand \tcos {théories cohé\-rentes\xspace}

\newcommand \tdy {théorie dyna\-mique\xspace}
\newcommand \tdys {théories dyna\-miques\xspace}

\newcommand \tel {théorie exis\-ten\-tielle\xspace}
\newcommand \tels {théories exis\-ten\-tielles\xspace}

\newcommand \telri {théorie exis\-ten\-tielle rigide\xspace}
\newcommand \telris {théories exis\-ten\-tielles rigides\xspace}

\newcommand \tf {de type fini\xspace}

\newcommand \tfo {théorie formelle\xspace}
\newcommand \tfos {théorie formelles\xspace}

\newcommand \tgm {théorie \gmq}
\newcommand \tgms {théories \gmqs}
\newcommand \tgmi {théorie \gmq infinitaire\xspace}
\newcommand \tgmis {théories \gmqs infinitaires\xspace}

\newcommand \Tho {Théo\-rème\xspace}
\newcommand \Thos {Théo\-rèmes\xspace}
\newcommand \tho {théo\-rème\xspace}
\newcommand \thos {théo\-rèmes\xspace}

\newcommand \thoc {théo\-rème$\mathbf{^*}$~}

\newcommand \tpe {théorie \peq}
\newcommand \tpes {théories \peqs}

\newcommand \trdi {treil\-lis dis\-tri\-bu\-tif\xspace}
\newcommand \trdis {treil\-lis dis\-tri\-bu\-tifs\xspace}

\newcommand \trel {trans\-for\-mation \elr}
\newcommand \trels {trans\-for\-mations \elrs}

\newcommand \umd {unimo\-du\-laire\xspace}
\newcommand \umds {unimo\-du\-laires\xspace}

\newcommand \unt {uni\-taire\xspace}
\newcommand \unts {uni\-taires\xspace}

\newcommand \usc {$1$-sécante\xspace}
\newcommand \uscs {$1$-sécantes\xspace}

\newcommand \uvl {uni\-versel\xspace}
\newcommand \uvle {uni\-ver\-selle\xspace}
\newcommand \uvls {uni\-versels\xspace}
\newcommand \uvles {uni\-ver\-selles\xspace}


\newcommand \vfn {véri\-fi\-cation\xspace}
\newcommand \vfns {véri\-fi\-cations\xspace}

\newcommand \vmd {vec\-teur \umd}
\newcommand \vmds {vec\-teurs \umds}

\newcommand \vgq {\vrt \agq} 
\newcommand \vgqs {\vrts \agqs}

\newcommand \vrp {\vrt \prov} 
\newcommand \vrps {\vrts \provs} 

\newcommand \vrt {variété\xspace}
\newcommand \vrts {variétés\xspace}

\newcommand \zed {z\'{e}ro-di\-men\-sionnel\xspace}
\newcommand \zede {z\'{e}ro-di\-men\-sion\-nelle\xspace}
\newcommand \zeds {z\'{e}ro-di\-men\-sion\-nels\xspace}
\newcommand \zedes {z\'{e}ro-di\-men\-sion\-nelles\xspace}

\newcommand \zedr {\zed réduit\xspace}
\newcommand \zedrs {\zeds réduits\xspace}

\newcommand \zmt {\tho de Zariski-Grothen\-dieck\xspace}


\newcommand \cof {cons\-truc\-tif\xspace}
\newcommand \cofs {cons\-truc\-tifs\xspace}

\newcommand \cov {cons\-truc\-tive\xspace}
\newcommand \covs {cons\-truc\-tives\xspace}

\newcommand \coma {\maths\covs}
\newcommand \clama {\maths clas\-siques\xspace}

\renewcommand \cot {cons\-truc\-ti\-vement\xspace}

\newcommand \matn {mathé\-ma\-ticien\xspace}
\newcommand \matne {mathé\-ma\-ti\-cienne\xspace}
\newcommand \matns {mathé\-ma\-ticiens\xspace}
\newcommand \matnes {mathé\-ma\-ti\-ciennes\xspace}

\newcommand \maths {mathé\-ma\-tiques\xspace}
\newcommand \mathe {mathé\-ma\-tique\xspace}

\newcommand \prco {démons\-tration \cov}
\newcommand \prcos {démons\-trations \covs}

\renewcommand\paragraph[1]{

\rdb\addcontentsline{toc}{subsubsection}{#1} \medskip \noindent $\bullet$ \textbf{#1}}

\newcommand \subsec[1]{\goodbreak\rdb\addcontentsline{toc}{subsection}{#1}\subsection*{#1}}

\newcommand \subsect[2]{\goodbreak\rdb\addcontentsline{toc}{subsection}{#2}
\subsection*{#1}}

\newcommand \subsubs[1]{

\goodbreak\rdb\medskip 

{\bf #1}

\smallskip }

\newcommand \subsubsec[1]{\goodbreak\rdb\addcontentsline{toc}{subsubsection}{#1}\subsubsection*{#1}}

\newcommand \subsubsect[2]{\goodbreak\rdb\addcontentsline{toc}{subsubsection}{#2}\subsubsection*{#1}}

\newcommand \Subsubsec[1]{\goodbreak\rdb\addcontentsline{toc}{subsection}{#1}\subsubsection*{#1}}

\newcommand \Subsubsect[2]{\rdb\addcontentsline{toc}{subsection}{#2}\subsubsection*{#1}}

\newcommand \Subsec[1]{\goodbreak\rdb\addcontentsline{toc}{section}{#1}\subsection*{#1}}

\newcommand\siBookdeux[1]{}
\newcommand\siFFR[1]{#1}
\newcommand\siDiviseurs[1]{}
\newcommand\Discussion[1]{}
\newcommand\sibrouillon[1]{}
\renewcommand \perso[1]{}
\renewcommand \entrenous[1]{}
\renewcommand \hum[1]{}
\renewcommand \ttt[1]{}
\newcommand\Today{}
\newcommand\siarticle[1]{#1}
\newcommand\sibook[1]{}

\newcommand \Exercices{\endinput}

\romanpagenumbers

\let\oldshowchapter\showchapter
\let\oldshowsection\showsection
\def\showchapter#1{\edef\temp{\thechapter}\def\ttemp{#1}\ifx\ttemp\temp\relax\def\showsection##1{##1}\else\let\showsection\oldshowsection\oldshowchapter{#1}\fi}


\setpagenumber1
~

\vskip2cm
\begin{center}\Huge\bf
\CMtitle 
\end{center}

\vskip2cm
\begin{center}{\large\bf
Thierry Coquand$^*$, Henri Lombardi$^\dag$, Claude Quitté$^\ddag$\\[.1em] \& Claire Tête$^\P$ 
}\vskip1cm
$*$. Chalmers, University of Göteborg, Sweden, email:~{\tt coquand@chalmers.se}
\\[.3em]
$\dag$. \'Equipe de Math\'ematiques, UMR CNRS 6623,
UFR des Sciences et Techniques, Universit\'e de Franche-Comt\'e,
25030 Besan{\c c}on cedex, France, email:~{\tt henri.lombardi@univ-fcomte.fr}
\\[.3em]
$\ddag$. Laboratoire de Math\'ematiques, SP2MI, Boulevard 3, T\'el\'eport 2, 
\\ BP~179, 86960 Futuroscope Cedex, France, email:~{\tt claude.quitte@math.univ-poitiers.fr}
\\[.3em]
$\P$. Laboratoire de Math\'ematiques, SP2MI, Boulevard 3, T\'el\'eport 2, BP~179, 86960 Futuroscope Cedex, France, \\email:~{\tt claire.tete@math.univ-poitiers.fr}

\end{center}

\romanpagenumbers
\setpagenumber0
\setcounter{tocdepth}{1}

\tableofcontents

\vskip1cm
\rdb
{\LARGE \bf Avant-propos}
\pagestyle{CMExercicesheadings}
\thispagestyle{CMchapitre}
\addcontentsline{toc}{chapterbis}{Avant-propos}
\markboth{Avant-propos}{Avant-propos}

\vskip.5cm

\begin{flushright}
\begin{minipage}{9cm}
\begin{flushright}
{\small
 Quant  à moi, je proposerais de s'en tenir
 aux règles
suivantes:


\item 1. Ne jamais envisager que des objets susceptibles
d'être définis
en un nombre fini de mots;
\\
\item 2. Ne jamais perdre de vue que toute proposition
 sur l'infini doit
être la traduction, l'énoncé abrégé
de propositions sur le
fini;
\\
\item 3. \'Eviter les classifications et les définitions
non-prédicatives.


\smallskip   Henri Poincaré,

dans  {\it La logique de l'infini }\\
(Revue de Métaphysique et de Morale, 1909).
\\
Réédité dans  {\it
Dernières pensées}, Flammarion.}
\end{flushright}
\end{minipage}

\vspace{1em}

\begin{minipage}{9cm}
\begin{flushright}
{\small
Le cadre historique naturel de la \gmt \agq est celui des \pols. Le développement de l'Algèbre moderne, commencé il y a près d'un siècle,
a renvoyé les anneaux de \pols au statut de cas particulier et les méthodes propres aux \pols, comme la \emph{Théorie de l'élimination}, au conservatoire. 
Mais \gui{les objets sont têtus} et les méthodes explicites ne cessent de ressurgir. Un calcul est toujours plus \gnl que le cadre théorique dans lequel 
on l'enferme à une période donnée.

\smallskip  Michel Demazure

dans {\it Résultant, discriminant.}\\
L'Enseignement Mathématique (2) {\bf 58} (2012), 333-373

}
\end{flushright}
\end{minipage}
\end{flushright}

\bigskip 
Ce mémoire puise sa source dans le livre  \cite[Northcott, Finite Free Resolutions]{NorFFR} (cité dans la suite comme [FFR]) et dans les articles \cite{CQ2012} et \cite{CT2018}. Il est écrit dans la continuation des ouvrages \cite[A course in constructive algebra]{MRR} et \cite[Commutative Algebra, Constructive Methods]{CACM} (cités dans la suite comme [MRR] et [CACM]).

L'ouvrage en anglais [CACM] est disponible à l'adresse \url{https://arxiv.org/abs/1605.04832}. Une version française légèrement plus étendue se trouve à l'adresse \url{https://arxiv.org/abs/1611.02942}. Si nécessaire nous la citerons comme [ACMC].

En écrivant son livre sur les résolutions libres finies et la théorie de la profondeur, Northcott semblait poursuivre plusieurs buts bien précis. Le premier est de libérer la théorie de toute référence à la \noet. Le deuxième est de se libérer des outils de l'algèbre homologique, notamment en adoptant la définition de Hochster pour la profondeur d'un module relativement à un \itf. Le troisième est de donner un traitement aussi algorithmique que possible des énoncés de la théorie.

De manière étrange ce livre, salué par la critique et régulièrement cité dans les commentaires bibliographiques comme référence dernière pour la théorie, n'est jamais vraiment utilisé dans les livres d'algèbre commutative ou de géométrie algébrique. 

L'exposé n'est pourtant pas plus compliqué que les exposés habituels de ces mêmes théories. Les rares raccourcis qu'offre l'usage intensif de la \noet restreignent la portée des \thos et surtout semblent interdire le traitement algorithmique du sujet.

Notre mémoire a pour but de réparer cet état des lieux en offrant un traitement complètement constructif de la théorie de Northcott, et, nous l'espérons, suffisamment clair et élégant. Quelquefois en effet, Northcott ne réussit pas à se débarrasser de l'usage des idéaux maximaux ou des idéaux premiers minimaux pour venir à bout de certains résultats. En outre il utilise quelquefois des arguments d'algèbre  homologique déguisés. Enfin il ne dispose pas à l'époque d'une définition constructive de la dimension de Krull.

Aujourd'hui les progrès de l'algèbre constructive, résumés dans [MRR], [CACM],  et \cite[Constructive Commutative Algebra, 2015]{Yen2015} nous permettent  de surmonter tous les obstacles qui ont empêché Northcott de remplir de manière complètement satisfaisante les buts qu'il s'était fixés.

Nous espérons, dans un autre mémoire à venir, rendre également justice à l'algèbre homologique, à sa fameuse suite exacte longue, et notamment aux complexes de \Cech et de Koszul. 
L'attention payée aux méthodes explicites devrait permettre dans ce cas de simplifier l'exposé de nombreux résultats et nous débarrasser de la plupart des suites spectrales qui encombrent souvent le paysage. Mais ne vendons pas la peau de l'ours avant de l'avoir tué.

\begin{flushright} \label{flushright}
Les auteurs

\today
\end{flushright}

\cleardoublepage
\pagestyle{CMheadings}
  
\arabicpagenumbers
\setpagenumber1


\setcounter{chapter}{-1}

\chapter{Quelques rappels}\label{chapRappelsFFR}

\minitoc

\intro

Ce chapitre est un rappel de \dfns et résultats
donnés dans \cite[CACM]{CACM} et qui seront souvent invoqués
dans cet ouvrage.

\smallskip La section \ref{secRappelDefs}  rappelle des \dfns
données en \coma pour des notions bien connues, ou inexistantes, 
en \clama. 

\smallskip La section \ref{secRappelThs}  donne quelques résultats,
qui sont souvent des versions \covs de résultats, parfois bien connus,
établis en \clama. 

Dans ce chapitre, comme dans tout l'ouvrage sauf mention expresse du
contraire, 

 \centerline{\fbox{les anneaux sont commutatifs et \unts},}

\smallskip\noindent   et les \homos entre anneaux respectent les $1$. En particulier, un sous-anneau a le même $1$ que l'anneau.


\section{Quelques \dfns et notations} \label{secRappelDefs}

Certaines de ces \dfns donnent des \prts \eqves, qui font l'objet de \thos dans [CACM].

 \fbox{$\gA$ désigne toujours un anneau}. 

\paragraph{Éléments et \moco}~
\begin{itemize}
\item On appelle \emph{\mo} de $\gA$ un \mo pour la multiplication (il peut contenir $0$).  On note $\gA_S$ ou $S^{-1}\gA$ le localisé de $\gA$ en $S$. 
\item Un \mo~$S$ dans  $\gA$ est dit
\emph{saturé} lorsque l'implication%
\index{monoide satur@monoïde!saturé}%
\index{sature@saturé!monoide@monoïde ---}
$$\preskip.4em \postskip.4em
\forall s, t \in \gA \;\;( st\in S \;\Rightarrow\; s\in S)
$$
est satisfaite. Un \mo saturé est \egmt appelé un \emph{filtre}.
On  note $\sat{S}$ le saturé du \mo $S$.
\item 
Nous appellerons \emph{filtre principal} un filtre engendré par un \elt:
il est constitué de l'ensemble des diviseurs d'une puissance de cet \elt.%
\index{principal!filtre --- d'un anneau commutatif}%
\index{filtre!d'un anneau commutatif}%
\index{filtre!principal d'un anneau commutatif}
\item   Des \elts $s_1$, $\dots$, $s_n\in\gA$ sont dits
\ixc{comaximaux}{elements@\elts~---} si
$\gen{1} = \gen{s_1,\dots,s_n}$. Deux \elts \com sont aussi appelés
\emph{étrangers}. \index{etrangers@étrangers!elements@\elts~---}
\item   Des \mos $S_1$, $\dots$, $S_n$ sont dits \ixc{comaximaux}{monoides@\mos~---}
si chaque fois que $s_1\in S_1$, \dots, $s_n\in S_n$, les $s_i$  sont \com.
\end{itemize}

\paragraph{Idéaux déterminantiels et \idfs}~ 
\begin{itemize}
\item Soient $m,n\in\NN\etl$,  $k\in\lrb{1..\min(m,n)}$ et $G\in\gA^{n\times m}$.
 L'{\em  \idd d'ordre $k$ de la matrice $G$}  est l'\id,
noté $\cD_{\gA,k}(G)$ ou $\cD_k(G)$,  engendré  par  les   mineurs d'ordre  $k$
de   $G$.
\index{ideaux determ@idéaux déterminantiels!d'une matrice}
Pour $k\leq 0$ on pose par convention~\hbox{$\cD_k(G)=\gen{1}$}, et pour
$k> \min(m,n)$,  $\cD_k(G)=\gen{0}$.
\item Une matrice $A$ est dite \emph{de rang $\leq k$} si $\cD_{k+1}(A)=\so0$,
elle est dite \emph{de rang $\geq k$}\hbox{ si $\cD_{k}(A)=\so1$}, de rang $k$ lorsqu'elle est à la fois de rang $\leq k$ et $\geq k$.
\item Si $G\in\gA^{q\times m}$ est une \mpn d'un \Amo $M$ donné par~$q$ \gtrs, les \emph{\idfs de} $M$ sont les \ids
$$\preskip.4em \postskip.4em 
\cF_{\gA,n}(M)=\cF_{n}(M):= \cD_{\gA,q-n}(G) 
$$
où $n$ est un entier arbitraire.\index{ideal@idéal!de Fitting}%
\index{Fitting!idéal de ---}
\end{itemize}

\paragraph{Matrices}~ 
\begin{itemize}
\item Deux matrices  sont dites
\emph{\eqves}
lorsque l'on  passe de l'une à l'autre en multipliant à droite et à
gauche par des matrices inversibles.
\item Une \emph{manipulation \elr de lignes} sur une matrice de $n$ lignes
consiste en le remplacement d'une ligne $L_i$ par la ligne $L_i+\lambda L_j$
avec~\hbox{$ i\neq j$}. On la note  aussi $L_i\aff L_i+\lambda L_j$. Cela correspond à la multiplication à gauche
par une matrice, dite \emph{\elr},  notée $\rE^{(n)}_{i,j}(\lambda)$
(ou, si le contexte le permet, $\rE_{i,j}(\lambda)$). Cette matrice est
obtenue à partir de $\In$ par la même manipulation \elr
de lignes.
\\
La multiplication à droite par la
même matrice $\rE_{i,j}(\lambda)$ correspond, elle,
à la \emph{manipulation \elr de colonnes} (pour une
matrice qui possède $n$ colonnes) qui
transforme la matrice $\In$ en $\rE_{i,j}(\lambda)$: $C_j\aff C_j+\lambda C_i$.
\item Le sous-groupe de $\SLn(\gA)$
engendré par les matrices \elrs est appelé le \emph{groupe \elr} et il est noté
$\En(\gA)$. Deux matrices sont dites \emph{\elrt \eqves}
lorsque l'on peut passer de l'une à l'autre
par des manipulations \elrs de lignes et de colonnes.%
\index{matrices!equiva@équivalentes}%
\index{equivalentes@équivalentes!matrices ---}%
\index{elementairement@élémentairement équi\-valentes!matrices ---}%
\index{matrices!elem@élémentairement équi\-valentes}%
\index{matrice!elem@élémentaire}%
\index{groupe elem@groupe élémentaire}%
\index{manipulation!elem@élémentaire}%
\index{elementaire@élémentaire!matrice ---}%
\index{elementaire@élémentaire!groupe ---}%
\index{elementaire@élémentaire!manipulation --- de lignes}%
\label{NOTAEn}%
\item La matrice $\I_{k,q,m}=\cmatrix{
    \I_{k}   &0_{k,m-k}      \cr
    0_{q-k,k}&     0_{q-k,m-k}      }$ est appelée une \emph{matrice simple standard}.
On note $\I_{k,n}$ pour $\I_{k,n,n}$ et on l'appelle une
\emph{\mprn standard}. Une matrice
est dite \emph{\nl} lorsqu'elle est \eqve à une
matrice $\I_{k,q,m}$.
\item Soient $E$ et $F$ deux \Amos,  et une \ali $\varphi:E\rightarrow F$.
Une \ali $\psi :F\rightarrow E$ est appelée un \ix{inverse
généralisé} de $\varphi$ si l'on a
\begin{equation}\label{eqdefIng}\preskip-.2em \postskip.4em
\varphi \circ\psi \circ\varphi =\varphi \quad \mathrm{et}  \quad \psi
\circ\varphi \circ\psi =\psi.
\end{equation}
%
\item  Une \ali est dite \emph{\lnl} lorsqu'elle possède un
\ing.  Une matrice est dite \emph{\lnl} lorsque c'est la matrice d'une \ali \lnl.%
\index{{localement!application linéaire --- simple}}
\end{itemize}

\paragraph{Anneaux}~  
\begin{itemize}
\item $\gA$ est dit \emph{discret} (resp. \emph{\fdi}) s'il est discret (resp. \fdi) en tant que \Amo.
\item $\gA$ est \emph{\coh}%
\index{anneau!cohérent}
 s'il est \coh comme \Amo. Autrement dit si tout \itf est \pf.
\item $\gA$ est un \emph{anneau de Bézout}%
\index{anneau!de Bézout}
 si tout \itf est principal.
\item $\gA$ est  \emph{intègre} si tout \elt  est nul ou \ndz,  \emph{\qi}%
\index{anneau!quasi intègre}
 (ou \emph{de Baer}) si l'annulateur de tout \elt est un \idm.
\item  $\gA$ est dit {\em \sdz}%
\index{anneau!sans diviseur de zéro}
 si lorsque $xy=0$, on a $x=0$ ou $y=0$.
\item  $\gA$ est dit {\em \lsdz}%
\index{anneau!localement sans diviseur de zéro} si pour tous $x$, $y$ tels que $xy=0$, il existe $s,t\in\gA$ avec $sx=0$,  $ty=0$ et $s+t=1$.
\item $\gA$ est appelé un \emph{\anar} si tout \itf 
est localement principal. Il revient au même de dire: pour tous $a$, $b\in\gA$, il existe $s$, $t$, $v$, $w$ tels que
\begin{equation}
\label{eqanar}
\preskip .1em \postskip -1em
sa=vb,\quad tb=wa\quad \hbox{et}\quad s+t=1.
\end{equation} 
\index{anneau!arithm@\ari}
\item $\gA$ est appelé un \emph{\adp}%
\index{anneau!de Prüfer}
 si tout \id est plat. Il revient au même de dire que $\gA$ est \ari et réduit. 
\item $\gA$ est  un \emph{\adpc} \ssi tout \itf est \pro. On dit aussi
que $\gA$ \emph{semihérédiaire}. Il revient au même de dire que $\gA$ est \ari et \qi. 
\item $\gA$ est dit \emph{\noe}%
\index{anneau!noethérien} s'il est \noe en tant que \Amo. 
\item  Un \ixx{anneau}{local} est un anneau $\gA$ où est vérifié
l'axiome suivant:\index{local!anneau ---}
$\forall x,\,y\in \gA \; x+y \in\Ati\;\Longrightarrow \; ( x\in\Ati\;\mathrm{ou}\;
y\in\Ati )\,.$
Il revient au même de demander:
$
\forall x\in \gA \; x \in\Ati\;\; {\rm  ou} \;\;  1-x \in\Ati\,.
$
Notez que selon cette \dfn l'anneau trivial est local.
Par ailleurs, les \gui{ou} doivent être compris dans leur sens \cof.

\item Un  \ixx{anneau}{local \dcd}
est  un \alo dont le corps
résiduel est un \cdi.
Un tel anneau $\gA$ peut être \care par l'axiome suivant:
$\forall x\in \gA \; x\in \Ati  \;\; {\rm  ou} \;\;
    1+x\gA\,\subseteq\,  \Ati.$

\item Un anneau est \emph{\zed} lorsqu'il vérifie l'axiome suivant:%
\index{zero-dimensionnel@\zed!anneau ---}\index{anneau!zéro-dimensionnel}
$\forall x\in \gA~\exists a\in\gA~\exists k\in
\NN\; x^{k}=ax^{k+1}.$
Un anneau est dit \emph{artinien}\index{artinien!anneau ---} s'il est \zed, \coh et \noe.%
\index{anneau!artinien}

%
\end{itemize}

\paragraph{Idéaux.}~ $\fa$ désigne un \id de l'anneau $\gA$ 
\begin{itemize}
\item Les \elts nilpotents dans un anneau $\gA$ forment un \id appelé
\emph{nilradical}, ou encore \ixx{radical}{nilpotent} de l'anneau.
Un anneau est \emph{réduit} si son nilradical est égal à $0$.
Plus \gnlt le nilradical d'un \id $\fa$ de $\gA$ est l'\id formé par les
$x\in\gA$ dont une puissance est dans $\fa$. Nous le noterons $\sqrt{\fa}$ ou
$\DA(\fa)$. Nous notons aussi $\DA(x)$ pour $\DA(\gen{x})$.
Un \id $\fa$ est appelé \emph{un idéal radical} lorsqu'il
est égal à son nilradical.
L'anneau $\gA/\DA(0)=\gA\red$ est \emph{l'anneau réduit associé à $\gA$}.%
\index{anneau!réduit}\label{NOTADA}%
\index{reduit@réduit!anneau ---}%
\index{nilradical!d'un anneau}%
\index{nilradical!d'un idéal}%
\index{radical!idéal ---}%
\index{ideal@idéal!radical}
\item L'ensemble des \elts  $a$ de $\gA$  qui
vérifient\label{eqDefRadJac}
$\forall x\in \gA~ 1+ax\in\Ati$
est un \id appelé le \ixx{radical}{de Jacobson} de $\gA$. Il sera noté $\Rad(\gA)$.

\item On dit que $\fa$ est \emph{localement principal} s'il existe
 $s_1$, \dots, $s_m$ \com  dans $\gA$ et des \elts $a_1$, \dots, $a_m$ de $\fa$ tels que $s_i\fa\subseteq \gen{a_i}$ (pour chaque~$i$).
 Dans ce cas, on a $\fa=\gen{a_1,\dots,a_m}$.%
\index{ideal@\id!localement principal}
\item \emph{L'idéal contenu d'un \pol $h\in \AXn$}, noté $\rc_\gA(h)$ ou $\rc(h)$, est l'\id de $\gA$
engendré par les \coes de $h$. Le \pol $h$ est dit \emph{primitif}   si 
$\rc_\gA(h)=\gA$.

\end{itemize}

\paragraph{Localisations, bords de Krull.}~ 
\begin{itemize}
\item Soient $U$ et $I$ des parties de l'anneau $\gA$. Nous notons $\cM(U)$ le \mo engendré par  $U$,
 et  $\,\cS(I,U)$  le \mo $\gen {I}_\gA + \cM(U)$.
De la même manière on note $\cS(a_1,\ldots,a_k;u_1,\ldots,u_\ell) = \gen{a_1,\ldots ,a_k}_\gA + \cM(u_1,\ldots,u_\ell)$. 
Nous disons qu'un tel \mo\ {\em admet une description finie}.
Le couple
$(\so{a_1,\ldots,a_k},\so{u_1,\ldots,u_\ell})$
est appelé
un \emph{\ipp fini}. L'\ipp  $\fq=(I,U)$ est fabriqué dans le but suivant: \emph{lorsqu'on localise
en $\cS(I,U)$, on obtient $U\subseteq \gA_\fq\eti$  et $I\subseteq \Rad(\gA_\fq)$.} 
De même, pour tout \idep $\fp$ tel que $I\subseteq\fp$ et $U\subseteq \gA\setminus \fp$, on a $U\subseteq \gA_\fp\eti$  et $I\subseteq \Rad(\gA_\fp)$. Le couple $\fq=(I,U)$ représente donc une information
partielle sur un tel \idep. Il peut être considéré comme une approximation de $\fp$. Ceci explique la terminologie d'\ipp.
\item (Bords de Krull)
\label{defZar2} Soient  $\gA$  un anneau commutatif, $x\in\gA$ et $\fa$ un \itf.
\begin{enumerate}
\item [$(1)$] Le \ix{bord supérieur de Krull} de $\fa$
dans $\gA$ est l'anneau quotient
\begin{equation}\label{eqBKAC}
\gA_\rK^{\fa}:=\gA/\JK_\gA(\fa)  \quad \hbox{où} \quad
 \JK_\gA(\fa):=\fa+(\sqrt{0}:\fa).
\end{equation}
On note  $\JK_\gA(x)$  pour $\JK_\gA(x\gA)$ et $\gA_\rK^{x}$
pour $\gA_\rK^{x\gA}$. Cet anneau est appelé le \emph{bord
supérieur de $x$ dans $\gA$}. \\
On dira  que $\JK_\gA(\fa)$ est
\emph{l'\id bord de Krull de $\fa$ dans $\gA$.}%
\index{ideal@idéal!bord de Krull}
\item [$(2)$] Le \ix{bord inférieur de Krull} de $x$ dans $\gA$ est l'anneau
localisé
\begin{equation}\label{eqBKAS}
\gA^\rK_{x}:=\SK_\gA(x)^{-1}\!
\gA \quad\hbox{où}\quad \SK_\gA(x)=x^\NN(1+x\gA).
\end{equation}
On dira que  $\SK_\gA(x)$ est le
\emph{monoïde bord de Krull de $x$} dans~$\gA$.%
\index{bord de Krull!idéal ---}
\end{enumerate}

\item 
Une version \gui{itérée} du \mo
$\SK_\gA(x)$ est le \mo \gui{bord de Krull itéré}
\begin{equation}\label{eqMonBordKrullItere}
\SK_\gA(\xzk):=
x_0^\NN(x_1^\NN\cdots(x_k^\NN (1+x_k\gA) +\cdots)+x_1\gA) + x_0\gA)
\end{equation}
Pour une suite vide, on définit $\SK_\gA()=\so{1}$.%
\index{monoide@monoïde!bord de Krull itéré}%
\index{bord de Krull!monoide --- it@monoïde --- itéré}%
 
%
\end{itemize}

\paragraph{Modules.}~ $M$, $N$ et $P$ désignent des \Amos. 
\begin{itemize}
\item Un ensemble $E$ est dit \emph{discret} s'il possède un test d'\egt, i.e. si l'on a explicitement: $\forall x,y\in E,\; x=y \hbox{ ou } x\neq y$.
\item  Un groupe (ou un module) $G$ est discret \ssi
il possède un test d'\egt à $0$.  
\item $M$ est dit \emph{\fdi}%
\index{module!fortement discret} si, pout tout sous-\mtf~$N$, le module quotient $M/N$ est discret.
\item  Pour $V=(v_1,\ldots ,v_n)\in M^n$,  on
appelle \ix{module des \syzys} \emph{entre les~$v_i$}, ou encore \emph{module des \syzys pour $V$}, le sous-\Amo
de $\Ae{n}$ noyau de l'\ali
$
\breve{V}:\Ae{n}\to M,\;   (\xn)\mapsto\som_ix_iv_i.
$
\item  Le module $M$ est  dit \ixc{cohérent}{module ---}
\index{module!cohérent} si tout sous-\Amo \tf est \pf. Autrement dit si pour tous $n\in\NN$ et $V\in M^n$,
le module des \syzys pour $V$ est \tf, i.e. si l'on~a:
\begin{equation}\label{eqMoCoh}\preskip.4em \postskip.4em
\formule{
\forall n\in\NN,\,\forall V\in M^{n{\times}1},\, \exists m\in\NN
,\, \exists G\in \Ae{m{\times}n},\,\forall X\in \Ae{1{\times}n}\,, \\[.3em]
\quad \quad
XV=0\quad \Longleftrightarrow\quad
\exists Y\in \Ae{1{\times}m},\; X=YG\;.}
\end{equation}
%
\item Un \Amo $P$ est dit \ixc{projectif}{module ---}
s'il vérifie la \prt suivante.
 Pour tous \Amos $M,\,N$, pour toute \ali surjective
\hbox{$\psi : M\rightarrow N$} et  toute \ali $\Phi:P\rightarrow N$, il existe une \ali $\varphi : P\rightarrow M$   telle que
$\psi\circ \varphi= \Phi$.
$$\preskip.4em \postskip.4em 
\xymatrix {
                                       & M\ar@{>>}[d]^{\psi} \\
P\ar@{-->}[ur]^{\varphi} \ar[r]_{\Phi} & N \\
} 
$$
\item $M$ est dit \emph{\noe}%
\index{module!noethérien} si toute  suite infinie croissante (au sens large) de sous-modules \tf de $M$  contient deux 
termes consécutifs égaux.
\item \emph{Modules plats}.
\begin{itemize}
\itbu Une \emph{\syzy dans $M$} est donnée par $L\in \Ae {1\times n}$ et $X\in M^{n\times 1}$ qui vérifient $LX=0$.
\itbu On dit que \emph{la \syzy $LX=0$ s'explique dans~$M$} si l'on trouve un vecteur $Y\in M^{m\times 1}$ et une matrice $G\in \Ae {n\times m}$ qui vérifient:
$$\label{eqdef.plat}
LG=0\quad {\rm et } \quad GY = X\,.
$$
\itbu Le \Amo $M$ est appelé un {\em  \mpl}\index{module!plat}\index{plat!module ---} si  toute \syzy  
dans~$M$ s'explique dans $M$.
(En langage intuitif: s'il y a une \syzy entre \elts de $M$ ce n'est pas la faute 
au module.) 
\end{itemize}

%
%
\end{itemize}
\paragraph{Modules projectifs de type fini, groupe de Grothendieck.}~ 
\begin{itemize}
\item Un \Amo \ptf à $n$ générateurs est un module $M$ isomorphe à l'image d'une matrice idempotente $F\in\Mn(\gA)$ ($F^2=F$). On dit encore une matrice de projection. On a l'\iso $\Ae n= \Im F \oplus \Ker F\simeq M\oplus K$. 
\item On note $\GKO \gA$ l'ensemble des classes d'\iso de \mptfs sur $\gA$.
C'est un semi-anneau\index{semi-anneau}\footnote{Ceci signifie que la structure est donnée par une addition, commutative et associative, une multiplication, commutative, associative et distributive par rapport à l'addition, avec un neutre $0$ pour l'addition et un neutre $1$ pour la multiplication.
Par exemple $\NN$ est un semi-anneau.} 
pour les lois héritées de $\oplus$ et $\otimes$.
Le  \textsf{G}  de $\GKO$ est en hommage à Grothendieck.%
\label{NOTAGKO}
Tout \elt de $\GKO \gA$ peut être représenté par une matrice \idme à
\coes dans~$\gA$. Le produit défini dans $\GKO \gA$ donne par passage au quotient un
produit dans $\KO \gA$, qui a donc une structure d'anneau
commutatif.
\item Le \mo additif (commutatif) de $\GKO \gA$ n'est pas toujours régulier.
Pour obtenir un groupe, on symétrise le \mo additif $\GKO\gA$ et l'on obtient
le \ix{groupe de Grothendieck} que l'on note $\KO \gA$.\label{NOTAK0}
La classe du \mptf $P$ dans $\KO \gA$ se note $[P]_{\KO (\gA)}$, ou~$[P]_{\gA}$, ou même $[P]$ si le contexte le permet.  Pour un \idm $e$ on notera aussi $[e]$ à la place de~$[e\gA]$, lorsque le contexte est clair.
Avec cette notation l'\elt neutre pour la multiplication est~$[1]$. 
\item
Tout \elt de $\KO \gA$ s'écrit sous forme $[P]-[Q]$.
Les classes de deux \mptfs $P$ et $P'$ sont égales dans~$\KO \gA$ \ssi
il existe un entier $k$ tel que $P\oplus\Ae k\simeq P'\oplus\Ae k$.
On dit dans ce cas que $P$ et $P'$ sont \ixc{stablement isomorphes}{modules ---}.
\end{itemize}

\paragraph{Déterminants.}~ 
\begin{itemize}

\item Un endomorphisme $\varphi$ de $M$ est caractérisé par l'endomorphisme~$\psi$ de $\Ae n= \Im F \oplus \Ker F\simeq M\oplus K$ obtenu en prolongeant $\varphi$ par $\psi(K)=0$. Si $H$ est la matrice de $\psi$ on a $H=HF=FH$. Le déterminant de $\varphi$ est par définition celui de l'endomorphisme $\theta$ obtenu en prolongeant~$\varphi$ par $\theta\frt K=\Id_K$. On a donc $\det \varphi= \det (H+\In-F)$.
\item Le \pol $\rR{M}(X)=\det(X\,\Id_{M})$ est appelé le {\em \polmu} du module $M$. Cette terminologie est justifiée par le fait que pour un
module libre de rang $k$ le \polmu est égal à~$X^k$.%
\index{rang!polynome@\pol --- d'un \mptf}%
\label{NOTAPolmu}%
\index{polynome@\pol!rang}
 Dans le cas où l'anneau possède des idempotents $\neq 0,1$ un polynôme rang a  la forme générale $\sum_{k=0}^n r_k X^k$, pour un système fondamental d'idempotents orthogonaux $(r_0,\dots,r_n)$.

\item Un \Amo $M$ est dit \emph{quasi libre}%
\index{quasi libre!module ---}\index{module!quasi libre} 
s'il est isomorphe à une somme directe finie d'idéaux $\gen{e_i}$ avec les $e_i$ \idms. C'est un cas particulier de \mptf.
   Pour un système fondamental d'idempotents orthogonaux $(r_0,\dots,r_n)$ le module quasi libre $\bigoplus_{k=1}^n(r_k\gA)^k$ a pour déterminant le \polmu $\sum_{k=0}^n r_k X^k$.
   
\end{itemize}

\paragraph{L'anneau des rangs généralisés $\HO(\gA)$.}~  
\begin{itemize}
\item On note $\HOp  \gA $ l'ensemble des classes d'\iso des modules
quasi libres sur $\gA$. Deux modules quasi libres stablement isomorphes sont isomorphes, de sorte que $\HOp \gA $  
s'identifie à un sous-semi-anneau de $\KO\gA$.
\item Un élément général de
$\HOp (\gA)$ s'écrit $\sum_{k=0}^n k [r_k]$ pour un système fondamental d'idempotents orthogonaux $(r_0,\dots,r_n)$. Si l'anneau n'est pas trivial,  le sous-semi-anneau de $\HOp (\gA)$ engendré par 1 est isomorphe à $\NN.$
\item Le \mo additif de $\HOp (\gA)$ est régulier, et le groupe symétrisé est muni d'une structure naturelle d'anneau qui étend la structure de $\HOp (\gA)$. Cet anneau s'appelle l'\ixx{anneau des rangs}{(\gnes) de \mptfs} sur $\gA$, et l'on le note  $\HO(\gA)$. C'est aussi le sous-anneau de
$\KO\gA$ engendré par les classes $[e\gA]$ pour $e$ idempotent.  
\item 
Si $M$ est un \Amo \ptf on appelle \emph{rang (\gne)} de $M$ et l'on note
$\rg_\gA(M) $ ou $\rg(M)$ l'unique
\elt de $\HOp(\gA)$, classe d'un module quasi libre qui a le même \polmu que 
$M$.\index{rang!(généralisé) d'un \mptf}
Deux \mptfs stablement isomorphes $P$ et $P'$ ont même rang
puisque $\rg(P\oplus\Ae k)=k+\rg(P)$. 
En conséquence, le rang (\gne) des \mptfs définit un \homo
surjectif d'anneaux
$\rg_\gA:\KO \gA\to\HO \gA$. 
   
%
\end{itemize}

\section{Quelques résultats préliminaires}\label{secRappelThs}

\paragraph{Lemmes de base}

\CMnewtheorem{plgbase}{Principe \lgb de base}{\itshape}
\begin{plgbase}\label{FFRplcc.basic}\index{principe local-global de base} 
\emph{(recollement concret de solutions d'un \sli, CACM~II-2.3)}
Soient $S_1$, $\dots$, $S_n$ des \moco de $\gA$,  $A$
une matrice \hbox{de $\MM_{m,p}(\gA)$} et $C$ un vecteur colonne de~$
\Ae{m}$.
Alors \propeq
\begin{enumerate}
\item  {Le \sli $AX=C$ admet une solution dans $\gA^{p}$}.
\item  {Pour $ i\in\lrbn$
le \sli $AX=C$ admet une solution dans~$\gA_{S_i}^{p}$}.
\end{enumerate}
Ce principe vaut \egmt pour les \slis à \coes dans un \Amo $M$.
\end{plgbase}

\CMnewtheorem{lemNak}{Lemme de Nakayama}{\itshape}
\begin{lemNak}\label{FFRlemNaka}  
\emph{(Le truc du \deter, CACM~IX-2.1)}%
\index{Nakayama!Lemme de ---} \\
Soient $M$ un \Amo \tf et $\fa$ un idéal de $\gA$.
\begin{enumerate}
\item  Si $\fa\,M=M$, il existe $x\in\fa$ tel que $(1-x)\,M=0$.
\item  Si en outre $\fa\subseteq\Rad(\gA)$, alors $M=0$.
\item  Si $N\subseteq M$,  $\fa\,M+N=M$ et $\fa\subseteq\Rad(\gA)$, alors $M=N$.
\item  Si $\fa\subseteq\Rad(\gA)$ et si  $X\subseteq M$ engendre $M/\fa M$ comme
$\gA/\fa$-module, alors~$X$ engendre $M$ comme \Amo.
\end{enumerate}
\end{lemNak}

\CMnewtheorem{lemDKME}{Lemme de \DKM}{\itshape}
\begin{lemDKME}\label{FFRlemdArtin}%
\index{Lemme de Dedekind-Mertens} \emph{(CACM~III-21)}\\
Pour $f,g\in\AT$ avec $m\geq\deg g$ on a
\fbox{$\rc(f)^{m+1}\rc(g)=\rc(f)^m\rc(fg)$}.
\end{lemDKME}

\paragraph{Matrices}~ 

\CMnewtheorem{lemmininv}{Lemme du mineur inversible}{\itshape}
\begin{lemmininv}\label{FFRlem.min.inv}\index{Lemme du mineur inversible}
\emph{(Pivot généralisé, CACM~II-5.9)}\\
Si une matrice $A\in\MM_{q,m}(\gA)$ possède un mineur d'ordre $k\leq \min(m,q)$
inversible
, elle est \eqve à une matrice

\snic{
\cmatrix{
    \I_{k}   &0_{k,m-k}      \cr
    0_{q-k,k}&  A_1      },}

avec $\cD_r(A_1)=\cD_{k+r}(A)$ pour tout $r\in\ZZ$. 
\end{lemmininv}

\CMnewtheorem{lemli}{Lemme de la liberté}{\itshape}
\begin{lemli}\label{NOTAIkqm} \label{FFRlem pf libre}\index{Lemme de la liberté}\emph{(CACM~II-5-10)}\\
Considérons une matrice $A\in\MM_{q,m}(\gA)$ de \hbox{rang $\leq k$} avec $1\leq k\leq \min(m,q)$. Si la
matrice~$A$ possède un mineur d'ordre~$k$ inversible, alors elle est
\eqve à la matrice
$$\preskip.0em \postskip.4em 
\I_{k,q,m}\;=\;
\cmatrix{
    \I_{k}   &0_{k,m-k}      \cr
    0_{q-k,k}&     0_{q-k,m-k}      }. 
$$
Dans ce cas, l'image, le noyau et le conoyau de $A$ sont libres,
respectivement de rangs $k$, $m-k$ et $q-k$. En outre l'image et le noyau
possèdent des supplémen\-taires libres.
\\ 
Si $i_1$, $\ldots$, $i_k$ (resp. $j_1$, $\ldots$, $j_k$) sont
les numéros de lignes (resp. de colonnes) du mineur inversible, alors
les colonnes $j_1$, $\ldots$, $j_k$ forment une base du module $\Im A$, et
$\Ker A$ est le sous-module défini par l'annulation des formes \lins
correspondant aux lignes $i_1$, $\ldots$, $i_k$.
\end{lemli} 

\begin{theorem}\label{FFRpropFactDirRangk}
\emph{(Matrices \lnls, CACM~II-5.14, II-5.20 \hbox{et II-5.26})}
Soit $A\in \MM_{n,m}(\gA)$. Notons $q=\min(m,n)$.
\begin{enumerate}
\item \Propeq
\begin{enumerate}
\item La matrice $A$ est de rang $k$.
\item La matrice $A$ est \lnl de rang $k$.
\item La matrice $A$ a même image qu'une \mprn de rang~$k$ de $\Mn(\gA)$.
\end{enumerate}
\item \Propeq
\begin{enumerate}
\item La matrice $A$ est \lnl.
\item Il existe une matrice $B\in\MM_{m,n}(\gA)$ telle que  $ABA=A$.
\item Le module $\Im A$ est  facteur direct dans $\Ae{n}$  (i.e., c'est l'image d'une \mprn $F\in\Mn(\gA)$)
\item Chaque \idd $\cD_k(A)$ est engendré par un \idm.
\item La matrice $A$ devient \nl après \lon en des
  \eco convenables.
\item \label{IFDi}
Chaque \idd $\cD_k(A)$ est engendré par un \idm~$e_k$ et,  après \lon en des  \eco convenables,  la
matrice  $A$ devient \eqve à 
$\Diag(e_1,e_2, \ldots ,e_{q})$  (complétée par des
lignes ou colonnes nulles si $n>m$ ou $m>n$).
\end{enumerate}
\end{enumerate}
\end{theorem}


\paragraph{Modules}~ 

\begin{lemma} \label{FFRfactchangesgrmpf} \emph{(Changement de \sgr, \cref{section IV-1})}\\
Lorsque l'on  change de \sgr fini pour un \mpf, les \syzys entre
les nouveaux \gtrs forment de nouveau un module \tf.
\end{lemma}

\begin{lemma}
\label{FFRlem pres equiv} \emph{(CACM~lemme IV-1.1. Matrices qui présentent un même module)} 
Pour deux matrices $G\in \MM_{q,m}(\gA)$ et $H\in \MM_{r,n}(\gA)$,
 \propeq
\begin{enumerate}\itemsep0pt
\item  Les matrices $G$ et $H$ présentent \gui{le même} module, \cad leurs conoyaux sont
isomorphes.
\item  Les deux matrices  de la figure ci-dessous sont
\elrt \eqves.
\item  Les deux matrices  de la figure ci-dessous sont
\eqves.
\end{enumerate}
\vspace{-10pt}
\begin{figure}[htbp]
{\small
\begin{center}
\[
\begin{array}{c|p{35pt}|p{15pt}|p{25pt}|p{20pt}|}
\multicolumn{1}{c}{} & \multicolumn{1}{c}{m} & \multicolumn{1}{c}{r} &
\multicolumn{1}{c}{q} & \multicolumn{1}{c}{n} \\
\cline{2-5}
\vrule height20pt depth13pt width0pt q\; & \hfil G \hfil &\hfil
0\hfil &\hfil 0\hfil &\hfil 0\hfil \\
\cline{2-5}
\vrule height15pt depth8pt width0pt r\; & \hfil 0  &\hfil
${\rm I}_{r}$\hfil & \hfil 0\hfil &\hfil0\hfil \\
\cline{2-5}
\end{array}
\]
\vspace{6pt}
\[
\begin{array}{c|p{35pt}|p{15pt}|p{25pt}|p{20pt}|}
\cline{2-5}
\vrule height20pt depth13pt width0pt q\; & \hfil 0 \hfil&\hfil
0\hfil &\hfil $\I_{q}$\hfil &\hfil 0\hfil \\
\cline{2-5}
\vrule height15pt depth8pt width0pt r\; &\hfil 0 \hfil&\hfil
0\hfil & \hfil 0\hfil & \hfil H\hfil \\
\cline{2-5}
\end{array}\]

\end{center}}
\caption{\emph{Les deux matrices}}
\label{fig}
\end{figure}
\end{lemma}

\begin{proposition}
\label{FFRpropPfSex} \emph{(CACM~IV-4.2, quotient d'un \mpf)}
Soit $N$ un sous-\Amo de $M$ et $P=M/N$.
\begin{enumerate}
\item Si $M$ est \pf et $N$ \tf, $P$ est \pf.
\item Si $M$ est \tf et $P$ \pf, $N$ est \tf.
\item Si $P$ et $N$ sont \pf, $M$ est \pf. Plus \prmt, si $A$ et
$B$ sont des \mpns pour $N$ et $P$, on a une \mpn 
$D=\blocs{1}{.8}{.7}{.6}{$A$}{$C$}{$0$}{$B$}$
pour $M$.
\end{enumerate}
\end{proposition}

\begin{theorem}
\label{FFRprop quot non iso} \emph{(Surjectif implique bijectif, CACM~IV-5.2)}
\\ Soit $M$ un \Amo \tf.
\begin{enumerate}
\item  Soit $\varphi\,:\,M\rightarrow M$ une \ali surjective.
Alors, $\varphi$ est un \iso et son inverse est un \pol en $\varphi$.%
\item  Si un
quotient $M/N$ de $M$ est isomorphe à $M$,  alors~$N=0$.
\item  Tout \elt $\varphi$ \iv à droite
dans $\End_\gA(M)$ est \iv, et son inverse est un \pol en $\varphi$.
\end{enumerate}
\end{theorem}

\CMnewtheorem{lemSchanuel}{Lemme de Schanuel}{\itshape}
\begin{lemSchanuel}\label{FFRlemSchanuel}  
\emph{[CACM~V-2.1 et V-2.8]}
\begin{enumerate}
\item On considère deux \Alis surjectives de même image
$$\preskip.2em \postskip.1em 
P_1  \vers{\varphi_1}   M  \rightarrow  0  \;\;\;\hbox{et} \;\;\;
   P_2  \vers{\varphi_2}   M  \rightarrow  0
$$
avec $P_1$ et $P_2$ \pros.
\begin{enumerate}
\item  Il existe des \isos réciproques
$$\preskip.2em \postskip-.4em
\alpha ,\,  \beta:P_1\oplus P_2\to P_1\oplus P_2
$$ 
tels que
$$\preskip.0em \postskip-.2em
(\varphi_1\oplus 0_{P_2})\circ \alpha=0_{P_1} \oplus \varphi_2\;\hbox{
et}\; \varphi_1\oplus 0_{P_2}=(0_{P_1} \oplus \varphi_2)\circ \beta.
$$
\item Notons $K_1=\Ker\varphi_1$ et $K_2=\Ker\varphi_2$. 
Par restriction de~$\alpha$ et~$\beta$ on obtient des \isos réciproques entre
$K_1\oplus P_2$ et $P_1\oplus K_2$.
\end{enumerate}
\item \emph{(Lemme de Schanuel)}
On considère deux  suites exactes:
\vspace{-1mm}
\[\preskip.2em \postskip.2em
\begin{array}{ccccccccc}
0 &\rightarrow& K_1& \vers{j_1} & P_1& \vers{\varphi_1} & M& \rightarrow& 0   \\
0 &\rightarrow& K_2& \vers{j_2} & P_2& \vers{\varphi_2} & M& \rightarrow& 0
 \end{array}
\]
avec les modules $P_1$ et $P_2$ \pros. Alors, $K_1\oplus P_2\simeq K_2\oplus P_1.$
\end{enumerate}
\end{lemSchanuel}

\begin{theorem} 
\label{FFRprop Fitt ptf 2}\relax \label{prop Fitt ptf 1} 
\emph{(Structure locale et \idfs d'un \mptf, CACM~V-6.1)}
\begin{enumerate}
\item Un \Amo $P$ \pf est \ptf \ssi ses \idfs sont (engendrés par des)
idempotents.
\item Plus \prmt pour la réciproque, supposons qu'un \Amo~$P$ \pf ait
ses \idfs \idms, et que la matrice $G\in\gA^{q\times n}$ soit une matrice de \pn de $P$,
correspondant à un \sys de~$q$ \gtrs. 
\\
Notons $f_h$ l'\idm qui engendre $\cF_h(P)$, \hbox{et~$r_h:=f_h-f_{h-1}$}. 
\begin{enumerate}
\item $(r_0,\ldots,r_q)$ est un \sfio.
\item Soient $t_{h,j}$ un mineur d'ordre $q-h$ de $G$, et $s_{h,j}:=t_{h,j}r_h$. Alors, le~$\gA[1/{s_{h,j}}]$-module $P[1/{s_{h,j}}]$ est libre de rang~$h$.
\item Les \elts  $s_{h,j}$ sont \com.
\item On a $r_k=1$ \ssi la matrice $G$ est de rang $q-k$.
\item Le module $P$ est \ptf.
\end{enumerate}
\item  En particulier, un \mptf devient libre après \lon en un nombre
fini d'\eco.
\end{enumerate}
\end{theorem}


\CMnewtheorem{lemnbgl}{Lemme du nombre de \gtrs local}{\itshape}
\begin{lemnbgl} 
\label{FFRlemnbgtrlo} \emph{(CACM~IX-2.4)}\\
Soit $\gA$ un anneau et  $M$ un \Amo \tf.
\begin{enumerate}
\item Supposons $\gA$  local. 
\begin{enumerate}
\item [a.\phantom{*}] Le module $M$  est engendré par $k$
\elts \ssi
son \idf $\cF_k(M)$ est égal à $\gA$.
\item [b.\phantom{*}]  
Si en outre $\gA$ est \dcd et $M$ \pf, le module admet une matrice de \pn dont tous les \coes sont
dans l'idéal maximal $\Rad\gA$.
\end{enumerate}
\item En \gnl, pour $k\in\NN$, \propeq
\begin{enumerate}
\item [a.\phantom{*}] L'\idd $\cF_k(M)$ est égal à $\gA$.
\item [b.\phantom{*}] Il existe des \eco $s_j$ tels que après extension des
scalaires
à chacun des $\gA[1/s_j]$, $M$ est engendré par $k$ \elts.
\item [c.\phantom{*}] Il existe des \moco $S_j$ tels que
chacun des  $M_{S_j}$ est engendré par $k$ \elts.
\item [d*.] Après \lon en n'importe quel \idep,  $M$ est engendré par
$k$ \elts.
\end{enumerate}
\end{enumerate} 
\end{lemnbgl}

\paragraph{Principe \lgb}~ 

\begin{plcc}
\label{plcc.ptf}
{\em  (Recollement concret de \prts de finitude pour les modules, CACM XV-2.2)}
Soient $S_1$, $\dots$, $S_n$  des \moco de $\gA$  et 
$M$ un \Amo.   Alors on a les \eqvcs suivantes.
\begin{enumerate}\itemsep0pt\parsep0pt
\item  $M$ est  \tf \ssi chacun des $M_{S_i}$
est un $\gA_{S_i}$-\mtf.
\item $M$ est \pf \ssi chacun des $M_{S_i}$ est un~$\gA_{S_i}$-\mpf.
\item  $M$ est plat \ssi chacun des $M_{S_i}$ est un~$\gA_{S_i}$-\mpl.
\item \label{iptfplcc.ptf}\relax $M$ est \ptf \ssi chacun des $M_{S_i}$ 
est un~$\gA_{S_i}$-\mptf.
\item $M$ est projectif de rang $k$ \ssi chacun des $M_{S_i}$ 
est un~$\gA_{S_i}$-module projectif de rang~$k$.
\item $M$ est cohérent \ssi chacun des $M_{S_i}$ est un
$\gA_{S_i}$-\coh.
\item  $M$ est \noe \ssi chacun des $M_{S_i}$ est un
$\gA_{S_i}$-module \noe.
\end{enumerate}
\end{plcc}

\begin{plcc}
\label{plcc.propaco}\relax
{\em  (Recollement concret de \prts des anneaux commutatifs, CACM XV-2.3)} 
Soient $S_1$, $\dots$, $S_n$  des \moco et~$\fa$ un \id de~$\gA$.
Alors on a les \eqvcs suivantes.
\begin{enumerate}\itemsep0pt\parsep0pt
\item  $\gA$ est  \coh \ssi chaque  $\gA_{S_i}$ est  \coh.
\item  $\gA$ est  \lsdz \ssi chaque  $\gA_{S_i}$ est  \lsdz.
\item  $\gA$ est  \qi \ssi chaque  $\gA_{S_i}$ est  \qi.
\item  $\gA$ est réduit \ssi chaque  $\gA_{S_i}$ est réduit.
\item  L'\id $\fa$ est \lop \ssi chaque  $\fa_{S_i}$ est \lop.
\item  $\gA$ est  \ari \ssi chaque  $\gA_{S_i}$ est  \ari.
\item  $\gA$ est  de Pr\"ufer \ssi chaque  $\gA_{S_i}$ est  de Pr\"ufer.
\item  L'\id $\fa$ est \icl \ssi chaque  $\fa_{S_i}$ est \icl.
\item  $\gA$ est normal \ssi chaque  $\gA_{S_i}$ est normal.
\item  $\gA$ est de \ddk $\leq k$ \ssi 
chaque  $\gA_{S_i}$ est de \ddk~$\leq k$.
\item  $\gA$ est  \noe \ssi chaque  $\gA_{S_i}$ est  \noe.
\end{enumerate}
\end{plcc}

\mni {\bf Machinerie locale-globale à \ideps.}\label{MethodeIdeps} (CACM section XV-5)\\
{\it  Lorsque l'on relit une \prco, donnée pour le
cas d'un \alo \dcd, avec un anneau $\gA$ arbitraire,
que l'on
considère au départ comme $\gA=\gA_{\cS(0;1)}$ et qu'à chaque
disjonction (pour un \elt $a$ qui se présente au cours du calcul
dans le cas local)

\sni \centerline{$ a\in\Ati\; \hbox{ou }\; a\in \Rad(\gA)$,}

\sni on remplace l'anneau \gui{en cours} $\gA_{\cS(I,U)}$ par les deux anneaux
 $\gA_{\cS(I;U,a)}$ et~$\gA_{\cS(I,a;U)}$ (dans chacun desquels le calcul
peut se poursuivre), on obtient, à la fin de la relecture, une
famille finie d'anneaux
$\gA_{\cS(I_j,U_j)}$ avec les \mos~\hbox{${\cS(I_j,U_j)}$} \com et $I_j,\;U_j$
finis. Dans chacun de ces anneaux, le calcul a été poursuivi avec succès
et a donné le résultat souhaité.
}

\paragraph{Dimension de Krull}~ 

Rappelons qu'en \clama la \ddk
d'un anneau est $-1$ \ssi l'anneau n'admet
pas d'\idep, ce qui signifie qu'il est trivial.

Le \tho suivant donne alors en \clama une \carn inductive \elr
de la \ddk d'un anneau commutatif.

\begin{theoremc} \emph{(CACM, XIII-2.2)}
\label{thDKA} Pour un anneau commutatif  $\gA$
et un entier $k\geq 0$ \propeq
\begin{enumerate}
\item  La \ddk de $\gA$ est $\leq k$.
\item  Pour tout $x\in \gA$ la \ddk de $\gA_\rK^{x}$ est
$\leq k-1$ (voir \paref{defZar2} pour $\gA_\rK^{x}$).
\item  Pour tout $x\in \gA$ la \ddk de $\gA^\rK_{x}$ est
$\leq k-1$.
\end{enumerate}
\end{theoremc}
NB. Ceci est un \tho de \clama qui ne peut pas admettre de \prco. \eoe

La \dfn \cov de la \ddk s'ensuit, une \dfn \eqve est la suivante.

\begin{definition} \label{defiddk} \emph{(\Dfn \cov de la \ddk)}
\begin{itemize}

\item Une suite $(x_0,\ldots ,x_k)$ dans  $\gA$ est dite
\ixd{singulière}{suite} si $0\in \SK_\gA(\xzk)$.
Autrement dit s'il existe $a_0$, \ldots, $a_k\in \gA$ et
$m_0$, \ldots, $m_k\in \NN$ tels que:
\hsd$x_0^{m_0}(x_1^{m_1}(\cdots(x_k^{m_k} (1+a_k x_k) +
\cdots)+a_1x_1) + a_0x_0) =  0.$

\item La \ddk de $\gA$ est dite $\leq k$
\ssi toute suite $(\xzk)$ dans $\gA$  est \sing (cette définition est équivalente à la définition usuelle en mathématiques classiques). On  note $\Kdim\gA\leq k$ pour dire que la \ddk de $\gA$ est $\leq k$.

\end{itemize}

\end{definition}
%

\begin{plcc}\label{FFRthDdkLoc}
\emph{(CACM~XIII-3.2, pour la \ddk)}\\
Soient $S_1$, $\ldots$, $S_n$ des \moco d'un anneau $\gA$ et $k\in\NN$.
\begin{enumerate}
\item Une suite est singulière dans $\gA$ \ssi elle est singulière dans
chacun des $\gA_{S_i}$.
\item L'anneau $\gA$ est de \ddk $\leq k$ \ssi les $\gA_{S_i}$ sont de \ddk $\leq k$.
\end{enumerate}
\end{plcc}

\begin{theorem}
\label{FFRthKroH} \emph{(\Tho de \KRN-Heitmann,
avec la dimension de Krull, non \noe, CACM~XIV-1.3)}
\begin{enumerate}
\item Soit $n\geq 0$.
Si $\Kdim \gA <n$ et $b_1$, \dots, $b_n\in\gA$,  il existe
$x_1$, \dots, $x_n$ tels que pour tout $a\in\gA$,
$\DA(a,b_1,\dots,b_n) = \DA(b_1+ax_1,\dots,b_n+ax_n)$.
\item En conséquence, dans un anneau de dimension de Krull $\leq n$, tout
\itf a m\^{e}me nilradical qu'un \id engendré par au plus $n+1$ \elts.
\end{enumerate}
\end{theorem}

\newpage \thispagestyle{empty}  

\chapter
{Théorie de la profondeur}\label{chapProfondeur}

\minitoc

\sibook{
\Intro

Ce chapitre  est une première introduction à 
la théorie de la profondeur d'un module relativement à un \itf.
Il est basé sur l'exposé qui en est donné dans le livre  remarquable \emph{Finite Free Resolutions} de Northcott. Nous avons  simplifié l'exposition \gnle, et comme dans l'ouvrage [CACM], toutes nos \dems sont \covs. 

Dans ce chapitre, les \srgs occupent une place prépondérante,
mais il s'avère qu'elles sont à considérer éventuellement dans des extensions \polles de l'anneau de base.
Cette remarque (qui est due à Hochster) est la clé qui permet de rendre 
la théorie vraiment élégante et \gnle (sans hypothèse \noee notamment).

\smallskip La  section \ref{seccor} introduit les \syss d'\elts \cor
et quelques unes de leurs \prts. Il s'agit d'un \gnn fort utile de la
notion de \sys d'\eco. Le principal intérêt est que la \prt d'être \cor est plus faible, et ses conséquences s'appliquent donc plus souvent.

\smallskip La  section \ref{secPrelimReg} donne 
des lemmes clés concernant les \srgs qui vont être
utilisés intensivement par la suite.

\smallskip La section \ref{secProfondeur}   introduit
 la notion de profondeur d'un module relativement à un \id
 et développe un certain nombre de \prts fondamentales de cette notion.

\smallskip La section \ref{secRecolle2} traite des \plgcs directement reliés à la théorie de la profondeur.

\smallskip 
\dots

\fbox{\`A COMPL\'ETER} 

}


\Grandcadre{Dans ce chapitre, $E$ et $F$ désignent des \Amos.}

\section
{Systèmes d'\elts \cor}
\label{seccor}
\subsec{Un \plg relatif à la régularité}

Dans l'ouvrage [CACM] les principales variantes du \plg étaient basées
sur les familles d'\eco, \cad les familles finies qui engendrent l'\id $\gen{1}$.
Une notion un peu plus faible est suffisante pour les questions de 
régularité: il s'agit des familles finies qui engendrent un \id \Erg. 

\begin{definition} \label{defiCoreg}~
\begin{enumerate}
\item Une famille finie $(\an)$ d'un anneau $\gA$ est appelée un \ix{système d'\elts coréguliers} si l'\id $\gen{\an}$ est fidèle\footnote{\`A ne pas confondre avec la notion de suite corégulière introduite par Bourbaki, comme notion duale de celle de \srg.}.%
\index{coreguliers@\cor!elem@\elts~---}
\item On dit qu'un \elt $a\in\gA$ est  \emph{\Erg} (ou  \emph{\ndz pour~$E$}) si:  

\snic{\forall x\in E,\; 
(ax=0\;\Longrightarrow\; x=0).}

\item Un \id  $\fa\subseteq \gA$ est dit \emph{\Erg} si: 

\snic{\forall x\in E,\;
(\fa\,x=0\;\Longrightarrow\; x=0).}
\end{enumerate}%
\index{E-reg@\Erg!ideal@\id ---}\index{E-reg@\Erg!elem@\elt ---}
\end{definition}

On utilisera souvent de manière implicite le lemme facile suivant, qui est une variante du lemme~[CACM, V-7.2], lequel était énoncé pour les \syss d'\eco.

\begin{fact} \emph{(Lemme des \lons \core successives)}
\label{factLocCasreg} \\
Si $s_1$, \ldots, $s_n$ sont des \elts \cor de $\gA$ et si pour chaque $i$,
on a des \elts
$$
\frac{s_{i,1}}{{s_{i}}^{n_{i,1}}},\; \ldots ,\;
\frac{s_{i,k_i}}{{s_{i}}^{n_{i,k_i}}},
$$
 \cor dans $\gA[1/s_i]$,
alors les $s_{i}s_{i,j}$ sont \cor dans~$\gA$.
\end{fact}

\begin{lemma} \label{lemCoreg1} \emph{(Une première astuce $(a,b,ab)$)}\\
On suppose que l'\id $\gen{a,c_2, \dots, c_n}$ et l'\elt $b$ 
sont \Ergs. Alors
l'\id~$\gen{ab,c_2, \dots, c_n}$ est \Erg. 
\end{lemma}
%
\begin{proof}
Soit $x\in E$ tel que $abx= c_1x=\cdots=c_nx=0$.\\
Alors~$abx=c_1bx=\cdots=c_nbx=0$,
donc~$bx=0$, donc~$x=0$.
\end{proof}

On a le corolaire  \imd suivant\footnote{On aurait pu aussi remarquer que pour $q$ assez grand, l'\id $\gen{\an}^{q}$, qui \hbox{est \Erg}, est contenu dans l'\id $\gen{a_1^p,\dots,a_n^p}$.}. 
%
\begin{lemma} \label{lemCoreg2}
Soient $\gen{\an}$ un \id \Erg et $p\in\NN$. \\
Alors l'\id $\geN{a_1^p,\dots,a_n^p}$
est \Erg. 
\end{lemma}

Notons que l'affirmation \gui{$\fb$ est \Erg} est stable par \lon 
lorsque~$\fb$ est \tf. Ceci donne l'implication dans le sens direct pour le point \emph{\ref{i3plcc.regularite}} dans le \plg qui suit.

\begin{plcc}\label{plcc.regularite} \emph{(Pour la régularité)}\\
Soient $E$ un \Amo, $b$, $a_1$, \dots, $a_n\in\gA$, $x\in E$,  et~$\fb$ un
\itf. On suppose que l'\id $\gen{\an}$ est \Erg.
\begin{enumerate}
\item \label{i1plcc.regularite} On a $x=0$   \ssi $x=0$ dans chaque $E[1/a_i]$.
\item \label{i2plcc.regularite} L'\elt $b$ est \Erg \ssi il est  $E[1/a_i]$-\ndz pour chaque~$i$.
\item  \label{i3plcc.regularite} L'\id $\fb$ est \Erg \ssi il est  $E[1/a_i]$-\ndz pour chaque~$i$.
\end{enumerate}
\end{plcc}
%
\begin{proof}
\emph{\ref{i1plcc.regularite}.} Si $x=0$ dans  $E[1/a_i]$ il y a un exposant $k_i$ tel que $a_i^{k_i}x=0$ dans~$E$. On conclut par le lemme \ref{lemCoreg2}  (avec le module~$\gA x$) que~$x=0$.

\noindent \emph{\ref{i3plcc.regularite}.} Supposons que $\fb$ est $E[1/a_i]$-\ndz pour chaque~$i$, et~$\fb\,x=0$.
Alors~$x=0$ dans chaque~$E[1/a_i]$, donc~$x=0$ par le point \emph{\ref{i1plcc.regularite}.}
\end{proof}
%

\subsec{\Tho de McCoy et variantes}

Comme application du \plgref{plcc.regularite}, nous donnons une nouvelle \dem d'un \tho de McCoy (\Cref{II-5.22~point~\emph{2.}}).

\CMnewtheorem{thoMccoy}{\Tho de McCoy}{\itshape}
\begin{thoMccoy} \label{propInjIdd} 
Une matrice $M\in\Ae{m\times n}$ 
est injective \ssi
l'\idd $\cD_n(M)$ est fidèle.  
\end{thoMccoy}
%
\begin{proof}
L'implication \gui{si} est simple. Montrons que si la matrice~$M$ est injective,
l'\id~$\cD_n(M)$ est fidèle. On raisonne par \recu sur le nombre de colonnes.
Puisque~$M$ est injective, les \coes de la première colonne
(qui représente l'image du premier vecteur de base), engendrent un \id fidèle.
Par le \plgrf{plcc.regularite}, il suffit donc de démontrer que~$\cD_n(M)$ est fidèle sur l'anneau~$\gA_a=\gA[1/a]$, où~$a$ est un \coe  de la première colonne. 
\\
Sur cet anneau il est clair que la matrice $M$ est \eqve à une matrice de la forme 
$\blocs{.4}{.6}{.4}{.9}{$1$}{$0$}{$0$}{$N$}$. 
En outre~$N$ est injective
donc par \hdr l'\id $\cD_{n-1}(N)$ est fidèle sur~$\gA_a$. 
Enfin~$\cD_{\gA_a,n-1}(N)=\cD_{\gA_a,n}(M)$.    
\end{proof}

\rems ~\\ 
1) La \dem donne aussi que si~$m<n$ et~$M$ injective, alors l'anneau
est trivial. En effet à chaque étape de \recu, quand on remplace~$M$
par~$N$ la différence~$m-n$ reste constante. Donc si~$m<n$ on obtient \gui{à l'initialisation} une application injective de~$\Ae0$ dans~$\Ae{n-m}$
ce qui implique~$1=0$ dans~$\gA$. Ceci est conforme à l'énoncé \gnl
de la proposition~\ref{propInjIdd}, car pour $m<n$, $\cD_n(M)=0$, et si $0$
est un \elt \ndz, l'anneau est trivial.
 
2) On trouve souvent dans la littérature le \tho de McCoy énoncée comme suit, sous forme contraposée (en apparence).
\\
 \emph{Si l'idéal n'est pas fidèle, l'application n'est pas injective}. 
 \\
 Ou encore de manière plus précise. 
\\
\emph{Si un \elt $x\in\gA$ non nul
annule $\cD_n(M)$, il existe un vecteur colonne non nul $C\in\Ae{m\times 1}$
tel que $MC=0$}.
\\
Malheureusement, cet énoncé ne peut être démontré qu'avec la logique classique, et l'existence du vecteur $C$ ne peut pas résulter d'un \algo \gnl. Voici un contre-exemple, bien connu des numériciens. Si $M$ est une matrice à \coes réels avec $m<n$, on ne sait pas produire un vecteur non nul dans son noyau tant que l'on ne conna\^{\i}t pas le rang de la matrice. Par exemple pour $m=1$ et $n=2$, on donne deux réels $(a,b)$, et l'on cherche un couple $(c,d)\neq (0,0)$ tels que $ac+bd=0$. Si le couple~$(a,b)$ est a priori indiscernable du couple $(0,0)$, il est impossible de fournir un couple~$(c,d)$ convenable tant que l'on n'a pas élucidé si
$\abs a+\abs b$ est nul ou non.
\sibook{\\
Des variantes \covs de la contraposée sont proposées dans les exercices \ref{exoMcCoyContr1} et \ref{exoMcCoyContr2}.}
\eoe
%


\subsec{Un (autre) lemme de McCoy}

\hum{il semble qu'il y a une preuve bien plus simple basée sur le \plgref{plcc.regularite}
\\
auquel cas il faudra renvoyer la preuve actuelle et la remarque qui suit en exercice
}

La proposition suivante reprend pour l'essentiel le 
corolaire \Cref{III-2.3}, qui était énoncé pour $E=\gA$.

\begin{proposition} \label{lemMcCoy} \emph{(Lemme de McCoy)}\\
Soient $E$ un \Amo et $f$  un \pol de $\AuX$. \Propeq
\begin{enumerate}
\item L'\elt $f$ est $E[\uX]$-\ndz, i.e. $(0_{E[\uX]}:f)_{E[\uX]}=0_{E[\uX]}$.
\item L'\id $\rc(f)$ est \Erg, i.e. $\big(0_{E}:\rc(f)\big)_{E}=0_{E}$.
\item Pour tout $x\in E$, $\,x\, f =0\;\Rightarrow\; x=0$, i.e. $(0_{E[\uX]}:f)_{E}=0_{E}$. 
\end{enumerate}
\end{proposition}
%
\begin{proof}
Les points \emph{2} et \emph{3} ont la même signification, et
il est clair que le point \emph{3} est un cas particulier du point \emph{1.}
Ce qui est vraiment l'objet du lemme est l'implication  \emph{3} $\Rightarrow$
\emph{1.} 

\noindent Il suffit de voir le cas des \pols en une seule variable
(le cas \gnl peut s'en déduire en utilisant l'astuce de \KRA).

\noindent Comme indiqué dans le corolaire \Cref{III-2.3}, 
c'est une conséquence facile du lemme de \DKM \ref{FFRlemdArtin}. 
\sibook{\Llec peut aussi consulter l'exercice~\ref{exoDDMcCoy}
et sa solution pour une preuve directe.}
\end{proof}

\rem En fait nous utilisons la variante \gui{pour les modules} du lemme de Dedekind-Mertens: \emph{Pour $f\in\AT$ et $g\in E[T]$ avec $m\geq\deg g$ on a}

\snic{\rc(f)^{m+1}\rc(g)=\rc(f)^m\rc(fg).}

\snii Ici le contenu $\rc(g)$ du \pol $g\in E[T]$ est le sous-\Amo de $E$
engendré par les \coes de $g$.
La variante pour les modules est une conséquence du lemme lui-même.
En effet ce lemme peut être vu comme une famille 
d'\idas reliant les \coes de $f$
et de $g$. Ces \idas sont toutes \lins par rapport aux \coes de~$g$.
Or toute \ida \lin par rapport à certaines des \idtrs peut être évaluée en spécialisant les \idtrs dans un anneau arbitraire $\gA$, sauf 
certaines des \idtrs \gui{\lins} qui peuvent être spécialisées en des \elts d'un \Amo arbitraire $E$.   
\eoe

\section
{Suites $E$-régulières, suites de Kronecker}
\label{secPrelimReg}


\subsec{Suites $E$-régulières}

On a défini la notion \srg dans un anneau en \cref{IV-2.3}. On la \gns comme suit.
\begin{definition} \label{defSERG}
Une suite finie $(a_1,\dots,a_n)$ dans $\gA$ est dite \emph{\Erge} si les conditions suivantes sont satisfaites:\index{E-regu@\Erge!suite ---}%
\index{suite!E-regu@\Erge}
\begin{itemize}
\item $a_1$ est \Erg, 
\item $a_2$ est \Erg modulo $a_1E$, i.e.,  $a_2$  est ($E\sur{a_1E}$)-\ndz,
\item  $a_3$ est \Erg modulo $a_1E+a_2E$,
\item  $\dots$
\item  $a_n$ est \Erg modulo $a_1E+\cdots+a_{n-1}E$.
\end{itemize}
 \end{definition}

 Nous n'imposons pas de condition négative du style $\gen{\an}E\neq E$.

La \dfn d'une \sErg pourrait aussi être dite comme suit: 
\begin{itemize}
\item $a_1$ est \Erg, 
\item $a_2$ est \Erg sur l'anneau $\aqo\gA{a_1}$,
\item  $a_3$ est \Erg sur l'anneau $\aqo\gA{a_1,a_2}$,
\item  $\dots$
\end{itemize}
    
\rdb
Dans le cas où $E=\gA$, on retrouve les notions 
usuelles d'\elt \ndz, d'\id fidèle
et de suite \ndze (\cref{IV-2.3}).%
\index{E-reguliere@\Erg!suite ---}\index{suite!E-reguliere@\Erg}%
\index{reguliere@régulière!suite ---}\index{suite!reguliere@régulière}%

\begin{examples} \label{exaregnoreg} \label{regnoreg}
{\rm  
 
 1) Dans $\aqo{\gk[X,Y,Z]}{X(Y - 1)} = \gk[x,y,z]$, la suite $\big(y, z(y-1)\big)$ est 
régulière, mais pas celle obtenue en changeant l'ordre des deux termes
(on a en effet $y(z-x)=0$ avec $z-x\neq 0$ dans le quotient considéré, si $\gk\neq 0$).
\\
 Ce phénomène désagréable sera contourné au moment de
définir la profondeur.

2) Si la suite $(a_1,a_2)$ est $\gA$-\ndze, alors d'une part $a_1$ est \ndz, et d'autre part, $a_1a_2$ est le ppcm de $a_1$ et $a_2$, \cad:
 $\gen{a_1}\cap\gen{a_2}=\gen{a_1a_2}$. Inversement~si $a_1$ et $a_2$ sont \ndzs, et si $a_1a_2$ est le ppcm de $a_1$ et $a_2$, alors les suites $(a_1,a_2)$ et $(a_2,a_1)$  sont $\gA$-\ndzes.
 \\
 Notons qu'il ne suffit pas que $a_1$ et $a_2$ aient pour pgcd $1$, comme le montre l'exemple des suites $(x^{3},x^{2})$ et $(x^{2},x^{3})$ qui ne sont pas \ndzes dans l'anneau $\ZZ[x^{2},x^{3}]$ (la première parce que $x^4 x^2 \equiv 0 \bmod x^3$ alors \hbox{que $x^{4}\not\equiv0 \bmod x^3$} car $\aqo{\ZZ[x^{2},x^{3}]}{x^3}=\aqo{\ZZ[x^{2}]}{x^6}$ , la seconde parce que $x^3 x^3 \equiv 0 \bmod x^2$ alors \hbox{que $x^{3}\not\equiv0 \bmod x^2$}
 car $\aqo{\ZZ[x^{2},x^{3}]}{x^2}=\aqo{\ZZ[x^{3}]}{x^6}$).

3) Un exemple particulièrement significatif, qui relie l'existence de \srgs  et la \ddk, est le suivant. \rdb\label{exempleProfKdim}
Soit $\gk$ un \cdi non trivial, $\fa$ un \itf de $\gA=\kXn$ et $\gB=\gA/\fa$. \hbox{Soit $d\in\lrb{-1..n}$} la \ddk de $\gB$ (on sait que c'est un entier bien défini). 
Alors, on va voir que si $d<n$, l'\id $\fa$ contient une \srg de longueur~\hbox{$n-d$}.
\\
Tout d'abord, si $d=-1$, l'\id $\fa$ contient $1$, il  contient des \srgs (formées de $1$) arbitrairement longues.\\
Si $d\in\lrb{0..n-1}$, en mettant l'\id en position de \Noe on obtient après un \cdv adéquat

\snic{\gA=\kYn$ et $\gB=\kYd[y_{d+1},\dots,y_n],}

avec les $y_i$ entiers sur $\gC=\kYd$.
Pour chaque $i>d$ soit $f_i(Y_i)\in\fa$ un \pol \unt de $\gC[Y_i]$ ($f_i$ annule 
$y_i$). 
\\
Montrons que la suite $(f_n,\dots,f_{d+1})$ est une \srg de $\gA$.
Le \pol $f_{n}$ est \unt en $Y_n$ donc \ndz. On considère le quotient 

\snic{\gB_1=\aqo\gA{f_{n}}=\gC[Y_{d+1},\dots,Y_{n-1},y_n]=\gA_1[Y_{n-1}].}

Le \pol $f_{n-1}$ est \unt en $Y_{n-1}$  donc \ndz dans $\gA_1[Y_{n-1}]$.\\
Et ainsi de suite.

4) Une \gnn partielle est la suivante. \rdb\label{exempleGenProfKdim}
Soit $\gk$ un anneau arbitraire,~$\fa$ un \itf de $\gA=\kXn$ et $\gB=\gA/\fa$. On suppose que~$\gB$ est finie sur $\gk$. 
Alors $\fa$ contient une \srg de longueur $n$.
La \dem est la même que pour l'exemple précédent, mais ici
la \gui{position de \Noe} est inutile puisqu'on peut prendre $\gC=\gk$. 
\eoe
}\end{examples}

\begin{lemma} \label{lemRegProdMod}
Soient $E$ et $F$ des \Amos et $(\ua)=(\an)$ une suite dans $\gA$. Alors
$(\ua)$ est $E\times F$-\ndze \ssi elle est \Erge et \Frge. 
\end{lemma}
%
\facile
%

\begin{lemma} \label{lem1sqlindep}
 Si  $(\an)$ est une suite  \Erge,  chaque $a_i$ est \Erg modulo les $a_j\neq a_i$.
 \end{lemma}
%
\begin{proof}
Comme $(a_i,\dots,a_n)$ est \Erge modulo  $\gen{(a_j)_{j<i}}$, il suffit de démontrer le cas $a_1$. Nous donnons la démonstration pour  $n=4$. Une démonstration par récurrence en bonne et due forme est possible, et plus courte.
Considérons une \syzy $\sum_{i=1}^4a_ix_i=0$. 
Nous devons montrer que $x_1\in \gen{a_2,a_3,a_4}E$. 

Comme  $a_4$ est \Erg modulo $\gen{a_1,a_2,a_3}$ on peut écrire $x_4=a_1y_1+a_2y_2+a_3y_3$, de sorte que
$$a_1(x_1+a_4y_1)+a_2(x_2+a_4y_2)+a_3(x_3+a_4y_3)=0.$$ Comme  $a_3$ est \Erg modulo $\gen{a_1,a_2}$ on peut écrire $x_3+a_4y_3=a_1z_1+a_2z_2$, de sorte que
$a_1(x_1+a_4y_1+a_3z_1)+a_2(x_2+a_4y_2+a_3z_2)=0$. Comme  $a_2$ est \Erg modulo $a_1$ on peut écrire $x_2+a_4y_2+a_3z_2=a_1t_1$. Donc
$a_1(x_1+a_4y_1+a_3z_1+a_2t_1)=0$. Comme  $a_1$ est \Erg, $x_1+a_4y_1+a_3z_1+a_2t_1=0$ et $x_1\in \gen{a_2,a_3,a_4}E$. 
\end{proof}

\subsec{Suite qui reste exacte modulo un \elt \ndz}

\begin{proposition} \label{propRegSex}~

\begin{enumerate}
\item 
On considère une \sex de \Amos

\snic {
 E \vvers {\varphi} F \vvers {\psi} G \vvers {\eta} H,
}

\snii  et $a\in \gA$ un \elt \Hrg.
 Alors la suite de \Amos

\snic {
 E/aE \vvers {\ov\varphi} F/aF \vvers {\ov\psi} G/aG 
}

\snii est exacte.

\item 
En conséquence, si l'on a une \sex

\snic {
 E_n \vvers {\varphi_n} E_{n-1} \vvers { } \;\cdots \cdots \;
 \vvers {\varphi_{m+1}} E_{m},
}

\snii et si $a\in\gA$ est $E_j$-\ndz pour $j\in\lrb{m..n-3}$,
alors on a aussi la \sex

\snic {
 E_n/aE_n \vvers {\ov{\varphi_n}} E_{n-1}/aE_{n-1} \vvers{ }  \; \cdots\cdots \;
 \vvers {\ov{\varphi_{m+2}}} E_{m+1}/aE_{m+1}.}
\end{enumerate}

\end{proposition}
%
\begin{proof}
Soit $y\in F$ tel que $\psi(y)\in aG$. On doit trouver un $x\in E$ tel
que 

\snic{\varphi(x)\equiv y\mod aF.}

\snii
On écrit $\psi(y)-az=0$ avec $z\in G$.
On applique $\eta$ ce qui donne $a\eta(z)=0$, donc $\eta(z)=0$, et il existe
$u\in F$ avec $\psi(u)=z$. Donc $\psi(y-au)=0$ et il existe $x\in E$ tel que
$\varphi(x)=y-au$.
\end{proof}
%

\begin{corollary} \label{lemHilBur2}
 Soit $H$ un \Amo et  $a\in \gA$ un \elt \Hrg. 
\begin{enumerate}
\item On suppose que l'on a une \sex
$0 \to F \vvers {\psi} G \vvers {\eta} H$.
Alors on a une nouvelle \sex 
$0 \to F/aF \vvers {\ov\psi} G/aG$.
%
%
\item Si la suite $0 \to F \vvers {\psi} G \vvers {\eta} H\to 0$ est exacte,
alors la suite

\snic{0 \to F/aF \vvers {\ov\psi} G/aG\vvers {\ov\eta} H/aH\to 0}

\snii
est aussi une \sex.
\end{enumerate}
\end{corollary}
%
\begin{proof}
\emph{2.} On remarque que $a$ est $\so{0}$-\ndz, on applique
deux fois le premier point de la proposition \ref{propRegSex}, et la surjectivité de $\ov\eta$ est claire.
\end{proof}
%

\subsec{Un lemme de simplification}

Le lemme suivant est une simple variante du lemme de Nakayama.
\begin{lemma} \label{lemSimplifidele}
Soit $\fb\subseteq \gA$ un \id et $E$ un \Amo \tf tel \hbox{que $E=\fb E$}.
\begin{enumerate}
\item Il existe $b\in\fb$ tel que $(1-b)\,E=0$.
\item \emph{(Lemme de simplification)} Si $E$ est fidèle, $1\in \fb$.
\item  \emph{(Nakayama)} Si $\fb\subseteq \Rad(\gA)$, $E=0$.
\end{enumerate} 
\end{lemma}
%
\begin{proof} Le point \emph{1} est un \gui{truc du \deter} donné dans le lemme
de Nakayama~\ref{FFRlemNaka}.
Le reste suit.
\end{proof}
%
\subsec{Changement d'anneau de base et \srg}

Il est clair qu'une suite dans un produit fini d'anneaux est \Erge \ssi
son \eds à chaque facteur est \Erge. 

\begin{plcc} \label{propLocSrg} \emph{(Pour les \srgs)}\\
Les \sErgs restent \Erges par \lon, et une suite qui est \Erge après \lon en des \moco est \Erge.
\end{plcc}

Plus \gnlt, on a le résultat suivant.

\begin{proposition} \label{propChgBasSrg}
Soit $\rho:\gA\to\gB$ une \Alg, $(\ua)=(\an)$ dans~$\gA$ et $E$ un \Amo.
Notons $(\uap)=\big(\rho(a_1),\dots,\rho(a_n)\big)$ \hbox{et $E'=\rho\ist(E)$}.
\begin{enumerate}
\item Si $\gB$ est plate sur $\gA$ et $(\ua)$ est une \sErg, alors $(\uap)$ est une suite $E'$-\ndze. 
\item Si $\gB$ est \fpte sur $\gA$, alors $(\ua)$ est une \sErg \ssi $(\uap)$ est une suite $E'$-\ndze.
\end{enumerate} 
\end{proposition}

%
\begin{proof}
En effet le fait que la suite est \ndze signifie que l'on a des suites exactes 

\snic{\!\! 
\begin{array}{cccccccccccc} 
 0&  \to & E   & \vvers{\cdot\times a_ 1}   & E     \\[.1mm] 
 0&  \to &  E/ a_ 1E  & \vvers{\cdot\times {a_ 2}}   &  E/ a_ 1E  \\[.1mm] 
 0&  \to &  E  / (a_ 1E+a_ 2E) & \vvers{\cdot\times {a_ 3}}   &E  /(a_ 1E+a_ 2E)    
\\[.1mm] 
&&\vdots&&\vdots
 \end{array}
}

\snii
On applique donc les résultats concernant la platitude (ou fidèle platitude) et les suites exactes.
\end{proof}

\hum{Si $\fa=\Ann_\gA(E)$, on peut considérer les anneaux $\gA'=\gA/\fa$
et $\gB'=\gB/\fa\gB$, de sorte que $E$ est un module sur $\gA'$.
La proposition précédente peut être énoncée avec $\gA'$ et $\gB'$
en place de $\gA$ et $\gB$. Elle est sans doute un peu plus \gnle sous cette forme. \`A vérifier. Exo?}

Dans la suite nous utilisons souvent des morphismes  
$\rho:\gA\to \gB$ comme changements d'anneau de base, et
nous commettons l'abus de langage suivant: étant donné un \Amo $E$,
\emph{nous disons qu'une suite
dans $\gB$ est \Erge pour signifier qu'elle est $\rho\ist(E)$-\ndze}.

\subsec{Lemmes d'échange, \sKr \Erge}

Dans la suite, sauf mention contraire, $E$ désigne un \Amo arbitraire.

\begin{lemma} \label{lemEchSeqReg} \emph{(Premier lemme d'échange)}
\begin{enumerate}
\item Soit $(b_1,b_2)$ une \sErg dans $\gA$. Si $b_2$ est \Erg, la suite~$(b_2,b_1)$ est \egmt~\Erge.
\item Plus \gnlt si $(b_1,\dots,b_m)$ est une \sErg  et $k\in\lrb{2..m}$, 
la suite obtenue en permutant $b_{k-1}$ et $b_k$ est aussi \Erge \ssi
l'\elt $b_k$ est \ndz modulo  $\sum_{j: j\leq  k-2}b_jE$.
\end{enumerate}
\end{lemma}
%
\begin{proof} 
\emph{1.} L'\elt  $b_2$ est \Erg par hypothèse, et $b_{1}$ est \Erg modulo~$b_2$ par calcul direct (ou d'après le lemme \ref{lem1sqlindep}).
\\
\emph{2.} L'implication directe est donnée par le lemme \ref{lem1sqlindep}.\\ 
Pour l'implication réciproque, il suffit de traiter le cas $m=k=2$,
ce qui est le point \emph{1.}  
\end{proof}
%

\begin{corollary} \label{cor0lemEchSeqReg} 
\emph{(Deuxième lemme d'échange)}\\
Soit $\gA$ un anneau \coh, $E$ un \mpf et $(b_1,\dots,b_m)$ une \sErg 
contenue dans $\Rad(\gA)$.
Alors toute suite obtenue par permutation des $b_i$ est \egmt \Erge.
\end{corollary}
%
\begin{proof}
D'après le premier lemme d'échange il suffit de montrer que \hbox{pour $k>1$,
$b_k$}  est \ndz
modulo le module 
 $$
   E_{k-2}= \som_{j: j\leq  k-2}b_jE.
 $$ 
Or en passant au quotient
par $\fa_{k-2}=\gen{(b_i)_{i\leq  k-2}}$ on a de nouveau un anneau  \coh avec les $b_j$
dans $\Rad(\gA\sur{\fa_{k-2}})$. Il suffit donc de faire la \dem pour $k=2$, \cad de voir que $b_2$
est \Erg. 

\noindent Supposons $b_2x=0$. Puisque $b_2$ est \Erg modulo $b_1$ on a $x=b_1y$
pour un certain $y$. On a $b_2b_1y=0$, donc $b_2y=0$. On vient de montrer
\hbox{que $(0:b_2)_{E}=b_1\,(0:b_2)_{E}$}. Comme $b_1\in\Rad(\gA)$ et $(0:b_2)_{E}$ est \tf,
on conclut par le lemme de Nakayama \ref{lemSimplifidele} que $(0:b_2)_{E}=\gen{0}$.
\end{proof}

\hum{La \dem fonctionne à l'identique dans un anneau gradué, en remplaçant le Nakayama pour les \mtfs par le Nakayama homogène, pour une suite d'\elts de degrés $\geq 1$. Il faudra incorporer cela dans le texte, après avoir brièvement introduit les anneaux et modules gradués.}

\rems 1) Considérons une \srg $(\an)$ dans un \alo \coh \dcd. Alors si $1\notin\gen{\an}$ le lemme d'échange s'applique. Mais naturellement  la \srg $(1,0)$ ne satisfait pas le lemme d'échange, à moins que l'anneau soit trivial.

\smallskip \noindent 2)  Le corolaire \ref{cor0lemEchSeqReg} est le plus souvent énoncé dans la littérature classique en suppossant que l'anneau $\gA$ est \noe. En \clama, tout anneau \noe est \coh, mais
ce n'est pas vrai en \gnl d'un point de vue \cof. La version \cov du \tho classique serait donc de demander que l'anneau soit \coh \noe. 
En fait, on voit que la \noet n'a rien à voir dans l'affaire. 
Dans l'exemple qui suit par contre,
on voit que l'hypothèse de \cohc est absolument \ncr.  
\eoe

\begin{example} \label{exacor0lemEchSeqReg} 
{\rm  Cet exemple a été construit avec l'aide d'un logiciel de calcul formel\footnote{Voir l'article \cite{CT2018}. En partant de l'anneau $\aqo{\gk[a,b,z_0]}{b_0}$ dans lequel $b$ divise $0$, on a cherché à forcer la suite $(a,b)$ à être régulière. Le logiciel nous indique d'abord que $b_0\in\gen{a}$ sans que $z_0\in\gen{a}$. Nous devons donc ajouter $z_1$ avec $az_1=z_0$ pour que $b$ soit régulier modulo $a$. Puis le logiciel nous indique  que $ab_1=0$ sans que $b_1=0$. D'où la nécessité d'ajouter $b_1=0$ pour que $a$ soit régulier. Et ainsi de suite. En suivant ce que nous dit de faire le logiciel pour construire le contre-exemple, nous pouvons aussi construire la \dem du corolaire sans avoir à faire aucun effort d'imagination. Inversement, on aurait pu construire le contre-exemple en regardant de près la \dem du corolaire.}. On considère un \cdi $\gk$ et l'anneau 
$$\preskip.4em \postskip.4em 
\gA=\aqo{\gk[a,b,z_0,z_1,\dots]}{b_0,z_0-az_1,b_1,z_1-az_2,b_2,\dots} 
$$ 
avec des \idtrs $a,b,(z_n)_{n\in\NN}$. On vérifie que l'\id $\fm=\gen{a,b}$ est maximal (le quotient est isomorphe à $\gk$). En posant $\gB=\gA_{1+\fm}$ on constate que $(a,b)$ une \srg dans $\Rad\gB$. Cependant $b$ n'est pas \ndz car $z_0$ est non nul dans $\gB$. Ainsi $\gB$ est \ncrt non \coh, et l'on  constate d'ailleurs que $\Ann_\gB(b)=\gen{(z_n)_{n\in\NN}}$ n'est pas \tf. 
A fortiori, $\gB$ n'est pas \noe: la suite des \ids $\fa_k=\gen{(z_k)_{k\leq n}}$ est strictement croissante.
\\
Si l'on note $\gC=\aqo{\gk[a,b,z]}{b}$ et $\varphi:\gC\to\gC$ le morphisme d'anneaux qui envoie $(a,b,z)$ sur $(a,b,az)$ on voit que $\gA$ est la limite inductive du diagramme $\big((\gC_n)_{n\in\NN},(\varphi:\gC_n\to\gC_{n+1})_{n\in\NN}\big)$ avec tous les $\gC_n$ égaux à $\gC$. 
\\
Notons que cet exemple \gui{purement \agq} est plus simple que celui habituellement cité  \cite[Dieudonné, 1966]{Die1966}. \eoe
}\end{example}

\begin{corollary} \label{corlemEchSeqReg} 
\emph{(Troisième lemme d'échange)}\\
Soit $\fa$ un \itf et $(b_1,\dots,b_m)$ une \sErg contenue
dans~$\fa$. Soit $f\in\AuX$ un \pol (en une ou plusieurs variables) de contenu  égal à $\fa$. \\
Alors on obtient une \sErg dans l'\id~$\fa[\uX]$ de~$\AuX$  en insérant~$f$ à n'importe quelle position dans la suite $(b_1,\dots,b_{m-1})$. 
\end{corollary}
%
\begin{proof}
On commence par montrer que la suite $(b_1,\dots,b_{m-1},f)$ est \ndze.
\\ Posons $E'=E/\gen{b_1,\dots,b_{m-1}}E$.
L'\id $\fa$ contient l'\elt $b_m$ qui est~\hbox{$E'$-\ndz}, donc $\fa$ est  $E'$-\ndz. Donc, par le lemme de McCoy~\ref{lemMcCoy}, le \pol $f$ est un \elt $E'$-\ndz de $\AuX$.

\noindent Ensuite on peut déplacer $f$ progressivement de droite à gauche dans la suite précédente
pour l'amener à la position voulue, ceci grâce au lemme d'échange \ref{lemEchSeqReg}. En effet l'argument que l'on vient de donner pour montrer que~$f$ est \Erg 
modulo $\gen{b_1,\dots,b_{m-1}}$ montre aussi que $f$ est \Erg 
modulo $\gen{b_1,\dots,b_{k}}$ pour tout $k<m$, et donc l'hypothèse du 
lemme~\ref{lemEchSeqReg} est satisfaite. 
\end{proof}

De là on déduit le \tho donnant une \sErg générique.

\begin{definition} \label{defiSGNQ} \emph{(Suite de Kronecker)}
\\
Soit $\fa$ un \itf de $\gA$.
\begin{itemize}
\itbu Un  \pol  dont l'idéal contenu est $\fa=\rc_\gA(f)$ est appelé un \emph{\pKr attaché à~$\fa$}.%
\index{polynome de Kro@\pKr!attaché à l'\id $\fa$}
\itbu Si des \pols $f_1$, \dots, $f_m\in\AuX$, sont tous de contenu~$\fa$, et si chacun porte sur un jeu de variables distinct des autres, on dit que la suite $(f_1,\dots,f_m)$ est  une \emph{\sKr de longueur $m$ associée à l'\itf $\fa$}.
%
%
\end{itemize}
\end{definition}

\begin{theorem} \label{thSRGNQ} \emph{(Suite de Kronecker \Erge)}
\\
Soit $\fa$ un \itf de $\gA$. 
\begin{enumerate}
\item Supposons que $\fa$  contient une \sErg de longueur $m$, éventuellement après une extension \fpte.\\
Alors toute {\sKr de longueur $m$ associée à $\fa$} est \Erge (après extension des scalaires à~$\AuX$).
\item Pour deux \sKrs de longueur $m$ associées à l'idéal $\fa$, l'une est \Erge \ssi l'autre est \Erge. 
\end{enumerate}
\end{theorem}
%
\begin{proof}
\emph{1.}
Il suffit d'appliquer $m$ fois le corolaire \ref{corlemEchSeqReg},
en tenant éventuellement compte du  point~\emph{2} 
de la proposition~\ref{propChgBasSrg}, concernant les extensions \fptes.

\snii
\emph{2.} Cela résulte du point \emph{1.}
\end{proof}
%


\goodbreak
\section[Profondeur, \prts fondamentales
]{La profondeur d'un module relativement à un \itf}
\label{secProfondeur}

Dans cette section nous introduisons la notion de profondeur d'un \Amo~$E$
relativement à un \itf $\fa$ et nous démontrons quelques \prts fondamentales reliées à cette notion.

\subsec{Ce que l'on veut réaliser}

Une idée de base est que si l'\id $\fa$ contient une \sErg de longueur $k$
alors la profondeur de $E$ relativement à $\fa$ est supérieure 
ou égale à $k$, ce que l'on 
notera $\Gr_\gA(\fa,E)\geq k$.
Par exemple si un \itf~$\fa$ contient un \elt \Erg, on a $\Gr_\gA(\fa,E)\geq 1$.

La notation $\Gr_\gA(\fa,E)$ est celle de Northcott, qui dit \gui{True Grade} à la place de profondeur (depth). 
\sibook {Plus loin, nous utiliserons
la notation~\hbox{$\Prf_\gA(\fa,E)$} pour la profondeur définie de manière homologique, avant de démontrer que les deux notions coïncident.} 

Une autre idée de base est que la notion de profondeur doit être
invariante par \eds \fpte, et en particulier par extension
à un anneau de \pols, ce qui donne l'\eqvc:

\snic{\Gr_\gA(\fa,E)\geq k \quad \hbox{ \ssi } \quad \Gr_\AuX(\fa[\uX],E[\uX])\geq k}

\snii
Comme un \id \Erg $\fa=\gen{\an}$ contient, après
\eds à l'anneau de \pols $\AT$, l'\elt \Erg $f=\sum_{k=0}^{n-1}a_{k+1}T^k$ (lemme de McCoy), on adopte la \dfn que 

\snic{
\Gr_\gA(\fa,E)\geq 1 \quad \hbox{ \ssi } \quad \fa  \hbox{ est \Erg,}
}

\snii
car cela équivaut au fait qu'après une \eds à un anneau de \pols,
l'\id $\fa$ contient un \elt \Erg.

Enfin une troisième idée de base est d'avoir une \dfn récursive de la profondeur,
comme pour les \srgs. Plus \prmt si $b$ est un \elt \Erg de $\fa$, alors
pour tout entier $k$, on doit avoir l'\eqvc

\snic{\Gr_\gA(\fa,E)\geq k+1 \;\Longleftrightarrow\; \Gr_\gA(\fa,E/bE)\geq k
}

\hum{rajouter ? $k\geq 1$ et: $\exists b$ $\Leftrightarrow$ $\forall b$}

\smallskip On remarque ici que si une \dfn de $\Gr_\gA(\fa,E)\geq k$ est possible
selon les exigences précédentes (pour les entiers $k\geq 1$), alors elle est 
sans ambiguité. Autrement dit les exigences en question sont suffisamment contraignantes. 
En effet, vu le \thref{thSRGNQ} et les contraintes posées, la profondeur 
de $E$ relativement à $\fa$ doit être $\geq k$ \ssi une \sKr de longueur $k$ attachée à l'\id est \Erge.

Le \pb est donc de savoir si une \dfn obéissant aux contraintes
précédentes est possible.
\sibook{Pour qu'elle soit possible il faut 
notamment que pour deux \elts \Ergs $b$ et $b'$
de $\fa$, l'\id $\fa$ soit $E/bE$-\ndz \ssi il est $E/b'E$-\ndz.
C'est l'objet de l'exercice \ref{exoregreg}.} 

\subsec{Définition de la profondeur à la Hochster-Northcott}

La \dfn suivante, motivée par la discussion précédente,  est rendue possible par le \thref{thSRGNQ}.

\begin{definota} \label{defiProfNor}\label{defiProfMon} Soit  $E$ un \Amo.
\begin{enumerate}
\item \emph{(Profondeur d'un module relativement à un \id)}\\ 
Soit  $\fa$ un \itf de $\gA$, $E$ un \Amo et $k\geq 1$.
\begin{itemize}
\itbu  On dit que l'\id \emph{$\fa$ est $k$ fois \Erg} si une \sKr de longueur $k$
attachée à $\fa$ est \Erge (cela ne dépend pas de la \sKr choisie).
\itbu  On dit aussi\footnote{Eisenbud parle de la \prof de $\fa$ sur $E$, et Matsumura de la $\fa$-\prof de~$E$. La terminologie adoptée ici est celle de Bourbaki, mais on la transgresse un peu plus loin.} que \emph{la profondeur de~$E$
relativement à $\fa$ est supérieure ou égale à $k$}.
\itbu  On écrit alors \fbox{$\Gr_\gA(\fa,E)\geq k$}.
\itbu  On dit (abusivement) que \emph{l'\id $\fa$ est de \prof $\geq k$}, pour signifier \emph{que $\Gr_\gA(\fa,\gA)\geq k$}.
\itbu  Pour $k=0$ on dit que $\Gr_\gA(\fa,E)\geq 0$ est toujours vrai.
\itbu  On note $\Gr_\gA(\fa)$ \hbox{pour $\Gr_\gA(\fa,\gA)$}.
\end{itemize}

\item \emph{(\'Eléments $k$ fois \cor)}
\begin{itemize}
\itbu  On note $\Gr_\gA(\an,E)$ pour  $\Gr_\gA\big(\!\gen\an\!,E\big)$.
\itbu On dit que des \elts $s_1$, \dots, $s_n$ sont   \emph{$k$ fois \cor}
(ou, forment un \sys $k$ fois \ndz) lorsque $\Gr_\gA(\sn)\geq k$.%
\index{coreguliers@\cor!elem@\elts $k$ fois ---}
\hum{pas très joli, mais comment éviter les phrases à rallonge?}
\end{itemize}

\item \emph{(Système de \mos $k$ fois~\hbox{\Erg})} 
\begin{itemize}
\itbu Un \sys~$(S_1, \dots, S_n)=(\uS)$ de \mos est dit  \emph{$k$ fois \Erg} si pour tous~$s_i\in S_i$ $i\in \lrbn$, on a~$\Gr_\gA(s_1,\dots,s_n,E)\geq k$.
\itbu  Si $E=\gA$, on dit simplement que le \sys~$(\uS)$ est \emph{$k$ fois \ndz}.
\itbu  Pour $k=1$, on dira que \sys $(\uS)$ est \Erg ou encore 
que les~$S_i$ sont des \mos  $E$-\cor
\itbu  Pour $k=\infty$ (voir le \thref{lemProfInfinie}),
on retrouve les \moco si $E=\gA$, et l'on dit que les \mos sont  $E$-\com lorsque \hbox{$\gen{s_1,\dots,s_n}E=E$} pour tous $s_i\in S_i$ $i\in \lrbn$.
\end{itemize}
\end{enumerate}
\index{E-reg@\Erg!ideal k@\id $k$ fois ---}%
\index{profondeur!d'un module relativement à un \itf}%
\index{profondeur!d'un \itf}%
\end{definota}

Concernant la notation $\Gr_\gA(\fa,E)=\infty$ voir le \thref{lemProfInfinie}. 

Notons que si $E\neq 0$, on a $\Gr_\gA(0,E)=0$.

\smallskip Du point de vue \cof, la phrase \gui{$\Gr_\gA(\fa,E)\geq k$} est bien définie, mais la profondeur n'est pas
toujours un entier bien défini\footnote{Une discussion analogue a eu lieu  dans \Cref{section XIII-2} concernant la dimension de Krull.}. En conséquence, une in\egt \fbox{$\Gr_\gB(\fb,F)\geq \Gr_\gA(\fa,E)$} est en fait une abréviation pour l'implication suivante:
$$\fbox{$\forall m\in\NN\quad \Gr_\gA(\fa,E)\geq m\;\Longrightarrow\; 
\Gr_\gB(\fb,F)\geq m.$}
$$


\begin{fact} \label{lemSdirEtProf} 
Soient $E$ et $F$ deux \Amos et $\fa$ un \itf. \\
Alors $\Gr_\gA(\fa,E)\geq k$ et 
 $\Gr_\gA(\fb,F)\geq k$ \ssi  $\Gr_\gA(\fa,E\times F)\geq k$. 
 En abrégé on  écrit 
 
 \snic{\fbox{$\Gr_\gA(\fa,E\times F)= \inf(\Gr_\gA(\fa,E),\Gr_\gA(\fa,F))$}.}
 
En particulier, \hbox{on a $\Gr_\gA(\fa)=\Gr_\gA(\fa,\Ae{n})$} si $n\geq 1$. 
\end{fact}
%
%
\begin{proof}
Résulte du lemme \ref{lemRegProdMod}.
\end{proof}

Pour un \id arbitraire $\fb$, on pourrait définir: $\Gr_\gA(\fb,E)\geq k$ \ssiz$\fb$ contient un \itf $\fa$ tel que $\Gr_\gA(\fa,E)\geq k$. Mais il y aura alors un \pb avec le cas de la profondeur infinie évoquée dans le \thref{lemProfInfinie}.
Cette \gnn, que nous n'utiliserons pas, serait justifiée par le
point \emph{2} dans le \tho suivant.

\begin{theorem} \label{factdefProf}
Soient  $\fa$, $\fa'$ deux \itfs et un module $E$.
\begin{enumerate}
\item On a $\Gr_\gA(\fa,E)\geq m$ \ssi après éventuellement une extension \fpte,
l'\id $\fa$ contient une \sErg de longueur $m$.
\item Si $\fa\subseteq \fa'$, alors $\Gr_\gA(\fa,E)\leq \Gr_\gA(\fa',E)$.
\end{enumerate}    
\end{theorem}
%
\begin{proof}
\emph{1.} On applique le point \emph{1} du \thref{thSRGNQ}.\\
\emph{2.} Conséquence du point \emph{1.}
\end{proof}

On verra par contre que lorsque $E\subseteq F$ les relations entre $\Gr_\gA(\fa,E)$, $\Gr_\gA(\fa,F)$ et $\Gr_\gA(\fa,F/E)$ sont assez subtiles 
(\thref{thSESPrf}).

\subsec{Profondeur $\geq 2$ et pgcd fort}

Soit $\fa=\gen{\an}$ un \itf. On a $\Gr(\fa,E)\geq 1$ \ssi $\fa$
est \Erg, \ssi $f(\uX)=\sum_ia_iX_i$ est \Erg. Supposant que $\Gr(\fa,E)\geq 1$,
on s'intéresse ensuite à la signification de l'in\egt $\Gr(\fa,E)\geq 2$: on demande  que~$\fa$ soit $E/fE$-\ndz (plus \prmt, que~$\fa$ soit 
$E[\uX]/fE[\uX]$-\ndz), ou ce qui revient au même,
que~$\fa$ soit \Erg après \eds de $\gA$ \hbox{à $\aqo\AuX f$}. 
En voici une \carn \gui{homologique}
simple très utile.

Soient $\bma=(\an)$ et $\bmv=(\vn)$  deux suites de même longueur 
respectivement dans 
$\gA$ et $E$.
Elles sont dites \emph{proportionnelles} si $a_iv_j=a_jv_i$ pour
tous~$(i,j)$. Lorsque $E=\gA$, cela signifie que $\bma\vi \bmv=0$ dans $\Vi^{2}\Ae n$.\index{suites proportionnelles}\index{proportionnelles!suites ---}

Le théorème suivant montre que la notion de profondeur $\geq 2$ ne dépend que de la structure multiplicative de l'anneau (et de celle du module relativement à l'anneau). 

\begin{theorem} \label{lemGRa>1} \emph{(\Carn de la profondeur $\geq 2$)}\\ 
Pour $\bma=(\an)$,
 \propeq
\begin{enumerate}
\item $\Gr(\an,E)\geq 2$. 
%
\item L'\id $\gen{\an}$ est \Erg, et si $\bmv$ est une suite dans $E$ proportionnelle à $\bma$, alors il existe  $v\in E$ tel que $\bmv=
\bma\,v$.
\item Si $\bmv$ est une suite dans $E$ proportionnelle à $\bma$, alors il existe un unique~$v\in E$ tel que $\bmv= \bma\,v$.
\end{enumerate}
\end{theorem}
%
\begin{proof} Soit $f(\uX)=\sum_ia_iX_i$. La condition \emph{1} signifie
que l'\id $\fa$ est \Erg, et \hbox{que $\fa$} est $E[\uX]/fE[\uX]$-\ndz.

\emph{1} $\Rightarrow$ \emph{2.}
Soient $v_1$, \dots, $v_n\in E$ qui vérifient $a_iv_j=a_jv_i$ pour $i\neq j$. \\
On pose $g=\sum_iv_iX_i$
et on veut montrer que $g\equiv 0 \mod f$. Or $a_ig=fv_i $, donc $a_ig= 0 \mod f$ pour tout $i$, donc
$\fa\, g=0 \mod f$. 
\\
Or $\fa$ est $E[\uX]/fE[\uX]$-\ndz, donc $g$ est multiple de $f$.

\snii \emph{2} $\Rightarrow$ \emph{1.} Montrons que $\fa$ est $E[\uX]\sur{f E[\uX]}$-\ndz. Pour cela soit $g\in E[\uX]$ tel que les $a_ig$ soient nuls dans $E[\uX]\sur{f E[\uX]}$. On peut donc écrire $a_ig=fg_i$ pour des $g_i\in E[\uX]$. On a alors $a_ia_jg=fa_ig_j =fa_jg_i $ dans $E[\uX]$ et puisque
$f$ est \Erg, on obtient $a_ig_j=a_jg_i$. Or la \prt \emph{2} reste manifestement vraie
par passage de $E$ à $E[\uX]$, il existe donc $h\in E[\uX]$ tel que $g_i=a_ih$.
Ceci donne $a_i(g-fh)=0$ \hbox{dans $E[\uX]$}, et comme $\fa$ \hbox{est \Erg},
chaque \coe de $g-fh$ est nul, \cad $g=fh$  et \hbox{donc $g=0$} \hbox{dans $E[\uX]\sur{f E[\uX]}$}. 
\end{proof}
\exl Lorsque $\gA$ est un anneau à pgcd intègre, $\Gr(\fa)\geq 2$
signifie que le pgcd de $\fa$ est égal à $1$. En effet, dans un tel anneau tout pgcd d'\elts non nuls est un pgcd fort, on utilise alors le point \emph{5} du 
lemme~\ref{lemGCDFORT}.
\eoe

\hum{le \tho précédent est une bonne motivation pour le début du complexe de Koszul, il faudra peut-être y songer, par exemple en analysant
à la main la signification de $\Gr(\an,E)\geq 3$}

\begin{propdef} \label{defiGCDFORT}
Soient  $a_1$, \dots, $a_n\in\gA$ \cor, et soit~\hbox{$a\in\gA$}. 
On dit que~$a$ est un \ix{pgcd fort} de $(\an)$, si 
\begin{itemize}
\item $a$ divise les $a_i$, 
\item pour tous $y\in\Reg(\gA)$, $x\in \gA$, si $y$ divise les $xa_i$, alors
$y$ divise $xa$.
\end{itemize}

\smallskip 
Dans ce cas l'\elt $a$
est dans $\Reg(\gA)$ et  $a$ est le pgcd fort de tout \sgr fini de l'\id $\fa=\gen{\an}$. On dit aussi
que \emph{$a$ est un pgcd fort de l'\id $\fa$}.
\end{propdef}

\begin{lemma} \label{lemGCDFORT}
Soit $n\geq 1$, $a_1$, \dots, $a_n\in\gA$ \cor, et $u$, $v$, $x\in\Reg(\gA)$. 
\begin{enumerate}
\item Si $u$ et $v$ sont pgcds forts de $(\an)$, ils sont associés
\hbox{(i.e., $v\in u \Ati\!$)}.
\item Si $u$ est pgcd fort de $(\ua)=(\an)$, $xu$ est pgcd fort de $x(\ua)$.
\item \Propeq
\begin{enumerate}
\item $(u,v)$ admet un pgcd fort.
\item $(u,v)$ admet un ppcm $m$ (i.e. $\gen{u}\cap\gen{v}=\gen{m}$).
\end{enumerate}
Dans ce
cas, l'\elt $g=uv/m$ est un pgcd fort de $(u,v)$.
\item Si $\Gr_\gA(\ua)\geq 2$, alors $1$ est un pgcd fort de $(\an)$.
\item Si l'\id $\gen{\ua}$ contient un \elt \ndz, les deux \prts sont \eqves.
\end{enumerate}
\end{lemma}
%
\begin{proof}
\emph{1}, \emph{2} et \emph{3.} Laissés \alec.

\emph{4.} Soient $y\in\Reg\,\gA$ et $x\in \gA$ avec $y$ qui divise les $xa_i$, autrement dit 

\snic{y\,(\ub)=y\,(\bn)=x\,(\an)=x\,(\ua)}

\snii
pour des $b_i\in\gA$. On a $yb_ia_j=xa_ia_j$ donc $yb_ia_j=yb_ja_i$, et puisque $y$ est \ndz, 
les vecteurs $(\ub)$ et $(\ua)$
sont proportionnels. \\
Donc il existe $z$ tel que $(\ub)=z\,(\ua)$, d'où 
 $y\,(\ub)=yz\,(\ua)=x\,(\ua)$. Puisque les~$a_i$ sont \cor, cela implique $yz=x$,
 autrement dit $y$ divise $x.1$. En conclusion $1$ est bien  pgcd fort des $a_i$.
 
\emph{5.} 
On suppose que 1 est pgcd fort de $(\ua)$ et que l'\id $\gen{\ua}$ contient un
\elt \ndz. Quitte à ajouter cet \elt \ndz à la liste $(\ua)$, on peut
supposer que $a_1$ est \ndz. Soient donc des $b_i$ tels que $a_ib_j =
a_jb_i$. 
\\
En particulier $a_1b_j = a_jb_1$, donc $a_1$ divise $b_1(\ua)$ et par
suite $a_1$ divise $b_1$, disons $b_1 = qa_1$. On reporte cela dans $a_ib_1 =
a_1b_i$ ce qui donne, après simplification par $a_1$, $b_i = qa_i$. En
définitive, $(\ub) = q\,(\ua)$, ce qu'il fallait démontrer.
\end{proof}
%

\subsec{Quelques \prts fondamentales de la profondeur}

On considère une \alg $\rho:\gA\to\gB$, un \itf $\fa$ de $\gA$ et un \Amo~$E$.
On parle de la \emph{profondeur de $E$  relativement à $\fa$ et $\gB$},
pour signifier la profondeur du \Bmo $\rho\ist(E)$  relativement à l'\id $\rho(\fa)\gB$, et l'on note~\hbox{$\Gr_\gB(\fa,E)$} au lieu de $\Gr_\gB\big(\rho(\fa)\gB,\rho\ist (E)\big)$.

\begin{proposition} \label{propProfchgbase} \emph{(Profondeur et \eds)}
Soit $k\geq 1$.
\begin{enumerate}
\item Si $\gB$ est une \Alg plate et $\Gr_\gA(\fa,E)\geq k$, on a 
$\Gr_\gB(\fa,E)\geq k$. \\
En abrégé on écrit $\Gr_\gB(\fa,E)\geq\Gr_\gA(\fa,E)$.
\item Si $\gB$ est une \Alg  \fpte  on a l'\eqvc

\snic{\Gr_\gA(\fa,E)\geq k \;\iff\; \Gr_\gB(\fa,E)\geq k.}

\snii
En abrégé on écrit $\Gr_\gB(\fa,E)=\Gr_\gA(\fa,E)$.
\end{enumerate}
\end{proposition}
%
\begin{proof} On considère une \sKr de longueur~$k$ attachée à l'\id $\fa$ de~$\gA$. Son image par $\rho$ est une \sKr de longueur~$k$ attachée à l'\id~$\rho(\fa)\gB$ de~$\gB$. 
On conclut par la proposition~\ref{propChgBasSrg}. 
\end{proof}

Les \plgcs relatifs à la \prof sont étudiés plus en détail
dans la section \ref{secRecolle2}.

\begin{lemma} \label{lemFondprof}
Soit $b\in\gA$ un \elt \Erg, et $\fa$ un \itf avec $\Gr_\gA(\fa,E)\geq 2$.
Alors $\fa$ est $E/bE$-\ndz. 
\end{lemma}
%
\begin{proof}
Soit $f$ un \pKr attaché à $\fa$,  par exemple $f\in\AX$. 
Par hypothèse, $f$ est un \elt $E[X]$-\ndz. \\
On doit montrer que $f$ est $(E/bE)[X]$-\ndz,
i.e. que si~$g$ \hbox{et $u\in E[X]$} vérifient une \egt $fu=bg$, alors $u\in b E[X]$.
En fait, d'après le lemme de McCoy \ref{lemMcCoy} appliqué avec $E/bE$,
il suffit de traiter le cas où $u\in E$. Puisque~$b$ est \Erg, si $f(X)=\sum_{k=0}^{m}a_kX^{k}$ on peut écrire $g=\sum_{k=0}^{m}u_kX^{k}$ avec les $u_k\in E$. On \hbox{a $a_iu=bu_i$} pour tout $i$. Donc $b\,(a_iu_j-a_ju_i)=0$, \hbox{puis $a_iu_j-a_ju_i=0$}
parce que~$b$ est \Erg. Le \thref{lemGRa>1} donne alors un $v\in E$ tel \hbox{que $u_i=a_i\,v$}
pour tout~$i$. \hbox{Donc $a_i\,u=a_i\,b\,v$}, et puisque~$\fa$ est \Erg, $u=b\,v$.
\end{proof}
%

\begin{example} \label{exalemFondprof} 
{\rm L'hypothèse \gui{$b$ est un \elt \Erg} est indispensable dans le lemme précédent. Considérons en effet une \klg \pf 
$$\preskip-.1em \postskip.4em 
\gA=\gk[a_1,a_2,u_1,u_2,x,b] $$
avec les seules relations
$$\preskip.4em \postskip.4em a_1x=u_1b,\,a_2x=u_2b,\,x^2=0,\,xu_1=0,\,xu_2=0. 
$$
Avec $\gk=\ZZ,\,\QQ$ ou le corps fini $\FF_2$, \llec peut vérifier avec l'aide d'un logiciel de calcul formel que la suite $(a_1,a_2)$ est régulière dans~$\gA$ et que l'\id $\gen{a_1,a_2}$ n'est pas régulier modulo $b$.   
}\end{example}

\begin{theorem} \label{thfondprof1} ~\\
Pour  $k\geq 1$ et $b$  un \elt \Erg de $\gA$, on a  l'implication

\snic{\Gr_\gA(\fa,E)\geq k+1 \;\Longrightarrow\; \Gr_\gA(\fa,E/bE)\geq k.}

\snii En fait, si $(f_1,f_2,\dots,f_{k+1})$ est une \sKr attachée à $\fa$,  on obtient une \sErg en rempla{ç}ant par $b$ n'importe quel terme de la suite $(f_1,f_2,\dots,f_{k+1})$.
\end{theorem}
\entrenous{Sans doute si $(b_1,\dots,b_r)$ est une \sErg, avec $r\leq k+1$
on obtient une \sErg en rempla{ç}ant par $b_1$, \dots, $b_r$
n'importe quels termes en ordre croissant dans la \sKr \ndze  $(f_1,\dots,f_{k+1})$

si c'est bien vrai il faudrait en faire un exo.}
\begin{proof} Il suffit de montrer le deuxième point, qui est un peu plus précis. 
\\
On commence avec $k=1$. 
L'hypothèse
est alors que $(f_1,f_2)$ est \Erge et que $b$ est \Erg, et l'on doit montrer 
que les suites $(b,f_2)$ et~$(f_1,b)$ sont \Erges. 
La première est donnée par le lemme \ref{lemFondprof}, qui donne aussi la régularité de $(b,f_1)$. 
Le premier lemme d'échange \ref{lemEchSeqReg}  
(point \emph{1}) dit alors que $(f_1,b)$ est \Erge. 
\\
Enfin on termine par \recu sur $k$, car d'après le cas $k=1$, 
on voit que~$b$ est \ndz sur $E_1=E/f_1E$, et cela permet d'appliquer l'\hdr
en utilisant le module $E_1$ pour lequel la suite $(f_2,\dots,f_{k+1})$ 
est \ndze. 
\end{proof}

\begin{theorem} \label{thfondprof} \emph{(\Tho fondamental de la profondeur)} \\
Pour  $k\geq 1$ et $b$  un \elt \Erg de $\fa$, on a  l'\eqvc

\snic{\Gr_\gA(\fa,E)\geq k+1 \;\iff\; \Gr_\gA(\fa,E/bE)\geq k.}

\end{theorem}
%
\begin{proof} Le \tho précédent donne l'implication directe
(sans même supposer que $b\in\fa$).
\\
L'implication réciproque est facile: avec les notations du \tho précédent, la suite $(b,f_1,\dots,f_k)$ est supposée \Erge, donc par le
 \thref{factdefProf}, on obtient $\Gr_\gA(\fa,E)\geq k+1$.
\end{proof}

Les \thos \ref{thfondprof1} et \ref{thfondprof} seront \gnes en \ref{thfondprof2}.

Il n'y a pas de souci à se faire avec la profondeur lorsque l'on raisonne modulo un \id, en raison du lemme suivant.
\begin{lemma} \label{lemprofmodulo} \emph{(Profondeur modulo un \id)}
Soient $\fb\subseteq \fc$ deux \ids de l'anneau $\gA$, $E$ un \Amo, $(\an)$ dans $\gA$
et $k$ un entier $\geq 1$. \Propeq
\begin{enumerate}
\item $\Gr_{\gA}(\an;E/\fc E)\geq k$.
\item $\Gr_{\gA/\fb}(\an;E/\fc E)\geq k$.
\item $\Gr_{\gA/\fc}(\an;E)\geq k$.
\end{enumerate}
\end{lemma}
\begin{proof}
La notation $\Gr_{\gA/\fc}(\an;E)\geq k$ est juste une abréviation pour 
$\Gr_{\gA/\fc}(\an;E/\fc E)\geq k$. Il faut donc prouver l'\eqvc de \emph{1}
et~\emph{2}. On procède par \recu sur $k$. Le passage de $k\geq 1$ à $k+1$
est assuré par le \tho fondamental \ref{thfondprof}, en utilisant l'\elt $f=\sum_{i=1^n}a_iT^{i-1}$ dans $\AT$ et en étendant les scalaires.
\\
Dans le cas $k=1$, les deux \prts signifient que si un \elt $x$ de $E$
satisfait $\gen{\an}x\in\fc E$ alors $x\in\fc E$.
\end{proof}

Le point \emph{3} du \tho qui suit est subtil, et très utile. 

\begin{theorem} \label{lemProfInfinie} 
Soit $\fa=\gen{a_1,\dots,a_k}\subseteq \gA$, $k\geq 1$.
\begin{enumerate}
\item \emph{(Profondeur infinie)} On a $\Gr_\gA(\fa,E)\geq k+1$ \ssi
$\fa\,E=E$. Dans ce cas, la profondeur est supérieure à tout entier.
On adopte la notation~\hbox{$\Gr_\gA(\fa,E)= \infty$} pour exprimer ce fait.
\item Si $\Gr_\gA(\fa,E)\geq k$, alors dans un anneau $\AuX$, $\fa$
est engendré par une \sErg de longueur $k$. 
\item  Si $\Gr_\gA(\fa,E)\geq \ell$ avec $\ell\leq k$, alors  dans un anneau $\AuX=\gA[X_1,\ldots,X_\ell]$, il existe une \sErg $(b_1,\dots,b_\ell)$ telle que 

\snic{\gen{a_1,\dots,a_k}\AuX=\gen{b_1,\dots,b_\ell, (a_j)_{\ell<j\leq k}}\AuX.}

Plus \prmt,  on peut prendre $\tra{\lst{b_1, \ldots, b_k}} = U
\tra{\lst{a_1, \ldots, a_k}}$ avec la matrice unitriangulaire supérieure $U \in \SL_k(\AuX)$:
{\small
$$
U \,=\, \cmatrix {
1 &X_1    &X_1^2 &\cdots &\cdots &\cdots &\cdots &X_1^{k-1} \\[.4em]
0 & 1     &X_2   &  &  &  &  &X_2^{k-2} \cr
\vdots &\ddots &1&\ddots&&&&\vdots\cr
\vdots &  &\ddots     &\ddots  &X_\ell &\cdots &\cdots &X_\ell^{k-\ell} \cr
\vdots &  &      &\ddots     &1  &0    &\cdots &0 \cr
\vdots & &      &     &\ddots      &\ddots &\ddots &\vdots \cr
\vdots & &      &&     &\ddots      &\ddots&0  \\[.4em]
0 &\cdots &\cdots      &\cdots     &\cdots&\cdots&   0  &1 \cr
}.
$$ 
}%
\end{enumerate}
\end{theorem}
%
\begin{proof} Le point \emph{2} est un cas particulier du point \emph{3.}

\snii
\emph{3.} On procède par \recu sur $\ell$ (pour n'importe quel $k\geq \ell$), le cas $\ell=0$ étant trivial.
\\
Passons de $\ell-1\geq 0$ à $\ell$.
On se place sur l'anneau $\gA[X_1]$. 
\\
L'\elt $b_1=a_1+a_2X_1+\cdots+a_kX_1^{k-1}$ est \Erg (lemme de McCoy). Par le \thref{thfondprof}, on a $\Gr_\gA(a_2,\dots,a_k;E/b_1E)\geq \ell-1$. Et l'on constate que $\gen{b_1,a_2,\dots,a_k}=\gen{a_1,a_2,\dots,a_k}$.
Par \hdr, sur l'anneau $\AXk$, la suite $(b_2,\dots,b_k)$ donnée dans l'énoncé est $E/b_1E$-\ndze et $\gen{a_2,\dots,a_k}=\gen{b_2,\dots,b_k}$.
En conséquence, la suite $(b_1,\dots,b_k)$ est \Erge et $\gen{a_1,a_2,\dots,a_k}=\gen{b_1,a_2,\dots,a_k}=\gen{b_1,b_2,\dots,b_k}$.

\smallskip 
\emph{1.} On applique le point \emph{3} avec la suite  $(\ak,0)$, on voit \hbox{que $b_{k+1}=0$}. On obtient donc
(après une \eds \fpte) que~$0$ \hbox{est \Erg} modulo $\gen{b_1,\dots,b_k}=\fa$,
i.e. la multiplication par $0$ est injective \hbox{sur $E/\fa E$} ce qui donne $E/\fa E=0$. 
\end{proof}
\comm
Pour le point \emph{2}, la \dfn de la profondeur donne le fait qu'une \sKr attachée 
à $(\ua)$ est \Erge, mais cette suite n'engendre pas \ncrt l'\id $\fa$
dans l'anneau de \pols considéré. 
Comme conséquence du point~\emph{3}, on obtient tout autre chose que la simple \dfn.
Par exemple on verra que l'anneau~\hbox{$\gA[X_1,\dots,X_{k-1}]$} est suffisant.\\
Le point~\emph{3} est nettement plus subtil qu'une simple explicitation de la \dfn. En effet, vu la forme de la matrice $U$, la suite $(b_1,\dots,b_\ell)$
\emph{n'est pas} en \gnl une \sKr attachée à $\fa$ \hbox{dans $\gA[X_1,\dots,X_\ell]$}, par exemple le contenu de $b_2$ n'a aucune raison d'être égal à $\fa$, et $b_2$ n'est peut-être pas un \elt \ndz.
\eoe

\medskip 
\rem
Voici une \dem alternative du point \emph{1} et d'une variante du point \emph{2}. 
On définit des \pols $g_j=b_1(X_j)$ (une \sKr). 
\\ Supposons d'abord $\Gr_\gA(\ak)\geq k$.
\\
Le passage de $(a_1,\dots,a_k)$ \hbox{à $(g_1,\dots,g_k)$} est une matrice de \hbox{VanderMonde} de \deter $\Delta=\prod_{1\leq i<j\leq k}(X_i-X_j)$. Sur l'anneau $\AXk[1/\Delta]$ on a donc l'\egt $\fa=\gen{\ak}=\gen{g_1,\dots,g_k}$.
Ceci montre que l'\id~$\fa$ est engendré par une 
\srg dans une extension \fpte de~$\gA$.
\\ Supposons maintenant que $\Gr_\gA(\ak)\geq k+1$.
La suite $(g_1,\dots,g_{k+1})$ est \Erge sur $\gA[X_1,\dots,X_{k+1}]$ et reste \Erge sur le localisé de Nagata $\gB=\gA(X_1,\dots,X_{k+1})$.  
\\
En outre $g_{k+1}\in \fa=\gen{g_1,\dots,g_k}$ sur $\gB$. Ainsi la multiplication par $g_{k+1}$ est à la fois nulle et injective sur $E/\fa E$.
Ce qui donne l'\egt  $E/\fa E=0$ sur~$\gB$.
Par \hbox{suite $E/\fa E=0$} sur~$\gA$, car $\gB$ est  \fpt sur $\gA$.
\eoe

\section{Suites \csces}
\begin{definition} \label{defiCseqFFR}
Soit $(\ua)=(\an)$ dans $\gA$ et $\fa=\gen{\ua} $. 
\begin{itemize}
\item On dit que $(\ua)$ est \emph{\csce} si $\Gr_\gA(\fa)\geq n$.
\item On dit que $(\ua)$ est \emph{complètement $E$-sécante}, ou \emph{\csce pour~$E$ (dans $\gA$)}, si $\Gr_\gA(\fa,E)\geq n$. 
\end{itemize}%
\index{completement sec@complètement sécante!suite --- dans $\gA$}%
\index{suite!complètement sécante dans $\gA$}%
\index{suite!complètement sécante pour $E$ dans $\gA$}%
\index{suite!complètement $E$-sécante}%
\end{definition}

La notion de \scs est meilleure que la notion de \srg car elle est indépendante de l'ordre des \elts dans la suite, tout en bénéficiant 
de la plupart des \prts intéressantes des \srgs.

Un exemple est donné par la suite $\big((y-1)x,  y,  (y-1)z\big)$
dans $\gk[x,y,z]$. Cette \srg engendre l'idéal $\gen{x,y,z}$
mais si l'on échange les deux derniers termes, elle n'est plus \ndze
(déjà envisagé dans l'exemple \ref{exaregnoreg}).

\medskip 
\rem Si $(\an)=(\ua)$ est \csce (resp.  \cEse),
tout suite de longueur $n$ qui engendre $\fa=\gen{\ua}$ est \csce (resp.  \cEse):
c'est clair d'après la \dfn \ref{defiCseqFFR}.
\hum{Il faudra signaler les énoncés similaires non complètement évidents. Par exemple pour la $1$-sécance.
\vspace{-3em}} 
\eoe

Une conséquence \imde du corolaire \ref{cor0lemEchSeqReg} est la proposition suivante, qui généralise un résultat classique pour les anneaux \cohs \noes locaux.
\begin{proposition} \label{prop}
Soit $\gA$ un anneau \coh et $E$ un \mpf.  Une suite $(b_1,\dots,b_m)$ contenue dans $\Rad(\gA)$ est \cEse \ssi elle est \Erge.
\end{proposition}

\begin{lemma} \label{lemscsclindep}
 Si  $(\ak)$ est une suite  \cEse,  chaque \elt $a_i$ est \Erg modulo les $a_j\neq a_i$.
 \end{lemma}
%
\begin{proof} Comme l'ordre des éléments n'importe pas dans une suite complètement sécante, il suffit de traiter le cas $i=k$.
Considérons une \syzy $\sum_{i=1}^ka_iy_i=0$ sur $E$. 
Nous devons montrer que $y_k\in \gen{a_1,\dots,a_{k-1}}E$.
On considère la suite $E[\uX]$-\ndze $(\bk)$ donnée dans le point \emph{3} du \thref{lemProfInfinie}, avec $k=\ell$. Par définition d'une suite régulière, l'élément  $b_k=a_k$ est $E[\uX]$-régulier modulo $\gen{b_1,\dots,b_{k-1}}$.
Donc $y_k\in \gen{b_1,\dots,b_{k-1}}E[\uX]$. En spécialisant $X_j=0$ ($j\in\lrb{1..k-1}$), on obtient  $y_k\in \gen{a_1,\dots,a_{k-1}}E$.
\end{proof}
%

\begin{theorem} \label{thfondprof2} \emph{(Généralisation des \thos~\ref{thfondprof1} et \ref{thfondprof})} \\
Soient $n$, $k\geq 1$, $E$ un \Amo, $\fb$ un \itf, $(\an)$  une  \scEs de~$\gA$ et $\fa=\gen{\an}$. 
\begin{enumerate}
\item 
On a  l'implication
$$\preskip.2em \postskip.4em 
\Gr_\gA(\fb,E)\geq n+k \;\Longrightarrow\; \Gr_\gA(\fb,E/\fa E)\geq k. 
$$
\item 
Si $\fa\subseteq \fb$ on a les \eqvcs 
$$\preskip.4em \postskip.4em 
\Gr_\gA(\fb,E)\geq n+k \;\iff\; \Gr_\gA(\fb,E/\fa E)\geq k\;\iff\; \Gr_{\gA/\fa}(\fb,E)\geq k. 
$$
\end{enumerate}
Comme cas particulier si $(\an)$ est \csce dans $\gA$, pour toute suite $(\bbm)$ on a  les \eqvcs
\[ \preskip.4em \postskip.2em 
\begin{array}{rl} 
&\Gr_\gA(\an,\bbm)\geq n+k \\[.3em] 
\iff  & \Gr_{\aqo\gA\ua}(\bbm)\geq k
\end{array}
\]
et
\[ \preskip-.2em \postskip.4em
\begin{array}{rl} 
  &  
(\an,\bbm)\hbox{ est \csce dans }\gA  \\[.3em] 
\iff  &  (\bbm)\hbox{ est \csce dans }\aqo\gA\ua 
 \end{array}
\]   
\end{theorem}
%
\begin{proof} \emph{1.} Supposons tout d'abord que $\an$ est une \sErg.
On applique plusieurs fois le \thref{thfondprof1}. 
\begin{itemize}
\item On a $\Gr_\gA(\fb,E/a_1E)\geq k+r-1$ car $a_1$ est \Erg.
\item On a $\Gr_\gA\big(\fb,E/(a_1E+a_2E)\big)\geq k+n-2$ car $a_2$ est 
$E/a_1E$-\ndz. 
\item \dots \, jusqu'à  $\Gr_\gA\big(\fb,E/(a_1E+\cdots+a_nE)\big)\geq k$.
\end{itemize}
Voyons le cas où $(\an)$ est \cEse. On considère une \sErg $(\hn)$
dans $\AuX$ 
avec $\gen{\hn}\AuX=\fa\,\AuX$
(point \emph{2} du \thref{lemProfInfinie}). 
Puisque la suite est \Erge on a 

\snic{\Gr_\AuX\big(\fb,E[\uX]/\!\gen{\hn}\!E[\uX]\big)\geq k,}

i.e.  $\Gr_\AuX\big(\fb,E[\uX]/\fa\,E[\uX]\big)\geq k$,
ce qui implique $\Gr_\gA(\fb,E/\fa\,E)\geq k$ par la proposition \ref{propProfchgbase}.

\emph{2.} On suppose $\Gr_\gA(\fb,E/\fa E)\geq k$. On remplace la suite $(\an)$ par une  \sErg  $(\hn)$ en étendant les scalaires à $\AuX$. 
\\
On étend les scalaires à $\AuX[\uY]$. Par hypothèse on a dans $\fb\AuX[\uY]$ une suite $(E/\fa E)$-\ndze  de longueur $k$. Les deux suites mises bout à bout font une \sErg dans 
$\gA[\uX,\uY]$.
\\
La dernière \eqvc est dans le lemme \ref{lemprofmodulo}.
\end{proof}
\rem L'implication du point \emph{1} n'est qu'un cas particulier
de l'implication correspondante du point \emph{2} car $\Gr_\gA(\fa,E)\geq r+k$ implique $\Gr_\gA(\fa+\fb,E)\geq r+k$ et $\Gr_\gA(\fa+\fb,E/\fb E)\geq k$
peut être vu sous la forme $\Gr_{\gA/\fb}(\fa,E/\fb E)\geq k$, \cad aussi bien $\Gr_\gA(\fa,E/\fb E)\geq k$.
\eoe

Voici, pour les \scEss, les analogues des résultats \ref{propLocSrg} et \ref{propChgBasSrg} pour les \sErgs. Cela ce déduit des résultats précédents en considérant une \sKr attachée à la suite considérée. 

\begin{plcc} \label{propLocScs} \emph{(Pour les \scss)} 
Les \scEss restent \cEses par \lon, et une suite qui est \cEse après \lon en des \moco est \cEse.
\end{plcc}

\begin{proposition} \label{propChgBasScs}
Soit $\rho:\gA\to\gB$ une \Alg, $(\ua)=(\an)$ dans~$\gA$ et $E$ un \Amo.
Notons $(\uap)=\big(\rho(a_1),\dots,\rho(a_n)\big)$ \hbox{et $E'=\rho\ist(E)$}.
\begin{enumerate}
\item Si $\gB$ est plate sur $\gA$ et $(\ua)$ est une \scEs, alors $(\uap)$ est une suite complètement $E'$-sécante. 
\item Si $\gB$ est \fpte sur $\gA$, alors $(\ua)$ est une \sErg \ssi $(\uap)$ est une suite complètement $E'$-sécante.
\end{enumerate} 
\end{proposition}

\sibook{
\begin{lemma} \label{lem a,b,ab csc} \emph{(Astuce $(a,b,ab)$ pour les \scss)}
Soient $(a, a_2, \ldots, a_n)$ et $(b, a_2, \ldots, a_n)$ deux suites \cEses.
Alors la suite $(ab, a_2, \ldots, a_n)$ est \cEse.  
\end{lemma}

%
\begin{proof}
Cela se déduit du résultat analogue pour les suites \Erges (exercice~\ref{lem a,b,ab regseq})
en utilisant le \thref{lemProfInfinie} (point \emph{3} avec $k=\ell$).
\end{proof}
}

\section{Le \tho de Wiebe}\label{wiebe}


\begin{theorem} \label{thWiebe} \emph{(\Tho de Wiebe)}
\\
Soient $(\ua) = (a_1, \dots, a_n)$ et $(\uc) = (c_1, \dots, c_n)$ deux suites de même longueur 
d'un anneau $\gA$ et soit $E$ un \Amo. 
On suppose que $\fc=\gen{\uc} \subseteq \fa=\gen{\ua}$, inclusion certifiée par une matrice carrée 
de \deter $\Delta$.\\
On suppose que la suite $(\uc)$ est \cEse. Alors on a l'égalité
:
$$
(\fc E: \Delta)_E \ = \ \fa E 
\qquad \hbox{et} \qquad 
(\fc E: \fa)_E \ = \   (\gen{\Delta}+ \fc) E. 
$$
En d'autres termes, la suite ci-dessous est exacte
$$\preskip.4em
\xymatrix @R=1em @C=4em{
0 \ar[r] & 
E/\fa E  \ar[r]^-{\times \Delta} & 
E/\fc E \ar[r]^-
{\Cmatrix{.4em}{a_1 \cr \vdots \cr a_n}} 
&
\big(E/\fc E\big)^n.
}
$$
En notant $\gC=\gA/\fc$ et $\ov E$ cela veut \egmt dire: 
$$\preskip.4em  
(0: \ov\Delta)_{\ov E} \ = \ \fa\ \ov E 
\qquad \hbox{et} \qquad 
(0: \ov \fa)_{\ov E} \ = \   \Delta\ \ov E.
$$
\vspace{-1.2em}
\end{theorem}
\begin{proof}
Comme la suite $(\uc)$ est \cEse,  il en est de même de~$(\ua)$ 
en raison de l'inclusion $\fc\subseteq \fa$.

Les inclusions $\supseteq$ sont évidentes. En effet, d'après l'égalité de Cramer, \hbox{on a 
$\Delta \fa \subseteq \fc$} donc $\Delta \fa E\subseteq \fc E$.
Il reste  à montrer les inclusions $\subseteq$, i.e., pour $x \in E$ :
$$
\hbox{\emph{1.} }\; \Delta x \in \fc E \ \Rightarrow \ x \in \fa E
\qquad \hbox{et} \qquad 
\hbox{\emph{2.} }\; \fa x \subseteq \fc E\ \Rightarrow \ x \in  \gen{\Delta, \, \uc} E
$$
On fait une \dem par \recu sur $n$, le cas $n=1$ étant clair.

On pose \fbox{$\gA'=\aqo\gA{ a_n }$,  $E'=E/a_nE$ et $F=E/\gen{ c_1, \dots, c_{n-1} }E$}. On note $x\mt \ov x$ (resp. $x\mt\wh x$) la surjection canonique $E\to E'$ (resp. $E\to F$). 

\emph{On traite d'abord le cas où les deux suites vérifient les points suivants.}
\begin{enumerate}
\item [i)] $c_n$ est \Frg. 

\item [ii)] $(c_1, \dots, c_{n-1}, a_n)$ est \cEse.

\item [iii)] $a_n$ est \Frg.

\item [iv)] $a_n$ est \Erg.
\end{enumerate}

On va appliquer l'\hdr aux suites 
$(\ov {a_1}, \dots, \ov {a_{n-1}})$ et $(\ov {c_1}, \dots, \ov {c_{n-1}})$
de l'anneau $\gA'$ et au $\gA'$-module $E'$. Pour pouvoir le faire, il faut s'assurer 
que l'hypothèse \gui {suite \cEse}  se propage. Et c'est bien le cas car d'après les points  (ii) et (iv) , on a:
$$
\Gr_{\gA}\big(c_1, \dots, c_{n-1} \, ;E/a_nE\big) \geq n-1,
\; \hbox{i.e.} \ \Gr_{\gA'}(\ov {c_1}, \dots, \ov {c_{n-1}}\, ;E') \geq n-1.
$$
Pour l'inclusion 
$(*)\,\gen{ \ov {c_1}, \dots, \ov {c_{n-1}} } \subseteq 
\gen{\ov {a_1},\dots,\ov {a_{n-1}}}$   
on  a un déterminant~$\delta$ vérifiant $\Delta a_n \equiv \delta c_n \bmod{\gen{ c_1, \dots, c_{n-1} }}$. En effet, notons $C$ (resp. $A$) 
le vecteur colonne des $c_i$ (resp. des $a_i$), si $U A = C$, on a $\wi U C=\Delta A $ et la dernière ligne donne $\Delta a_n \equiv \delta c_n \bmod{\gen{ c_1, \dots, c_{n-1} }}$. Or $\delta$ est le cofacteur de~$U$ en position $(n,n)$, i.e. le \deter de la matrice $V$ extraite de $U$ qui certifie l'inclusion $(*).$   

\emph{Le point 1.}
\\
Soit $x \in E$ tel que $ \Delta x\in \fc E$. Si  l'on a $ \delta \ov x \in \gen{ c_1, \dots, c_{n-1}} E'$,
d'après l'\hdr, on aura
$\ov x \in \gen{ a_1, \dots, a_{n-1}}E'$, i.e. $x \in \fa E$.
Montrons donc que l'on a $\delta x \in \gen{ c_1, \dots, c_{n-1}, a_n} E$, \cade $\delta \wh x\in  a_n F$.
Raisonnons modulo $\gen{ c_1, \dots, c_{n-1}}$.
On a $\Delta a_n \equiv \delta c_n$, d'où $ \Delta  a_n x \equiv  \delta c_n x$.
\hbox{Mais $\Delta x \equiv c_n y$} pour un certain $y \in E$, parce que $\Delta x \in \fc E$.
Ainsi $c_n a_n y \equiv  \delta c_n x$. 
Or $c_n$ est \Frg (cf. (i)), donc $ a_n y \equiv  \delta x$, 
\hbox{d'où $ \delta \wh x\in  a_n F$}, 
ce que l'on voulait.

\emph{Le point 2.}
\\
Soit $x \in E$ tel que $\fa x  \subseteq \fc E$. On veut $x \in \gen{ \Delta,\, \uc } E$.
\\
On va montrer qu'il existe $y \in E$ 
vérifiant $ c_n y\equiv a_n x \bmod {\gen{ c_1, \dots, c_{n-1} }}$ et tel que 
$ \gen{ a_1, \dots, a_{n-1} }\ov y\subseteq \gen{ c_1, \dots, c_{n-1} }E'$.
Donc par \hdr, 
$\ov y \in \gen{ \delta, c_1,\dots, c_{n-1} }E'$, 
c'est-à-dire $y \in \gen{ \delta, c_1,\dots, c_{n-1},a_n }E$.
On multiplie alors cette \syzy par $c_n$ et l'on raisonne modulo $\gen{ c_1, \dots, c_{n-1} }$. 
Comme $c_n y \equiv a_nx$ et $\delta c_n \equiv \Delta a_n$, on obtient 
$ a_n x\in a_n\gen{\Delta,c_n}F$. Mais $a_n$ est \Frg (cf. (iii)), donc
$x \in \gen{ \Delta, c_n }F$. 
Autrement dit $x \in \gen{ \Delta, \uc }E$.

Tout revient donc à montrer qu'il existe un $y \in E$ tel que 
$c_n \wh y = a_n\wh x$ \hbox{et  
$\gen{ a_1, \dots, a_{n-1}, a_n }y \subseteq \gen{ c_1, \dots, c_{n-1},a_n }E$}.
Pour cela, on traduit l'hypothèse $ \fa x \subseteq \fc E$ en disant qu'il existe 
 $y_1,\dots,y_n \in E$ tels que 
$(\star)\,a_i\wh x = c_n\wh {y_i}$ 
et l'on pose $\wh y = \wh {y_n}$ de sorte que $ a_n\wh x \equiv  c_n\wh y$.
En multipliant $(\star)$ par $a_n$, on obtient $a_n  a_i \wh x = a_n  c_n\wh {y_i}$,
c'est-à-dire $c_n  a_i \wh y = a_n  c_n\wh {y_i}$. 
D'après (i), $c_n$ \hbox{est \Frg} 
donc $ a_i\wh y = a_n \wh {y_i}$, d'où $\gen{ a_1, \dots, a_n } y\subseteq \gen{ c_1, \dots, c_{n-1},a_n }E$.

\emph{Il reste à expliquer pourquoi on peut se ramener au cas de 
deux suites $(\ua)$ et $(\uc)$ vérifiant les points (i), (ii), (iii) et (iv).} 
\\
On utilise des extensions \polles. 
Tout d'abord, en application du point \emph{3} du \thref{lemProfInfinie}, après extension à $\gA[X_1,\dots,X_{n-1}]=\AuX$, l'\id $\fc$ est engendré par la \sErg $(c'_1,\dots,c'_n)$ où 
$$\preskip.4em \postskip.4em 
c'_i=c_i+c_{i+1}X_i+\cdots+c_nX_i^{n-i}. 
$$
Comme $\fc\subseteq \fa$ et $(c'_1,\dots,c'_n)$ est 
\Erge, on a 
$$\Gr_\AuX(\fa;E/\gen{c'_1,\dots,c'_{n-1}}E)\geq\Gr_\AuX(\fc;E/\gen{c'_1,\dots,c'_{n-1}}E)\geq 1.$$  
On ajoute encore une \idtr $Y$ et l'on considère le \pol
$$\preskip.4em \postskip.4em 
a'_n=a_n+a_{n-1}Y+\cdots+a_1Y^{n-1}, 
$$
dont le contenu est égal à $\fa$. On en déduit que $a'_n$ est à la fois \Erg \hbox{et \Frg} (ici $F=E/\gen{c'_1,\dots,c'_{n-1}}E$). Ce dernier point implique que la suite $(c'_1,\dots,c'_{n-1},a'_n)$ est \cEse (en fait elle est même \Erge). 
On définit la suite $(a'_1,\dots,a'_n)$ en posant $a'_i=a_i$ pour $i<n$.
On obtient donc un peu mieux que (i), (ii), (iii) et (iv), car $(c'_1,\dots,c'_n)$ \hbox{et $(c'_1,\dots,c'_{n-1},a'_n)$} sont \Erges.
On remarque pour terminer que les matrices de passage entre les anciennes et nouvelles suites étant unitriangulaires, le \deter $\Delta$ ne change pas. 
\end{proof}
\rem La \dem précédente montre un cas typique où l'on supprime une hypothèse trop forte \gui{anneau local \noe} que l'on trouve dans les preuves usuelles données en \clama, en utilisant la bonne \dfn de la profondeur, à la Hochster-Northcott.
\eoe

%

\section{Profondeur et \seco}\label{secprofetsexc}

\begin{theorem} \label{thSESPrf} \emph{(Profondeur et \seco)}\\
Soit $k\in\NN$, si l'on a une \seco de \Amos 

\snic{0\to E\vvers{\iota} F\vvers{\pi} G\to 0}

\snii on obtient les implications suivantes pour un \itf $\fa=\gen{\an}$
\begin{enumerate}
\item
\label{i2thSESPrf}   $(\Gr_\gA(\fa,F)\geq k+1$ et $\Gr_\gA(\fa,G)\geq k)$ 
$\Longrightarrow$ $\Gr_\gA(\fa,E)\geq k+1$.
\item
\label{i3thSESPrf}   $(\Gr_\gA(\fa,E)\geq k$ et $\Gr_\gA(\fa,G)\geq k)$ 
$\Longrightarrow$ $\Gr_\gA(\fa,F)\geq k$.
\item
\label{i1thSESPrf}  $(\Gr_\gA(\fa,E)\geq k+1$ et $\Gr_\gA(\fa,F)\geq k)$ 
$\Longrightarrow$ $\Gr_\gA(\fa,G)\geq k$.
\end{enumerate}
En abrégé, avec 
${g_E=\Gr_\gA(\fa,E),\; g_F=\Gr_\gA(\fa,F) \;\hbox{et}\; g_G=\Gr_\gA(\fa,G)}$:

\fnic{g_E\geq \inf(g_G+1,g_F),\phantom{\big)}
g_F\geq \inf(g_E,g_G),\;
g_G\geq \inf(g_E-1,g_F)\,.}
\end{theorem}
%
\begin{proof} Rappelons que si un \elt $f$ est \Grg, la \sex dans l'énoncé donne lieu à la \sex

\snic{0\to E/fE\lora F/fF\lora G/fG\to 0,}

\snii
d'après le corolaire  \ref{lemHilBur2}. Ceci nous permet de raisonner par \recu.

\emph{3.} On considère une \sKr $(f_1,\dots,f_{k+1})$ attachée à l'\id $\fa$. Par hypothèse $(f_1,\dots,f_{k+1})$ est \Erge
 et $(f_1,\dots,f_{k})$ est \Frge. Pour $k=0$, il n'y a rien à démontrer. On prend $k\geq 1$ et on montre d'abord que $\fa$ est $G$-\ndz. Soit $z\in G$ avec $a_i z = 0$ pour tout $i$. On doit montrer~\hbox{$z=0$}. On écrit $z=\pi(y)$. Puisque $\pi(a_i y)=0$ on a un $x_i\in E$ tel que~\hbox{$a_iy=u(x_i)$}.
 \hbox{Donc $\iota(a_jx_i-a_ix_j)=0$} pour \hbox{tous $(i,j)$}. Puisque $\iota$ est injective,
$a_jx_i-a_ix_j=0$ pour tous $(i,j)$. Puisque $\Gr_\gA(\fa,E)\geq 2$, on peut écrire $x_i=a_ix$ pour un $x\in E$. Enfin $a_iy=a_i\iota(x)$ pour tout $i$, donc 
$y=\iota(x)$, et $z=\pi(\iota(x))=0$.\\
Ainsi $f_1$ est $G$-\ndz. On a donc la \sex

\snic{0\to E_1=E/f_1E\lora F_1=F/f_1F\lora G_1=G/f_1G\to 0,}

\snii et l'on peut appliquer l'\hdr avec $k-1$.
\\ En effet  $(f_2,\dots,f_{k+1})$
est $E_1$-\ndze et  $(f_2,\dots,f_{k})$  $F_1$-\ndze.

\emph{1.} On initialise avec $k=0$. Si $\fa$ est \Frg, 
puisque $E$ peut être considéré comme un sous-module de $F$, a fortiori
$\fa$ est \Erg.\\
 Si $k\geq 1$, par hypothèse la suite  $(f_1,\dots,f_{k+1})$ est \Frge
 et $(f_1,\dots,f_{k})$ \hbox{est \Grge}. Puisque $f_1$ est \Grg, la \recu fonctionne. 

\emph{2.} Ici aussi, il suffit de vérifier l'initialisation avec $k=1$.
On suppose que $\fa$ est \Erg et \Grg, on doit montrer qu'il est \Frg.
\\
Soit $y\in F$ avec $\fa y=0$. L'\elt $z=\pi(y)$ est annulé par $\fa$, donc $z=0$ et~$y$  s'écrit~\hbox{$y=\iota(x)$}. Mais puisque $\iota$ est injective,  $\fa x=0$,
donc $x=0$, \hbox{et $y=0$}. 
\end{proof}

\rem 
On peut aussi présenter les in\egts concernant les profondeurs par trois descriptions analogues de type \gui{ultramétrique} 
lorsque les entiers~$g_i$ sont bien définis. Chacune se suffit à elle-même:

\begin{itemize}
\item [\emph{1'.}] \quad $g_E\geq \inf(g_G+1,g_F)$, \,\,avec \egt si $g_F\neq g_G$.
\item [\emph{2'.}] \quad $g_F\geq \inf(g_E,g_G)$, \,\,\hspace{1.4em} avec \egt si $g_E\neq g_G+1$.
\item [\emph{3'.}] \quad $g_G\geq \inf(g_E-1,g_F)$, \,\,avec \egt si $g_E\neq g_F$.
\end{itemize}

On peut visualiser ceci sous forme d'un tableau, en fixant $g_E$ (ici  $g_E=4$),  en se basant sur \emph{3'}, et en faisant cro\^{\i}tre $g_G$.
{\small\[ \!\!\!
\begin{array}{cccccccccccccccccccccccccccccccccccccccc} 
g_E & | &  4 & 4 & 4   & \! |&  4  &  4 &  4  &  4 & \dots &\! |&  4  & 4  & 4   & 4& \!\dots  \\[1mm] 
g_F & | &  0 & 1 & 2   & \! |&  3  & 4  & 5   & 6  &\dots &\! |&  4  &  4 &  4  & 4&  \!\dots    \\[1mm] 
g_G & | &  0 & 1   & 2  & \!|& 3   & 3  & 3   & 3 &\dots &\! |&  4  & 5  & 6   & 7&  \!\dots    \\[1mm] 
& | &  g_G & <    & g_E-1  &\! |& g_G & =    & g_E   &-\,1 &    &\! |&  g_G & >    & g_E   & -\,1  &   \\[1mm] 
 \end{array}
\] } 
\vspace{-3em}
\eoe

\vspace{3em}
\exl En particulier, on voit que si $G$ est quotient d'un module libre par un sous-module
libre, on obtient $\Gr_\gA(\fa,G)\geq \Gr_\gA(\fa)-1$.
L'in\egt dans l'autre sens est donnée pour un cas particulier dans le
lemme \ref{lem1SESPrf}.
\eoe


\section{Profondeur et \ddk}\label{secprofetddk}

Cette section n'utilise pas les résultats des sections \ref{wiebe} et \ref{secprofetsexc}.

\sibook{
Dans cette section \dots}

\subsec{Profondeur d'un produit d'\ids}

\begin{lemma} \label{lemProdEtProf}
Soient $\fa$ et $\fb$ deux \itfs et $k\geq 1$. \\
Si $\Gr_\gA(\fa,E)\geq k$ et 
 $\Gr_\gA(\fb,E)\geq k$, alors  $\Gr_\gA(\fa\fb,E)\geq k$.
\\
 En particulier, si $\sqrt\fa=\sqrt\fb$ alors $\Gr_\gA(\fa,E)= \Gr_\gA(\fb,E) $, et si $\fa\subseteq \sqrt\fb$, alors $\Gr_\gA(\fa,E)\leq \Gr_\gA(\fb,E) $. 
\end{lemma}
%
\begin{proof} La proposition est \imde pour $k=1$. \\
On procède par récurrence pour $k\geq 2$. Soit $f(X)$
un \pol de Kronecker pour $\fa$ et $g(Y)$ un \pol de Kronecker pour $\fb$.
\\
Alors $h=fg$ est un  $E[X,Y]$-\ndz et appartient à l'\id $\fa\,\fb \,E[X,Y]
$. 
\\
Donc par le \thref{thfondprof1}, on obtient sur $\gA[X,Y]$

\snic{\Gr_\gA(\fa,E/hE)\geq k-1$ et $\Gr_\gA(\fb,E/hE)\geq k-1.}

Par hypothèse de récurrence on a $\Gr_\gA(\fa\fb,E/hE)\geq k-1$. Par le \thref{thfondprof} on obtient $\Gr_\gA(\fa\fb,E)\geq k$.  
\end{proof}
\begin{corollary} \label{corlemProdEtProf}
Soit  $a=(\an)$ une suite dans $\gA$. Si $\Gr_\gA(a,E)\geq k$  (resp. $a$  est une suite \cEse) et si $e=(e_1,\dots,e_n)$ est une suite d'entiers $>0$, alors en notant $a^e=(a_1^{e_1},\dots,a_n^{e_n})$, on a \egmt $\Gr_\gA(a^e,E)\geq k$  (resp. $a^e$ est également \cEse).   
\end{corollary}
%
\begin{proof}
Soit $\fa=\gen{\an}$ et $\fb=\gen{a_1^{e_1},\dots,a_n^{e_n}}$. Pour un entier $N$ assez grand on a l'inclusion $\fa^N\subseteq \fb$. On conclut par le lemme \ref{lemProdEtProf} que $\Gr(\fb,E)\geq k$ (resp. $\Gr(\fb,E)\geq n$).
\end{proof}

Le lemme \ref{lemProdEtProf} justifie la \dfn suivante.

\begin{definition} \label{defiProfNor2}  
\emph{(Profondeur d'un module relativement à un \id, 2)}\\
Soit $k\geq 1$ et $\fa$ un \id \emph{radicalement \tf}. On dit que  \emph{$\fa$ 
est~$k$ fois \Erg}, ou encore que que \emph{la profondeur de~$E$
relativement à~$\fa$ est supérieure ou égale à $k$}, 
et l'on écrit $\Gr_{\gA}(\fa,E)\geq k$,
si l'on a cette \prt pour un (donc pour tout) \itf $\fb$ tel que $\sqrt\fa=\sqrt\fb$. 
\end{definition}

Notons aussi qu'un \id radicalement \tf est \Erg \ssi il est une fois \Erg
au sens de la \dfn précédente.

\subsec{La profondeur est majorée par la \ddk}

\begin{lemma} \label{lemmasupertrick0}
On considère une \scEs $(\ua)=(\an)$, $x\in\gA$, \hbox{et  $(e_1,\dots,e_n)$} dans $\NN$.\\
Si l'on a \fbox{$x\, a_1^{ e_1} \dots a_n^{ e_n} \in 
 \geN{a_1^{ e_1+1},\dots, a_n^{ e_n+1}} $},
 alors $x\in\gen{\an}$.
\end{lemma}
\begin{proof}
Si $e_n>0$ on réécrit l'appartenance encadrée sous la forme 

\snic{(x\, a_1^{ e_1} \dots a_{n-1}^{ e_{n-1}} + \bullet\, a_n)\, a_n^{ e_n}
\ \in \ 
 \geN{a_1^{ e_1+1},\dots, a_{n-1}^{ e_{n-1}+1}} .}

Comme la suite $(a_1^{ e_1 + 1}, \dots, a_{n-1}^{ e_{n-1}+1}, a_n^{ e_n})$ 
est \cEse (par le corolaire \ref{corlemProdEtProf}),
l'\elt $a_n^{ e_n}$ est \Erg modulo $\geN{a_1^{ e_1 + 1}, \dots, a_{n-1}^{ e_{n-1}+1}}$ (lemme \ref{lemscsclindep}) 
et donc 

\snic{x\, a_1^{ e_1} \dots a_{n-1}^{ e_{n-1}}  +\bullet\, a_n
\ \in \  
\geN{a_1^{ e_1 + 1}, \dots, a_{n-1}^{ e_{n-1}+1}} 
}

d'où 

\centerline{\fbox{$x\, a_1^{ e_1} \dots a_{n-1}^{ e_{n-1}}
\ \in \ 
\geN{ 
a_1^{ e_1 + 1}, \dots, a_{n-1}^{ e_{n-1}+1}, a_n}$}.
}

On se retrouve dans la même situation que la situation initiale
encadrée, mais
avec le \mom $a^{ e'}$ où $ e' = ( e_1, \dots,  e_{n-1}, 0)$.
\\
On vient de remplacer $ e_n$ par $0$ (en supposant qu'il ne l'était pas déjà) dans l'appartenance encadrée. Le même processus
peut être appliqué successivement à chacun des $ e_j>0$ et l'on obtient à la fin $x\in\gen{\an}$ comme souhaité.
\end{proof}

\begin{lemma} \label{lem-secsing}
Une suite qui est à la fois \csce et singulière est \umd.
\end{lemma}
\begin{proof}
On applique le lemme \ref{lemmasupertrick0} avec $E=\gA$ et $x=1$.
\end{proof}

\begin{theorem} \label{thDimKrullGr}
Soit $\gA$ un anneau de \ddk $\leq r$. 
 Tout \itf de profondeur $>r$ contient $1$. 
\end{theorem}
%
\begin{proof}
L'\itf $\fa$ en question est radicalement engendré par $r+1$ de ses \elts d'après le \tho de Kronecker \ref{FFRthKroH}: $\sqrt{\fa}=\sqrt{\fb}$ avec $\fb=\gen{b_1,\dots,b_{r+1}\subseteq \fa}$. 
La suite $(b_1,\dots,b_{r+1})$ est une \csce de longueur $r+1$
puisque $\Gr(\fb)=\Gr(\fa)>r$ (lemme \ref{lemProdEtProf}). D'autre part cette suite est singulière puisque la \ddk de $\gA$ est $\leq r$.
On conclut par le lemme \ref{lem-secsing}.
\end{proof}

%

\begin{corollary} \label{cor0thDimKrullGr}
Soient $n\in\NN$ et $k\in\lrb{1..n+1}$. \\Si $\Kdim\gA\leq n$ et  $\Gr_\gA(\fa)\geq k$, alors~\hbox{$\Kdim(\gA/\fa)\leq n-k$}.
\\On peut écrire ceci sous la forme \fbox{$\Kdim(\gA/\fa)\leq \Kdim\gA-\Gr(\fa)$}, en notant que $\Gr_\gA(\fa)\geq k$ peut se relire $-\Gr_\gA(\fa)\leq -k$.  
\end{corollary}
\rem Ce serait plus joli  d'écrire 
  \fbox{$\Kdim(\gA/\fa)+\Gr(\fa)\leq \Kdim\gA$}.  À condition de dire ce que cela signifie lorsque l'on n'a pas affaire à des entiers bien définis, i.e. ce qui est dans l'énoncé.\eoe

\begin{proof} On note $\gB=\gA/\fa$. Pour des \elts arbitraire $a_0$, \dots, $a_{n-k}$ de $\gA$
on doit démontrer  $0\in S=\SK_\gB(a_0,\dots,a_{n-k})$.  Comme \lon et passage au quotient commutent, on doit montrer que l'anneau $(S^{-1}\gA)/(S^{-1}\fa)$
est trivial, \cad que $S^{-1}\fa$ contient $1$ (dans $S^{-1}\gA$).
\\
Or $\Gr_{S^{-1}\gA}(S^{-1}\fa)\geq k$ (proposition~\ref{propProfchgbase}) et $\Kdim(S^{-1}\gA)< k$, car $S$ est le \mo bord itéré.
On conclut par le \thref{thDimKrullGr} avec $r=k-1$. 
\end{proof}
%

\subsubsec{Cas des anneaux \gmqs}

Comme $\gK$ est un \cdi et $\fa$ un \itf, 
si $\gK$ est non trivial, les trois entiers dans l'encadré du \tho 
suivant sont bien définis.

\begin{theorem} \label{thDimKrullGr-KXn}
Soit $\gK$ un \cdi non trivial et $\gA=\KXn$. Soit $\fa$ un \itf de $\gA$ et $\gB=\gA/\fa$. Soit $r\in\lrb{-1..n}$ la \ddk de $\gB$  et soit $q=n-r$. 
\\
Alors $\Gr_\gA(\fa)\geq q$ avec \egt si $r\geq 0$ (i.e. si $\gB$ est non nulle). Plus \prmt:
\begin{enumerate}
\item Si $r=-1$, alors $\fa=\gen{1}$ et $\Gr_\gA(\fa)=\infty$.
\item Si $r=n$, alors $\fa=\gen{0}$ et $\Gr_\gA(\fa)=0$.
\item Si $r\in\lrb{0..n-1}$,  alors  $\Gr_\gA(\fa)=q$ et  l'\id $\fa$ contient une \srg de longueur $q$
\end{enumerate}
En abrégé:
$
\fbox{$1\in\fa$ ou $\Kdim\big(\gA/\fa\big)+\Gr_\gA(\fa)=\Kdim\gA$}\ . 
$
\end{theorem}
%
\begin{proof} \emph{1.} Clair.
\emph{2.}
Si $r=n$, l'\id est nul et par convention, il est de profondeur nulle (en tant qu'\itf non fidèle). 

\emph{3.}
D'après le corolaire \iref{cor0thDimKrullGr}, si la profondeur $\Gr_\gA(\fa)$
était $> q$, cela impli\-querait que la \ddk de $\gB$ soit $<r$.\\
Il nous reste à voir que $\fa$ contient une \srg de longueur $q$.
C'est ce qui a été fait dans l'exemple \ref{exaregnoreg}~(3). 
\end{proof}

\hum{1. Notons que comme $\gK$ est un \cdi et $\fa$ un \itf, 
si $\gK$ est non trivial, les trois entiers dans l'encadré du \tho sont bien définis. Ainsi la phrase avec négation dans la \dem pour assurer que $\Gr_\gA(\fa)\leq q$ ne pose pas \pb. 

Utiliser une négation pour $\Gr_\gA(\fa)\leq q$ dans le cas \gnl d'un \itf arbitraire sur un anneau arbitraire est sans  doute inévitable car on n'a pas défini de manière positive ce que doit signifier $\Gr(\fa)\leq k$.
On pourrait cependant dire ce qui suit (un peu comme quand on a défini $\rg(A)\geq k$ pour une matrice $A$). Pour un \itf $\fa$, on dira que $\Gr_\gA(\fa,E)\leq k$ (en un sens fort)
si toute \scEs   $(b_1,\dots,b_{k+1})$ contenue dans $\fa$
est $E$-\umd, cad $\gen{\ub} E=E$.

Peut-être les choses ne deviennent vraiment limpides qu'avec la \carn 
homologique de la régularité, du moins lorsque les modules d'homologie de Koszul sont calculables explicitement. C'est le cas pour les anneaux \cohs \fdis.

2. Les \thos précédents suggéreraient peut-être de s'intéresser de manière plus \gnle
à la profondeur des \ids bords itérés, \dots, lorsqu'ils veulent bien être radicalement \tf, par exemple pour les anneaux \noes \cohs (les \ids bords itérés à la Richman sont alors \tf).}

Nous laissons en exercice le \tho analogue lorsque $\gk$ est un \azrd.

\begin{corollary} \label{corthDimKrullGr-KXn}
Soit $\gK$ un \cdi non trivial, $\gA=\KXn$, $\lfq\in\gA$ $(q\in\lrbn)$ et $\gB=\gA/\!\gen{\uf}$. On pose  $r=n-q$. \Propeq
\begin{enumerate}
\item La suite $(\lfq)$ est \csce dans $\gA$.
\item $\gB$ est nulle ou de dimension $r$.
\end{enumerate}
\end{corollary}

\begin{definition} \label{defthDimKrullGr-KXn} \emph{(Intersection complète dans le cas \gmq)}
\begin{enumerate}
\item Dans le contexte du \thref{thDimKrullGr-KXn}, si $\gB$ est de dimension $r$
et si~$\fa$ possède un \sgr de $q$ \elts,  on dit que $\gB$ est \emph{intersection complète globale sur $\gK$}.
\item On dit que $\gB$ est \emph{intersection complète locale sur $\gK$} s'il existe des \eco $s_1$, \dots, $s_\ell$ de $\gB$ tels que chaque \Klg $\gB[1/s_j]$
est intersection complète globale sur $\gK$.
\end{enumerate}
\end{definition}

Notez que le point \emph{2} de la \dfn est légitimé par les deux faits suivants. Si $\gB=0$, alors $\fa=\gen{1}$ est engendré par une \srg de longueur $q$, donc~$\gB$ est intersection complète globale.  Si $\gB$ est intersection complète globale, alors pour tout $s\in\gB$, $\gB[1/s]$ est intersection complète globale.

\hum{Un exemple du genre suivant serait bienvenu. \\
Si $\gB_1=\aqo{\KXn}{f_1,\dots,f_q}$ et $\gB_2=\aqo{\KXm}{g_1,\dots,g_p}$ sont deux intersections complètes globales non nulles sur $\gK$ avec $m-p\neq n-q$, $\gB_1\times \gB_2$ est intersection complète locale mais pas globale.}

\subsec{Profondeur d'un anneau}

Les \dems du paragraphe précédent suggèrent d'introduire une
nouvelle dimension.

\begin{definition} \label{defProfAnn} \emph{(Profondeur d'un anneau)}\\
Pour $d\geq -1$ on dit qu'un anneau $\gA$ est \emph{de profondeur $\leq d$} si pour tout \elt $s$ de $\gA$, toute \scs de longueur $>d$ dans $\gA[1/s]$ est \umd. On note ceci sous la forme~$\Gr(\gA)\leq d$.%
\index{profondeur!d'un anneau}
\end{definition}

En particulier  $\Gr(\gA)=-1$ \ssi l'anneau est trivial.

Notons que la \dfn a été formulée de façon à ce que la \prt \gui{$\Gr(\gA)\leq d$} satisfasse le \plgc \gui{en haut} pour la \lon en des \eco.

\hum {Dans le cas d'un \alo $(\gA,\fm)$ avec $\fm$ \rtf, la profondeur de $\gA$
est la même que celle de $\fm$ lorsqu'elles sont connues de manière exacte. En effet, si l'on a $\Gr(\gA)\leq d$ et $\Gr_\gA(\fm)> d$ pour un $d\geq -1$, alors l'\alo est trival. ??}

\section{Principes \lgbs concrets et profondeur}\label{secRecolle2}

\subsec{Nouveau regard sur les \plgcs}

\sibook{
\rems 1) Si $\fa=\gen{\an}$ est $k$ fois \Erg, les \mos
$a_i^{\NN}$ forment un \sys $k$ fois \Erg (d'après le corolaire      \ref{corlemProdEtProf}).

2) Si $(S_1, \dots, S_n)$ est $(n+1)$-fois \Erg, les
\mos $S_i$ sont \emph{$E$-\com}.
\eoe
}

On a déjà vu avec   le \plg \ref{plcc.regularite}
que certains \plgs s'appliquent avec des \ecr,
et pas seulement avec des \eco. Nous allons voir maintenant le cas
des \syss de \mos~$k$ fois~\hbox{\Ergs}.\\ 
Les cas $k=1$ (\mos \cor), $k=2$ et $k=+\infty$ (\moco) semblent les plus importants.

\begin{proposition} \label{propProf2div} \emph{(Un \plg pour la \dve)}\\
On considère des \mos $(S_1,\dots,S_n)$ qui forment un \sys deux fois \ndz,  $y\in\Reg(\gA)$, $x$, $b\in\gA$ \hbox{et $b_1$, \dots, $b_m$} \cor dans $\gA$.
\begin{enumerate}
\item Si $y$ divise $x$ après \lon en chaque $S_i$, alors $y$ divise $x$.
\item Si $b$ est pgcd fort de $(\bbm)$ après \lon en chaque $S_i$, alors~$b$
est pgcd fort de $(\bbm)$.
\end{enumerate}
\end{proposition}
%
\begin{proof}
\emph{1.} Par hypothèse $y$ divise  $xs_i$ pour des $s_i\in S_i$. Les $s_i$ forment une suite de profondeur $\geq 2$, donc de pgcd fort $1$. 
Donc $y$ divise $x$.

\emph{2.} Tout d'abord $b$ divise les $b_j$ par le point \emph{1}. Supposons que $y$  divise les~$xb_j$. 
Après \lon en chaque $S_i$, $y$ divise $xb$. Donc $y$ divise $xb$ par le point~\emph{1.}
\end{proof}
%
 
\begin{plcc} \label{plccProfondeur} \emph{(Pour la profondeur d'une suite)}\\
On considère une suite $(s_1,\dots,s_n)$  $k$ fois \Erge
(où $k\geq 1$), et un \itf $\fa$.
Alors $\fa$ est $k$ fois  \Erg  \ssi il est~$k$ fois  \Erg après \lon en les $s_i$.\\
Comme cas particulier, si $s_1$, \dots, $s_n$ vérifient $\gen{s_1,\dots,s_n}E=E$, et \hbox{si 
$\Gr_{\gA[1/s_i]}(\fa,E)\geq k$} pour $i\in\lrbn$, alors $\Gr_\gA(\fa,E)\geq k$. 
\end{plcc}
%

%
\begin{proof}
Pour $k=1$ c'est le \plgc \ref{plcc.regularite}. Pour $k\geq 2$
on suppose que  $\fa$ est $k$-fois  $E[1/s_i]$-\ndz. Soit $\fb=\gen{s_1,\dots,s_n}$.
On considère comme dans la \dem du lemme \ref{lemProdEtProf} un \pKr~$h$ attaché à $\fa\,\fb$ qui est  \Erg. 
Alors,  par le \thref{thfondprof1}, on obtient sur~$\AuX$:
$\Gr_\gA(\fb,E/hE)\geq k-1$ et $\Gr_{\gA[1/b_i]}(\fa,E/hE)\geq k-1.$
\\
On peut donc terminer par \recu, en utilisant le \thref{thfondprof}.
\end{proof}
\hum{Je ne sais pas si (et je ne crois pas que) lorsqu'un \itf est
$k$-fois \Erg après \lon en $S$ alors il est $k$-fois \Erg après \lon en un $s\in S$. 
Cela faciliterait la vie, mais je suis un peu perdu.
Donc il faut faire très gaffe dans les énoncés.
L'avantage de considérer des \mos est surtout la similitude
qui s'opère avec le cas des \moco.
Cela rend aussi certaines preuves plus lisibles, comme dans le \plgc
qui va suivre.}

 Voici une variante des lemmes des \lons successives \Cref{V-7.2},
 \Cref{XV-1.3} et~\ref{factLocCasreg}.

\begin{fact} \emph{(Lemme des \lons successives, avec la profondeur $k$)}
\label{lelosuccprof} \\
Si $(s_1, \ldots, s_n)$ est un \sys dans $\gA$ qui est  $k$ fois \Erg
 et si pour chaque $i$,
on a un \sys
$
(s_{i,1},\dots,s_{i,k_i})
$
dans $\gA$ qui est $k$ fois $E[1/s_i]$-\ndz,
alors les $s_{i}s_{i,j}$ forment un \sys  $k$ fois-\Erg.
\end{fact}
\begin{proof}
D'après le \plgrf{plccProfondeur}, il suffit de vérifier que les~$s_is_{ij}$ sont~$k$ fois \Ergs
après \lon en des \elts   $k$ fois \Ergs. Cela fonctionne avec les \elts $s_i$. 
\end{proof}
%

\subsec{Recollements en profondeur 2}

Cette sous-section reprend le contenu des sections \Cref{XV-8, XV-9}, qui ont été ajoutées par rapport à l'édition française originale. 

Nous reprenons maintenant le \plg \Cref{XV-4.2} en remplaçant l'hypothèse selon laquelle les \mos sont \com par une hypothèse plus faible 
(\sys de \mos deux fois \ndz).

Le contexte est le suivant.
On considère  $(\uS)=(S_i)_{i\in\lrbn}$ un \sys de \mos.
\\
Nous notons $ \gA_i:= \gA_{S_i}$ et  $ \gA_{ij}:= \gA_{S_iS_j}$ ($i\neq j$)
de sorte que $ \gA_{ij}= \gA_{ji}$. 
\\Nous notons~\hbox{$\varphi_i: \gA\to  \gA_i$} et
$\varphi_{ij}: \gA_i\to  \gA_{ij}$ les
\homos naturels. 
\\
Dans ce qui suit des notations comme $(E_{ij})_{i<j\in\lrbn}$ et $(\varphi_{ij})_{i\neq j\in\lrbn})$ signifient que l'on a $E_{ij}=E_{ji}$ mais pas (a priori)
$\varphi_{ij}=\varphi_{ji}$.

\begin{plcc}
\label{plcc.modules1bis}  {\em (Recouvrement d'un module par des localisés) }  
On considère le contexte décrit ci-dessus.
\begin{enumerate}
\item \label{i1plcc.modules1bis} 
On suppose $(\uS)$  deux fois \ndz. On considère un \elt $(x_i)_{i\in\lrbn}$ de  $\prod_{i\in\lrbn}  \gA_i$.
Pour qu'il existe un~\hbox{$x\in  \gA$} vérifiant $\varphi_i(x)=x_i$ dans chaque $ \gA_i$, il faut et
suffit que pour chaque~\hbox{$i<j$} on ait $\varphi_{ij}(x_i)=\varphi_{ji}(x_j)$ dans $ \gA_{ij}$. En outre, cet $x$
est alors déterminé de manière unique.
\\
En d'autres termes l'anneau $ \gA$ (avec les \homos $\varphi_{i}$) est la limite projective du diagramme:

\snic{\big(( \gA_i)_{i\in\lrbn},( \gA_{ij})_{i<j\in\lrbn};(\varphi_{ij})_{i\neq j\in\lrbn}\big)}

\item \label{i2plcc.modules1bis} 
Soit $E$ un \Amo. On suppose $(\uS)$ deux fois \Erg. 
\\Notons $E_i:=E_{S_i}$ et  $E_{ij}:=E_{S_iS_j}$ ($i\neq j$)
de sorte que $E_{ij}=E_{ji}$. 
\\Notons $\varphi_i:E\to E_i$ et
$\varphi_{ij}:E_i\to E_{ij}$ les
\alis naturelles.
Alors le couple $\big(E,(\varphi_{i})_{i\in\lrbn}\big)$ donne la limite projective du diagramme
suivant dans la catégorie des \Amos:

\snic{\big((E_i)_{i\in\lrbn},(E_{ij})_{i<j\in\lrbn};(\varphi_{ij})_{i\neq j\in\lrbn}\big)}

\vspace{-.5em}
$$
\xymatrix @C=3.5em @R=1.5em 
          {
                &&& E_i \ar[r]^{\varphi_{ij}}\ar[ddr]_(.3){\varphi_{ik}}
                & E_{ij}\\
F\ar[urrr]^{\psi_i} \ar[drrr]^{\psi_j}\ar[ddrrr]_{\psi_k}
\ar@{-->}[rr]^(.6){\psi!} &&E \ar[ur]_{\varphi_i}\ar[dr]^{\varphi_j} \ar[ddr]_(.5){\varphi_k} &&\\
  &&& E_j \ar[uur]_(.7){\varphi_{ji}}\ar[dr]
          &E_{ik}\\
&&& E_k \ar[ur]\ar[r]_{\varphi_{kj}} 
   & E_{jk}\\
}
$$
\end{enumerate}
\end{plcc}

\begin{proof}
 \emph{\ref{i1plcc.modules1bis}.} Cas particulier de \emph{\ref{i2plcc.modules1bis}.}

\emph{\ref{i2plcc.modules1bis}.} Soit un \elt $(x_i)_{i\in\lrbn}$ de  $\prod_{i\in\lrbn}  E_i$. On doit montrer que pour qu'il existe un $x\in  E$ 
vérifiant $\varphi_i(x)=x_i$ dans chaque $ E_i$ il faut et
suffit que pour chaque $i<j$ on ait $\varphi_{ij}(x_i)=\varphi_{ji}(x_j)$ dans $E_{ij}$. En outre, cet $x$ doit être unique.
\\
 La condition est clairement \ncr. Voyons qu'elle est suffisante.
\\
Montrons l'existence de $x$. Notons tout d'abord qu'il existe des $s_i\in S_i$ et des $y_i$ dans $E$ tels
que l'on ait $x_i=y_i/s_i$ dans chaque $E_i$.
\\
Si $\gA$ est intègre, 
$E$ sans torsion et \hbox{les $s_i\neq 0$},
on a dans l'\evc obtenu par \eds au corps
des fractions les \egts
$$\preskip.4em \postskip.4em 
\frac{y_1}{s_1}=\frac{y_2}{s_2}=\cdots
=\frac{y_n}{s_n}, 
$$
et vue l'hypothèse concernant les $s_i$ il existe un  $x\in E$
tel que $xs_i=y_i$ pour chaque $i$. 
\\
Dans le cas \gnl on fait à peu près la même
chose.
\\
Pour chaque couple $(i,j)$  avec $i\neq j$, le fait que $x_i/1=x_j/1$ dans $E_{ij}$
signifie que pour certains
$u_{ij}\in S_i$ et  $u_{ji}\in S_j$ on a
 $s_j u_{ij} u_{ji} y_i = s_i u_{ij} u_{ji} y_j $.
Pour chaque $i$, soit  $u_i\in S_i$ un multiple commun des $u_{ik}$ (pour $k\neq i$). 
\\
On a
alors $(s_j u_{j}) (u_{i}  y_i) = (s_i u_{i}) (u_{j} y_j) $.
Ainsi le vecteur des $u_{i}  y_i$ est proportionnel au vecteur des $s_i u_{i}$.
Puisque le \sys $(\uS)$ est deux fois \Erg, il existe \hbox{un $x\in E$} tel que 
$u_{i}  y_i= s_i u_{i} x$ pour tout $i$, ce qui donne les \egts~\hbox{$\varphi_i(x)=\fraC{u_{i}  y_i}{s_i u_{i}}=\fraC{  y_i}{s_i} =x_i$}.
\\
Enfin cet $x$ est unique parce que les $S_i$ sont $E$-\cor.
\end{proof}

Voici maintenant une variante du \plg \Cref{XV-4.4}.
Cette variante appara\^{\i}t cette fois-ci comme une réciproque du \plg précédent.

\begin{plcc}
\label{plcc.modules2bis}\relax {\em (Recollement concret de modules) }
\\
 Soit $(\uS)=(S_1, \dots, S_n)$ un \sys de \mos de $\gA$. \\
 On note $\gA_i=\gA_{S_i}$,
$\gA_{ij}=\gA_{S_iS_j}$ et $\gA_{ijk}=\gA_{S_iS_jS_k}$.
Supposons donné dans la catégorie des \Amos un diagramme commutatif

\snic{\big((E_i)_{i\in \lrbn}),(E_{ij})_{i<j\in \lrbn},(E_{ijk})_{i<j<k\in \lrbn};(\varphi_{ij})_{i\neq j},(\varphi_{ijk})_{i< j,i\neq k,j\neq k}\big)}

(comme dans la figure ci-après)
avec les \prts suivantes. 
\begin{itemize}
\item Pour tous $i$, $j$, $k$ (avec $i<j<k$), $E_i$ est un $\gA_i$-module,  $E_{ij}$ est un~$\gA_{ij}$-module et $E_{ijk}$ est un~$\gA_{ijk}$-module.
Rappelons que selon nos conventions de notation on pose $E_{ji}=E_{ij}$, $E_{ijk}=E_{ikj}=\dots$

\item Pour $i\neq j$,  $\varphi_{ij}:E_i\to E_{ij}$ est un  \molo en $S_j$ (vu dans $\gA_i$).
\item Pour $i\neq k$, $j\neq k$ et $i<j$, $\varphi_{ijk}:E_{ij}\to E_{ijk}$ est un  \molo en $S_k$ (vu dans $\gA_{ij}$).
\end{itemize}

\smallskip {\small\hspace*{6em}{
$
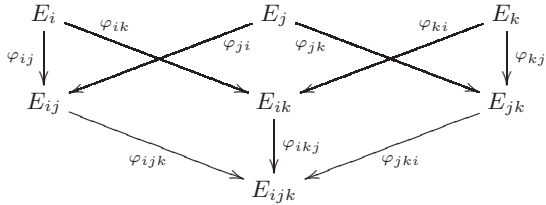
\xymatrix @R=2em @C=7em{
 E_i\ar[d]_{\varphi _{ij}}\ar@/-0.75cm/[dr]^{\varphi _{ik}} &
     E_j\ar@/-1cm/[dl]^{\varphi _{ji}}\ar@/-1cm/[dr]_{\varphi _{jk}} &
        E_k\ar@/-0.75cm/[dl]_{\varphi _{ki}}\ar[d]^{\varphi _{kj}} &
\\
 E_{ij} \ar[rd]_{\varphi _{ijk}} & 
    E_{ik}  \ar[d]^{\varphi _{ikj}} & 
      E_{jk}  \ar[ld]^{\varphi _{jki}} 
\\
   &  E_{ijk} 
}
$
}}
 Alors, si $\big(E,(\varphi_i)_{i\in\lrbn}\big)$ est la limite projective du diagramme, on a  les résultats suivants.
\begin{enumerate}
\item 
Chaque morphisme $\varphi_i:E\to E_i$ est
un  \molo en~$S_i$.  
\item 
Le \sys $(\uS)$ est deux fois \Erg.  
\item  Le \sys $\big(E,(\varphi_{i})_{i\in\lrbn}\big)$ est,
à \iso unique près, l'unique
\sys  $\big(F,(\psi_{i})_{i\in\lrbn}\big)$ avec les $\psi_i\in\Lin_\gA(F,E_i)$
vérifiant les points suivants:
\begin{itemize}
\item le diagramme est commutatif,
\item  chaque $\psi_i$
un  \molo en $S_i$,
\item  le \sys $(\uS)$ est deux fois \Frg.
\end{itemize}
\end{enumerate}
\end{plcc}
%
\begin{proof} \emph{1.} Cette \prt est valable sans aucune hypothèse sur le \sys de \mos considéré (voir la \dem du \plg \cref{XV-4.4}\siBookdeux{ rappelée en \ref{propRcmModLoc}}).

\emph{2.}  On considère des $s_i\in S_i$ et une suite $(\bmx_i)_{i\in\lrbn}$ dans $E$ proportionnelle à $(s_i)_{i\in\lrbn}$. Notons $\bmx_i=(x_{i1},\dots,x_{in})$. La proportionnalité des deux suites signifie que $s_ix_{jk}=s_jx_{ik}$
dans $E_k$ pour tous $i$, $j$, $k$. On pose $\bmx=(\fraC{x_{ii}}{s_i})_{i\in\lrbn}$. On vérifie ensuite que $s_i\bmx=\bmx_{i}$: i.e., que
$s_i\fraC{x_{jj}}{s_j}=x_{ij}$ dans chaque $E_j$. En effet,
cela résulte de l'\egt de proportionnalité  $s_ix_{jk}=s_jx_{ik}$ pour $k=j$.

\emph{3.} 
Puisque~$E$ est la limite projective du diagramme, il y a une unique \Ali  $\psi:F\to E$ telle que
$\psi_i=\varphi_i\circ \psi$ pour tout~$i$.
\\ 
En fait on a $\psi(y)=\big(\psi_1(y),\dots,\psi_n(y)\big)$.
\\ Montrons d'abord que $\psi$ est injective. \hbox{Si $\psi(y)=0$}
tous les $\psi_i(y)$ sont nuls, et puisque $\psi_i$ est un  \molo en~$S_i$, il existe des  $s_i\in S_i$ tels que $s_iy=0$. 
Puisque  $(\uS)$ est un \sys \Frg, on a $y=0$. 
\\
Comme $\psi$ est injective on peut supposer $F\subseteq E$ et 
$\psi_i=\varphi_i\frt F$. \\
Dans ce cas montrer que $\psi$ est bijective revient à montrer que $F=E$. 
\\
Soit $\bmx\in E$. Comme~$\psi_i$ et~$\varphi_i$
sont deux  \molos en $S_i$, il y a des  
$u_i\in S_i$ tels que $u_i\bmx\in F$. Puisque  $(\uS)$ est deux fois \Frg, 
et que la suite des $u_i\bmx$ est proportionnelle à la suite des
$u_i$, il existe un $y\in F$ tel \hbox{que $u_i\bmx=u_iy$} pour tout $i$, donc $y=\bmx\in F$.
\end{proof}

\hum{Il semblerait intéressant de faire une exploration systématique
des \plgs divers et variés et voir s'il y en a d'autres qui
fonctionnent avec une hypothèse affaiblie concernant le \sys des \mos.}

\Exercices

\begin{exercise}
\label{exoDDMcCoy}
{\rm Donner une \dem directe du lemme de McCoy \ref{lemMcCoy}.
 
}
\end{exercise}

\begin{exercise}
\label{exoMcCoyContr1} {(\Tho de McCoy contraposé, version pénible)}\\
{\rm 
Soient $\gA$ un anneau discret non trivial et
une matrice $M\in\Ae{m\times n}$.
\begin{enumerate}
\item Si $\cD_n(M)$ est fidèle, $M$ est injective.
\item Si l'on conna\^{\i}t un entier $k<n$ et un $x\in\gA$ non nul, tels
que 

\snic{x\cD_{k+1}(M)=0\hbox{  et  }\cD_k(M)\hbox{  est fidèle},}

\snii alors on peut construire 
un vecteur non nul dans le noyau de $M$.  
\end{enumerate}
 
}
\end{exercise}

\begin{exercise}
\label{exoMcCoyContr2} (\Tho de McCoy contraposé, version digeste)\\
{\rm Soient $\gA$ un anneau \coh discret non trivial et
une matrice $M\in\Ae{m\times n}$.
\begin{enumerate}
\item Ou bien $\cD_n(M)$ est fidèle, et $M$ est injective.
\item Ou bien on peut construire  dans le noyau de $M$
un vecteur avec au moins une \coo dans $\Atl$.  
\end{enumerate} 
}
\end{exercise}



\begin{exercise} \label{exoregreg}
{\rm
Soient deux \elts \Ergs $b$ et $b'$,
avec $b$ dans l'\id $\fa$ de $\gA$. 
\begin{itemize}
\item [\emph{1}.] Si l'\id $\fa$ est $E/bE$-\ndz,
 il est $E/b'E$-\ndz.
\item [\emph{2}.]  En outre $(bE:\fa)/bE=0$ \ssi $\fa$ est \Erg.
\item [\emph{3}.]  Si $b'$ est aussi dans $\fa$, on a un \iso $(bE:\fa)/bE\simarrow (b'E:\fa)/b'E$.
\end{itemize}
}
\end{exercise}

\begin{exercise} \label{exoWarningReduction}
{(Réduction d'une \sex modulo un \elt \ndz)}\\
{\rm  
Soit une \sex de \Amos

\snic {
0 \to E \to F \to G \to 0
}

\snii  et $a\in \gA$ un \elt \Grg. Le
corolaire \ref{lemHilBur2} dit que la réduction modulo $a$ donne une suite $0 \to E/aE \to F/aF \to G/aG \to 0$ qui est encore exacte.
\\
On considère $\gA = \gk[x,y]$, $a = y$ et la \sex

\snic {
0 \to \gk[x,y] \vvvvers {\tra{\vab {-y}{x}}} \gk[x,y]^2 
\vvers {\vab x y} \gen {x,y} \to 0.
}

\snii
Modulo $y$, on obtient la suite

\snic {
0 \to \gk[x] \vvvvers {\tra{\vab {0}{x}}} \gk[x]^2 
\vvers {\vab x 0} \gen {x} \to 0
}

\snii
qui n'est pas exacte car $\tra{\vab 0 1}$ est dans
le noyau de la forme \lin $\vab x 0$ sans être dans
l'image de la flèche de gauche. Où est l'erreur?

}

\end{exercise}


\begin{exercise}\label{lem a,b,ab regseq}
{(Astuce $(a,b,ab)$ pour les \srgs)}
{\rm
\vspace{.2em}
\begin{itemize}
\item [\emph{1}.]
Soient $(a, a_2, \ldots, a_n)$ et $(b, a_2, \ldots, a_n)$ deux \sErgs.
Montrer qu'il en de même de la suite $(ab, a_2, \ldots, a_n)$.  

\item [\emph{2}.]
Soit $(a_1, \ldots, a_n)$ une \sErg.  \\
Montrer que pour $e_1$, \ldots, $e_n\ge 1$, la suite $(a_1^{e_1}, \ldots, a_n^{e_n})$
est $E$-\ndze.
\end{itemize}
}

\end{exercise}


\begin{exercise} \label{exoprof2gcd} (Profondeur $\geq 2$ pour un anneau à pgcd intègre)\\
{\rm
Soit $\gA$ un anneau à pgcd intègre non trivial. On rappelle que
pour tout \itf $\fa$, on a l'\eqvc:

\snic{\Gr_\gA(\fa)\geq 2\iff \fa$ admet $1$ pour pgcd$.}
\begin{enumerate}
\item [\emph{1.}] Toute matrice $A$ de rang $\leq 1$ peut s'écrire $\tra u\, v$,
où $u$ et $v$ sont des vecteurs lignes et $\Gr_\gA(\gen{u})\geq 2$.
On a alors $\Gr_\gA\big(\cD_1(A)\big)\geq 2\iff \Gr_\gA(\gen{v})\geq 2$. 
\item [\emph{2.}] Donner les énoncés sans négation lorsqu'on ne suppose pas $\gA\neq 0$.
\end{enumerate}

}
\end{exercise}

\begin{exercise}\label{exoMonomialSyzygies} {(Syzygies entre \moms)}
\\
{\rm  
Soit $\kuX = \gk[\Xn]$ un anneau de \pols (où $\gk$ est un anneau quelconque)
et $s$ \moms $m_1, \ldots, m_s$ de $\kuX$. On note $m_i \vi m_j$ le pgcd de
$(m_i, m_j)$ et:
$$\preskip.0em \postskip.4em 
{m_{ij} = {m_j \over m_i \vi m_j}}
\quad \hbox {de sorte que} \quad
\framebox [1.1\width][c]{$m_{ij}\cdot m_i = m_{ji}\cdot m_j$} 
$$
En notant $(\vep_1, \ldots, \vep_s)$ la base canonique de $\kuX^s$, on voit que l'encadré
fournit les \syzys $m_{ij}\vep_i - m_{ji}\vep_j$ pour $(m_1, \ldots, m_s)$.
\\
Montrer que ces \syzys engendrent le module des \syzys pour
$(m_1, \ldots, m_s)$. 

}

\end{exercise}

\begin{exercise} \label{exoGenericPolsRegSequence}
{(Une suite de \pols \gnqs est une suite \ndze)}\\
{\rm  
On considère $s$ \pols \gnqs $P_1$, \ldots, $P_s \in \gA[\uX]=\AXn$, \hmgs de
degrés respectifs $d_1$, \ldots, $d_s \ge 1$; \gui{\gnqs} signifie
que les \coes de chaque $P_i$ sont des \idtrs sur un sous-anneau $\gk$ de
$\gA$ et que le jeu des \idtrs pour $P_i$ est disjoint de celui pour $P_j$ si
$i \ne j$. On va montrer que la suite $(P_1, \ldots, P_s)$ est \ndze lorsque
$s \le n$.

\snii\emph {1.}
Soit le cas particulier $n=s=2$, $d_1=2$, $d_2=3$:

\snic {
P_1 = aX^2 + bXY + cY^2, \qquad
P_2 = a'X^3 + b'X^2Y + c'XY^2 + d'Y^3
}

Montrer que la suite:

\snic {
(\ b, \ c,\ a',\ b',\ c',\ X-a,\  Y-d',\ P_1,\  P_2\ )
}

est \ndze. En déduire que la suite $(P_1, P_2)$ est \ndze.

\snii\emph {2.}
Traiter le cas \gnl.

\snii\emph {3.}
Si $s > n$, la suite $(P_1, \ldots, P_s)$ est-elle
encore \ndze?

}

\end{exercise}




\bonbreak
\sol

\exer{exoDDMcCoy} \emph{(Démonstration directe du lemme de McCoy \ref{lemMcCoy}).}\\
On traite le cas des \pols en une seule variable
(le cas \gnl peut s'en déduire en utilisant l'astuce de \KRA). 
On a $f\in\AT$ et $g\in E[T]$. On suppose $fg=0$ et on veut montrer $g=0$.
On fait une \recu  sur le degré formel $m$ de $g$. Pour  
$m=0$ le résultat est clair. Passons de $m-1$ à $m$. Appelons~$f_i$ les \coes de $f$ et $g_j$ ceux de $g$.  
On va montrer que $f_ig$ est nul pour chaque~$i$. Cela implique alors 
que tous les $f_ig_j$ sont nuls, et puisque $\Ann_E(\rc(f))=0$ que tous les 
$g_j$ sont nuls.
Pour montrer que~$f_ig$ est nul on fait une \recu descendante sur~$i$, 
qui s'initialise avec~$i=n+1$ sans \pb ($n$ est le degré formel de~$f$). Supposons donc avoir montré 
que~$f_ig=0$ pour tous les~$i>i_0$ et montrons le pour~$i_0$.
On a 
  $$ 
      fg = \left(\som_{i\leq i_0}f_iX^i\right) \,g 
  $$
Le \coe de degré~$m+i_0$ de ce \pol est égal à~$g_mf_{i_0}$ et
donc~$g_mf_{i_0}=0$. Donc le \pol $\tilde{g}=f_{i_0}g$ est de degré 
$\leq m-1$ et, évidemment,~$f\tilde{g}=f_{i_0}fg=0$. On peut donc 
appliquer l'\hdr avec~$\tilde{g}$, on conclut~$f_{i_0}g=0$, et on a gagné.  


\exer{exoMcCoyContr1} \emph{(\Tho de McCoy contraposé, version pénible)}\\
\emph{1.} Déjà vu.\\
\emph{2.}  On a $x\neq 0$,  $\cD_k(M)$ est fidèle et l'anneau est discret, 
donc il existe un mineur~$\mu$ d'ordre $k$ de $M$ tel que $x\mu\neq 0$.
Supposons par exemple que $\mu$ soit le mineur nord-ouest et notons $C_1$, \dots, $C_{k+1}$ les premières colonnes de $M$, notons~$\mu_i$ ($i\in\lrbk$) les \deters convenablement signés des matrices extraites sur les lignes $\lrbk$
et les colonnes précédentes, sauf la colonne d'indice $i+1$. Alors
les formules de Cramer donnent l'\egt $\sum_{i=1}^{k}x \mu_iC_i+x\mu C_{k+1}=0$. Comme $x\mu\neq 0$, ceci donne un vecteur non nul dans le noyau de $M$.

\comm En \clama, si $\cD_n(M)$ n'est pas fidèle, comme $\cD_0(M)=\gen{1}$
est fidèle, il existe un $k<n$ tel que $\cD_k(M)$ est fidèle 
et~$\cD_{k+1}(M)$ n'est pas fidèle. Toujours en \clama, si $\cD_{k+1}(M)$ n'est pas fidèle, il existe un $x\neq 0$ tel que $x\cD_{k+1}(M)=0$.
Pour que ces choses deviennent explicites, il faut par exemple disposer d'un test pour la fidélité des \itfs, en un sens assez fort. 
\eoe


\exer{exoMcCoyContr2}\emph{(\Tho de McCoy contraposé, version digeste)}\\
Puisque les \idds sont des \itfs, et l'anneau est \coh, leurs annulateurs
sont \egmt des \itfs Et l'on peut tester la nullité d'un \itf parce que l'anneau est discret. Les hypothèses de l'exercice \ref{exoMcCoyContr1} sont donc satisfaites.\\
NB: l'alternative \gui{\emph{1} ou \emph{2}} est exclusive car l'anneau est non nul, ceci justifie le \gui{ou bien, \dots, ou bien} de l'énoncé. 




\exer{exoregreg}
\\
  Soit $x \in E$ tel que $\fa x
\subseteq b'E$. Il faut prouver que $x \in b'E$. \hbox{Puisque $b \in\fa$}, on \hbox{a $bx\in b'E$}, disons $bx = b'y$ avec $y \in E$. Montrons d'abord \hbox{que $\fa y
\subseteq bE$}. \hbox{Soit $a \in \fa$}, en utilisant $\fa x \subseteq b'E$, on a
$ax = b'z$ avec $z \in E$. 
On \hbox{a
$abx=ab'y = bb'z$}, et comme~$b'$ \hbox{est \Erg}, $ay = b \in bE$.  \\
Puisque $\fa y \subseteq bE$, et que $\fa$ est
$E/bE$-\ndz, on obtient~\hbox{$y \in bE$}, donc $b'y \in bb'E$, i.e. $bx \in bb'E$.
Enfin,  $x \in b'E$ car $b$ est \Erg.


\exer{exoWarningReduction} \emph{(Réduction d'une \sex modulo un \elt \ndz)}
\\
Le $\gk[x,y]$-module $\gen {x,y}/y\gen{x,y}$ est certes un
$\gk[x]$-module mais ce n'est pas l'\id~$\gen {x}$! Réduire 
modulo $y$, ce n'est pas faire $y = 0$ n'importe où!
Voici la \dcn du $\gk[x]$-module $M = \gen {x,y}/y\gen{x,y}$ en
composantes \hmgs:

\snic {
M = 0 \oplus (\gk\ov x \oplus \gk\ov y) \oplus\gk \ov x^2 
\oplus\gk \ov x^3 \oplus \cdots
\qquad \hbox {avec $x \cdot \ov y =0$}
}

\snii
Si on note $N = \aqo{\gk[x]}{x}$, alors $\gk \ov y \simeq N$ via $\ov 1
\leftrightarrow \ov y$ donc $M \simeq N \oplus \gen {x}$. Si on tient compte
de la graduation, il faut décaler $N$ pour obtenir 
$M \simeq N(-1) \oplus \gen {x}$.



\exer{lem a,b,ab regseq} \emph{(Astuce $(a,b,ab)$ pour les \srgs)}
\\
\emph{1.}
Le produit de deux \elts \Ergs de $\gA$ est clairement \Erg.  
Pour chaque $i\in\lrb{2..n}$ on doit vérifier que $a_i$ est $E/(abE+\sum_{k:2\leq k<i}a_kE)$-\ndz.
Autrement dit, il faut
montrer qu'une \egt $a_i x = ab x_1 + a_2 x_2 + \cdots + a_{i-1} x_{i-1}$,
avec~\hbox{$x$, $x_1$, \ldots, $x_{i-1} \in E$}, entra\^ine $x \in ab E + \cdots +
a_{i-1} E$. La suite $(a, a_2, \ldots, a_n)$ est \Erge,
donc:
$$
x = a y_1 + a_2 y_2 + \cdots + a_{i-1} y_{i-1}, \qquad  y_j \in E,
\leqno (\star)
$$
de sorte que:

\snic {
a_i(a y_1 + a_2 y_2 + \cdots + a_{i-1} y_{i-1}) = a_ix
 =
ab x_1 + a_2 x_2 + \cdots + a_{i-1} x_{i-1}.
}

\snii
D'où:

\snic {
a (a_i y_1 - bx_1) + a_2(a_iy_1 - x_2)  + \cdots + a_{i-1}(a_iy_{i-1} - x_{i-1}) = 0.
}

\snii
Comme la suite $(a, a_2, \ldots, a_i)$ est \Erge,  
le lemme~\ref{lem1sqlindep} donne:

\snic {
a_i y_1 - bx_1 \in a_2E + \cdots + a_{i-1}E, \quad \hbox {donc} \quad
a_i y_1 \in bE + a_2E + \cdots + a_{i-1}E,
}

\snii
puis, comme la suite $(b, a_2, \ldots, a_i)$ est
\Erge, $y_1 \in bE + \cdots + a_{i-1}E$.  En reportant ceci dans
$(\star)$, on obtient enfin

\snic {
x \in ab E + a_2 E + \cdots + a_{i-1} E.
}

\snii\emph{2.} Du point \emph{1} on déduit que  la suite $(a_1^e, a_2, \ldots, a_n)$
est \Erge pour $e\geq 1$.
Le résultat voulu se montre  alors par \recu sur $n$.

\exer{exoprof2gcd} \emph{(Profondeur $\geq 2$ pour un anneau à pgcd intègre)} 
\\
\emph{1.}
On suppose \spdg la matrice non nulle. Les colonnes sont proportionnelles
par hypothèse. Supposons la première colonne non nulle, et écrivons 
là sous forme $a\tra u$, où $a$ est un pgcd de ses \coos.
Les \coos de $\!\tra u$ ont pour pgcd $1$ donc forment une suite de profondeur $\geq 2$.
Toute colonne non nulle est proportionnelle à $\!\tra u$, donc
 multiple de $\!\tra u$ (point \emph{3} du \thref{lemGRa>1}), ce qui donne le résultat souhaité.
\\
La dernière affirmation résulte du \thref{factdefProf} et du  lemme \ref{lemProdEtProf}.

\emph{2.}
Les énoncés sont inchangés. 
Les preuves \egmt, à condition de remplacer \gui{\elt non nul}
par \gui{\elt \ndz}. Par ailleurs on peut noter que dans le cas d'un anneau nul, toutes
les \prts envisagées sont vraies, donc \eqves.


\exer{exoMonomialSyzygies} \emph{(Syzygies entre \moms)}
\\
Soit~$r \in \kuX^s$; dire que~$r \in \sum_{i,j} \kuX (m_{ij}\vep_i - m_{ji}\vep_j)$
c'est dire qu'il existe des~$r_{ij} \in \kuX$ vérifiant~$r_{ii} = 0$,
$r_{ij} + r_{ji} = 0$ et:

\snic {
r = \sum_{i,j} r_{ij}m_{ij}\vep_i = \sum_i \left( \sum_j r_{ij}m_{ij}\right) \vep_i
}

\snii
Soit une \syzy~$\sum_i u_im_i = 0$ avec~$u_i \in \kuX$. En considérant la
composante sur un \mom quelconque fixé~$m$, on obtient un terme~$a_im'_i$ de
$u_i$ avec~\hbox{$a_i \in \gk$,~$m'_i$} \mom tel que~$m'_im_i = m$ et~$\sum_i a_im'_i
m_i = 0$, i.e.~$\sum_i a_i = 0$. Il est entendu \hbox{que~$a_i = 0$} si~$m_i \nedivi
m$ et l'on peut se limiter dans la suite aux~$m_i$ tels que~\hbox{$m_i \divi m$}. Il
suffit donc de montrer que~$\sum_i a_im'_i \vep_i$ est dans le module engendré
par les~\hbox{$m_{ij}\vep_i - m_{ji}\vep_j$}.

\snii
Puisque~$m'_im_i = m = m'_jm_j$, le \mom $m$ est divisible par~$m_i$, par
$m_j$ donc par leur ppcm~$m_i \vee m_j$; par conséquent, on peut écrire:

\snic {
m = q_{ij} (m_i \vee m_j),  \quad\hbox {et l'on a donc}\quad q_{ij} = q_{ji}.
}

\snii
Il vient:

\snic {
m'_im_i = q_{ij} (m_i \vee m_j), \quad\hbox {donc}\quad
m'_i = \displaystyle {q_{ij} {m_i \vee m_j \over m_i} = 
q_{ij} {m_j \over m_i \vi m_j} = q_{ij} m_{ij}}.
}

\snii
D'autre part, puisque~$\sum_i a_i = 0$, il existe une matrice
antisymétrique~\hbox{$(a_{ij}) \in \MM_s(\gk)$} telle que 
$a_i = \sum_j a_{ij}$ (la somme sur la ligne~$i$ vaut~$a_i$).
Par exemple, pour~$s = 4$:

\snic {
\cmatrix {
0 & -a_2 & -a_3 &-a_4 \cr
a_2 & 0 & 0 & 0 \cr
a_3 & 0 & 0 & 0 \cr
a_4 & 0 & 0 & 0 \cr} \cmatrix {1\cr 1\cr 1\cr 1\cr} =
\cmatrix {a_1\cr a_2\cr a_3\cr a_4\cr}
}

\snii
On écrit alors (comme par magie en ayant gratté un peu sur son brouillon):

\snic {
a_im'_i = \sum_j a_{ij} m'_i = \sum_j a_{ij} q_{ij} m_{ij}. 
}

\snii Et il ne reste plus qu'à poser~$r_{ij} = a_{ij}q_{ij}$:
on a bien~$r_{ii} = 0$,~$r_{ij} + r_{ji} = 0$ \hbox{et
$\sum_i a_i m'_i \vep_i = \sum_{i,j} r_{ij} m_{ij} \vep_i$}.


\exer{exoGenericPolsRegSequence} \emph{(Une suite de \pols \gnqs est une suite \ndze)}
\\
\emph {1.}
Il est clair que la suite $(\uw)$ des 7 premiers termes est \ndze.  En ce qui
concerne $(\uw, P_1, P_2)$ on peut, par \dfn d'une \seqreg, remplacer $P_i$
par~$Q_i$ avec $Q_i \equiv P_i \bmod \gen {\uw}$; on a $P_1 = X^3 \bmod \gen
{\uw}$ et $P_2 = Y^4 \bmod \gen {\uw}$. La suite~$(\uw, X^3, Y^4)$ est \ndze
parce que l'anneau peut être vu comme un anneau de \pols en les variables
$( b, \ c,\ a',\ b',\ c',\ X-a,\  Y-d',\ X,\  Y )$.
Donc la suite~\hbox{$(\uw, P_1, P_2)$} est \ndze. Or cette suite est constituée de
\pols \hmgs et reste donc \ndze par toute permutation de ses termes.

\hum{1. Dans le dernier argument se fait sentir la nécessité d'une section du cours consacrée au cas gradué. On pourrait \egmt dire que la régularité d'une suite ne change pas lorsque l'on remplace un terme par une puissance?

2. Dans ce raisonnement on a surtout besoin que $a$ et $d'$ soient deux
\idtrs distinctes, distinctes des autres \idtrs. Par contre les autres
\coes peuvent être des \pogs arbitraires en n'importe quel jeu d'\idtrs
distinctes de $a$ et $d'$? }

\snii\emph {2.}
Analogue au cas particulier. Puisque $s \le n$, on peut isoler dans chaque
$P_i$ le terme $a_{i,\varphi_i} X^{\varphi_i}= a_{i,\varphi_i} X_i^{d_i}$  (i.e., $\varphi_i = (0, \ldots,
0,d_i,0, \ldots,0)$). On considère alors la suite $(\uw)$ constituée d'une
part des \coes des $P_i$ autres que les $a_{i,\varphi_i}$ et d'autre part des
binômes $X_i - a_{i,\varphi_i}$. Il est clair que $(\uw)$ est une \seqreg et
\hbox{que $P_i \equiv X_i^{d_i+1} \bmod \gen {\uw}$}.  Puisque la suite $(\uw,
X_1^{d_1+1}, \ldots, X_s^{d_s+1})$ est \ndze, il en est de même de $(\uw, P_1,
\ldots, P_s)$ et, par permutation, de la suite $(P_1, \ldots, P_s, \uw)$, a
fortiori de $(P_1, \ldots, P_s)$.

\snii\emph {3.}
Non: prendre $n=1$, $s=2$, $P_1 = aX$, $P_2 = bX$.


\Biblio

\newpage \thispagestyle{empty}
\incrementeexosetprob


\chapter
{Résolutions libres finies}
\label{chapRLF}
\perso{compilé le \today}

\minitoc

\sibook{
\Intro

Ce chapitre  est une première introduction à l'\alg homologique.
Il est basé sur le livre  remarquable \emph{Finite Free Resolutions} de Northcott. Nous avons  simplifié l'exposition \gnle, et comme d'habitude dans cet ouvrage, toutes nos \dems sont \covs.

\smallskip La section  \ref{secIdesCars} introduit les \idcas 
de certains complexes de modules libres (notamment les \rlfs). 

\smallskip La section   \ref{secRSNLF}  donne quelques premiers résultats
clés concernant les \rlfs.

\smallskip La section \ref{secRangStable}  traite le rang stable des matrices dans une \rlf,
les \rsns libres minimales et les nombres de Betti
dans le cas des \alos.

\smallskip La section \ref{secAusBu} donne le développement \cof qui 
aboutit à une version forte du \tho d'\ABH.

\smallskip La section \ref{secCQRCMLE} donne le \tho d'Eisenbud-Evans
qui caractérise l'exactitude d'un complexe exact de \mlrfs en termes de la profondeur de ses \idcas.

\smallskip La section \ref{secRLFGRAD}  \dots

\smallskip 
\dots

\fbox{\`A COMPL\'ETER} 

}

\section{Complexes}\label{secComplexes}


\subsec{Modules d'homologie d'un complexe}
\begin{definition} \label{defiCompHom}~
\begin{enumerate}
\item Un \ix{complexe} \emph{(descendant)} $C_{\bullet}$ est donné par une suite d'\alis indexée par $\ZZ$,\index{complexe!descendant}%
\begin{equation}
\label{eqnComplexe}
C_{\bullet}:\quad  \cdots\; C_{n+2}\vvvers{u_{n+2}}C_{n+1}\vvvers{u_{n+1}}C_n\vvvers {u_{n}} C_{n-1} \;\cdots
\end{equation} 
dans laquelle $u_n\circ u_{n+1}=0$ pour tout~$n$.
Le complexe est dit \emph{borné} si les~$C_k$ sont nuls pour~$k\leq k_0$
et~$k\geq k_1$ (pour deux entiers~$k_0$,~$k_1$).
 L'\ali $u_n$ est souvent appelée la \emph{\dile} en~$n$ du complexe, et notée $\delta_n$, ou $d_n$, ou $\partial_n$.%
\index{complexe!borné} 
\item Le \emph{module d'homologie en $n$} d'un complexe tel que \pref{eqnComplexe}  est le module quotient~$\Ker u_{n}/\Im u_{n+1}$, noté~$\rH_n(C_{\bullet})$.
\hum{Définir $\rH(C_{\bullet})$ comme un module gradué (sur l'anneau de base)
ou simplement comme la famille $(\rH_n(C_{\bullet}))_{n\in\ZZ}$, et lui donner quel nom?} 
\item 
Le complexe est dit \emph{exact}, ou \emph{acyclique} si tous les groupes
d'homologie sont nuls.%
\index{groupes d'homologie!d'un complexe descendant}%
\index{exact!complexe ---}\index{acyclique!complexe ---}%
\index{complexe!exact}\index{complexe!acyclique}%
\index{decalage@décalage!d'un complexe} 
\item  On peut \emph{décaler d'un rang $k$} un complexe descendant $C_{\bullet}$, on obtient le complexe
noté $C_{\bullet}[k]$ défini par $C_{\bullet}[k]_n=C_{k+n}$, avec une  modification éventuelle du signe pour la \dile: $u'_n=(-1)^{k}u_{n+k}$. 
\item On définit de manière analogue un \emph{complexe montant}, 
que l'on note~$C\bul$, en demandant que
$u_n$ soit une \ali de~$C_n$ vers~$C_{n+1}$.
\begin{equation}
\label{eqnComplexeM}
C\bul:\quad  \cdots\; C_{n-2}\vvvers{u_{n-2}}C_{n-1}\vvvers{u_{n-1}}C_n\vvvers {u_{n}} C_{n+1} \;\cdots
\end{equation} 
Le \emph{module de cohomologie en $n$} d'un tel complexe  est le module quotient~$\Ker u_{n}/\Im u_{n-1}$, noté~$\rH^{n}(C\bul)$.
\\
Le \emph{décalage de rang $k$ d'un complexe montant} $C\bul$ donne le complexe noté $C\bul[k]$
défini \hbox{par $C\bul[k]_n=C_{-k+n}$}.\\
Tout complexe montant donne lieu à un complexe descendant\footnote{Et vice versa.} obtenu
en changeant les indices par l'involution\footnote{Mais il faut faire attention à la numérotation des \diles: la convention universellement acceptée est qu'une \dile a toujours pour indice celui de sa source (que le complexe soit montant ou descendant).} $n\mapsto -n$. 
Cette remarque permet de ne pas répéter systématiquement les \dfns ou \thos concernant les complexes d'une sorte, pour les complexes de l'autre sorte.%
\index{complexe!montant}\index{groupes de cohomologie!d'un complexe montant}%
\end{enumerate}
\end{definition}

\rems ~
\\
1) On dit souvent \emph{groupe d'homologie} (ou de cohomologie) plutôt que \emph{module d'homologie} (ou de cohomologie).

\snii 
2) Si l'on a un complexe défini seulement pour des  indices consécutifs
formant une partie stricte de~$\ZZ$, on peut toujours le compléter avec des modules nuls et des applications nulles pour les indices où il n'est pas défini.

3) On décrit parfois un complexe $C_{\bullet}$ tel que \pref{eqnComplexe} à travers le module somme directe 
$C=\bigoplus_{n\in\ZZ} C_n$ et l'application $u:C\to C$
somme directe des \Alis $u_n$. La condition $u\circ u=0$ résume alors
les conditions~\hbox{$u_n \circ u_{n+1}=0$}, et le module~\hbox{$\rH_{\bullet}(C)=\Ker u/\Im u$} est la somme directe \hbox{des $\rH_n(C_{\bullet})$}.

4) On  peut définir les complexes et leurs \gui{groupes d'\hml} sur n'importe quelle catégorie abélienne. Il en va de même pour les autres notions qui seront définies par la suite, avec les modifications adéquates évidentes.
\eoe

\subsec{Morphismes de complexes}

Un \emph{morphisme} (de degré $0$)
\index{morphisme!de complexes} du complexe (descendant) $C_{\bullet}$  vers le complexe $C'_{\bullet}$ est donné par une famille $\varphi=(\varphi_n)$ d'\Alis $\varphi_n:C_n\to C'_n$ qui fait commuter tous les diagrammes convenables: on en déduit des \Alis 
$\rH_n(\varphi):\rH_n(C_{\bullet})\to\rH_n(C'_{\bullet})$ entre les \mhmls.

Un \emph{morphisme de degré $k$}  du complexe (descendant) $C_{\bullet}$  vers le complexe $C'_{\bullet}$ est un morphisme de $C_{\bullet}$  vers  $C'_{\bullet}[k]$.

\sibook{
\subsubsection*{Somme directe de complexes}

\fbox{écrire le minimum}
}
\subsec{Changement d'anneau de base}

 Rappelons qu'une extension plate préserve les \sexs.
Voici un énoncé du même style. 

\begin{theorem} \label{propEdsHomologie} 
\emph{(Changement d'anneau de base)}\\
Considérons un complexe $C_{\bullet}$ de \Amos et un morphisme $\rho:\gA\to\gB$.
Par \eds  on obtient un complexe $\rho\ist(C_{\bullet})$.
\begin{enumerate}
\item Si $\gB$ est plat sur $\gA$ les modules d'homologie de $\rho\ist(C_{\bullet})$ sont obtenus par \eds à partir des modules d'homologie de~$C_{\bullet}$. 
Autrement dit pour chaque $n$, l'\Ali naturelle 
$$
\rH_n(\rho):\rH_n(C_{\bullet})\to \rH_n(\rho\ist(C_{\bullet}))
$$ 
est un morphisme d'\eds. 
En particulier on obtient les résultats suivants. 
\begin{enumerate}
\item Si $C_{\bullet}$ est exact, il en va de même pour $\rho\ist(C_{\bullet})$.
\item Si $S$ est un \mo de $\gA$, $\rH_n((C_{\bullet})_S)$ s'identifie à $\rH_n(C_{\bullet})_S$. 
\end{enumerate}

\item Si $\gB$ est \fpte sur $\gA$ et si $\rho\ist(C_{\bullet})$ est exact, alors 
 $C_{\bullet}$ est exact. 
%
\end{enumerate} 
\end{theorem}
%
\facile

Comme corolaire on a \tho de recollement suivant.

\begin{theorem} \label{propRecolHomologie} 
\emph{(Recollement de groupes d'homologie)}\\   
Soient  $C_{\bullet}$ un complexe de \Amos et  $S_1$, \dots, $S_n$ des  \moco de $\gA$. 
\\
Notons $C_{\bullet}^{(i)}$ le complexe obtenu par \lon en $S_i$ et $C_{\bullet}^{(i,j)}=C_{\bullet}^{(j,i)}$ le complexe obtenu par \lon en $S_iS_j$. 
Dans un tel cas, chaque module d'homologie~$\rH_n(C_{\bullet})$ est obtenu \gui{par recollement
des modules $\rH_n(C_{\bullet}^{(i)})$ le long des~$\rH_n(C_{\bullet}^{(i,j)})$}. 
De manière précise:
\begin{enumerate}
\item chaque morphisme naturel $\rH_n(C_{\bullet})\to\rH_n(C_{\bullet}^{(i)})$ est un \molo en $S_i$,
\item chaque morphisme naturel $\rH_n(C_{\bullet}^{(i)})\to\rH_n(C_{\bullet}^{(i,j)})$ est un \molo en $S_j$ (vu dans $\gA_{S_i}$), pour $i\neq j$,
\item le module $\rH_n(C_{\bullet})$ et les morphismes naturels $\rH_n(C_{\bullet})\to\rH_n(C_{\bullet}^{(i)})$
donnent la limite projective (dans la catégorie des \Amos) du diagramme formé par les $\rH_n(C_{\bullet}^{(i)})$, les
 $\rH_n(C_{\bullet}^{(i,j)})$ et les \molos $\rH_n(C_{\bullet}^{(i)})\to\rH_n(C_{\bullet}^{(i,j)})$.
\end{enumerate}
\end{theorem}

\subsec{Homotopies} \label{homotopie}

Soient $(C_{\bullet},d)$, $(C'_{\bullet},d')$ deux $\gA$-complexes montants et $f$, $g : C_{\bullet} \to C'_{\bullet}$
deux morphismes de complexes de degré 0.  \\
Une \emph{homotopie de $C_{\bullet}$ dans $C'_{\bullet}$ reliant $f$ à $g$} est une application  $h : C_{\bullet} \to C'_{\bullet}$ de degré
$-1$, i.e. telle que $h(C_i) \subseteq C'_{i-1}$ pour tout $i$, vérifiant:%
\index{homotopie!reliant deux morphismes de complexes}

\smallskip \centerline {
\begin{minipage}[c]{4cm}
 \fbox{
$d' \circ h + h \circ d = g-f$}
\end{minipage}
\qquad
\begin{minipage}[c]{6cm}
$\xymatrix {
\ar[r] & C_{i-1}\ar@/_6pt/[d]_g\ar@/^6pt/[d]^f\ar[r]^{d} & C_i\ar[dl]_(.4)h\ar@/_6pt/[d]_g\ar@/^6pt/[d]^f\ar[r]^{d}
     & C_{i+1}\ar@/_6pt/[d]_g\ar@/^6pt/[d]^f\ar[dl]_(.4){h} \ar[r] &
\\
\ar[r] & C'_{i-1} \ar[r]_{d'} & C'_{i} \ar[r]_{d'}
     & C'_{i+1} \ar[r] &
\\
}
$
\end{minipage}
}

\begin{fact} \label{factHomotopHomolg}
Dans la situation décrite ci-dessus, on a 

\centerline{\fbox{$\rH(f) = \rH(g)$}.}

\end{fact}
 En effet, soit $x \in \Ker d$;
alors 

\snic{(g-f)(x) = (d' \circ h + h \circ d)(x) = d'(h(x)),}

\snii
et
donc $(g-f)(x)$ est un $d'$-bord. Il s'ensuit que $\rH(g-f) = 0$,
\hbox{i.e. $\rH(f) = \rH(g)$}.
\\
Notez que l'application $h$, purement ensembliste, n'est soumise à aucune condition de type \gui{\lin}.

\snii
On a une notion analogue pour deux complexes descendants
$(C_{\bullet},\partial)$, $(C_{\bullet}',\partial')$ et des morphismes de complexes $f$, $g : C_{\bullet} \to
C_{\bullet}'$.  \\
Une \emph{homotopie de $C_{\bullet}$ dans $C_{\bullet}'$ reliant $f$ à $g$} est une application
 $h : C_{\bullet} \to C_{\bullet}'$ de degré $1$, i.e. $h(C_i) \subseteq C'_{i+1}$ pour tout $i$,
vérifiant:

\smallskip \centerline {
\begin{minipage}[c]{4cm}
 \fbox{
$\partial' \circ h + h \circ \partial = g-f$}
\end{minipage}
\qquad
\begin{minipage}[c]{6cm}
\xymatrix {
\ar[r] & C_{i+1}\ar@/_6pt/[d]_g\ar@/^6pt/[d]^f\ar[r]^{\partial} & C_i\ar[dl]_(.4)h\ar@/_6pt/[d]_g\ar@/^6pt/[d]^f\ar[r]^{\partial} 
     & C_{i-1}\ar@/_6pt/[d]_g\ar@/^6pt/[d]^f\ar[dl]_(.4){h} \ar[r] &
\\
\ar[r] & C'_{i+1} \ar[r]_{\partial'} & C'_{i} \ar[r]_{\partial'} 
     & C'_{i-1} \ar[r] &
\\
}
\end{minipage}
}

On déduit des considérations précédentes le résultat très utile suivant.

\begin{propdef} \label{propEquivHomotop}
Une \emph{équivalence d'homotopie} du complexe~$C_{\bullet}$ vers le complexe~$C'_\bullet$ (tous deux montants, ou tous deux descendants)
est donnée par deux couples $(f,g)$ et $(h,h')$ où%
\index{equivalence d'ho@\eqvc d'homotopie!entre deux complexes} 
\begin{itemize}
\item $f:C_{\bullet}\to C'_\bullet$ et $g:C'_\bullet\to C_{\bullet}$ sont des morphismes de complexes, 
\item $h$ est une homotopie reliant $g\circ f$
à $\Id_{C_{\bullet}}$, et 
\item  $h'$ est une homotopie reliant $f\circ g$
à $\Id_{C'_\bullet}$.
\end{itemize}
On dit alors que les deux complexes sont \emph{homotopiquement \eqvs},
ou encore qu'ils ont \emph{le même type d'homotopie}.%
\index{homotopiquement \eqvs!complexes ---}%
\index{type d'homotopie!d'un complexe}
\\
Dans ce cas $\rH(f)$ et $\rH(g)$ sont des \isos réciproques.
\end{propdef}

\subsec{Caractéristique d'Euler-Poincaré}

\begin{definition} \label{defiCarEuPoComp}
La  \emph{\cEP} d'un complexe
borné de modules libres\index{caracteristique@caractéristique d'Euler-Poincaré!d'un complexe borné de modules libres} 
$$
C_{\bullet}:\quad \quad  \gA^{p_n}\vvers{u_n} \cdots\cdots\cdots\;\gA^{p_1} \vvers{u_1} \gA^{p_0}
$$ 
est l'entier
$\sum_{k=0}^{n}(-1)^{k}p_k$, noté $\chi_\gA(C_{\bullet})$ ou $\chi(C_{\bullet})$.
C'est un entier bien défini si l'anneau est non trivial.
\end{definition}

Supposons que l'anneau $\gA$ soit intègre non trivial.
\'Etendons les scalaires au corps des fractions $\gK$, 
et supposons que le complexe
$$ 0\to \gK^{p_n}\vvers{u_n} \cdots\cdots\cdots\;\gK^{p_1} \vvers{u_1} \gK^{p_0} 
$$
soit exact\footnote{On dit alors que $C_{\bullet}$ est \gnqt exact.}. 
On voit que
$u_n$ est de rang $p_n$ sur $\gK$ et que $p_{n-1}=p_n+r_{n-1}$ pour un certain
entier $r_{n-1}\geq 0$. D'où ensuite $u_{n-1}$ est de rang $r_{n-1}$, \hbox{et  
 $p_{n-2}=r_{n-1}+r_{n-2}$}  pour un certain
entier $r_{n-2}\geq 0$. On a ainsi une suite d'entiers $r_k\geq 0$  avec
$p_k=r_k+r_{k-1}$ pour tout $k$, et $\chi_\gA(C_{\bullet})=r_0\geq 0$.
Si en outre $u_1$ est surjective (pour $\gK$), on a $\chi_\gA(C_{\bullet})=r_0= 0$. 

Ces résultats  seront \gnes pour un anneau commutatif arbitraire (ce sera l'objet du \thref{thStrLocResFin}).

Nous nous contenterons pour le moment des faits suivants.

\begin{fact} \label{factCompExPro} ~
On considère un complexe exact de modules \pros 
$$
P_{\bullet}:\quad \quad \cdots\cdots\lora  P_n\vvers{u_n} \cdots\cdots\cdots\;P_1 \vvers{u_1} P_0\lora 0.
$$ 
Alors le complexe est \emph{entièrement décomposé} (chaque noyau est facteur direct),
et il est isomorphe à une somme directe de complexes \gui{triviaux} 
\[ 
\begin{array}{rcccccccccccccccc} 
 \cdots\cdots  \lora  & 0  & \lora  & Q_{k}  & \vvers{\Id_{Q_k}}  & Q_k  &   \lora & 0 & \lora \cdots\cdots \\[1mm]
 &k+1 & & k & & k-1 && k-2
 \end{array}
\]
pour $k\geq 1$, où les
$Q_k$ sont des modules \pros. Enfin si les $P_i$ sont \tf, \stls ou de rang constant, il en va de même des $Q_i$.    
\index{complexe!trivial}\index{trivial!complexe (exact) ---}%
\index{complexe!entièrement décomposé}\index{entierement d@entièrement décomposé!complexe (exact) ---}
 \end{fact}
%
\begin{proof}
Puisque les modules $P_i$ sont \pros, on voit de proche en proche en partant de $u_1$ que la suite est entièrement décomposée.
 On obtient successivement avec  $Q_1=P_0$ 
 
\snic{
P_1\simeq Q_2\oplus Q_1,\;\;P_2\simeq Q_3\oplus Q_2,\;\; \dots,\;\; P_{n-1}\simeq Q_{n}\oplus Q_{n-1}, \;\dots
\,.}

Ceci donne l'\iso avec une somme directe de complexes triviaux. Matriciellement, on peut décrire les \alis $u_j$ comme suit
$$
{u_{k}} =\blocs{.7}{1.1}{1.1}{.9}{$0$}{$\Id_{Q_k}$}{$0$}{$0$}\;, \;\;\;
{u_{1}} =\blocs{1.5}{1}{1}{0}{$0$}{$\Id_{Q_1}$}{}{}\;.
$$
Enfin il est clair que les $P_i$ sont \tf, de rang constant ou \stl \ssi les $Q_i$  le sont. 
\end{proof}
%
\begin{fact} \label{factCharCompEx0} ~
\begin{enumerate}
\item Soit un complexe borné exact de modules libres 
$$
L_{\bullet}:\quad \quad 0\lora  L_n\vvers{u_n} \cdots\cdots\cdots\;L_1 \vvers{u_1} L_0\lora 0.
$$ 

Si $\chi(L_{\bullet})\neq 0$, l'anneau~$\gA$ est trivial.

\item 
Si un complexe exact borné  $L_{\bullet}$ comme ci-dessus est formé de \mptfs, on a:
\begin{enumerate}
\item ${
\sum_{i}(-1)^{i}\rg_\gA(L_i)=0 \hbox{ dans }\HO(\gA),}$  
\item le complexe est entièrement décomposé et il est isomorphe à une somme directe finie de complexes triviaux.    
\end{enumerate}%
\end{enumerate}

 \end{fact}
%
\begin{proof}
Il suffit de montrer le point \emph{2}. Or \emph{2b.} est un cas particulier
du lemme~\ref{factCompExPro}, et \emph{2a.} résulte clairement de \emph{2b.}
\\
Matriciellement, on peut décrire cette fois-ci les \alis $u_j$ comme suit
$$
{u_{n}} =\blocs{0}{1.3}{1.3}{.6}{}{$\Id_{Q_n}$}{}{$0$}\;,\;\;\;
{u_{k}} =\blocs{.7}{1.1}{1.1}{.9}{$0$}{$\Id_{Q_k}$}{$0$}{$0$}\;, \;\;\;
{u_{1}} =\blocs{1.5}{1}{1}{0}{$0$}{$\Id_{Q_1}$}{}{}\;.
$$
Alors avec  $Q_{0}=Q_{n+1}=0$, on a $\rg_\gA(P_k)=\rg_\gA(Q_k)+\rg_\gA(Q_{k+1})$ pour tout $k\in\lrbn$.
\end{proof}

\hum{Un complexe somme directe de triviaux a même type d'homotopie que quoi?
}

\bonbreak
\section{Modules $n$-présentables, \lrsbs} \label{sec-n-Pres}

\subsec{Modules $n$-présentables
}
Une \emph{$n$-\pn (libre finie)
du module $M$}   est donné par une
\sex
\begin{equation} \label {eqnpres}
L_n\vvers{u_n} L_{n-1}\vvers{u_{n-1}} \cdots\cdots \vvers{u_1} L_0\vvers{u_0} M \to 0
\end{equation}
où les $L_i$ sont des \mlrfs.
La \sex \pref{eqnpres} est aussi appelée un \emph{début de \rlf de longueur $n$   du \Amo $M$}. 

\begin{definota} \label{definotaLdimmodules}
Un \Amo qui possède une $n$-\pn est dit \emph{$n$-\pfb}.
On écrit ce renseignement sous la forme condensée suivante: $\Ld_\gA(M)\geq n$, \hbox{ou $\Ld(M)\geq n$}. On écrit $\Ld(M)= \infty$ pour dire que
l'on a $\Ld(M)\geq n$ pour tout~$n$.
\index{n-pres@$n$-\pn!d'un module} 
\index{module!n-pres@$n$-présentable}  
\end{definota}

\entrenous{Dans \cite{Glaz}, il y a la notation $\lambda(M)\geq n$, d'où ici le \textsf{Ld}, mais ce serait peut-être mieux d'utiliser \textsf{Pres}.
Et dans Bourbaki?

Par ailleurs, les \pns projectives non \ncrt \tf semblent pertinentes
pour la définition des Ext, mais en \coma on peut se rabattre sur la \dfn à la Yoneda, qui fonctionne pour toute catégorie abélienne. 

On serait donc tenté d'introduire
une notation $\Ld_\gA(M,\textsf{pro})\geq n$ pour dire l'existence
d'une \sex~\pref{eqnpres} avec des $L_i$ \pros
(en \clama, on a toujours $\Ld_\gA(M,\textsf{pro})=\infty$, un cas possible avec des modules \pros non \tf en \coma est celui des modules discrets
sur un anneau discret).

Si l'on se résout à introduire la notation $\Ld_\gA(M,\mathsf{pro})\geq n$
en vue de la \dfn des Ext, il faudra alors dire que $\Ld_\gA(M)\geq n$
est une abréviation pour $\Ld_\gA(M,\mathsf{fif})\geq n$, où \textsf{fif}
est mis pour \gui{finite free}.
}

Un \Amo est $0$-\pfb  \ssi
il est \tf, il est~\hbox{$1$-\pfb}  \ssi
il est \pf. \\
Une $n$-\pn plus longue permet d'avoir plus de 
renseignements sur le module.

Un anneau est \coh lorsque tout \mpf $M$ satisfait~\hbox{$\Ld(M)= \infty$}. En effet, comme le noyau de toute matrice
est un \mtf, on peut construire de proche en proche des \mlrfs~$L_k$ et des \alis $L_k\to L_{k-1}$ qui rendent la \sex.

Lorsque l'on a une \sex \pref{eqnpres} avec pour $L_i$ 
des modules plats
(resp. des modules \pros, resp. \ptfs), on dit que l'on a une~\emph{$n$-\pn plate}
(resp. une \emph{$n$-\pn projective}, resp. une~\emph{$n$-\pn projective \tf}) du module $M$. 
\\
En fait toute $n$-\pn  projective \tf se transforme facilement en une $n$-\pn libre finie (lemme~\ref{lempnptfpnlibre}).

\subsec{Modules \lrsbs, \lorsbs, dimension projective finitaire}
Une $n$-\pn du module $M$ est appelée une \emph{\rlf (de longueur~$n$)} lorsque l'\ali $u_n:L_n\to L_{n-1}$
est injective, \cad lorsque la suite  
\begin{equation}\label{eqrlf} 
  0\to L_n\vvers{u_n} L_{n-1}\vvers{u_{n-1}} \cdots\cdots \vvers{u_1} L_0\vvers{u_0} M \to 0
\end{equation}
est exacte. 

Lorsque l'on a une \sex \pref{eqrlf} avec pour $L_i$ des modules plats
(resp. des modules \ptfs, des modules \stls), on dit que l'on a une \emph{\rsn plate de longueur finie}
(resp. une \emph{\rsf}, une \emph{\rsn \stl finie}) du module $M$.

\begin{definota} \label{definresol}\label{defiPd}
\begin{enumerate}
\item Si un \Amo $M$ admet une \rsn libre de longueur $n$,  nous dirons qu'il est \emph{\lnrsb}.
\index{resolution libre@\rsn libre!d'un module} 
\item Si $M$ admet \rsf de longueur $n$, nous dirons qu'il est \emph{\lonrsb}. 
On dira aussi qu'il est de 
\emph{\pdi (finitaire)}~$\leq n$. Dans ce cas, on adopte la notation~\hbox{\gui{$\Pd_\gA(M)\leq n$}}, \hbox{ou \gui{$\Pd(M)\leq n$}}.
Ainsi $\Pd_\gA(M)\leq 0$ signifie que $M$ est \ptf.
Dans la suite on omet en \gnl le mot finitaire après \gui{\pdi}.%
\index{dimension projective!(finitaire) d'un module}
\item Si l'entier $n$ n'est pas précisé on dira 
dans le premier cas que le module est \emph{librement \rsb}, et dans le second cas qu'il est \emph{\lot \rsb}.%
\index{module!localement r@\lorsb}%
\index{module!librement r@\lrsb}%
\index{module!librement $n$-\rsb}%
\index{module!localement $n$-\rsb}%
\index{librement res@\lrsb!module ---}%
\index{localement res@\lorsb!module ---}%
\index{resoluble@\rsb!module ---}%
\index{resolution proj@\rsn projective!d'un module} 
\item Enfin $\Pd_\gA(M)= -1$ signifie que $M=0$.
\end{enumerate}
\end{definota}
\comm 
A priori, la \pdi d'un \mlorsb n'est pas un entier
bien défini\footnote{Une discussion analogue a eu lieu  dans \Cref{section XIII-2} concernant la dimension de Krull.}: 
on a seulement défini la phrase 
\gui{le \Amo~$M$ est de \pdi~\hbox{$\leq k$}}.

En \clama, on définit $\Pd_\gA(M)$ pour n'importe quel module comme un \elt de l'ensemble~\hbox{$\so{-1}\cup\NN\cup\so{+\infty}$}: la borne inférieure des $k$ tels que $\Pd_\gA(M)\leq k$.

Si $\gA$ est \fdi et non trivial, lorsqu'un module~$M$ est \lorsb, on peut décider quel est le meilleur entier~$k$ tel que $\Pd(M)\leq k$,
ce qui permet de définir \cot la phrase \gui{\emph{le \Amo~$M$ est de \pdi $k$}}. Ceci résulte du \tho de permanence \ref{corLScha} (qui se ramène pour l'essentiel au lemme de Schanuel, voir \ref{lemSchanuelVariation}) et du fait que dans le cas où $\gA$ est \fdi et non trivial, on est capable de tester si un \mpf est \pro.
On adoptera alors la notation~\hbox{$\Pd(M)= k$}. 
\\
Dans le cas où l'on ne connait pas $\Pd(M)$ de manière certaine,
il faut reformuler \cot certaines inégalités exprimées en \clama, comme on l'a déjà fait pour la \ddk. La formulation \cov s'applique en toute généralité sans recours au principe du tiers exclu et elle est \eqve
en \clama à la formulation classique. 
\\ 
Par exemple l'inégalité \fbox{$\Pd(M)\leq \sup\big(\Pd(N),\Pd(P)\big)$} est une abréviation pour l'implication: 
$$
\fbox{$\forall k\geq -1, \,\big(\Pd(N)\leq k \hbox{ et} \Pd(P)\leq k\big) \;\Longrightarrow\; \Pd(M)\leq k.$}
$$
Cette formulation \cov signifie \prmt ceci: 
\emph{si $N=P=0$, \hbox{alors $M=0$} et, pour $k\geq 0$, si l'on connait des \rsfs
de longueurs $\leq k$ pour $N$ et $P$, on peut construire une \rsf
de longueur $\leq k$ pour $M$}.
\eoe

\hum{Dans la littérature usuelle on n'impose pas aux modules \pros dans la \rsn d'être \tf. Si on se réfère à Glaz p.61, il faudrait sans doute utiliser ``small finitistic projective dimension'', ce qui n'est vraiment pas élégant. En outre Glaz p.244, dont le lien avec Glaz p.61 n'est pas évident, me laisse perplexe.
Il faudra s'expliquer sur de sujet.

On introduira plus tard $\pd_\gA(M)$ pour la \pdi ``usuelle''
dans la littérature, tout en émettant quelques bémols sur sa pertinenece.

En fait dans le chapitre V du volume 1, on aurait sans doute déjà d\^u émettre des bémols concernant l'acceptabilité de la \dfn \gnle de \gui{module \pro} qui nécessite une quantification sur la classe de tous les \Amos, laquelle n'est pas un ensemble.} 

\medskip \goodbreak\exls
\\
1) Le \tho des syzygies de Hilbert affirme que tout
\mgpf  sur~$\kXn$ admet une \rlf graduée de longueur~\hbox{$\leq n$} \ttt{(par des morphismes de degré $0$?)} (ici~$\gk$ est un \cdi).
Ce \tho a été \gne au cas non homogène, puis à d'autres situations qui ont conduit à la notion d'anneau \corg\siBookdeux{ (que l'on étudiera dans le chapitre \ref{chapAnnCohReg})}.

\smallskip 2) Par ailleurs un \id engendré par une \srg est librement résolu par le complexe de Koszul descendant associé à cette 
suite\siBookdeux{ (chapitre~\ref{chapHomalgBasic})}.  

\smallskip  3) Soit $\gk$ un \cdi $\gA=\gk[x,y]=\aqo{\gk[X,Y]}{XY}$ et~$\gB=\gA_{1+\gen{x,y}}$. L'\id $x\gB$ de~$\gB$ admet une \rsn infinie par des modules libres de rang~$1$:
$$ \cdots\cdots\lora\gB\vvers x\gB\vvers y\gB\vvers x\gB\vvers y\gB\vers{x}x\gB\to 0
$$
(le $x$ au dessus de la flèche signifie \gui{multiplication par~$x$}).
Comme les \gui{matrices} dans cette \rsn sont à \coes dans l'\idema de l'\alo
$\gB$, la \rsn est essentiellement unique (cf. le \thref{thBetti})
et donc le module~$x\gB$ n'admet pas de \rlf. 

\smallskip  4) Si $P\oplus Q= L=\Ae n$, le module \pro $P$ admet une \pn de longueur infinie
$$
\cdots\lora L \vvers{\pi_P} L \vvers{\pi_Q} L \vvers{\pi_P} L \vvers{\pi_Q}L\vvers{\pi } P \lora 0,
$$
où $\pi_P:L\to L$ est le \prr d'image $P$ et de noyau $Q$,  $\pi_Q=\Id-\pi_P$ \hbox{et $\pi(x)=\pi_P(x)$} pour $x\in L$.

\smallskip 5)  Par \dfn un module stablement libre admet une \rlf de longueur~$1$,
mais il n'admet une \rlf de longueur~$0$ que s'il est libre.
\eoe

\medskip 
Dans une certaine mesure l'étude des \rsfs se ramène à l'étude
des \rlfs. D'une part, on peut toujours remplacer une \rsf par une \rsn
où tous les modules sont libres finis, à l'exception du premier
module, noyau de la première flèche (proposition \ref{cor2lemModifComplexe}). \sibook{Voir \egmt à ce sujet l'exercice \ref{exoProjResolutionToFreeResolution}.}
D'autre part, le \tho de structure locale des \mptfs implique
qu'après \lon en des \eco une \rsf devient une \rlf comme
expliqué dans la proposition suivante (voir aussi le
lemme \ref{lemlorsbloclresb}). 

\begin{proposition} \label{propRSFRLF}
\'Etant donnée une \rsf d'un \Amo $M$, on peut calculer un \sys d'\eco tels qu'après \lon en chacun de ces \elts, la \rsn devienne une \rlf. 
Réciproquement, si après \lon en des \eco le module $M$ possède
chaque fois une \rlf, alors il possède une \rsf.
\end{proposition}

Notons que l'on peut aussi passer au localisé de Nagata,
qui est une extension \fpte, et la \rsf devient une \rsn quasi libre finie. 

\subsec{Résolutions  très courtes}

\paragraph{Résolutions libres finies}~

\begin{enumerate}
\item Un module est \lrf \ssi  il  admet une \rlf de longueur~$0$.
\item Un anneau $\gA$ est un \cdi \ssi tout module $\aqo\gA a$ est libre, \ssi tout \mpf est libre.
\item Les  anneaux de Bezout intègres sont \cares par le fait que tout \itf est un module libre de rang~$1$ ou~$0$, ou encore que sur un tel anneau toute matrice admet pour noyau un \mlrf.
\end{enumerate}

%

\paragraph{Résolutions projectives finies}~

\begin{enumerate}
\item Un module est de \pdi $\leq 0$ \ssi il est \ptf.
\item Il est de \pdi $\leq 1$ \ssi il est le quotient d'un \mptf par un sous-module \ptf.
\item Un anneau est un \adp \coh \ssi tout \itf est \pro \ssi tout \mpf $M$ satisfait $\Pd(M)\leq 1$.
\end{enumerate}

%

\section[Idéaux \caras d'un complexe]{Idéaux caractéristiques d'un complexe}\label{secIdesCars}

\subsec{Modifications  \elrs d'un complexe}

 Nous décrivons dans le lemme qui suit des transformations \elrs d'un complexe qui ne changent pas les groupes d'homologie. Il s'agit en fait d'un cas particulier très simple d'homotopie entre complexes.
 
\begin{lemma} \label{lemModifComplexe} \emph{(Modification  \elr  d'un complexe)}%
\index{modification \elr!d'un complexe}\index{elementaire@\elr!modification --- d'un complexe}\\
{1.} Soit $E$ un \Amo. Considérons le complexe $\wi{C_{\bullet}}$  obtenu par une modification  d'un  complexe  $C_{\bullet}$ comme suit 
%
$$
\wi{C_{\bullet}}:\quad \cdots\; C_{k+2}\vvvers{\wi{u_{k+2}}}C_{k+1}\times  E\vvvers{\wi{u_{k+1}}}C_k\times  E\vvvers {\wi{u_{k}}} C_{k-1} \;\cdots
$$
%
où les \alis $\wi{u_{k+2}}$, $\wi{u_{k+1}}$, $\wi{u_{k}}$ sont définies
par
 $$
 \wi{u_{k+2}}(x)=(u_{k+2}(x),0),\quad  \wi{u_{k+1}}(y,z)=(u_{k+1}(y),z),\quad \wi{u_{k}}(t,z)=u_{k}(t),
 $$
matriciellement:
$$\preskip.0em \postskip.2em
\wi{u_{k}} =\blocs{1.1}{.6}{.8}{0}{$u_{k}$}{$0$}{}{}\;,\;\;\;
\wi{u_{k+1}} =\blocs{1.3}{.6}{1.1}{.6}{$u_{k+1}$}{$0$}{$0$}{$\I_E$}\;, \;\;\;
\wi{u_{k+2}} =\blocs{0}{1.5}{1.3}{.6}{}{$u_{k+2}$}{}{$0$}\;.
$$
On a alors les résultats suivants.
\begin{enumerate}
\item [{a.}] \gui{L'homologie ne change pas}, \cad 
que l'on a des
\isos canoniques~$\rH_j(C_{\bullet})\simeq \rH_j(\wi{C_{\bullet}})$
\hbox{pour $k-1\leq j\leq k+2$}.
\item [{b.}] Si les modules considérés sont libres, avec 
$E$ de rang $r$, on obtient pour tout $\ell$

\snic{\cD_{\ell}(u_k)=\cD_{\ell}(\wi{u_k}), \;
\cD_{\ell}(u_{k+1})=\cD_{\ell+r}(\wi{u_{k+1}}), \;
\cD_{\ell}(u_{k+2})=\cD_{\ell}(\wi{u_{k+2}}) \,.
}
\end{enumerate}

\snii  
{2.} Inversement, supposons que l'on a des \dcns

\snic{C_{k+1}=\ov C_{k+1}\oplus E \;\hbox{  et  } \;C_k=\ov C_k\oplus F,}

\snii
avec $\formule{u_{k+1}(\ov C_{k+1})\subseteq \ov C_k,\;\;u_{k+1}(E)\subseteq F, \\[1.5mm]
               u_{k+1} \hbox{ réalise un \iso de } E \hbox{ sur } F .}$

\snii 
Considérons le complexe $\ov{C_{\bullet}}$  obtenu par une telle modification \elr
%

\snic{
\cdots\; C_{k+2}\vvvers{\ov u_{k+2}}\ov C_{k+1}\vvvers{\ov u_{k+1}}\ov C_k\vvvers {\ov u_{k}} C_{k-1} \;\cdots,
}

%
où $\ov u_{k+2}$, $\ov u_{k+1}$ et $\ov u_{k}$ sont obtenues comme induites par $u_{k+2}$, $u_{k+1}$, $u_{k}$. \\
On a alors les résultats suivants.
\begin{enumerate}
\item [{a.}] \gui{L'homologie ne change pas}, \cad 
que l'on a des
\isos canoniques~$\rH_j(C_{\bullet})\simeq \rH_j(\ov{C_{\bullet}})$
\hbox{pour $k-1\leq j\leq k+2$}.
\item [{b.}] Si les modules considérés sont libres, avec 
$E$ de rang $r$, on obtient pour tout $\ell$

\snic{\cD_{\ell}(u_k)=\cD_{\ell}(\ov{u_k}), \;
\cD_{\ell+r}(u_{k+1})=\cD_{\ell}(\ov{u_{k+1}}), \;
\cD_{\ell}(u_{k+2})=\cD_{\ell}(\ov{u_{k+2}}) \,.
}
\end{enumerate}
Matriciellement:
$$\preskip.0em \postskip.2em
{u_{k}} =\blocs{1.1}{.6}{.8}{0}{$\ov{u_{k}}$}{$0$}{}{}\;,\;\;\;
{u_{k+1}} =\blocs{1.3}{.6}{1.1}{.6}{$\ov{u_{k+1}}$}{$0$}{$0$}{iso}\;, \;\;\;
{u_{k+2}} =\blocs{0}{1.5}{1.3}{.6}{}{$\ov{u_{k+2}}$}{}{$0$}\;.
$$
\end{lemma}
\begin{proof}
La \dem est laissée \alec. Notez que dans le point \emph{2}
on a $\Im(u_{k+2})\subseteq \ov{C_{k+1}}$ et $F\subseteq \Ker(u_k)$ parce que $C_{\bullet}$ est un complexe.
\end{proof}

Notons que les \egts en \emph{1b} et \emph{2b} concernant les \idds restent valables dans toutes les situations en utilisant la convention habituelle selon laquelle les \idds
 d'une matrice vide sont égaux à~$\gen{1}$.

Le corolaire suivant sera appliqué de manière systématique dans la suite.

\begin{corollary} \label{corlemModifComplexe}
Soit un complexe (descendant) dans lequel on a
pour un certain indice 
$m$, $C_{m+1}=\Ae r$, $C_m=\Ae s$. 
On suppose que $u_{m+1}$ est représenté par une matrice  $A_{m+1}\in \Ae{s\times r}$ qui possède un mineur
d'ordre $k$ inversible. 
Alors on peut remplacer $C_{m+1}$, $C_m$,
 $u_{m+1}$, $u_{m+2}$ \hbox{et $u_m$}  \hbox{par $C'_{m+1}=\Ae{r-k}$},  $C'_m=\Ae{s-k}$ et 
des \alis convenables~$u'_{m+1}$, $u'_{m+2}$ et $u'_m$ 
 sans changer l'homologie du complexe. 
 \\
 En outre $\cD_\ell(u'_{m+1})=\cD_{\ell+k}(u_{m+1})$ pour tout $\ell$.
 \\
Notons $C_\bullet'$ le complexe modfié. Alors $C_{\bullet}$ est isomorphe à la somme directe du complexe
$C_\bullet'$ et du complexe trivial $\dots\to0\to E \vvers{\Id_E}E\to 0 \dots$ ($E=\Ae k$)
avec les indices $m+1$ et $m$ pour $E$.  
\end{corollary}
%
\begin{proof}
$\!$Utiliser le lemme du mineur inversible \ref{FFRlem.min.inv} 
et le  lemme~\ref{lemModifComplexe}, point~\emph{2.} 
\end{proof}

Deux conséquences importantes du corolaire \ref{corlemModifComplexe} sont données dans les \thos~\ref{cor3lemModifComplexe} et~\ref{thStrLocResFin}.

L'écart entre une \rlf et une \rsf est très faible comme le montre la proposition qui suit. \sibook{\`A  ce sujet voir aussi l'exercice \ref{exoProjResolutionToFreeResolution}.}

D'un certain point de vue, la notion de \rsf est un meilleur concept
car c'est un concept \lgb (voir le principe \ref{plcc.resf}), alors qu'un \mptf qui n'est pas \stl ne peut pas admettre de \rlf.

D'un autre point de vue, la notion de \rlf est plus fondamentale car elle aboutit à des résultats plus précis, et aussi parce qu'une \rsf n'est jamais
qu'une \rsn qui devient libre finie après \lon en des \eco.

\begin{proposition} \label{cor2lemModifComplexe}
Si l'on a une \rsf  

\snic {
0 \to P_n \lora  P_{n-1} \lora \; \cdots\cdots \;  \lora P_0 \lora M
\to 0,
}

\snii
(avec $n\geq 1$) d'un \Amo $M$, on a aussi une \rsf de même longueur

\snic {
0 \to P'_n \lora  L_{n-1} \lora \; \cdots\cdots \;  \lora L_0 \lora M
\to 0,
}

\snii
où $L_{n-1}$, \dots, $L_0$ sont \lrfs. En outre si les $P_i$ pour $i\leq k$ sont déjà libres, il suffit de modifier les $P_j$ pour $j> k$.
\end{proposition}
%
\begin{proof} On applique de manière répétée la première partie du lemme \ref{lemModifComplexe}.
On complète $P_0$ en $L_0=P_0\oplus Q_0$ pour en faire un module libre,
simultanément on remplace $P_1$ par $\wi{P_1}=P_1\oplus Q_0$. Ensuite on complète $\wi{P_1}$ pour en faire un module libre $L_1$ et l'on modifie simultanément $P_2$. On continue jusqu'à rendre tous les termes $L_i$ libres sauf éventuellement
le terme $P'_n$. \\
Le même raisonnement s'applique si les  $P_i$ sont libres pour $i\leq k$:
on commence la modification à $P_{k+1}$.
\end{proof}

De la même manière on obtient le résultat suivant.
\begin{lemma} \label{lempnptfpnlibre}
Toute $n$-\pn  projective \tf d'un module donne, par transformations
\elrs du complexe, une $n$-\pn (libre finie) du même module.   
\end{lemma}

\begin{lemma} \label{lemlorsbloclresb}
Si $E$ est un \mlonrsb, il existe des \eco $(s_i)$ tels qu'après
\lon en chaque $s_i$, le module devient  \lnrsb. 
\end{lemma}
%
\begin{proof}
On considère une \rsn de $E$ par des \mlrfs, sauf le premier
qui est \ptf (proposition \ref{cor2lemModifComplexe}). Ce module
devient libre après \lon en des \eco.
\\
En fait, n'importe quelle \rsn de $E$ par des \mptfs devient une \rlf après \lon en des \eco convenables. 
\end{proof}
%

\subsec{Complexes \gnqt exacts et autres}

Considérons  un  complexe descendant de \mlrfs, borné, se terminant à l'indice $0$.

\fnic{L_{\bullet}: \quad 0 \to L_m \vvers{A_m}  L_{m-1} \vvvers{A_{m-1}}\;  \cdots \cdots \; \vvers{A_2}  L_1 \vvers{A_1} L_0,\;\; L_k=\gA^{p_k}. 
}

\smallskip
Lorsque l'anneau de base $\gA$ est intègre, on dit que le complexe est \emph{\gnqt exact}
s'il devient exact  quand on étend les scalaires au corps de fractions~$\gK$.\index{complexe!generiquement@\gnqt exact}%
\index{generiquement@\gnqt exact!complexe}
Dans un tel cas,  les modules d'homologie sont des modules de torsion sur~$\gA$, annulés par un même \elt de $\Reg(\gA)$: en effet, ils sont \pf et ils deviennent nuls par \eds à $\gK$.
Et si la matrice~$A_k$ est de rang $r_k$ sur $\gK$,  
on obtient pour tout $k\in\lrbm$

\smallskip \centerline{\hbox{$L_k= \gA^{r_{k+1}+r_k}$}\,,} 

à condition de définir  $r_{m+1}=0$ et $r_0=p_0-r_1$ (cf. supra la discussion après la \dfn \ref{defiCarEuPoComp}).
 Notons  que $r_0=\chi(L_{\bullet})$.

\smallskip 
Dans la suite on ne suppose pas \ncrt l'anneau intègre, 
on considère un complexe borné de \mlrfs,
et l'on garde l'hypothèse que l'on a des entiers $r_k\in\NN$ tels que \fbox{$p_k=r_{k+1}+r_k$} pour tout $k\in\NN$, \hbox{avec  $p_\ell=0$} \hbox{pour $\ell> m$}.\\
Sous ces hypothèses on dira que $r_k$ est le \ix{rang stable attendu}\footnote{Le rang stable d'une matrice sur un anneau intègre est le rang sur le corps de fractions. Plus \gnlt le rang stable d'une matrice est défini en \ref{defiRangStable}, et l'on démontre en \ref{thRangStRLF}
que si $L_{\bullet}$ est exact chaque matrice $A_k$ est de rang stable~$r_k$.} de la matrice~$A_k$.\\
Si l'anneau $\gA$ est non trivial, on a \ncrt $r_\ell+r_{\ell+1}=0$
\hbox{pour $\ell>m$}, donc~$r_{m+1}=0$ et $r_m$ est le rang 
du module $L_{m}$: la suite des~$r_k$ est  définie sans ambig\"uité à partir de la suite
des $p_k$, elle-même bien définie. Il reste cependant une contrainte
à imposer à la suite des $p_k$, c'est que la suite des~$r_k$ qui en résulte soit entièrement dans $\NN$ (en fait seuls importent les
indices~\hbox{$k\in\lrb{0..m}$}). On se simplifie donc un peu la vie
en donnant les~$r_k$ plutôt que les $p_k$.

\smallskip 
On s'intéresse maintenant aux invariants d'un tel complexe lorsqu'on lui fait subir des modifications  \elrs
telles que celles décrites dans le paragraphe précédent (lemme \ref{lemModifComplexe}, avec $E$ libre), ou encore des changements de base pour les modules $L_k$. 

\begin{definition} \label{defiCompeltequiv}
Deux complexes bornés de \mlrfs sont dits \emph{\elrt \eqvs} si l'on 
peut passer de l'un à l'autre (à \iso près) 
par une suite finie de modifications
\elrs comme dans le lemme~\ref{lemModifComplexe}, avec $E$ \lrf, rajouté en somme directe, ou retranché, à deux modules $L_k$ et $L_{k+1}$ pour un $k\geq 0$.%
\index{elementairement equivalents@\elrt \eqvs!complexes de modules libres
de rang fini ---} 
\end{definition}

On note pour commencer que dans une modification  \elr 
opérée dans le lemme \ref{lemModifComplexe}, avec $E$ libre,  le seul rang stable attendu modifié \hbox{est  $r_{k+1}$}. En particulier, la \cEP $r_0$ ne change pas.
On note aussi en posant \fbox{$M=\Coker A_1$}  que si l'on considère une modification  \elr de $L_{\bullet}$ avec $k=1$ ou $0$,
 on obtient des transformations de la matrice $A_1$
qui conservent son conoyau à \iso près (voir \Cref{section IV-1}) et que l'on a des \ids invariants 

\smallskip 
\centerline{\fbox{$\cD_{r_1}(A_1)=\cF_{r_0}(M)$}  et   \fbox{$\cD_{r_1-\ell}(A_1)=\cF_{r_0+\ell}(M)$}.}

Les autres modifications \elrs (avec $k>1$) ne changent pas la matrice $A_1$,
donc le module $M$ reste le même et les $\cD_{r_1-\ell}(A_1)$ sont toujours inchangés.

Plus \gnlt, on constate le fait suivant.

\begin{propdef} \label{propdefidecars} On considère des complexes descendants bornés de \mlrfs $L_k$, limités à droite à l'indice $0$, donnés avec  des entiers $r_k\in\NN$ tels \hbox{que $L_k\simeq \Ae{r_k+r_{k+1}}$} pour tout $k\in\NN$, et $r_\ell=0$ si $\ell$ est assez grand.
Si deux tels complexes de modules libres sont \elrt \eqvs, 
ils ont les mêmes modules d'homologie. En outre,
en modifiant corrélativement les rangs stables attendus, les deux complexes ont  les mêmes \idds\label{idecars} pour les numérotations suivantes:  

\smallskip 
\centerline{\fbox{$\fD_k=\fD_k(L_{\bullet})   \eqdefi\cD_{r_k}(A_k)$}   et  
\fbox{$\fD_{k,\ell}=\fD_{k,\ell}(L_{\bullet})  \eqdefi \cD_{r_k-\ell}(A_k)\;$}.}

Nous appellerons ces \ids les \emph{\idcas du complexe $L_{\bullet}$}.%
\index{ideaux cara@\idcas!d'un complexe de modules libres de rang fini}%
\\
Pour $k$ suffisamment grand, les modules et les \alis sont nulles: on  a alors
  $\fD_k=\gen{1}$, et $\fD_{k,\ell}=\gen{1}$ et  
$\fD_{k,-\ell}=0$ pour $\ell>0$.\\
\sibook {Enfin, on utilisera éventuellement la convention $\fD_0(L_{\bullet})=0$.
\ttt{???}}
\end{propdef}
NB: 
les deux complexes ont aussi même \cEP égale à $r_0$, et
les conoyaux des flèches de droite sont isomorphes. 

Si \ncr, pour lever une ambig\"uité, on notera $\fD_{\gA,\ell}(L_{\bullet})$ et $\fD_{\gA,k,\ell}(L_{\bullet})$
en rappelant l'anneau sur lequel on considère les matrices.

\medskip 
\emph{Note.} L'indice  $0$  utilisé pour limiter le complexe à droite est  une convention naturelle si l'on considère que le complexe résout un module donné.  C'est la raison pour laquelle les modifications \elrs du complexe envisagées ne  rallongent jamais sur la droite la liste des modules non nuls.  En pratique, comme on l'a déjà signalé, lorsque l'anneau est non trivial, les rangs stables attendus sont connus sans ambig\"uité.
En bref, la structure de donnée correspondant à la proposition \ref{propdefidecars} est un \emph{complexe borné de \mlrfs
avec rangs stables attendus}; mais lorsque l'anneau est non trivial\footnote{Si l'on ignore si l'anneau est trivial ou non, par exemple après une \lon mal maitrisée, il pourrait s'avérer utile d'avoir les $r_k$ à sa disposition.}, les~$r_k$ sont connus à partir de $L_{\bullet}$, et le fait de les donner a priori est simplement une manière de se simplifier la vie. 
\eoe 

\medskip 
Le  \thref{thRangStRLF} dit que lorsque le complexe $L_{\bullet}$ est exact, les \ids~$\fD_k$ sont fidèles et $\fD_{k,\ell}=0$ pour $\ell<0$.
Signalons aussi le remarquable \thref{cor3thABH1} qui caractérise de manière précise, par des \prts des \ids~$\fD_{k}$, les complexes $L_{\bullet}$  qui sont exacts, et le non moins remarquable \thref{thdetCay} qui précise la \gui{structure multiplicative} des \idcas lorsque le complexe est un complexe de Cayley, a fortiori s'il est exact.

\begin{lemma} \label{lemfDkchgbase} 
Après une \eds $\rho:\gA\to\gB$ les \ids \caras $\fD_{\gA,k}(L_{\bullet})$ et $\fD_{\gA,k,\ell}(L_{\bullet})$, ainsi que le module conoyau de la dernière flèche, sont remplacés par  leurs images  dans $\gB$.
\end{lemma}

Les \edss plates se distinguent ici seulement par le fait qu'elles
préservent l'homologie.

\bonbreak
\section{Résolutions libres minimales}\label{secRSNLF}

\subsec{Existence de \rsns minimales}

\begin{thdef} \label{cor3lemModifComplexe} \emph{(Existence de \rsns minimales)}\\
Soit  un \alrd $\gA$. 
\begin{enumerate}
\item Pour tout complexe de $\gA$-\mlrfs
$$
L_{\bullet}:\quad  L_n \vvers{u_n}  \;   \cdots\cdots \; \vvers{u_1} L_0,
$$
on sait calculer un complexe \elrt \eqv  où 
les matrices~$A_i$ des \alis $u_i$ ($i\in\lrbn$) sont à \coes dans~$\Rad(\gA)$.  
\item Si l'on a une $n$-\pn d'un module, le calcul du point  {1} la ramène à une  $n$-\pn \emph{minimale}, \cad où les matrices
de la \pn sont à \coes dans~$\Rad(\gA)$.%
\index{minimale!$n$-\pn --- d'un module} 
\item Si l'on a une \rlf d'un \Amo $M$,  
$$
0 \to L_n \vvers{u_n}  L_{n-1} \vvers{u_{n-1}}\;  \cdots \cdots \; \vvers{u_2}  L_1 \vvers{u_1} L_0 \vers{\pi} M
\to 0,
$$
la \rsn obtenue  au point  {1}  est dite \emph{minimale}.%
\index{minimale!\rsn libre --- d'un module}%
\index{resolution libre@\rsn libre!minimale d'un module} 
\end{enumerate}
Rappelons que dans cette réduction, la \cEP et les \idcas du complexe $L_{\bullet}$ ne   changent pas
 (mais la longueur~$n$ peut diminuer). 

\end{thdef}
%
\begin{proof}
Puisque l'anneau local est \dcd on peut déterminer pour n'importe quelle matrice un entier $k\geq 0$ et  un mineur inversible\footnote{Rappelons qu'un mineur d'ordre $0$ est égal à $1$.} d'ordre $k$ tels que tous les
mineurs d'ordre $k+1$ sont résiduellement nuls. On applique alors le corolaire \ref{corlemModifComplexe}
successivement avec  $u_1$, \dots , $u_n$. 
\end{proof}

Le résultat précédent généralise le point \emph{1} 
du lemme du nombre de \gtrs local \ref{FFRlemnbgtrlo}.

Soit $\gk=\gA\sur{\Rad(\gA)}$
le corps résiduel. Si 
dans une \rsn minimale de~$M$ on a $L_0=\Ae r$, alors 

\snic{\cF_r(M)=\cD_0(u_1)=\gen{1}  \hbox{ et } \cF_{r-1}(M)=\cD_1(u_1)\in\Rad(\gA).}

\snii
Si l'anneau $\gA$ n'est pas trivial, l'entier $r$ est donc bien défini. C'est aussi la dimension du \kev $\,\gk\otimes_\gA M\simeq M\sur{\Rad(\gA)M}$, et c'est le nombre d'\elts de tout \sgr minimal de $M$.\\
Ce résultat va être \gne dans les \thos~\ref{thBetti} et~\ref{thStrLocResFin}.

\subsec{Unicité des \rsns libres minimales}

Le \thref{cor3lemModifComplexe} affirme dans le cas d'un \alo \dcd l'existence de \rlfs 
pour lesquelles toutes les matrices sont à \coes dans $\Rad(\gA)$.
Plus \gnlt sur un anneau arbitraire on donne la \dfn suivante.
\begin{definition} \label{defireslibmin}
Une \rlf d'un \Amo est dite \ixc{minimale}{resolution@\rsn~---} si  les \alis
$u_1$, \dots, $u_n$ de la \rsn sont données par des matrices à \coes dans $\Rad(\gA)$. 
\end{definition}

\begin{lemma} \label{lemPNMini}
Soit $\fa$ un \id de $\gA$, $M$ un \Amo \pf et deux \pns de $M$ données par les suites exactes

\snic {
\xymatrix {
&\Ae m \ar[r]^{u} &\Ae n\ar[r]^{\pi} &M\ar[d]^{\Id_M}\ar[r] & 0
\\
&\Ae {m'} \ar[r]^{u'} &\Ae {n'}\ar[r]^{\pi'} &M\ar[r] & 0
\\
}}

\snii avec les matrices de $u$ et $u'$ à \coes dans $\fa$.
\begin{enumerate}
\item On a $n=n'$ dans $\HO(\gA\sur\fa)$. En particulier 
\begin{enumerate}
\item si $n\neq n'$
alors $1\in\fa$:
\item si $\fa$ est un \id strict, alors $n=n'$ dans $\NN$.
%
%
\end{enumerate}  
\item Si $\fa\subseteq \Rad(\gA)$ alors $\Ae n\simeq \Ae{n'}$ et il existe un \auto $v$ de $\Ae n$
tel que $\pi'\circ v=\pi$ et $v(\Ker\pi)=\Ker\pi'$.
\end{enumerate}
\snic {
\xymatrix {
&\Ae m \ar[r]^{u} &\Ae n\ar[d]^v\ar[r]^{\pi} &M\ar[d]^{\Id_M}\ar[r] & 0
\\
&\Ae {m'} \ar[r]^{u'} &\Ae {n}\ar[r]^{\pi'} &M\ar[r] & 0
\\
}}

\end{lemma}
\hum{On dit qu'un \id $\fa$ est propre si $1\in\fa\Rightarrow1=_\gA0$.
Le point \emph{1b} peut alors être énoncé sans négation comme suit: si $\fa$ est un \id propre, alors $n=n'$ dans $\HO\gA$}
\begin{proof}
\emph{1.} En étendant les scalaires à $\gB=\gA\sur\fa$ on obtient deux
suites exactes de \Bmos

\snic{\gB^ m\vers{0}\gB^ n \vers {\ov\pi} M/\fa M\lora 0,\quad$  et $\quad \gB^ {m'}\vers{0}\gB^ {n'} \vers {\ov{\pi'}} M/\fa M\lora 0.}

\snii Donc  $\ov{\pi}$ et $\ov{\pi'}$ sont des \isos et $\gB^n\simeq\gB^{n'}$.

\noindent \emph{2.} Si $n\neq n'$, $1\in\Rad\gA$ donc $\gA$ est trivial. On peut supposer \spdg $n=n'$.
On pose $K=\Im(u)=\Ker(\pi)$ et $K'=\Im(u')=\Ker(\pi')$.

\noindent  
L'\iso composé 

\snic{v_1:\Ae n \sur K\simarrow M\vvvers {\Id_M} M \simarrow \Ae n \sur K'}

\snii est représenté par une matrice $V\in\Mn(\gA)$
et l'\iso réciproque $w_1$ par une matrice $W$. On a $WV-\In=0 \mod K$
. En prenant le \deter on obtient  $\det(WV-\In)\in \fa$, donc $\det(WV)\in1+\fa\subseteq \Ati$. Ainsi~$V$ et~$W$  sont \ivs Notons $v$ l'\auto de matrice $V$.
Puisque $V$ représente $v_1$, on a $\pi'\circ v=\pi$. Et ceci implique $v(K)=K'$ car

\snic{K=\Ker\pi=\Ker(\pi'\circ v)=v^{-1}(\Ker\pi')=v^{-1}(K')}

\snii et $v$ est un \iso.
\end{proof}
%

\begin{theorem} \label{thBetti} \emph{(Unicité des \rlfs minimales)}
\\
Soient deux \rlfs  minimales d'un même module $M$
\[ 
\begin{array}{ccccccccccccccccccccc} 
0 &\to& L_n &\vvers{u_n}&  L_{n-1} &  \cdots \cdots    &\vvers{u_1}& L_0 &\vers{\pi} &M&
\to &0,
 \\[1mm] 
0 &\to& L'_n& \vvers{u'_n} & L'_{n-1} &  \cdots \cdots   & \vvers{u'_1}& L'_0 &\vers{\pi'}& M&
\to& 0.
\end{array}
\]
On a les résultats suivants.
\begin{enumerate}
\item Pour tout $i\in\lrb{0..n}$, $L_i$ et $L'_i$ sont isomorphes 
(donc si l'anneau est non trivial leur rang dans $\NN$ est le même).
\item Les deux complexes sont isomorphes au sens des complexes, i.e. il existe des \isos $v_i:L_i\to L'_i$, $(i\in\lrb{0..n})$
tels que $v_{i-1}u_i=u'_iv_{i}$  (on pose $u_0=\pi$, $u'_0=\pi'$ et $v_{-1}=\Id_M$). \\
En particulier $v_{i-1}(\Im(u_{i}))=\Im(u'_{i})$.
\end{enumerate}
\end{theorem}
{\small\[ 
\begin{array}{ccccccccccccccccccccc} 
0 &\to& L_n &\vvers{u_n}&  L_{n-1} &  \cdots   \cdots  &\vvers{u_1}& L_0 &\vers{\pi} &M&
\to &0 
 \\[1.5mm] 
 & & 
 \downarrow \!{v_n}&& 
 \downarrow \!v_{n-1}&& 
 &
 \downarrow \!v_0&&
 \downarrow \!\Id\\[.5mm] 
0 &\to& L'_n& \vvers{u'_n} & L'_{n-1} &  \cdots   \cdots & \vvers{u'_1}& L'_0 &\vers{\pi'}& M&
\to& 0 
\end{array}
\]
}
\begin{proof}
On démontre les points \emph{1} et \emph{2}
 par \recu sur $i$. 

\noindent Pour $i=0$, les rangs de $L_0$ et $L'_0$
sont égaux d'après le point \emph{1} du lemme~\ref{lemPNMini},
et  le point \emph{2} donne $\Im(u_1)\simeq \Im(u'_1)$ via un \iso $v_0$ de $L_0$ sur $L'_0$.

\noindent 
Pour $i=1$,  on note  que (la matrice de) $u_2$ (resp. $u'_2$) est une \mpn pour $\Im(u_1)$ (resp. $\Im(u'_1)$).
 D'après le point \emph{1}  du lemme \ref{lemPNMini} on 	a $\rg(L_1)= \rg(L'_1)$, et d'après le point \emph{2} un \iso $v_1:L_1\to L'_1$ qui envoie $\Im(u_2)$ sur $\Im(u'_2)$ et qui fait commuter le diagramme.
 
 \noindent 
On termine par \recu de la même manière.
\end{proof}

\rems 1) Dans le \tho précédent, on peut supprimer la parenthèse dans le point \emph{1} en considérant le rang comme un \elt de $\HO(\gA)$ plutôt que comme un \elt de $\NN$.

\snii 2) On n'a  exigé ni $\rg_\gA(L_n)>0$, ni $\rg_\gA(L'_n)>0$. 
Une conséquence est donc
que si l'on a	deux \rlfs minimales avec des rangs tous $>0$, elles ont même longueur (plus \prmt, si ce n'est pas le cas,  \hbox{alors $1=_\gA0$}). En particulier toute \rlf minimale d'un \mlrf est \gui{de longueur nulle} (plus \prmt, les $L_k$ sont nuls \hbox{pour $k\geq 1$}).

\snii 3) Dans la \dem précédente, si l'on suppose seulement que les matrices des $u_\ell$ sont à \coes dans un \id $\fa$, on peut considérer l'anneau  $\gA'=\gA_{1+\fa}$ pour lequel
$\fa\subseteq \Rad(\gA')$ et l'on voit que les \isos construits utilisent effectivement uniquement des \denos dans $1+\fa$. 
Ceci justifie le point~\emph{2} dans la \dfn  \ref{defiBetti} ci-après.

\snii 4) Pour plus de précisions concernant le \tho précédent, voir
le \thref{thcomparaisonRSFS} dans le cadre des \rsfs 
\eoe

\medskip 
 
Comme conséquence des \thos \ref{cor3lemModifComplexe} et \ref{thBetti}
on obtient le corolaire suivant. 

\begin{corollary} \label{corthBetti}
Sur un \alo \dcd non trivial, une \rsn libre minimale d'un module est de longueur minimum
parmi toutes les \rlfs du module.
\end{corollary}

\hum{J'ai fait appel au \thref{cor3lemModifComplexe} et mis \dcd 
pour que l'on soit certain
d'arrêter au bon endroit une \rlf, mais ce n'est peut-être pas pertinent. \`A voir}

Ainsi dans le cas d'un \alo \dcd, la {\pdi d'un module~$M$ \lrsb}, notée $\Pd_\gA(M)$ (voir la \dfn~\ref{definresol}), est un entier bien défini.\index{dimension projective!d'un module sur un anneau local}

\subsec{Nombres de Betti}

\hum{Peut-être les histoires de \rsns $\fa$-minimales
peuvent attendre ... plus tard}

La \dfn suivante est légitimée par le \thref{thBetti}
\begin{definota} \label{defiBetti} 
\begin{enumerate}
\item Soit $\gA$ un anneau non trivial. Si un \Amo $M$ admet une \rsn libre minimale (\dfn \ref{defireslibmin}) 
$$
0 \,\to\, L_n \,\vvers{u_n}\,  L_{n-1} \,  \cdots\cdots \cdots    \,L_1 \,\vvers{u_1}\, L_0 \,\vers{\pi} \,M\,
\to \,0,
$$
les entiers $\rg_\gA(L_i)$
pour $i\in\lrb{0..n}$ ne dépendent que de $M$, ils sont appelés les \ix{nombres de Betti du \Amo $M$}.
On les note $\beta_{\gA,i}(M)$ \hbox{ou $\beta_{i}(M)$}. 
\item Plus \gnlt, si un \Amo $M$ admet une \rlf avec toutes les matrices à \coes dans un \id $\fa$,  les \elts~{$\rg_{\gA_{1+\fa}}(L_i)$ de $\HO(\gA_{1+\fa})$}  sont appelés les \ix{nombres de Betti du \hbox{\Amo $M$} relativement à l'\id $\fa$}.
On les note $\beta_{\gA,\fa,i}(M)$ ou $\beta_{\fa,i}(M)$. 
\end{enumerate}
\end{definota}

\rem 1) 
Pour un module $M$ qui admet une \rsn libre minimale, les types des matrices des $u_i$ (à \eqvc près)
et leurs \idds sont tous des invariants attachés au module $M$.

2)
Dans le point \emph{2}, si $1\notin\fa$ on peut identifier les 
$\beta_{\fa,i}(M)$ à des entiers naturels ordinaires.
\eoe 

\hum{Il faut donner des exemples de nombres de Betti.
Cela peut venir de la topologie \agq et un commentaire sur des exemples 
d'origine topologique serait bienvenu. On pourrait peut-être en faire une
section à part en début ou en fin du chapitre. 
}

\subsec{\Idcas d'un module}\label{subsecidecaras2}

Dans le cas d'un \alo \dcd $(\gA,\fm)$ les \idcas dans une \rsn libre minimale d'un module $M$ sont les mêmes que ceux de toute \rlf du module,
ceci parce que toute \rlf peut être réduite à une \rsn minimale par des modifications \elrs.
Ainsi ces  \idcas ne dépendent pas de la \rlf considérée.
\\
Notons aussi que si $M$ n'est pas libre, la longueur $n=\Pd_\gA(M)\geq 1$ de la \rsn libre minimale est caractérisée par
$\fD_n\subseteq \fm$ et $\fD_{n+1}=\gen{1}$.

\begin{thdef} \label{thidecarnresol} \emph{(Idéaux \caras d'un \mlrsb)}
Soient $\gA$ un anneau arbitraire, $M$  un \Amo \lrsb, et
 $L_{\bullet}$ un complexe exact de \mlrfs qui résout~$M$.
\begin{itemize}
\item Les \idcas $\fD_{\gA,k}(L_{\bullet})$ et  $\fD_{\gA,k,\ell}(L_{\bullet})$ du complexe ne dépendent que du module $M$.
\item On les note $\fD_{\gA,k}(M)$ et  $\fD_{\gA,k,\ell}(M)$
et on les appelle les \emph{\idcas du \mlrsb~$M$}.%
\index{ideaux cara@\idcas!d'un module librement \rsb}
\item De même la \cEP de $L_{\bullet}$ ne dépend que du module~$M$. 
\end{itemize}
\end{thdef}
%
\begin{proof}
Puisque c'est vrai dans le cas d'un \alo \dcd, il suffit d'appliquer la machinerie \lgbe \gnle à \ideps (\paref{MethodeIdeps}).
\sibook{\ttt{donner quelques précisions.}}
\end{proof}

\rem
Notons $r_0=\chi(L_{\bullet})$, on a alors par \dfn l'\egt
  $$\fD_{1,\ell}(L_{\bullet})=\cF_{r_0+\ell}(M)$$
pour tout $\ell$.
Ceci éclaire le résultat d'invariance précédent sur
un cas particulier. 
On va voir bientôt que les $\fD_{k,\ell}(L_{\bullet})$ sont nuls pour $\ell<0$ et fidèles \hbox{pour $\ell\geq 0$}. 
Il suffirait de le démontrer dans le cas d'un \alo
\dcd, mais la \dem ne semble pas facilitée par une telle hypothèse.
\eoe

\sibook {\ttt { terminologie: \idcas du module ? \\
 \idfs \gnes du module?.}}

\medskip 
On étendra la \dfn des \idcas ainsi que le  résultat d'invariance à tout \mlorsb en \ref{thidecarnrsf}.

%

\section{Rang, rang stable}\label{secRangStable}

Dans cette section nous définissons le rang d'un \mlrsb et le mettons en relation avec le rang stable des matrices 
présentes dans la \rsn.

\subsec{Rang stable de matrices}\label{subsecRangStable}

La notion d'\ali stable (entre \mlrfs) que nous introduisons ici est une \gnn
de la notion d'\ali injective.
Elle est en quelque sorte aux \alis injectives ce que les \alis de rang 
entier\footnote{Dans cet ouvrage, à la suite de [CACM], on dit qu'une \ali $\varphi$ entre modules libres et de rang $k$ si $\cD_k(\varphi)=\gen{1}$ et $\cD_{k+1}(\varphi)=\gen{0}$.}
sont aux \isos.

\hum{J'introduis le rang stable d'un \mpf, cela facilite un peu la vie je trouve.}

\begin{definition} \label{defiRangStable} ~
\begin{enumerate}
\item Une matrice $A$ est dite de \emph{rang stable supérieur ou égal à $ r$} si l'\idd $\cD_r(A)$ est fidèle. On écrit $\rgst_\gA(A)\geq r$ \hbox{ou $\rgst(A)\geq r$}. 
\item La matrice $A$ est dite de \emph{rang stable $r$},  et l'on \hbox{écrit $\rgst_\gA(A)= r$}, si en outre $\cD_{\gA,r+1}(A)=\gen{0}$, \cad si $\rg_\gA(A)\leq r$. On dit dans ce cas que la matrice est \emph{stable}. 
\end{enumerate}%
\index{rang stable!d'une matrice}\index{rang stable!d'une \ali entre modules libres de rang fini}\index{rang stable!d'un \mpf}
Ces \dfns s'étendent naturellement aux \alis entre \mlrfs
\begin{enumerate}\setcounter{enumi}{2}
\item Un \Amo \pf $M$ est dit de \emph{rang stable $r$},  et l'on \hbox{écrit $\rgst_\gA(M)= r$}, si $\cF_{\gA,r}(M)$ est fidèle
et $\cF_{\gA,s}(M)=0$ \hbox{pour $s<r$}. Ainsi, si l'on a une \pn de $M$ sous la forme
d'une \sex
$$
\gA^{m+r}\vers{A}\gA^{r+n}\vers\pi M\to 0
$$
on a $\rgst(M)=r$ \ssi $\rgst(A)=n$.
\end{enumerate}
\end{definition}

Si l'on a une matrice qui est de rang $\leq r-1$ et de rang stable $\geq r$,
alors l'anneau est trivial.
Quand l'anneau est non nul, si une matrice est stable, son rang stable est donc un entier bien défini.

\medskip 
\exls
1) Une matrice injective est toujours stable, de rang stable égal à son nombre de colonnes.
\\
2) 
Un anneau est intègre \ssi toutes les matrices sur cet anneau sont stables.
Et alors le rang stable d'une matrice n'est autre que le rang de la matrice sur le corps des fractions.
\\
3)
Une matrice est de rang stable $0$ \ssi elle est nulle.
\eoe

\medskip 
\rem Posons $\gB=\Frac(\AX)$. Alors, pour toute matrice $A$ on a

\snic{\rgst_\gA(A)\leq k\Leftrightarrow \rg_\gB(A)\leq k,\hbox{ et }\rgst_\gA(A)= k\Leftrightarrow \rg_\gB(A)= k,}

\snii
ce que l'on peut abréger en $\rgst_\gA(A)=\rg_\gB(A)$. Ceci résulte du lemme de McCoy~\ref{lemMcCoy}.
\eoe

\medskip 
On pourra comparer les \idfs d'un module stable
de rang $r$ avec ceux d'un \mrc~$r$.
Dans les deux cas les \ids~\hbox{$\cF_s(\bullet)$} pour $s<r$ sont nuls. 
La  différence est que l'\egt $\cF_r(P)=\gen{1}$, valable pour un module
$P$ \pro de rang constant~$r$, est ici remplacée par la condition \gui{$\cF_r(M)$ est fidèle}
pour un module stable.

\begin{plcc}\label{plcc.RangStable} \emph{(Pour le rang stable)}\\
Soient $a_1$, \dots, $a_n$ des \ecr de $\gA$. \\
Soient $A\in\Ae{m\times \ell}$, $r$ un entier $\geq 0$ et $M$ un \Amo \pf.
\begin{enumerate}
\item  On a $\rgst(A)\geq r$ sur $\gA$ \ssi $\rgst(A)\geq r$ sur chaque 
localisé $\gA[1/a_i]$.
\item  On a $\rg(A)\leq r$ sur $\gA$ \ssi $\rg(A)\leq r$ sur chaque $\gA[1/a_i]$.
\item  On a $\rgst(A)= r$ sur $\gA$ \ssi $\rgst(A)= r$ sur chaque localisé $\gA[1/a_i]$.
\item  On a $\rgst(M)= r$ sur $\gA$ \ssi $\rgst(M[1/a_i])= r$ pour chaque
\lon en $\gA[1/a_i]$..
\end{enumerate}
\end{plcc}
%
\begin{proof}
Conséquence \imde du \plgrf{plcc.regularite}.
\end{proof}
%

\begin{proposition} \label{prop.rgsta.rgconstant} ~
\begin{enumerate}
\item Une matrice $A$ est stable de rang stable $r$ \ssi il existe 
des \ecr tels qu'après \lon en chacun d'eux la matrice soit de rang $r$. 
\item Un \mpf $M$ est stable de rang stable $\ell$ \ssi il existe 
des \ecr tels qu'après \lon en chacun d'eux $M$ soit de libre de rang $\ell$. 
\item Si $\rgst(A)=r$, pour tout $k\in\lrbr$, $\Vi^{k}\!A$ est stable, de rang stable~${r \choose k}$.
\end{enumerate} 
\end{proposition}
%
\begin{proof} \emph{1} et \emph{2.} Si l'on inverse un mineur d'ordre $r$ de $A$, la matrice devient de rang~$r$, et son conoyau est libre
(lemme de la liberté \ref{FFRlem pf libre}).
\\ 
\emph{3.} Appliquer le point \emph{1} en notant que les puissances 
extérieures se comportent bien par \lon.
\end{proof}

\rem On aurait des résultats analogues à celui du point \emph{3} pour toutes les constructions usuelles (produit tensoriel, puissance \smq, modules d'\alis \dots).
\eoe

On a le résultat suivant pour les \secos de modules stables.

\begin{proposition} \label{propSecoStab} \emph{(Rang stable et \seco)}\\
Si $N\subseteq M$, et si les trois modules $M$, $N$ et $M/N$ sont stables, 
 on a 
 
 \snic{\rgst(M)= \rgst(N)+\rgst(M/N)}
 
 dans $\HOp\gA$. 
\end{proposition}
%
\begin{proof} Notons $\rho_1=\rgst(N)$, $\rho_2=\rgst(M)$, $\rho_3=\rgst(M/N)$.
Un produit d'\itfs fidèles est un \itf fidèle. On applique alors
la proposition \ref{prop.rgsta.rgconstant}
simultanément pour les trois modules. Ainsi, en localisant en des \ecr 
$a_i$ on obtient $\rho_{1}$, $\rho_2$ et $\rho_3$ comme rangs de modules 
libres $N_i$, $M_i$ et $M_i/N_i$ sur l'anneau $\gA[1/a_i]$, ce qui donne
un \iso $M_i\simeq N_i\times (M_i/N_i)$. \\
Si $\rho_2\neq \rho_1+\rho_3$
dans $\NN$, alors on \hbox{a $1=0$} dans tous les $\gA[1/a_i]$, 
et puisque les~$a_i$ sont \cor, cela donne~\hbox{$1=0$} dans $\gA$. 
\\
Ainsi, on a toujours
 $\rho_2=\rho_1+\rho_3$ dans $\HO\gA$.
\end{proof}
\rem
L'appel à $\HO\gA$ permet d'éviter de supposer que l'anneau
est non trivial dans l'énoncé précédent. En fait cela suggère que l'on pourrait élargir la \dfn et dire qu'un module est \emph{\lot stable}
s'il est stable après \lon en des \eco. Dans ces conditions le rang stable serait un \elt de $\HOp\gA$ et l'\egt de la proposition \ref{propSecoStab}
serait valable pour toute \seco de modules \lot stables.
\eoe

\medskip  La \dem du lemme crucial suivant est tout à fait analogue à celle de la
proposition~\ref{propInjIdd}.

\CMnewtheorem{lemmrangstab}{Lemme du rang stable}{\itshape}
\begin{lemmrangstab}
\label{lemSuitExStable}~\\
Soit une \sex $\Ae m \vers{u} \Ae n \vers{v} \Ae \ell $ et  $k$, $h\in\NN$
avec $k+h=n$.  Alors $u$
est stable, de rang stable~$k$, \ssi $v$ est stable, de rang stable~\hbox{$h$}.
Et dans ce cas, on a ${\cD_{h}(v)}\subseteq \sqrt[\gA]{\cD_{k}(u)}$. 
\end{lemmrangstab}
NB: ainsi, si $v$ est stable de rang stable $h$ et si $n-h>m$, alors l'anneau est trivial.  
\begin{proof} 
Supposons d'abord $u$ stable de rang stable $k$.
Par hypothèse $\cD_k(u)$ est fidèle. Si $k>m$, $\cD_k(u)=0$ et l'anneau est trivial.
On suppose $k\leq m$. 
Après \lon en un mineur arbitraire $a$ d'ordre $k$ de $u$, on se ramène (lemme du mineur inversible \ref{FFRlem.min.inv}) au cas où~$u$
est simple, de matrice $\I_{k,n,m}$. \\
Comme $v\circ u=0$, $v$ admet une matrice de la forme $\blocs{.6}{.7}{.9}{0}{$0$}{$B$}{}{}$\,, et comme la suite est exacte, $B$ représente une \ali injective de $\gA[\fraC1a]^{h}$ dans~$\gA[\fraC1a]^\ell$, donc est stable de rang $h$. On conclut avec le \plgref{plcc.RangStable} que $v$ est stable de rang $h$.
\\
Supposons ensuite $v$ stable de rang stable $h$. Par hypothèse $\cD_h(v)$ est fidèle. Si $h>\ell$, $\cD_h(v)=0$ et l'anneau est trivial.
On suppose $h\leq \ell$. 
Après \lon en un mineur arbitraire $a$ d'ordre $h$ de $v$, on se ramène (lemme du mineur inversible \ref{FFRlem.min.inv}) au cas où~$v$
est simple, de matrice $\I_{h,\ell,n}$.\\ 
Comme $v\circ u=0$, $u$ admet une matrice de la forme $\blocs{0}{1.2}{.6}{.7}{}{$0$}{}{$C$}$\,, et comme la suite est exacte, $C$ représente une \ali surjective de~$\gA[\fraC1a]^{m}$ dans~$\gA[\fraC1a]^h$, donc est de rang $h$. On conclut avec le \plgref{plcc.RangStable} que $u$ est stable de rang $k$.
En outre pour chaque \gtr $a$ de $\cD_{\gA,h}(v)$, on vient de voir que $1\in\cD_{\gA[\fraC1a],k}(u)$, \cad $a\in\sqrt[\gA]{\cD_{\gA,k}(u)}$. Ceci prouve l'inclusion
${\cD_{\gA,h}(v)}\subseteq \sqrt[\gA]{\cD_{\gA,k}(u)}$.
\end{proof}
%

\subsec{Rang d'un module  \lrsb}

Peut-être l'appellation \gui{\tho de structure locale} pour le \tho qui suit est un bien grand mot
pour un résultat qui ne donne pas beaucoup d'informations nouvelles
concernant les \mpfs lorsque l'anneau est intègre. Néanmoins, dans le cas d'un anneau arbitraire, il mérite le respect, comme on le verra avec les corolaires qui suivront.
\begin{thdef} \label{thStrLocResFin} \label{thRangStRLF} 
\emph{(\Tho de structure \gui{locale}
pour les modules  \lrsbs)}\\
Soit  $M$ un \Amo qui admet une \rlf

\snic
{0 \to L_n \vvers{u_n}  L_{n-1} \vvers{u_{n-1}}\;  \cdots \cdots \; \vvers{u_2}  L_1 \vvers{u_1} L_0 \vers{\pi} M
\to 0.}

\snii
pour laquelle  $L_i=\gA^{p_i}$.  Notons $r_\ell=\som_{k=\ell}^n(-1)^{k-\ell}p_k$ la \cEP du complexe $0\to L_n\lora\; \cdots\cdots\;\lora L_\ell$.
\begin{enumerate}\setcounter{enumi}{-1}
\item Si l'un des $r_\ell$ est $<0$ dans $\ZZ$, l'anneau $\gA$ est trivial. \\
Autrement dit, on a toujours $r_\ell\geq 0$ dans $\HO(\gA)$.
\end{enumerate}
Dans les points qui suivent on suppose \spdg que les~$r_\ell$ sont $\geq 0$
(s'il y en a un $<0$ l'anneau est nul, on peut prendre tous les~$r_\ell$ nuls,
et tout est correct, mais sans intérêt)
\begin{enumerate}
\item Chaque application $u_\ell$ est stable, de rang stable $r_\ell$.
 Autrement dit en termes d'\idcas, $\fD_\ell$ est fidèle et $\fD_{\ell,-1}=0$.
\item Le module $M$ est stable, de rang stable $r_0$ avec $\fD_1=\cF_{r_0}(M)$. L'entier~$r_0$, vu comme \elt de~$\HO(\gA)$, ne dépend pas de la \rlf considérée.  
On dit que~$r_0$ est le \emph{rang du \mlrsb~$M$} 
et on le note~$\rg_\gA(M)$ ou $\rg(M)$.
\item Si $N\subseteq M$, et si $N$ et $M/N$ admettent aussi des \rlfs, 
alors on a $\rg(M)= \rg(N)+\rg(M/N)$ dans $\HO(\gA)$. 
\end{enumerate}\index{rang!d'un module qui admet une \rlf}
\end{thdef}
\emph{Note.} Rappelons que si l'anneau est non trivial, les entiers de $\HO(\gA)$ s'identifient aux entiers  de $\ZZ$. Par ailleurs si $M$ est \ptf, il résulte du point \emph{2} que le rang que l'on vient de définir co\"{\i}ncide bien avec le rang que l'on avait déjà défini pour les \mptfs
\eoe
\begin{proof} \emph{0} et \emph{1.}
Si $n=0$ le \tho est clair. Supposons $n>0$. Puisque $u_n$ est une \ali injective, elle est stable, de rang stable $r_n$.
Par le lemme du rang stable, $u_{n-1}$ est stable, de rang stable $r_{n-1}$, et si $r_{n-1}<0$ l'anneau est trivial.
On continue de la même manière jusqu'à $u_0$.

\emph{2}. Résulte du point \emph{1} et des résultats sur les modules stables.
Notez que l'invariance de $r_0$ a déjà été établie au \thref{thidecarnresol}, par un argument plus mystérieux. La notion de rang stable 
éclaircit la situation.

\emph{3.} On applique la proposition \ref{propSecoStab}.
\end{proof}

\rem Que se passe-t-il lorsque l'anneau $\gA$ est intègre et non trivial?
Soit $\gK$ le corps des fractions de $\gA$. Par \eds de $\gA$ à $\gK$, le \mpf $M$ devient libre de rang $r$ et les \alis $u_\ell$ deviennent simples
de rangs 

\snic{r_n=p_n, \,r_{n-1}=p_{n-1}-p_n,\, r_{n-2}=p_{n-2}-p_{n-1}+p_n,\, \dots \,}

\snii
En fait on voit facilement que cette réduction au cas des corps nécessite un seul dénominateur~$a$, qui remplace la famille d'\ecr évoquée dans le \tho. Ainsi, comme nous le disions avant d'énoncer le \tho, celui-ci 
peut nous donner des informations non triviales, voire surprenantes (cf. les corolaires qui suivent), mais uniquement 
dans le cas d'un anneau non intègre.
\eoe

L'énoncé du \thref{thStrLocResFin} fait un peu mal à la tête
à cause du cas de l'anneau trivial. 
Nous proposons ici l'énoncé alternatif suivant, un peu moins \gnl mais plus agréable, où les
entiers~$r_\ell$, qui doivent être $\geq 0$ si l'anneau n'est pas trivial,
sont donnés en hypothèse comme des \elts de $\NN$.
Par contre nous ne supposons pas que l'anneau est non trivial.

\THO{thStrLocResFin}
{Soit $(r_0,\dots,r_n,r_{n+1})$ dans $\NN$ avec $r_{n+1}=0$, et $M$ un \Amo qui admet une \rlf

\snic
{0 \to L_n \vvers{u_n}  L_{n-1} \vvvers{u_{n-1}}\;  \cdots \cdots \; \vvers{u_2}  L_1 \vvers{u_1} L_0 \vers{\pi} M
\to 0}

avec   $L_k\simeq\gA^{r_{k+1}+r_k}$ 
pour tout $k\in\lrbn$.
\begin{enumerate}
\item Chaque application $u_\ell$ est stable, de rang stable $r_\ell$.
 Autrement dit en termes d'\idcas, $\fD_\ell$ est fidèle et $\fD_{\ell,-1}=0$.
\item Le module $M$ est stable, de rang stable $r_0$ avec $\fD_1=\cF_{r_0}(M)$. L'entier~$r_0$, vu comme \elt de~$\HO\gA$, ne dépend pas de la \rlf considérée.  
On dit que~$r_0$ est le \emph{rang du \mlrsb~$M$} 
et on le note~$\rg(M)$.
\item Si $N\subseteq M$, et si $N$ et $M/N$ admettent aussi des \rlfs, 
alors on a $\rg(M)= \rg(N)+\rg(M/N)$ dans $\HO\gA$. 
\end{enumerate}
}

\medskip
Le \tho suivant, dû à Vasconcelos \cite{Vas1971} est une conséquence \imde du  \thref{thStrLocResFin}.
\begin{theorem} \label{corVasco} \emph{(Annulateur et rang)}\\
Soit $M$ un module qui admet une \rlf

\snic
{0 \to L_n=\gA^{p_n} \lora   \cdots \cdots \;  \lora L_0=\gA^{p_0} \to M
\to 0.}

\snii et $r=\som_{k=0}^n(-1)^kp_k$. 
\begin{enumerate}
\item Si $r=0$, $\cF_0(M)$ est fidèle (i.e., $\Ann(M)$ est fidèle).
\item Si $0<r\leq p_0$, $M$ est fidèle.
\item Si $r<0$ ou $r> p_0$, l'anneau est trivial.
\end{enumerate}
En particulier, si l'anneau est non trivial, le rang de $M$ est toujours~\hbox{$\geq 0$}
dans~$\NN$, il est nul \ssi l'annulateur de $M$ est fidèle, et il est $>0$ \ssi l'annulateur est nul.
\end{theorem}
%

Si $\fa$ est un \id de $\gA$, en appliquant le \tho précédent au module~\hbox{$M=\gA/\fa$}, on obtient le résultat qui suit (penser que $r$ est remplacé  \hbox{par $1-r$} et que la \rsn est raccourcie d'un cran).

\begin{corollary} \label{corVascon} \emph{(Vasconcelos)}\\
Soit $\fa$ un \id qui admet une \rlf

\snic
{0 \to L_n=\gA^{p_n} \lora\;   \cdots \cdots \;  \lora L_0=\gA^{p_0} \to \fa
\to 0.}

\snii et $r=\som_{k=0}^n(-1)^kp_k$. 
\begin{enumerate}
\item Si $r=0$, $\fa=0$.
\item Si $r=1$, $\fa$ est fidèle.
\item Si $r\neq 1,0$, l'anneau est trivial.
\end{enumerate}
\end{corollary}
%

Comme conséquence \imde on obtient le \tho suivant.

\begin{theorem} \label{corcorVascon} 
Si dans l'anneau $\gA$ tout \idp admet une \rlf, alors $\gA$ est intègre.
\end{theorem}
\rem La réciproque est évidente
\eoe

\medskip 
Notons que l'on a un résultat encore plus spectaculaire:
\emph{si dans l'anneau~$\gA$ tout \id $\gen{a,b}$ admet une \rlf, alors $\gA$ est un anneau à pgcd intègre.}
Mais nous devrons plus nous fatiguer pour le démontrer
(voir le  \thref{corthdetCay2}).

\subsec{Retour sur les \ids \caras}

\begin{proposition} \label{propfDk=1}
On reprend les hypothèses et notations du \thref{thStrLocResFin}~bis. 
\begin{enumerate}
\item Si $\fD_{k}=\gen{1}$, alors $\Pd_\gA(M)\leq k-1$ et $\fD_{k+s}=\gen{1}$ pour tout $s>0$.
\item Pour tout $\ell\geq 1$, on a $\fD_{\ell}\subseteq \DA(\fD_{\ell+1})$.
\end{enumerate}
\end{proposition}
%
\begin{proof} \emph{1.} Puisque $\rg(u_k)\leq r_k$, si $\cD_{r_k}(u_k)=\gen{1}$, l'\ali est de rang~$r_k$, donc \lot simple (\thref{FFRpropFactDirRangk}). Puisque~$\Im(u_k)$ est facteur direct, on obtient une \rsn de $M$ de 
longueur~$k-1$
$$\preskip.0em \postskip.4em
 0 \to \Im(u_k) \vvers{u'_{k-1}} L_{k-2} \vvers{u_{k-2}}\; \cdots\cdots 
\; L_0 \lora M \to 0 
$$
où $L_{k-1}=\Im(u_k)\oplus P'_{k-1}$ et $u'_{k-1}$ est l'injection canonique.
Les modules de cette \rsn de $M$ sont bien \ptfs
\\
Quant au complexe exact   
$$\preskip.0em \postskip.4em
{0 \to L_n \vvers{u_n}  L_{n-1} \vvers{u_{n-1}}\;  \cdots \cdots \; L_k\vvers{u'_{k}}  \Im(u_k)
\to 0,}
$$
il relève du  fait \ref{factCharCompEx0} (point \emph{2}). 
En conséquence les \alis~$u_{k+s}$ ($s>0$) sont toutes \lot simples. Donc leur rang stable correspond à un \idd égal à $\gen{1}$.
Ceci montre que les $\fD_{k+s}$ sont égaux~à~$\gen{1}$.

\emph{2.} Soit $\mu$ un \gtr de $\fD_\ell$. On passe à l'anneau $\gA[1/\mu]$
et l'on se retrouve dans la situation du point \emph{1.} Donc $\fD_{\ell+1}=\gen{1}$ sur l'anneau $\gA[1/\mu]$. Cela signifie que sur l'anneau $\gA$, une puissance de $\mu$
est dans $\fD_{\ell+1}$. 
\end{proof}

\rem Le résultat $\fD_k\subseteq \sqrt{\fD_{k+1}}$  sera renforcé dans le 
\thref{thResFinIdFac}, mais ici la preuve est vraiment simple.
\eoe

\begin{theorem} \label{corthDimKrullGr} \emph{(Dimension de Krull et \rsfs)}
\\
Si $\Kdim\gA\leq r$, tout \mlrsb possède une
\rsf de longueur~\hbox{$\leq r$}. 
\end{theorem}
%
\begin{proof}
Notons $M$ le module. On sait que $\Gr\big(\fD_{r+1}(M)\big)\geq r+1$. Par le \thref{thDimKrullGr}, $\fD_{r+1}(M)$ contient $1$. Par la proposition \ref{propfDk=1}, $\Pd(M)\leq r$. 
\end{proof}
%

\subsec{Trivialisation de certains complexes}

\begin{lemma} \label{lemSuitExStable2}
Soit un complexe 

\snic{E=\Ae {m+n} \vers{u} F=\Ae {n+r} \vers{v} G=\Ae {r+s}}

\snii avec $\rg(u)\leq n$, et $\rg(v)\leq r$. Soit $\mu$ un mineur d'ordre $n$ de la matrice $A$ de $u$, et $\nu$ un mineur d'ordre $r$ de la matrice $B$ de $v$. 
\\
Alors sur l'anneau $\gA[1/(\mu\nu)]$, le complexe devient \emph{trivialement exact}%
\index{complexe!trivialement exact}%
\index{trivialement exact!complexe ---}: i.e., 
\begin{itemize}
\item $\Ker u$ est libre facteur direct d'un module libre dans $E$,
\item  $\Im u=\Ker v$ est libre facteur direct d'un module libre dans $F$,
\item  $\Im v$
est libre facteur direct d'un module libre dans $G$.
\end{itemize}
Ainsi pour des \suls $E_1$, $F_1$ et~$G_1$ arbitraires, $u$ réalise un \iso de $E_1$ sur son image, $v$ réalise un \iso de $F_1$ sur son image, et pour des bases convenables des modules~$\Ker u$,
$E_1$, $\Im u$, $F_1$, $\Im v$ et~$G_1$, 
 on obtient les nouvelles matrices
$$ B'=\blocs{1}{.8}{.8}{.6}{$0_{r,n}$}{$\I_r$}{$0_{s,n}$}{$0_{s,r}$}\quad\hbox{ et }\quad  
   A'=\blocs{1.2}{1}{1}{.8}{$0_{n,m}$}{$\I_n$}{$0_{r,m}$}{$0_{r,n}$}.
$$  
\end{lemma}
%
\begin{proof}
Par hypothèse $\rg(A)\leq n$, donc si l'on inverse un mineur d'ordre $n$, la matrice
$A$ devient simple de rang $n$, avec son noyau et son image libres en facteur direct de libres (lemme de la liberté \ref{FFRlem pf libre}).
La même chose vaut pour la matrice $B$. 
\\
Notons $K=\Ker v$ et $I=\Im u$.
Il reste à prouver que $K=I$.\\
On sait que $I\subseteq K$ et que $F/K$ et $F/I$ sont isomorphes à $\Ae r$.
Ainsi $F/I$ est un \mtf isomorphe à son quotient $F/K\simeq (F/I)/(K/I)$.
Ceci implique $K/I=0$, \cad $K=I$ (\thref{FFRprop quot non iso}). 
\end{proof}
On notera que l'on n'a pas besoin de savoir si $\gA[1/(\mu\nu)]$ est trivial ou non.
Le lemme reste valable dans tous les cas, avec la \dem inchangée.

\begin{theorem} \label{thRangStRLF2}
Soit $(r_k,\dots,r_{n+1})$ dans $\NN$ et un complexe

\snic
{ L_n \vvers{u_n}  L_{n-1} \vvers{}\;  \cdots \cdots \; \vvers{}  L_{k+1} \vvers{u_{k+1}} L_{k}}

\snii
avec   $L_j\simeq\gA^{r_{j+1}+r_j}$ 
pour tout $j\in\lrb{k..n}$. 
\Propeq
\begin{enumerate}
\item Chaque \ali $u_j$ est stable de rang stable $r_j$.
\item Il existe des \ecr $s_1$, \dots, $s_m$ tels que, après \lon en chaque  
$s_i$, le complexe devient  trivialement exact (comme dans le lemme \ref{lemSuitExStable2}) avec chaque $u_j$ de rang $r_j$.
\end{enumerate}
\end{theorem}
On notera que le résultat s'applique en particulier (avec $r_{n+1}=0$) si $u_n$ est injective et si le complexe est exact. 
Dans ce cas, cela ressortirait aussi de l'examen détaillé de la \dem du \thref{thRangStRLF}. 
\begin{proof} \emph{1} $\Rightarrow$ \emph{2.}
Notons $A_j$ la matrice de $u_j$. Si l'on inverse un mineur d'ordre $r_j$ 
de~$A_j$ pour chaque $j$, le lemme \ref{lemSuitExStable2} nous dit que sur 
l'anneau localisé le complexe devient complètement trivial.
Comme chaque $\cD_{r_j}(A_j)$ est fidèle par hypothèse, il en va de même
du produit de ces \ids, et les \gtrs de ce produit forment bien un \sys d'\ecr.\\
 \emph{2} $\Rightarrow$ \emph{1.} Appliquer le \plg \ref{plcc.RangStable} pour le rang stable.   
\end{proof}
%

\section
{Profondeur et \pdi}\label{secAusBu}
\sibook{Dans cette section \dots
}
\subsec{Profondeur des \idcas}

On améliore maintenant le \thref{thRangStRLF} en donnant des informations sur la profondeur des \idcas \fbox{$\fD_k=\cD_{r_k}(u_k)$} dans une \rlf
avec  $L_k\simeq\gA^{r_{k+1}+r_k}$ pour tout $k$, 
et $r_{m+1}=0$.

\fnic{\;0 \to L_m \vvers{u_m}  L_{m-1} \vvvers{u_{m-1}}\;  \cdots \cdots \; \vvers{u_2}  L_1 \vvers{u_1} L_0\;}

\sibook{ 
On donnera une réciproque dans le \thref{cor3thABH1}.}

\begin{theorem} \label{thRangStProfRLF} \emph{(Profondeur des \ids \caras dans une \rlf)}
 Dans une \rlf comme ci-dessus, chaque \ali $u_\ell$ est stable de rang stable
$r_\ell$ et $\Gr(\fD_\ell)\geq \ell$.
\end{theorem}
%
\begin{proof} 
 On sait que les \ids~$\fD_k$ sont de profondeur $\geq 1$.
 On fait une \dem par \recu sur $m$, le cas $m=1$ est clair.
Supposons $m\geq 2$ et passons de $m-1$ à~$m$.  
\\
Quitte à passer à une extension \polle de $\gA$,
on peut supposer que~$\fD_m$ contient un \elt~$f_m$ \ndz,
donc $L_k$-\ndz pour tout~$k$. 
La proposition~\ref{propRegSex} nous dit que le complexe
$$\preskip.3em \postskip.4em 
0\to L_m/f_m L_m \vvers{u_m}  L_{m-1}/f_m L_{m-1} \vvers{u_{m-1}}\;  \cdots \cdots \; \vvers{u_2}  L_1/f_m L_1 
$$
est exact. 
On  applique l'\hdr avec l'anneau $\gB=\aqo\gA {f_m }$
et le complexe ci-dessus. On obtient,  que  $\Gr_\gB(\fD_m)\geq m-1$, \cade que $\Gr_\gA(\fD_m,\aqo\gA {f_m})\geq m-1$. 
\\
On conclut avec le \thref{thfondprof} que $\Gr_\gA(\fD_m,\gA)\geq m$.
\\
Il reste à montrer $\Gr_\gA(\fD_\ell,\gA)\geq \ell$ pour $\ell\geq 2$.
Soit $\delta$ un mineur d'ordre maximal $r_m$ de la matrice de $u_m$.
Plaçons nous sur l'anneau $\gC=\gA[1/\delta]$. 
Après  changements de bases, la matrice $U_m$ de $u_m$ est du type 
$\Cmatrix{2pt}{\I_{r_m}\cr 0}$, donc celle de $u_{m-1}$ du type 
$\Cmatrix{2pt}{0& V_{m-1}}$, ce qui fournit sur $\gC$ une \rlf 
$$\preskip.0em \postskip.4em 
0 \to  L'_{m-1} \vvvers{v_{m-1}}L_{m-2} \vvvers{u_{m-2}}\;  \cdots \cdots \; \vvers{u_2}  L_1 \vvers{u_1} L_0 
$$
avec $L'_{m-1}$ libre de rang ${r_{m-1}}$ et $v_{m-1}$ de matrice $V_{m-1}$.
\\
Par \hdr on obtient $\Gr_\gC(\fD_\ell,\gA)\geq \ell$ pour tout $\ell<m$.
On termine en remarquant que les mineurs $\delta$ de $U_m$ forment une suite de profondeur $\geq m>\ell$, et l'on applique le  \plgref{plccProfondeur}.   
\end{proof}
%

\subsec{Le \tho de Peskine et Szpiro}\label{secPesSzpi}

\begin{theorem}\label{thoPesSzpi}
Soit  $C_{\bullet}$ un complexe de \Amos
$$\preskip.4em \postskip.4em
0\lora C_n\vvers{d_n}\cdots\cdots \vvers{d_2} C_1\vvers{d_1} C_0
$$
tel que $\Gr(\fa,C_k)\geq k$ et $\fa\, \rH_k(C_{\bullet}) = 0$ pour $k\in\lrbn$ alors $C_{\bullet}$ est exact.
\end{theorem}

\begin{proof}
On note $B_k = \Im d_{k+1}$ et $Z_k = \Ker~d_k$, de sorte que $\rH_k(C_{\bullet}) = Z_k/B_k$, et l'on écrit $\rH_k$ pour $\rH_k(C_{\bullet})$.
\\
Puisque $\rH_n = Z_n$ est un sous-module de $C_n$ et $\Gr(\fa,C_n)\geq 1$ et $\fa\, \rH_n = 0$ on a
$\rH_n = 0$ et donc $0\rightarrow C_n\rightarrow C_{n-1}$ est exact. En outre, $B_{n-1} = \Im  d_n$
est isomorphe à $C_n$ et donc $\Gr(\fa,B_{n-1})\geq n$.
\\ 
Supposons $n\geq 2.$
Puisque $n-1\geq 1$, l'\id $\fa$ est $C_{n-1}$-\ndz.
Puisque~$Z_{n-1}$ est un sous-module de $C_{n-1}$ il est aussi $Z_{n-1}$-\ndz.
Le point \emph{\iref{i1thSESPrf}} du \thref{thSESPrf} et la suite exacte courte 
$$\preskip.4em \postskip.4em
0\rightarrow B_{n-1}\rightarrow Z_{n-1}\rightarrow \rH_{n-1}\rightarrow 0
$$
montrent alors que $\Gr(\fa,\rH_{n-1})\geq 1$ et donc $\rH_{n-1} = 0$ puisque $\fa \,\rH_{n-1} = 0$. 
En utilisant de nouveau le point \emph{\iref{i1thSESPrf}} du \thref{thSESPrf} avec la \seco 
$$\preskip.2em \postskip.4em
0\rightarrow B_{n-1}\rightarrow C_{n-1}\rightarrow B_{n-2}\rightarrow 0
$$
on obtient $\Gr(\fa,B_{n-2})\geq n-1$. 
\\
Si $n\geq 3$ 
on obtient de même $\rH_{n-2} = 0$ \hbox{et 
$\Gr(\fa,B_{n-3})\geq n-2.$}
\\
 Et ainsi de suite. 
\end{proof}

\subsec{Le \tho d'Auslander-Buchsbaum-Hochster}\label{secAusBu1}

\hum{il semble que  le
lemme qui suit fonctionne sous l'hypothèse $\cD_n(\varphi)\subseteq \fa$}
\begin{lemma} \label{lem1SESPrf} \emph{(Voir [FFR, Chap. 6, Th. 1])}\\
Soit $k\in\NN$,  $\fa$ un \itf, et une \seco 

\snic{0\to \Ae n\vvers{\varphi}\Ae \ell\lora G \to 0.}

\snii
Si $\ell>0$, $n>0$ et si $\cD_1(\varphi)\subseteq \fa$ (autrement dit les \coes de la matrice de $\varphi$ sont dans $\fa$),
alors $\Gr_\gA(\fa,G)\geq k$ \ssi $\Gr_\gA(\fa)\geq k+1$.
\end{lemma}
%
\begin{proof} 
On note $E=\Ae n$ et $F=\Ae \ell$ pour retrouver les notations du \thref{thSESPrf}. On a alors $\Gr_\gA(\fa,E)=\Gr_\gA(\fa,F)=\Gr_\gA(\fa)$.
Le point \emph{1} du \tho donne l'implication 
$$
{\Gr_\gA(\fa)\geq k+1\;\Longrightarrow\;\Gr_\gA(\fa,G)\geq k.}
$$
Puisque $\varphi$ est injective $\cD_1(\varphi)$ est $\gA$-\ndz
(les \coes d'une seule colonne suffisent).
L'hypothèse $\cD_1(\varphi)\subseteq \fa$ implique donc que l'\id $\fa$ \hbox{est $\gA$-\ndz}, \cad $\Gr_\gA(\fa)\geq 1$, ce qui donne l'implication réciproque \hbox{pour $k=0$}. 
\\
Pour $k\geq 1$, la \recu dont on a pris l'habitude fonctionne, mais il faut faire attention. On considère un  \pol de Kronecker $f_1$ attaché à~$\fa$. Puisque $f_1$ est \Grg, on a la \sex

\snic{ 0 \to E/f_1E\lora F/f_1F\lora G/f_1G\to 0,}

\snii avec $\Gr_\gA(\fa,G/f_1G)\geq k-1$. Ici, pour que l'appel à l'\hdr soit correct, nous devons changer d'anneau de base et passer à $\gA_1=\aqo\gA{f_1}$. \\
On a alors $E/f_1E\simeq \gA_1^{n}$,  $F/f_1F\simeq \gA_1^{\ell}$
et $\Gr_{\gA_1}(\fa,G/f_1G)\geq k-1$. On en déduit $\Gr_{\gA_1}(\fa)\geq k$,
puis $\Gr_{\gA}(\fa)\geq k+1$.
\end{proof}

Rappelons que pour un \alo $\gA$, on dit qu'un \Amo $E$ est {de \pdi $\leq m$}, et l'on écrit $\Pd_\gA(E)\leq m$, 
s'il admet une \rlf de longueur $m$.
Lorsque l'\alo est non trivial et \dcd, la \pdi
(d'un \mlrsb) est bien définie en raison de l'existence d'une \rsn libre minimale essentiellement unique.

La  \pdi est définie  en~\ref{defiPd}  pour tout
module  \lorsb, sur un anneau arbitraire, 
et le \tho d'Auslander-Buchsbaum sera \gne à ce cadre 
dans la section  \ref{secAusBuResProj}.

\medskip Le \tho d'Auslander-Buchsbaum a été démontré dans un cadre \noe,
il a été ensuite étendu au cadre \gnl par Hochster.
Son énoncé \gnl en \clama est le suivant (voir [FFR, Chap. 6, \tho~2, page~176]).

\THo{$\!\!\etl$ d'\ABH} {Soit $\gA$ un anneau local, $\fm=\Rad(\gA)$ et $E$ un \Amo non nul qui admet une \rlf. Alors on a l'\egt

\snic{\Gr(\fm,E)+\Pd(E)=\Gr(\fm).}
}

\hum{Si $\fm$ n'est pas \rtf il faut sans doute définir $\Gr(\fm)$
comme le sup des $\Gr(\fa)$ pour $\fa$ \tf contenu dans~$\fm$.
C'est bien cela? 

Mais nous avons refusé de donner cette \dfn à cause
de l'ambig\"uité qu'elle produit concernant la profondeur infinie.
\\
Bref, il me semble qu'il va falloir rajouter un commentaire quelque part
pour indiquer que nos \prcos impliquent bien le \tho ci-dessus même lorsque $\fm$ n'est pas \rtf, vue la \dfn adoptée par Northcott.

Notre profondeur infinie est plus infinie, car elle signifie $1\in\fm$.}

\medskip 
Nous démontrons dans cette section deux \thos de \coma 
(\ref{thABH1} et \ref{thABH2}) qui impliquent le \tho
précédent en \clama. Ces deux \thos peuvent donc être vus comme 
donnant un contenu \algq précis au \tho d'\ABH.
Ces \thos donnent en fait des résultats plus \gnls. Ils concernent un anneau arbitraire, non \ncrt local, et un \id arbitraire de cet anneau. 
En \clama, ils ne semblent pas pouvoir se déduire simplement du \tho d'\ABH ci-dessus.

Ce sont les deux \thos suivants.

\begin{theorem} \label{thABH1} \emph{(\ABH, direct)}\\
Soit $k\in\NN$ et $\fa$ un \itf de $\gA$. Si un \Amo $E$ admet une \rlf de longueur $m$ 
$$ 
    0\to L_m\vvers{u_m} L_{m-1}\llra \cdots\cdots\cdots \vvers{u_1} L_0\vers \pi E\to 0
$$
et si  $\Gr(\fa)\geq k+m$
alors  $\Gr(\fa,E)\geq k$.
\end{theorem}
%
\begin{proof} Pour $k=0$, il n'y a rien à prouver. On suppose $k\geq 1$.\\
On décompose la \sex longue en \secos
\[ 
\begin{array}{ccccccccccccccccccccccccccccccc} 
 0 & \to  & E_1 & \lora	 & L_0 & \lora & E &\to & 0\\
 0 & \to  & E_2 & \lora & L_1 & \lora & E_1 &\to & 0\\
   & \vdots   &   &   &   &  &   &\vdots \\
 0 & \to  & E_{m-1} & \lora & L_{m-2} & \lora & E_{m-2} &\to & 0\\
 0 & \to  & L_m & \lora & L_{m-1} & \lora & E_{m-1} &\to & 0\\
 \end{array}
\]
avec $E_i=\Im u_{i+1}=\Ker u_i\subseteq L_i$.
Le \thref{thSESPrf}, point \emph{\ref{i1thSESPrf}} donne tout d'abord $\Gr(\fa,E_{m-1})\geq k+m-1$, puis $\Gr(\fa,E_{m-2})\geq k+m-2$, jusqu'à~\hbox{$\Gr(\fa,E)\geq k$}.
\end{proof}

\hum{il semble que  le
\tho qui suit fonctionne sous l'hypothèse $\cD_{r_m}(u_m)\subseteq \fa$,
et qu'il puisse arriver que $  \sqrt{\cD_{r_m}(u_m)}\subsetneq \sqrt {\cD_{1}(u_m)}$ . 

\`A vérifier. Si c'est le cas on peut légèrement améliorer le \tho}

\begin{theorem} \label{thABH2}  \emph{(\ABH, réciproque)}\\
Soit $k\in\NN$ et $\fa$ un \itf de $\gA$.  On considère un \Amo $E$ 
qui admet une \rlf de longueur $m\geq 1$
$$ 
    0\to L_m\vvers{u_m} L_{m-1}\lora \cdots\cdots\cdots \vvers{u_1} L_0\vvers \pi E\to 0
$$
avec chaque $L_j$ de rang $>0$. Si  $u_m(L_m)\subseteq \fa L_{m-1}$ 
 (i.e si $\cD_1(u_m)\subseteq \fa$), et si $\Gr(\fa,E)\geq k$, alors
$\Gr(\fa)\geq k+m$. 
\end{theorem}

Notez que dans le cas d'un anneau local \dcd, on peut considérer une \rsn libre minimale  (\thrf{cor3lemModifComplexe}), 
et le \thref{thABH2} affirme quelque chose de non trivial pour tout \itf $\fa$ contenant~$\cD_1(u_m)$ et contenu dans $\fm$. 

\begin{proof}
Le  cas $m=1$ est donné par le 
lemme \ref{lem1SESPrf}.
 Supposons $m\geq 2$.  
 \\
 Voyons le cas $k=0$. 
Le \thref{thRangStProfRLF} nous dit que $\Gr\big(\cD_{r_m}(u_m)\big)\geq m$,
\hbox{où $L_m=\gA^{r_m}$},
or $\cD_{r_m}(u_m)\subseteq \cD_1(u_m)\subseteq \fa$ \hbox{car $r_m\geq 1$}, \hbox{donc $\Gr(\fa)\geq m.$} 
\\
Montrons enfin que le \tho passe de $k-1$ à $k$. \\
On va utiliser le lemme \ref{lem1SESPrf} et le \thref{thSESPrf}, point \emph{\ref{i2thSESPrf}.}
On décompose la \sex longue en \secos
\[ 
\begin{array}{ccccccccccccccccccccccccccccccc} 
 0 & \to  & E_1 & \lora	 & L_0 & \lora & E &\to & 0\\
 0 & \to  & E_2 & \lora & L_1 & \lora & E_1 &\to & 0\\
   & \vdots   &   &   &   &  &   &\vdots \\
 0 & \to  & E_{m-1} & \lora & L_{m-2} & \lora & E_{m-2} &\to & 0\\
 0 & \to  & L_m & \lora & L_{m-1} & \lora & E_{m-1} &\to & 0\\
 \end{array}
\]
avec $E_i=\Im u_{i+1}\subseteq L_i$. 

\snii

\snii 
On a $\Gr(\fa)\geq k+m-1\geq k+1$,
donc $\Gr(\fa,L_j)\geq k+1$
pour tous les~$j$.\\
Du \thref{thSESPrf} on déduit que $\Gr(\fa,E_j)\geq k+1$ 
pour $j=1$, \dots, $m-1$.  Le lemme \ref{lem1SESPrf} nous donne $\Gr(\fa)\geq k+2$ avec la dernière \sex. \hbox{Donc $\Gr(\fa,L_j)\geq k+1$}
pour tous les $j$. \\
Du \tho \ref{thSESPrf} on déduit  $\Gr(\fa,E_j)\geq k+2$ 
 pour $j=2,\dots,m-1$. Le lemme \ref{lem1SESPrf} donne  $\Gr(\fa)\geq k+3$.
 \\
Et ainsi de suite jusqu'à $\Gr(\fa)\geq k+m$.
\end{proof}

Voici maintenant un énoncé   \cof non \noe très proche du \tho original,
un peu plus fort que celui-ci, et moins fort cependant que la réunion des deux \thos que nous venons de voir.

\begin{theorem} \label{thABH3} \emph{(\Tho d'\ABH, original, version \cov)}\\
Soit $(\gA,\fm)$ un \alo  non trivial et $E$ un \Amo
qui admet une \rsn libre minimale
$$ 
    0\to L_m\vvers{u_m} L_{m-1}\lora \cdots\cdots\cdots \vvers{u_1} L_0\vvers \pi E\to 0.
$$ 
Soit $\fa$ un \id radicalement \tf contenant $\cD_1(u_m)$ et contenu dans~$\fm$. 
Alors on a l'\egt

\snic{\Gr_\gA(\fa,E)+\Pd_\gA(E)=\Gr_\gA(\fa),}

autrement dit on a l'\eqvc suivante: 

\snic{\fbox{$\forall k\in\NN,\;\;\; \Gr_\gA(\fa,E)\geq k\iff \Gr_\gA(\fa)\geq \Pd_\gA(E)+k.$}}
 
\end{theorem}

\emph{Notes.} 1) L'entier $m=\Pd_\gA(E)$ et l'\id $\cD_1(u_m)$  sont bien définis en vertu du \thref{thBetti}.
 
2)
On peut toujours prendre $\fa=\cD_1(u_m)$, ou $\fa=\fm$ si $\fm$ est radicalement \tf, ce qui est le cas lorsque  $\gA$ 
est  \noe \coh \fdi\siBookdeux{\ (\thref{nilreg2})}.
 
3)
Si $\gA$ est \dcd, on peut construire une \rsn minimale à partir d'une \rlf arbitraire (\thref{cor3lemModifComplexe}). 
 
4)
Il est sans doute 
possible d'avoir deux \itfs $\fa_1$ et $\fa_2$ contenus dans $\fm$
et contenant~$\cD_1(u_m)$, mais
qui ne soient pas de même profondeur. Ceci ne contredit pas a priori
le \tho.
\eoe

Voici un corolaire du \thref{thABH1}.

\begin{proposition} \label{corthABH1}\emph{(\Tho de Rees, voir [FFR, Chap. 6, Th. 4]\footnote{Northcott démontre plutôt que si $E\neq 0$ et $\fa\,E=0$, alors $\Gr(\fa)\leq m$. C'est un énoncé avec deux négations puique la signification de \gui{$\Gr(\fa)\leq m$} c'est que \gui{$\Gr(\fa)\geq m+1$
 est absurde}. Ce sont les énoncés $\Gr(\fa)\geq k$ qui sont définis en premier.})}\\
Soit $\fa$ un \itf de $\gA$ et $E$ un \Amo  qui admet une \rlf de longueur $m$
$$ 
    0\to L_m\lora L_{m-1}\lora \cdots\cdots\cdots \lora L_0\vvers \pi E\to 0.
$$
Si $\Gr(\fa)\geq m+1$ et $\fa\,E=0$ alors
$E=0$. 
 
\end{proposition}
%
\begin{proof}
Le \thref{thABH1} donne $\Gr(\fa,E)\geq 1$, donc $\fa\,E=0$ implique $E=0$.
\end{proof}

Voici une application, dans laquelle on précise la proposition \ref{propfDk=1}.
\begin{proposition} \label{propfDk=1+} 
Soit $(r_0,\dots,r_n,r_{n+1})$ dans $\NN$ avec $r_{n+1}=0$, et $M$ \hbox{un \Amo} qui admet une \rlf

\smallskip 
\centerline{\fbox{$L_{\bullet}:\;\;0 \to L_n \vvers{u_n}  L_{n-1} \vvers{u_{n-1}}\;  \cdots \cdots \; \vvers{u_2}  L_1 \vvers{u_1} L_0 \vers{\pi} M
\to 0$.}}

\snii
avec   $L_k\simeq\gA^{r_{k+1}+r_k}$. 
\begin{enumerate}
\item Si $\fD_{k}(M)=\gen{1}$, alors $\Pd_\gA(M)\leq k-1$ et $\fD_{k+s}=\gen{1}$ pour tout $s>0$. Si en outre $r_0=0$ et $\Gr(\fD_1)\geq k$, alors $M=0$.
\item Pour tous $\ell\geq 1$, on a $\fD_{\ell}\subseteq \DA(\fD_{\ell+1})$.
Si en outre  $r_0=0$ on obtient l'\eqvc suivante pour tout $p\geq 2$

\snic{\Gr(\fD_1)\geq p \iff \DA(\fD_{1})=\dots=\DA(\fD_{p}).}
\end{enumerate}
\end{proposition}
%
\begin{proof}
Par rapport à \iref{propfDk=1}, il reste à prouver les assertions concernant
le \hbox{cas $r_0 = 0$}. Dire que $r_0$ est nul, c'est dire que $u_1$
est de rang stable~\hbox{$r_0+r_1$}, auquel cas $\fD_1$ est le $0$-Fitting $\cF_0(M)$ 
de $M$. \\ 
 On a 
$M = 0 \Leftrightarrow 1 \in \fD_1$. 
 En \gnl, 
puisque $\cF_0(M) \subseteq \Ann(M)$,  $\fD_1 M = 0$. 

\emph{1.}
Clair si $k=1$ car $1 \in \fD_k=\fD_1$. Supposons $k\ge 2$.
\\
On a $\Gr(\fD_1) \ge k$, $\Pd(M) \le k-1$ et $\fD_1 M = 0$. Le \tho de Rees~\iref{corthABH1} nous dit que $M=0$.  En fait, le contexte du \tho de Rees est celui des
\rlfs mais il s'applique \egmt aux \rsfs comme on le voit en utilisant
des \lons \come convenables.

\emph{2.}
L'implication $\Leftarrow$ résulte du fait que les \ids $\fD_1$ et $\fD_p$ ont même profondeur (car même radical) et de l'inégalité $\Gr(\fD_p) \ge p$.

\snii
Pour l'implication $\Rightarrow$, il suffit de voir que $\fD_p \subseteq \DA(\fD_1)$. Si $\mu \in \fD_p$, on 
applique le point \emph{1}  sur $\gA[1/\mu]$: on obtient $1 \in \fD_1[1/\mu]$, donc $\mu\in\DA(\fD_1)$.
\end{proof}
%
 
\bonbreak
\section{Ce qui rend un complexe de modules libres exact}\label{secCQRCMLE}

Le \tho suivant reprend le \thref{thRangStProfRLF}, en affirmant la réciproque.

\begin{theorem} \label{cor3thABH1}\emph{(Voir [FFR, Chap. 6, Th. 15])}\\
On considère un complexe de modules libres

\smallskip 
\centerline{\fbox{$L_{\bullet}:\quad 0 \to L_n \vvers{u_n}  L_{n-1} \vvvers{u_{n-1}}\;  \cdots \cdots \; \vvers{u_2}  L_1 \vvers{u_1} L_0\;$}}

\smallskip 
où  $L_k\simeq\gA^{r_{k+1}+r_k}$ pour tout $k$, 
avec $r_{n+1}=0$. On rappelle la notation~\hbox{$\fD_k=\cD_{r_k}(u_k)$}
pour les \idcas du complexe.\\
Le complexe est exact \ssi $\Gr(\fD_\ell)\geq \ell$ pour tout $\ell$.
\end{theorem}
%
\begin{proof}  
Vu le \thref{thRangStProfRLF}, il reste à montrer que si $\Gr(\fD_\ell)\geq \ell$ pour tout $\ell$, alors
le complexe est exact. \\
On fait une \dem par \recu sur $n$. Le cas $n=1$ est clair. Passons \hbox{de $n-1$} \hbox{à $n$}.
Par \hdr, le complexe tronqué à~$L_1$ est exact, on
a donc une \rlf de $L_1/\Im(u_2)$:
$$
{0 \to L_n \vvers{u_n}  L_{n-1} \vvvers{u_{n-1}}\;  \cdots \cdots \; L_2 \vvers{u_2}  L_1\vers{\pi}  L_1/\Im(u_2). 
}
$$
Le \thref{thABH1} implique alors que $\Gr\big(\fD_n,L_1/\Im(u_2)\big)\geq 1$.\\
Montrons que $L_{\bullet}$ est exact après \lon en chaque 
\gtr~$\delta_i$ de~$\fD_n$. Or quand on inverse un mineur maximal de $u_n$,
le complexe peut être raccourci de son premier terme $L_n$ sans changer l'homologie (corolaire~\ref{corlemModifComplexe}), et l'\hdr s'applique
sur ce nouveau complexe avec l'anneau~$\gA[1/\delta_i]$.\\
Il reste enfin à voir  que si $L_{\bullet}$ est exact après \lon en chaque \gtr $\delta_i$ de~$\fD_n$,
il est exact. Il faut juste voir l'exactitude en $L_1$. 
\\
Soit $x\in L_1$
tel que $u_1(x)=0$. Après \lon en $\delta_i$, on a $x=0$ modulo~$\Im(u_2)$. Et $\Gr\big(\fD_n,L_1/\Im(u_2)\big)\geq 1$ signifie que $\fD_n$ est \ndz
modulo~$\Im(u_2)$. \hbox{Donc $x=0$}.
\end{proof}
%

\siarticle{ }
\section{Le cas des \algs et modules gradués}\label{secRLFGRAD}

\centerline{\fbox{\Large \`A \'Ecrire}}

\medskip 
\hum{on espère traiter parallèlement le cas gradué pur et le cas gradué saturé et faire comprendre des choses par cette comparaison}

\Exercices

\begin{exercise} \label{exoProjResolutionToFreeResolution}
{(D'une \rsf de $E\subseteq\gA^r$ à une \rlf de $E\oplus\gA^r/E$)}\\
{\rm  
Soit $E$ un sous-module de $\gA^r$ et $E'=\gA^r/E$.

\snii\emph {1.}
On suppose que $E$ admet une \rsf. Montrer que $E\oplus E'$
admet une \rlf et que sa \cara d'Euler-Poincaré est $r$.

\snii\emph {2.}
En déduire que $E$ admet une \rsf \ssi~\hbox{$E\oplus E'$} admet une \rlf.

Nous devons cet exercice à Jean-Pierre Jouanolou.

}

\end{exercise}



\begin{problem} \label{exoF=UX+VY} {(Résolution infinie de l'\id d'un point singulier
d'une courbe plane)}
\entrenous{Provisoirement dans ce chapitre, pour le cas où l'on n'aura pas mis de \rsns dans le chapitre sur les \algs et modules gradués}
{\rm  
Soit $p_0 = (x_0,y_0)$ un point du plan $\AA^2(\gk)$ et $F \in \gk[X,Y]$
vérifiant $F(p_0) = 0$. On note $\gA = \aqo{\gk[X,Y]}{F} = \gk[x,y]$ l'anneau des \coos de la courbe $\{F=0\}$.
Si le point $p_0$ est un point singulier de la courbe,
on veut illustrer le fait que son \id $\gen {x-x_0, y-y_0}$, considéré
comme \Amo, admet des résolutions libres \gui {minimales} de
longueur arbitairement grande:

\snic{
L_m \to L_{m-1} \to \cdots \to L_1 \to L_0 \twoheadrightarrow 
\gen {x-x_0,y-y_0} \to 0,
}

\snii minimale signifiant ici que les matrices de la résolution sont
à \coes dans l'\id $\gen {x-x_0,y-y_0}$. En fait si l'on considère
l'\alg locale de la courbe en~$p_0$, \cad $\gA'=\gA_{1+\gen{x-x_0,y-y_0}}$,
on a $\Rad(\gA')\supseteq \gen{x-x_0,y-y_0}_{\gA'}$, et la \rsn est bien minimale
au sens donné dans le cours.

\snii
Dans la suite, on suppose \spdg que $p_0 = (0,0)$. L'anneau de base $\gk$ est
quelconque et l'on demande que $F$ soit \ndz. 
Puisque $F(p_0) = 0$ on peut écrire $F = UX + VY$.

\snii
\emph {1.}
On note $\gB = \gk[X,Y]$ et l'on considère $\gen {x,y}$ comme \Bmo.
Montrer que:

\snic {
0 \to \gB^2 \vvers {M} \gB^2 \vvers {[x,y]} \gen {x,y} \to 0
\qquad \hbox {avec} \quad M = \cmatrix {Y & U\cr -X & V}
}

\snii est une \rsn. Autrement dit le module
des $\gB$-\syzys pour $(x,y)$ est un \Bmo libre
de rang~2 (les colonnes de $M$ en forment une base).

\snii\emph {2.}
Dans cette question, on considère $\gen {x,y}$ comme $\gA$-module.
On note encore par abus $M=\cmatrix {y & u\cr -x & v}$ la matrice $M$ vue dans $\gA$.
Montrer que l'on a une résolution libre de $\gen{x,y}$ de la forme suivante:

\snic {
\cdots \to 
\gA^2 \vers {\wi M} \gA^2 \vvers {M}  \gA^2 \vers {\wi M} 
\gA^2 \vers {M} \gA^2 \vers {\wi M}  \gA^2 \vers {M} 
\gA^2 \vers {[x,y]} \gen {x,y} \to 0
}

\snii
On notera que les \coes de $M$ et $\wi M$ sont dans $\gen {x,y}$
\ssi la composante \hmg de degré $1$ de $F$ est nulle,
\cade si la multiplicité de $p_0$ dans $\{F=0\}$
est $\ge 2$.

\snii\emph {3.}
Histoires de poids: cette question, indépendante des autres, 
n'est pas essentielle à l'exercice mais elle permet de se
familiariser avec les poids et les décalages. On gradue
$\gk[X,Y]$ en posant $\poids(X) = w_1$, $\poids(Y) = w_2$ et
l'on suppose que~$U$ est \hmg de poids $d-w_1$ et $V$ \hmg
de poids $d - w_2$, et donc~$F$ est \hmg de poids $d$
(avec $w_1$, $w_2$, $d\in\NN\etl$). \\
Dans la \rsn 
$0 \to L_1 \to L_0 \to \gen {x,y} \to 0$ de la première
question, déterminer les décalages de $L_0$ et $L_1$ pour
en faire une \rsn graduée. Déterminer \egmt
les séries d'Hilbert-Poincaré de $\gB$, $\gA$, $L_0$,
$L_1$ et $\gen {x,y}$.
}

\end{problem}


\begin{problem} \label{exoMonomialTaylorResolution} {(La résolution monomiale de Taylor)}
\\
{\rm  
\noindent
Ici, $\gk$ est un anneau quelconque. Soit $\um = (m_1, \ldots, m_r)$
une suite de  monômes dans $\gA = \gk[\Xn]$. 
Pour $J \subseteq \lrbr$, on note
$\um_J$ la restriction de la famille aux indices $j \in J$.
\`A la famille $\um$, on associe un complexe libre:

\snic {
0 \to L_r \vers {d_r} L_{r-1} \to \cdots  \to L_2 \vers {d_2} 
L_1=\gA^r \vers {d_1} L_0=\gA \to \aqo\gA\um \to 0
}

\snii
Pour  $k\in\lrbr$, $L_k$ est un \Amo libre de rang 
${r \choose k}$ dont une base $(e_J)_J$ est indexée par 
$\cP_{k,r}$. 
La \dile $d : L_k \to L_{k-1}$ est fournie par
$$
d(e_J) = \som_{j \in J}\; (-1)^{n_{j,J}} \;
\frac{\ppcm(\um_J)} {\ppcm(\um_{J \setminus \{j\}})}
\, e_{J \setminus \{j\}},
$$
où $n_{j,J}=\#\sotq{j\in J}{j<k}$.  On
remarquera que $d_1(e_j)= m_j$. 
\\
Pour $j$, $k \in J$,  
$m_k \mid m_j\ppcm(\um_{J \setminus \{j\}})$,  donc  
$\ppcm(\um_J) \divi m_j\ppcm(\um_{J \setminus \{j\}})$.
En conséquence, dans
la \dfn de $d$, le quotient des deux ppcm est toujours un diviseur de
$m_j$. Si on remplace ce quotient par $m_j$ (ce qui revient à remplacer le
ppcm par le produit), 
on obtient le complexe de Koszul descendant de $(m_1, \ldots, m_r)$ (en \gnl
non exact).

\snii On va montrer que le complexe $(L_{\bullet}, d)$ est exact; on l'appelle
la \emph{résolution de Taylor de $\aqo\gA\um$ (ou par abus, de $\um$)}.

\snii
 De manière précise, on va construire
une homotopie contractante \hbox{$h : L_{\bullet} \to L_{1+\bullet}$} qui est
$\gk$-\lin (mais pas $\gk[\uX]$-\lin). \\
Comme cas particulier, le complexe de Koszul de la suite~\hbox{$(\Xn)$} sur $\gk[\uX]$ admet une homotopie contractante $\gk$-\lin. Mais le
complexe de Koszul d'une \srg d'un anneau $\gA$ est exact sans admettre pour autant une homotopie contractante $\gA$-\lin.

\snii \emph {1.}
Pour un \mom $p \in \gk[\uX]$ et $J\in\cP_{k,r}$, il
faut définir $h(pe_J) \in L_{k+1}$; $h$ sera ensuite prolongée par
$\gk$-linéarité.  Soit $i$ le plus petit  indice tel
que $m_i \divi \ppcm(\um_J)p$. Si $J = \emptyset$, cet indice $i$ peut ne pas
exister; sinon, il est clair que $i \le J$.  On pose alors:

\snic {
J' = J \cup \{i\}, \qquad\quad
h(pe_J) = \cases {\displaystyle{\ppcm(\um_J)p \over \ppcm(\um_{J'})}\, e_{J'} 
&si $i$ existe et $i \notin J$,\cr 
\noalign {\smallskip}
0 &sinon.}
}

\snii
Vérifier que $d \circ h + h \circ d = \Id$.

\snii \emph {2.}
Comment graduer le complexe de Taylor de fa\c{c}on à ce que
chaque différentielle soit graduée de degré $0$?

\snii \emph {3.} Montrer que la \rsn de Taylor est minimale \ssi 

\snic{\forall J\subseteq \lrbr \;\;\;\forall j\in J, \;\;\;\; m_j \nedivi \ppcm \um_{J\setminus \so j} .}

\snii
Donner un exemple simple pour lequel la \rsn de Taylor n'est pas minimale.

\snii \emph {4.}
Soit $\um = (x^2y, xy^3, x, yz)$. Expliciter la résolution de Taylor de $\um$
et la comparer avec la résolution minimale.

}
\end{problem}

\begin{problem} \label{exoPolynomialSyzygies} {(Quand les \moms dominants sont 
premiers entre eux)}
\\
{\rm  
Dans ce \pb, $\gk$ est anneau quelconque et $\kuX = \gk[\Xn]$ un anneau
de \pols avec un ordre monomial $\preceq$ fixé. \\ 
Si $m=\uX^{\varphi}$ est un \mom et~$f$
un \pol, la notation $f \preceq m$ signifie que $f$ est une
combinaison $\gk$-\lin de \moms $\preceq m$; convention analogue pour $f \prec m$.
Enfin $f \in \kuX$ est dit \emph{$\preceq$-unitaire} si  $f = m + r$
où $m$ est un \mom et $r \prec m$.

\snii
Soient  $f_1$, \ldots, $f_s \in \kuX$. On suppose que chaque $f_i$ est
$\preceq$-unitaire, de \mom dominant $m_i$ et enfin que $\pgcd(m_i,m_j) = 1$
pour $i \ne j$.  \\
On note $S$ le sous-$\gk$-module de $\kuX$ admettant base les \moms
$m$ tels que $m \notin \gen {m_1, \ldots, m_s}$:

\snic {
\displaystyle {S = \bigoplus_{m \notin \gen {m_1, \ldots, m_s}} \gk\,m}
}

\snii
Un des objectifs du \pb est de montrer que 
$$
\kuX = \gen {f_1, \ldots, f_s} \oplus S
\leqno (\star)
$$
La \dem ci-dessous fournit les moyens d'écrire tout \pol $f$ dans la
décomposition $\fa \oplus S$.

\snii
Pour ceux qui connaissent les \bdgs: si $\gk$ est un corps, la
décomposition~$(\star)$ ci-dessus est classique; on dit parfois que les
\moms de $S$ (les \gui {\moms sous l'escalier}) sont standard: tout \pol $f$
est donc \eqv modulo $\fa$ à un et un seul \pol $\gk$-combinaison de \moms
standard, combinaison qualifiée de forme normale de $f$ modulo $\fa$
(relativement à l'ordre monomial~$\preceq$). Dans ce contexte, l'\id momonial
$\gen {m_1, \ldots, m_s}$ est l'\id initial de $\fa$ relativement à l'ordre
monomial~$\preceq$ (\id engendré par les \moms dominants des \pols
de $\fa$).

\snii\emph {1.}
Pour un \mom $m$, on introduit deux $\gk$-modules $\fa_{\prec m}$ et
$\fa_{\preceq m}$ qui sont des sous-$\gk$-modules de l'\id $\fa = \gen {f_1,
\ldots, f_s}$,

\snic {
\fa_{\prec m} = \bigl\{ \sum_i g_i \,|\, g_i \in \gen{f_i}
\hbox { et } g_i \prec m \bigr\}
\ \subset \
\fa_{\preceq m} = \bigl\{ \sum_i g_i \,|\, 
g_i \in \gen{f_i} \hbox { et } g_i \preceq m \bigr\}
}

\snii
Il est clair que $\fa$ est la réunion des $\fa_{\prec m}$ a fortiori
des $\fa_{\preceq m}$.

\begin {itemize}
\item [\emph {a.}]
Il existe des $r_i \prec m_i$ tels que  $m_jf_i
- m_if_j = r_jf_i - r_if_j$.

\item [\emph {b.}]
Soient $m'_i, m'_j$ deux \moms tels que $m'_im_i = m'_jm_j$.\\
En notant $m$
le \mom $m'_im_i = m'_jm_j$, montrer que $m'_if_i - m'_jf_j \in \fa_{\prec m}$.

\item [\emph {c.}]
Plus \gnlt, pour une partie $I \subseteq \{1,\ldots,s\}$,
supposons disposer de \moms $(m'_i)_{i \in I}$ et d'un \mom $m$
tels que $m'_im_i = m$. Alors si
$(a_i)_{i \in I}$ est une famille finie d'\elts de $\gk$ de somme nulle,
on a $\sum_{i \in I} a_im'_if_i \in \fa_{\prec m}$.
\\
On pourra utiliser une \gui {transformation d'Abel}:

\snic {
a_1 b_1 + \cdots + a_kb_k = \sum_{j=1}^{k-1} s_j(b_j-b_{j+1})}

 avec  $s_j = \sum_{i=1}^j a_i$
 dès que  $\sum a_i = 0$.

\item [\emph {d.}]
Soit $g \in \gen {f_i}$ et un \mom $m$ tel que $g \preceq m$.  Selon que $m_i$
divise ou ne divise pas $m$, montrer que, ou bien $g \prec m$ ou bien $g -
am'f_i \in \fa_{\prec m}$ pour un $a \in \gk$ et un \mom $m'$ tel que $m'm_i =
m$.

\end {itemize}

\snii\emph {2.}
Soit $m$ un \mom.
Montrer que si $f \in \fa_{\preceq m}$ et $f \prec m$, alors
$f \in \fa_{\prec m}$. En déduire
que
$S \cap \fa_{\preceq m} \subset \fa_{\prec m}$. 
 Puis  $S \cap \gen {f_1, \ldots, f_s} = 0$.

\snii\emph {3.}
Montrer que $\kuX = \gen {f_1, \ldots, f_s} + S$.  Plus
\prmt: si $f \in \kuX$ \hbox{et $f \preceq m$} pour un \mom $m$, alors $f \in
\fa_{\preceq m} + S$.  Bilan: $\kuX = \gen {f_1, \ldots, f_s} \oplus S$.  
\\
En outre $\gen {m_1, \ldots, m_s}$ est \emph{l'\id initial} de $\gen {f_1, \ldots, f_s}$
au sens suivant: si \hbox{pour $a \in \gk$} et $m$ \mom on a $am + h \in \gen {f_1,
\ldots, f_s}$ avec $h \prec m$, alors $am $ appartient à $\gen {m_1, \ldots, m_s}$.

\snii\emph {4.}
Le module des \syzys pour $(f_1, \ldots, f_s)$ est engendré
par les $f_i\vep_j - f_j\vep_i$ \hbox{où $(\vep_1, \ldots, \vep_s)$} est la base
canonique de $\kuX^s$.

\snii\emph {5.}
La suite $(m_1, \ldots, m_s)$ est \ndze car il en
est de même de $(f_1, \ldots, f_s)$.

\snii\emph {6.}
On suppose $m_i \ne 1$ pour tout $i$.  Soit $M$ l'ensemble des \moms
appartenant à~$S$.  Montrer que $\kuX$ est un $\gk[f_1, \ldots, f_s]$-module
libre de base les $m \in M$ et que~\hbox{$(f_1, \ldots, f_s)$} sont $\gk$-\agqt
indépendants.

\snii\emph {7.}
Soient, dans $\gk[X,Y]$, les deux \pols $f_1 = X^2$ et $f_2 = XY + aY^2$
\hbox{avec $a \in \gk$}, tous les deux \hmgs de degré $2$ et unitaires
pour l'ordre lexicographique~\hbox{$Y \prec X$}. 

\begin {itemize}
\item [\emph {a.}]
On a $a^2 Y^3 \in \gen {f_1,f_2}$, et  $Y^3 \in \gen {f_1,f_2}\iff a\in\gk\eti$.
\item [\emph {b.}]
On suppose $a$ \ndz. Alors l'\id $\gen {f_1,f_2}$ est facteur direct
dans $\gk[X,Y]$ \ssi $a$ est \iv.
\end {itemize}
}
\end{problem}

\begin{problem} \label{exoResolutionPfafienne}
{(Résolution pfaffienne)} 

{\rm  
Soit $X \in \Mn(\gA)$ une matrice antisymétrique avec $n$ impair; pour $i
\in \lrb{1..n}$, la matrice obtenue en supprimant la ligne $i$ et la colonne
$i$ de $X$ est une matrice antisymétrique de taille paire $n-1$ dont on peut
considérer le pfaffien. On note $q_i$ ce pfaffien multiplié par $(-1)^i$
et par $Q$ la matrice ligne $[\, q_1,\, \ldots,\, q_n\,]$. Par exemple,
pour $n = 3$:

\snic {
X = \cmatrix {0 &x_{12}& x_{13}\cr 
-x_{12} &0 &x_{23}\cr  -x_{13} &-x_{23} &0\cr}, \quad
q_1 = -x_{23}, \quad q_2 = x_{13}, \quad q_3 = -x_{12}
}

\snii\emph {1.}
Montrer que $\wi X = \tra{Q} Q$. En déduire que $Q X = 0$
et que $\cD_{n-1}(X) = \cD_1(Q)^2$.

\snii\emph {2.}
On suppose que $\gA$ est un anneau de \pols sur $\gk$ à $n(n-1)/2$
\idtrs $x_{ij}$ avec $1 \le i < j \le n$ et on considère
la  matrice antisymétrique \gnq~\hbox{$X \in \Mn(\gA)$}.
\begin {enumerate}
\item [\emph {a.}]
Montrer que $X$ est une matrice stable de rang $n-1$.
\item [\emph {b.}]
Soient $p_1 = q_1$, $p_2 = q_{n+1 \over 2}$ et $p_3 = q_n$.
Montrer que $(p_1, p_2, p_3)$ est une suite \ndze.

\item [\emph {c.}]
Montrer que   $\Gr_\gA(\cD_{n-1}(X)) \ge 3$.

\item [\emph {d.}]
Montrer que la suite ci-dessous est exacte:

\snic {
0 \to L_3 = \gA \vers {u_3} L_2 = \gA^n \vers {u_2} 
L_1 = \gA^n \vers {u_1} \gA  
\qquad \hbox {avec} \quad
\left\{\begin {array} {c}
u_3 = \tra{Q} \\ 
u_2 = X \\
u_1 = Q \\
\end {array}
\right.
}
\end {enumerate}

}

\end{problem}


\bonbreak
\sol

\exer{exoProjResolutionToFreeResolution} \emph{(D'une \rsf de $E\subseteq\gA^r$ à une \rlf de $E\oplus\gA^r/E$)}

\snii\emph {1.}
On peut choisir une \rsf de $E$ où les $L_i$ sont \lrfs et $P$ \ptf:

\snic {
0 \to P \to L_n \to \cdots \to L_1 \to L_0 \to E \to 0
}

On en déduit une \rsf de $E'=\gA^r/E$ que l'on juxtapose à celle de $E$

\snic {
\setlength{\arraycolsep}{2pt}
\begin {array} {ccccccccccccccccccccc}
  &   &0 &\to&P   &\to&L_n    &\to&L_{n-1} &\to&\cdots &\to&L_2
  &\to&L_1 &\to&L_0   &\to&E  &\to&0 \\
0 &\to&P &\to&L_n &\to&L_{n-1}&\to&L_{n-2} &\to&\cdots &\to&L_1 
  &\to&L_0 &\to&\gA^r &\to&E' &\to&0 \\
\end {array} 
}

On en fait la somme directe, ce qui donne, en posant $L_{-1}=\gA^r$

\snic {
\def\To {\kern-1pt\to\kern-1pt}
0 \To P\To P\oplus L_n\To L_n\oplus L_{n-1}\To\cdots\To
L_2\oplus L_1\To L_1\oplus L_0\To L_0\oplus L_{-1}\To E\oplus E' \To 0
}

Soit $Q$ tel que $L_{n+1} := Q\oplus P$ soit \lrf.
Par ajout de $Q$ aux deux extrémités gauches, il vient une
résolution de $E\oplus E'$:

\snic {
\def\To {\kern-1.5pt\to\kern-1.5pt}
\def\To {\to}
0 \To L_{n+1}\To L_{n+1}\oplus L_n\To L_n\oplus L_{n-1}\To\cdots\To
L_1\oplus L_0\To L_0\oplus L_{-1}\To E\oplus E' \To 0
}

Par télescopie (comme disent les étudiants), on voit que
la \cara d'Euler-Poincaré est $r = \rg(L_{-1})$.

\snii\emph {2.}
Immédiat.



\prob{exoF=UX+VY} \emph{(Résolution  de l'\id d'un point singulier
d'une courbe plane)}

\emph {1.}
On a $\det(M) = F$ donc $M$ est injective. L'\egt $[x,y] M = 0$
est \imde.\\
Montrons que
$\ker [x,y] \subseteq \Im M$. Soient $P$, $Q \in \gB$ tels
que $PX + QY \equiv 0 \bmod F$, disons $PX + QY = (UX + VY)K$.
De $(P-KU)X + (Q-KV)Y = 0$, on tire l'existence de $H \in \gB$
tel que:

\snic {
\cmatrix {P-KU\cr Q-KV\cr} = H \cmatrix {Y\cr -X}, 
\; \hbox {donc} \quad
\cmatrix {P\cr Q\cr} = H \cmatrix {Y\cr -X} + K \cmatrix {U\cr V} \in \Im M.
}

\snii\emph {2.}
Il est clair, sur $\gA$, que $M \circ \wi M = \wi M \circ M = 0$.  \\
Il reste à
voir, sur $\gA$, \hbox{que $\ker M \subseteq \Im \wi M$} et $\ker \wi M \subseteq \Im M$.\\
Montrons la première inclusion. Soit $\ov V \in
\ker M$; donc, sur~$\gB$,  $M \cdot V \in F\gB^2$. \\
On écrit
$F\gB^2 = M\wi M\gB^2$ et l'on simplifie par $M$ (injective):~\hbox{$V \in \wi M\gB^2$}.

\snii\emph {3.}
On a:
$$
S_\gB = {1 \over (1-t^{w_1})(1-t^{w_2})}, \qquad
S_\gA = {1 - t^d \over (1-t^{w_1})(1-t^{w_2})}
$$
\`A droite, il faut prendre $L_0 = \gB(-w_1) \oplus \gB(-w_2)$
pour que $[x,y]$ soit graduée de poids $0$. \`A gauche, il faut
prendre $L_1 = \gB(-w_1-w_2) \oplus \gB(-d)$ pour que la matrice
$M$ soit graduée de poids $0$. \Llec peut vérifier
cela en utilisant le fait qu'une matrice $M :
\bigoplus_j \gB(-m_j) \to \bigoplus_i \gB(-n_i)$ est graduée de degré
$0$ \ssi sa $j$-ième colonne $M_j :  \gB(-m_j) \to 
\bigoplus_i \gB(-n_i)$ est graduée de degré~$0$ ce qui signifie que $M_{ij}$
est \hmg de poids $m_j-n_i$. On trouve alors
$$
S_{L_0} = {t^{w_1} + t^{w_2} \over (1-t^{w_1})(1-t^{w_2})}, \quad
S_{L_1} = {t^d + t^{w_1 + w_2} \over (1-t^{w_1})(1-t^{w_2})} \; \;\hbox{et}\;\;
S_{\gen {x,y}} = S_{L_0} - S_{L_1}. 
$$


\prob{exoMonomialTaylorResolution} \emph{(La résolution monomiale de Taylor)}

\snii\emph {1.}
On montre que $(h\circ d + d\circ h)(pe_J) = pe_J$ dans le cas où l'indice
$i$ intervenant dans la \dfn de $h(pe_J)$ n'appartient pas à $J$; le cas
où $i \in J$ est laissé au lecteur. On note $|J|$ pour $\ppcm(\um_J)$,
$J \setminus j$ au lieu de $J \setminus \{j\}$ et $J \cup i$ au lieu \hbox{de
$J \cup \{i\}$}. 
\\
Calculons d'abord 
{$d\circ h$} sur
$pe_J$; par \dfn, avec $J' = J \cup i$, on a:

\snic {
h(p e_J) = {|J|p \over |J'|}\, e_{J'}.
}

\snii
Comme $i$ est le plus petit \elt de $J'$, on a:

\snic {
d(e_{J'}) = {|J'| \over |J|} e_J - \sum_{j \in J} (-1)^{n_{j,J}}
{|J'| \over |J' \setminus j|} e_{J' \setminus j},
}

\snii ce qui donne

\snic {
(d\circ h)(pe_J) = pe_J - \sum_{j \in J} (-1)^{n_{j,J}}
{|J|p \over |J' \setminus j|} e_{J' \setminus j}.
}

\snii
Déterminons maintenant 
{$h\circ d$}
sur $pe_J$. D'abord:

\snic {
d(pe_J) = \sum_{j \in J} (-1)^{n_{j,J}} q_J e_{J\setminus j} 
\quad \hbox { où $q_J$ est le \mom\ } {|J|p \over |J\setminus j|}.
}

\snii
Pour déterminer $(h \circ d)(pe_J)$, il faut déterminer pour chaque
$j \in J$ le plus petit~$k$ tel que $m_k \divi |J\setminus j|q_J$.
Mais $|J\setminus j|q_J = |J|p$, donc cet indice $k$ c'est $i$
(pour \hbox{tout $j \in J$}). D'où:

\snic {
h(q_J e_{J\setminus j}) = {|J \setminus j| q_J \over 
|J \setminus j \cup i|} e_{J \setminus j \cup i} =
{|J|p \over |J' \setminus j|} e_{J' \setminus j}.
}

\snii En additionnant $(d\circ h)(pe_J)$ et $(h\circ d)(pe_J)$, il reste $pe_J$.

\sni\emph {2.}
Il suffit de donner à $e_J$ le poids $\deg(\ppcm(\um_J))$.

\sni\emph {3.} En effet la \rsn est minimale \ssi aucun des 
${\ppcm(\um_J) \over \ppcm(\um_{J \setminus \{j\}})}$ n'est égal à $1$.

Comme exemple simple de \rsn de Taylor non minimale on peut prendre celle correspondant aux \moms
$(xy,x,yz)$ dans $\gk[x,y,z]$. La différentielle coupable est $d_3$ qui admet pour matrice
$\bordercmatrix[\lbrack\rbrack] {
   & 123 \cr
12 & 1   \cr
13 & -1   \cr
23 & 1   \cr}
$.
La \rsn minimale est donnée par la \sex
%
%
$
0\to\gA^{2}\vvers{d_2}\gA^{3}\vvers{d_1}\gA\lora\aqo\gA{xy,x,yz}\to0
$
%
\hbox{avec $d_2=\cmatrix{z&0\cr-y&y\cr0&-x}$}.

\sni
\goodbreak
\emph {4.}
On a $d_1 = \cmatrix {x^2y & xy^3 &  x &  yz}$ et:

\snic {
d_2 = \bordercmatrix[\lbrack\rbrack] {
    &12_5  &13_4 &14_4  &23_5  &24_5  &34_3 \cr
1_3 &-y^2  &  -z &  -z  &   .  &   .  &   .\cr
2_4 &   x  &   . &   .  &  -z  &  -z  &   .\cr
3_2 &   .  & xy  &   .  & y^3  &   .  &  -y\cr
4_2 &   .  &  .  &  x^2 & .    & xy^2 &   x\cr}
}

\snii
En indice, on a mentionné le poids du vecteur de base. Ainsi, $e_{12}$
est de poids $5$ et son image par $d_2$ est $-y^2 e_1 + x e_2$ (\egmt de poids~5). 
Avec ces notations:

\snic {
d_3 = \bordercmatrix[\lbrack\rbrack] {
    &123_6 &124_6 &134_4 &234_5 \cr
12_5 &   z &   z  &   .  &  .\cr
13_4 &-y^2 &   .  &   1  &  .\cr
14_4 &   . & -y^2 &  -1  &   .\cr
23_5 &   x &   .  &  .   &  1\cr
24_5 &   . &   x  &  .   & -1\cr
34_3 &  .  &   .  &  x   &  y^2\cr}
\qquad
d_4 = \bordercmatrix[\lbrack\rbrack] {
      &1234_6 \cr
123_6 &  -1\cr
124_6 &  1\cr
134_4 & -y^2\cr
234_5 &  x\cr}
}

\snii
On utilise la \rsn de Taylor pour déterminer la \rsn (graduée) minimale de
$\um$.  Toute colonne de $d_3$ est une \syzy entre les colonnes de
$d_2$. Utilisons les deux dernières colonnes de $d_3$ qui contiennent $\pm
1$. En notant $C_1, \ldots, C_6$ les 6 colonnes de $d_2$, on a $C_2 - C_3 +
xC_6 = 0$ (fournie par la colonne 3 de $d_3$) et $C_4 - C_5 + y^2C_6 = 0$ (via
la dernière colonne de $d_3$). Dans $d_2$, on peut donc supprimer par exemple $C_3,
C_5$ pour obtenir $d'_2$ et actualiser en conséquence $d_3$ pour obtenir
$d'_3$. Ce qui donne une \rsn libre minimale (graduée) de $\um$ de longueur
3 ($d_1$ est inchangée):

\snic {
d'_2 = \bordercmatrix[\lbrack\rbrack] {
    &12_5 &13_4  &23_5 &34_3 \cr
1_3 &-y^2 &  -z  &  .  &.\cr
2_4 &   x &   .  & -z  &.\cr
3_2 &   . & xy   &y^3  &-y\cr
4_2 &   . &   .  &  .  &x\cr
}
\quad
d'_3 = \cmatrix {
   z\cr
-y^2\cr
   x\cr
   0\cr
}
}

\prob{exoPolynomialSyzygies} \emph{(Quand les \moms dominants sont 
premiers entre eux)}

\emph {1a.}
Soient $f_i = m_i - r_i$ avec $r_i \prec m_i$; alors

\snic{m_jf_i - m_if_j = (f_j+r_j)f_i - (f_i+r_i)f_j = r_jf_i - r_if_j.}

\snii\emph {1b.}
Par hypothèse $m'_im_i = m'_jm_j$; comme $\pgcd(m_i, m_j) = 1$, on a $m_i
\mid m'_j$.  On dispose donc d'un \mom $q$ défini par $q = m'_j/m_i =
m'_i/m_j$.  On écrit alors

\snic {
m'_if_i - m'_jf_j = q(m_jf_i - m_if_j) = q(r_jf_i - r_if_j) = qr_jf_i - qr_if_j
}

\snii
Et l'on a $qr_jf_i \prec qm_jm_i = m'_im_i = m$; de même $qr_if_j \prec m$.
\\
Bilan: $m'_if_i - m'_jf_j \in \fa_{\prec m}$.

\snii\emph {1c.}
On peut supposer $I = \{1,\ldots,k\}$; on écrit la transformation
d'Abel:

\snic {
\sum_{i \in I} a_im'_if_i = \sum_{j=1}^{k-1} 
s_j(m'_jf'_j - m'_{j+1}f'_{j+1})
}

\snii
D'après la question précédente $m'_jf'_j - m'_{j+1}f'_{j+1}$ appartient
à $\fa_{\prec m}$, et il en est de même de leur somme $\sum_{i \in I}
a_im'_if_i$.

\snii\emph {1d.}
On écrit $g = qf_i \preceq m$; $q$ est donc une $\gk$-combinaison de \moms
$m'$ tels que $m'm_i \preceq m$. Si $m_i \nedivi m$, on a $m'm_i \prec m$ donc
$g \prec m$. Si $m_i \divi m$, on écrit $m = m'm_i$; si $a \in \gk$ est le
\coe de $m'$ dans $q$, on a $q - am' \prec m'$.  Donc $g = qf_i = am'f_i +
(q - am')f_i$ avec $(q - am')f_i \prec m'm_i = m$.

\snii\emph {2.}
Soit $f \in \fa_{\preceq m}$. On écrit $f = \sum_i g_i$ avec
$g_i \in \gen {f_i}$ et $g_i \preceq m$. On coupe cette somme
en deux:

\snic {
f = \sum_{j \notin I} g_j + \sum_{i \in I} g_i 
\quad \hbox {avec} \quad
I = \sotq {i \in \{1,\ldots,s\}} {m_i \hbox { divise } m}
}

\snii
Si $m_j \nedivi m$, on a $g_j \prec m$ donc $\sum_{j \notin I} g_j \in
\fa_{\prec m}$. Pour les $i$ tels que $m_i \divi m$, on écrit $m = m'_im_i$ et
comme dans la question \emph {1d}, il y a un $a_i \in \gk$ tel \hbox{que $g_i -a_im'_if_i \in \fa_ {\prec m}$}.  Bilan:

\snic {
f = \hbox {un \elt de $\fa_{\prec m}$} + \sum_{i \in I} a_im'_i f_i 
}

\snii
On utilise maintenant le fait que $f \prec m$. Ceci implique 
$\sum_{i \in I} a_i = 0$. D'après la question \emph {1c} on
a~$\sum_{i \in I} a_im'_i f_i \in \fa_{\prec m}$ donc 
$f \in \fa_{\prec m}$.

\snii
Soit $f \in S \cap \fa_{\preceq m}$, donc $f = \sum_i g_i$
avec~$g_i \in \gen {f_i}$ et~$g_i \preceq m$, a fortiori~$f \preceq m$. 
Si~$m_i \nedivi m$ pour chaque~$i$, alors~$g_i \prec m$ donc
$f \in \fa_{\prec m}$. Si~$m_i \mid m$ pour un indice~$i$ i.e.
si~$m \in \gen {m_1, \ldots, m_s}$, alors la composante de~$f$
sur~$m$ est nulle car~$f \in S$. Comme~$f \preceq m$, on a
$f \prec m$; d'après le début de cette question, on en
déduit que~$f \in \fa_{\prec m}$.

\snii
On a $S \cap \fa_{\preceq m} \subset \fa_{\prec m}$; et pour tout~$f \in
\fa_{\prec m}$, il existe un \mom $m' \prec m$ tel \hbox{que $f \in \fa_{\preceq
m'}$}.  On en déduit au bout d'un nombre fini d'étapes que~$f = 0$, \hbox{i.e.  $S \cap \fa_{\preceq m} = \{0\}$}. Comme~$\fa$ est la réunion des~$\fa_{\preceq
m}$, on a bien~$S \cap \fa = \{0\}$.

\snii\emph {3.}
Il s'agit d'une technique de division d'un \pol par le \sys $(f_1, \ldots, f_s)$
au sens des \bdgs. On utilise le fait que si~$f \prec m$ pour un \mom $m$,
alors~$f \preceq m'$ pour un \mom~$m' \prec m$.  Supposons le résultat à
montrer vrai pour les \pols $\prec m$ et montrons le pour~$f\preceq m$. On
écrit~$f = am + g$ avec~$a \in \gk$ et~$g \prec m$.  Si~$m \in \gen {m_1,
\ldots, m_s}$, alors~$m = m'm_i$ pour un~$i$ et on remplace~$f$ par~$f -
am'f_i \prec m$.  Si~$m \notin \gen {m_1, \ldots, m_s}$, alors~$m \in S$ a
fortiori~$am \in S$, et on remplace~$f$ par~$f - am \prec m$.

\snii\emph {4.}
Soit $E$ le sous-module de $\kuX^s$ engendré par les~$f_i\vep_j -
f_j\vep_i$. Soit~$m$ un \mom, des \pols $u_i$ tels que
$\sum_i u_i f_i = 0$ et~$u_i f_i \preceq m$. Il suffit de montrer
qu'il existe des \pols $v_i$ tels que

\snic {
\sum_i u_i\vep_i = \hbox {un \elt de $E$} + \sum_i v_i\vep_i
\hbox { avec } v_if_i \prec m
}

\snii
Soit $I = \sotq {i \in \{1,\ldots,s\}} {m_i \hbox { divise } m}$.
Pour~$i \notin I$, on a~$u_if_i \prec m$. Pour~$i \in I$, on
définit le \mom $m'_i = m/m_i$; si~$a_i$ est la composante de~$u_i$
sur~$m'_i$, on a~$(u_i - a_im'_i) \prec m'_i$. Alors
$\sum_{i\in I} a_im_im'_i = 0$ i.e.~$\sum_{i\in I} a_i = 0$.
En suivant la solution de l'exercice \ref {exoMonomialSyzygies},
on pose, pour~$i,j \in I$:

\snic {
q_{ij} = m'_i/m_j = m'_j/m_i = m/(m_im_j)
}

\snii
Il existe une matrice antisymétrique $(a_{ij})_{I \times I}$ à \coes
dans $\gk$ telle que $a_i = \sum_{j \in J} a_{ij}$ de sorte que

\snic {
\sum_{i\in I} a_im'_i\vep_i = \sum_{i \in I} 
\bigl(\sum_{j \in I} a_{ij} q_{ij} \und {m_j}\bigr) \vep_i
\qquad\quad \hbox {(ici $m_{ij} = m_j/(m_i \vi m_j) = m_j$)}
}

\snii 
Pour $i \in I$, on définit $w_i = \sum_{j \in I} a_{ij} q_{ij} \und {f_j}$
de sorte que $\sum_{i\in I} w_i\vep_i \in E$ et $w_if_i \preceq q_{ij}m_jm_i = m$.
Enfin on pose:

\snic {
v_i = \cases {u_i-w_i& \hbox {si $i \in I$} \cr u_i& \hbox {sinon}}
}

\snii On a $v_if_i \prec m$ (car pour $i \in I$, la composante de $v_if_i$ sur
$m$ est $a_i - \sum_{j\in J} a_{ij} = 0$).  Et c'est gagné puisque:

\snic {
\sum_i u_i\vep_i = \sum_{i \in I} w_i\vep_i + \sum_i v_i\vep_i =
\hbox {un \elt de $E$} + \sum_i v_i\vep_i 
}

\snii\emph {5.}
Le fait que la suite de \moms $m_1, \ldots, m_s$ soit \ndze n'est pas
difficile et laissé à la lectrice. Montrons le pour $f_1, \ldots, f_s$. Il
suffit de voir que $uf_s \in \gen {f_1, \ldots, f_{s-1}}
\Rightarrow u \in \gen {f_1, \ldots, f_{s-1}}$.  Encore une fois on raisonne
par \recu sur l'ordre monomial en écrivant $u = am + v$ avec $v \prec m$. En
multipliant par $f_s$, on a donc un $w \prec mm_s$ tel que $amm_s + w \in \gen
{f_1, \ldots, f_{s-1}}$.  Comme $\gen {m_1, \ldots, m_{s-1}}$ est l'\id
initial de $\gen {f_1, \ldots, f_{s-1}}$, on a $amm_s \in \gen {m_1, \ldots,
m_{s-1}}$; la suite $m_1, \ldots, m_s$ étant \ndze, on a $am \in \gen {m_1,
\ldots, m_{s-1}}$, disons $am = qm_i$ avec $i < s$.  Alors $(u - qf_i) f_s
\in \gen {f_1, \ldots, f_{s-1}}$ avec $u - qf_i \prec m$; donc, par \recu,
$u - qf_i \in \gen {f_1, \ldots, f_{s-1}}$ d'où l'on tire $u \in \gen
{f_1, \ldots, f_{s-1}}$.

\snii\emph {6.}
Soit $\gA$ le $\gk[f_1, \ldots, f_s]$-module engendré par les \moms
de~$M$. Il faut d'abord montrer que $\kuX = \gA$. On procède par \recu sur
l'ordre monomial en supposant que tout \pol $g \in \kuX$ tel que $g \prec m$
est dans $\gA$ et en montrant que c'est vrai pour un \pol $f \in \kuX$ tel que
$f \preceq m$. On écrit $f = \sum_i u_if_i + r$ avec $u_if_i \preceq m$ et $r
\in S \subset \gA$. On a $u_i \prec m$ (car $m_i \ne 1$) donc $u_i \in \gA$
par \recu. Bilan: $f \in \gA$.

\snii
Pour $e = (e_1, \ldots, e_s) \in \NN^s$ notons $f^e = f_1^{e_1}\cdots
f_s^{e_s}$. On va montrer simultanément que les $m \in M$ forment une base
de $\kuX$ sur $\gk[f_1,\ldots, f_s]$ et que les $f_i$ sont $\gk$-\agqt
indépendants en prouvant, pour une famille $(a_{e,m})_{e\in\NN^s, m\in M}$
d'\elts de $\gk$ à support fini, l'implication:
$$
\sum_{e,m} a_{e,m} f^e\, m = \sum_m \bigl( \sum_e a_{e,m} f^e\bigr) m = 
\sum_e \bigl( \sum_m a_{e,m} m\bigr) f^e = 0 
\ \Rightarrow\  a_{e,m} = 0
\leqno (\star)
$$
La clef réside dans les différentes fa\c{c}ons de regrouper les sommes;
ceci est lié au fait d'écrire que les $m \in M$ constituent un \sgr de
$\kuX$ sur $\gk[f_1, \ldots, f_s]$:

\snic {
\begin {array} {rcl} 
\kuX &=& \sum_{m \in M} \gk[f_1, \ldots, f_s]m = \sum_{m\in M} 
\bigl(\sum_{e\in \NN^s} \gk f^e\bigr)m 
\\
&=&
\sum_{e\in \NN^s} f^e \sum_{m \in M} \gk m =
\sum_{e\in \NN^s} f^e S
\\
\end {array}
}

\snii En utilisant le fait que la suite $f_1, \ldots, f_s$ est \ndze et le fait
que $\gen {f_1, \ldots, f_s} \cap S = 0$, on se convainc que la dernière
somme est directe. Par exemple, avec $2$ \pols $f_1, f_2$ et des 
$s_{ij} \in S$, on veut:

\snic {
s_{00}\cdot 1 + s_{10}\cdot f_1 + s_{01}\cdot f_2 + 
s_{20}\cdot f_1^2 + s_{11}\cdot f_1f_2 + s_{02}\cdot f_2^2 = 0
\ \Rightarrow\ s_{ij} = 0
}

\snii On a $s_{00} \in S \cap \gen {f_1,f_2}$ donc $s_{00} = 0$. Ensuite,
on raisonne modulo $f_1$, et en utilisant le fait que $f_2$ est
\ndz modulo $f_1$, il vient $s_{01} + s_{02} f_2 \equiv 0 \bmod f_1$ donc
$s_{01} = 0$ puis $s_{02} = 0$. On peut alors simplifier par $f_1$ et ainsi de
suite.

\snii
Bilan: on a bien justifié $(\star)$, ce qui prouve le résultat escompté.

\snii\emph {7.}
On a $a^2Y^3 = Yf_1 + (aY-X)f_2$. Si $Y^3 \in \fa := \gen {f_1,f_2}$, comme
les \pols sont \hmgs, on a $Y^3 = (\varphi X + \beta Y) f_1 + (\gamma X +
\delta Y) f_2$ avec $\varphi, \beta, \gamma, \delta \in \gk$.  L'examen de la
composante sur $Y^3$ donne $1 = \delta a$.

\snii 
Supposons $a$ \ndz. Si $\fa$ est facteur direct, $\gk[X,Y] = \fa \oplus S$,
alors $Y^3 = f + r$ avec $f \in \fa$ et $r \in S$. Comme $a^2Y^3 \in \fa$, on
a $a^2r = 0$ puis $r = 0$; donc $Y^3 \in \fa$ et $a$ \iv. Réciproquement,
si $a$ est \iv, on considère $f_1$, $a^{-1} f_2 = Y^2 + a^{-1} XY$ et
l'ordre lexicographique avec $X \prec Y$: les \moms dominants sont $X^2$,
$Y^2$, premiers entre eux et on peut appliquer l'étude précédente.

\prob{exoResolutionPfafienne} \emph{(Résolution pfaffienne)}

\emph {1.}
On se met en situation \gnq en utilisant le corps des fractions rationnelles
sur~$\QQ$ à~$n(n-1)/2$ \idtrs $x_{ij}$ avec~$1 \le i < j \le n$. On a alors une
matrice \iv $P$ telle~$X = \tra {P} X' P$ où:

\snic {
X' = \cmatrix {
0_{m} & \Id_{m} & 0_{m, 1} \cr
-\Id_{m} & 0_{m} & 0_{m, 1} \cr
0_{1, m}  & 0_{1, m} & 0_{1} \cr}
\hbox{ avec } n = 2m+1
}

\snii
La matrice~$\wi {X'}$ est la matrice~$n \times n$ ayant un seul coefficient non
nul, égal à 1, en position~$(n,n)$; quant à~$Q'$, c'est
$[\,0,\,\ldots,\,0,\,1\,]$ et l'on a bien~$\wi {X'} = \tra{Q'}Q'$.  Une fois
obtenue l'\idt $\wi X = \tra{Q}Q$, on la multiplie par~$X$ à droite, on
obtient les \egts $\tra{Q}(QX) = \wi X X = \det(X)\In = 0$.  Comme chacune des composantes
de la matrice colonne~$\tra{Q}$ est le pfaffien \gnq, c'est un
\pol \ndz; on en déduit~$QX = 0$. 
On a~$\cD_{n-1}(X) = \cD_1(\wi X) = \cD_1(Q)^2$.

\snii\emph {2a.}
On a~$\cD_n(X) = \gen {\det(X)} = 0$ et~$\cD_{n-1}(X) = \cD_1(Q)^2$ est un \id
fidèle.

\snii\emph {2b.}
Pour l'ordre gradué lexicographique inversé, les termes dominants
de $p_1$, $p_2$, $p_3$ sont, avec $n = 2q - 1$:

\snic {
-x_{2,n}x_{3,n-1}\cdots x_{q,q+1}, \qquad
\pm x_{1,n}x_{2,n-1}\cdots x_{q-1,q+1}, \qquad
-x_{1,n-1}x_{2,n-2}\cdots x_{q-1,q}
}

\snii
Ils sont étrangers deux à deux; d'après le problème
\ref {exoPolynomialSyzygies}, la suite $(p_1, p_2, p_3)$
est \ndze (ainsi que toute permutation de $(p_1, p_2, p_3)$).

\snii\emph {2c.}
D'après l'exercice \ref{exoExponentiationRegSequence}, la suite $(p_1^{2},p_2^{2},p_3^{2})$ est \ndze.\\
Et elle est contenue   dans $\cD_1(Q)^2=\cD_{n-1}(X)$.

\snii\emph {2d.}
Les matrices $u_1$, $u_2$, $u_3$ sont stables avec $\rgst(u_3) = \rg(L_3)=1$,
puis

\snic {
\rgst(u_2) + \rgst(u_3) = \rg(L_2) = n, \quad\hbox{et }
\rgst(u_1) + \rgst(u_2) = \rg(L_1) = n.
}

\snii
Notons $r_k = \rgst(u_k)$; alors $\Gr(\cD_{r_k}(u_k)) \ge 3$ a fortiori
$\Gr(\cD_{r_k}(u_k)) \ge k$. On peut alors appliquer le résultat de
Northcott (proposition \ref{cor3thABH1}).



\Biblio

\newpage \thispagestyle{empty}
\incrementeexosetprob

\chapter
{Déterminant de Cayley}\label{chapCayley}
\perso{compilé le \today}

\minitoc

\sibook{
\Intro

Ce chapitre 

\smallskip La section \ref{secCompGenEx}

\smallskip La section \ref{secMacRae}  \dots

\smallskip La section \ref{secCayleyAlgExt}  \dots

\smallskip La section \ref{secDetCayley}  \dots

\smallskip La section \ref{secStrMulRsl}  \dots

\smallskip La section \ref{secHilbBu}  \dots

\smallskip La section \ref{secDetCayModGr}  \dots

\smallskip La section \ref{secCayleyAnnexe}  \dots

\smallskip 
\dots

\fbox{\`A COMPL\'ETER}

\section{Complexes \gnqt exacts}\label{secCompGenEx}

Dans cette section \dots\,\dots\,\dots

Un peu d'histoire \dots

Cayley \dots\dots\dots

}

\section{Modules de MacRae}\label{secMacRae}

Dans cette section, on donne une légère \gnn de la théorie des
modules qui admettent un \emph{invariant de MacRae}.

La théorie développée par Northcott dans [FFR, chapitre III]  
puis étendue aux  \mlrsbs de rang nul peut être en effet encore un peu \gnee.
Ceci n'offre pas de difficulté à partir du moment où la notion
de profondeur des \itfs a été suffisamment développée.

\subsec{Définition et \prt de base}

Les modules de MacRae sont des \mpfs de torsion particuliers.

\begin{definition} \label{defiMacRae} ~
\begin{enumerate}
\item Un \itf $\fa=\gen{\an}$ est appelé un \emph{\iMR} 
s'il existe un \elt~$e\in\Reg\gA$  tel que $\fa \subseteq \gen {e}$ et $\Gr(\fa/e) \ge 2$.
On a alors $\Ann \,\fa = 0$ et $e$ est un pgcd fort de $\fa$ (voir ci-après). En particulier l'\id $\gen{e}$  ne dépend que de $\fa$. 
\item Un \Amo $E$ est appelé un \emph{\mMR} s'il est \pf et si son \idf $\ff_0=\cF_0(E)$ est un \iMR.%
\index{ideal@idéal!de MacRae}\index{MacRae!ideal@idéal de ---}%
\index{module!de MacRae}\index{MacRae!module de ---}%
\index{MacRae!invariant de ---}\index{invariant de MacRae}
Dans ce cas, si $e$ est le pgcd fort de $\ff_0$, l'\id $\gen{e}$ (qui ne dépend que de~$E$) est appelé
\emph{l'invariant de MacRae du module $E$}, et il est noté $\fG(E)$.
\item Un cas particulier simple est un module $E$ isomorphe au conoyau d'une matrice carrée dont le \deter est \ndz. Un tel module est appelé un \emph{module \elr}.
\end{enumerate}
\end{definition}
Justifions le fait que $e$ est pgcd fort de $\fa$ dans le point \emph{1}
On définit $q_i=a_i/e$ pour $i\in\lrbn$.
On doit montrer, pour $x$, $y \in \gA$ et $y$ \ndz que $y \mid xa_i$
pour tout $i$
entra\^ine $y \mid xe$. On a des $b_i$ tels que $xa_i = yb_i$ i.e.
$xq_i e = yb_i$; on en déduit que $yb_iq_j= yb_jq_i\, (= xeq_iq_j)$ et 
comme $y$ est \ndz, $b_iq_j = b_jq_i$. Puisque $\Gr(q_1,\ldots,q_n) \ge 2$,
on a \hbox{un $c \in \gA$} tel que $b_i = cq_i$. On reporte cela dans
l'\egt $xq_ie = yb_i$, ce qui donne~\hbox{$(xe-cy)q_i = 0$}, d'où $xe = cy$.
Et l'on a bien montré que $y \mid xe$.

\medskip 
\rem Lorsque l'anneau est intègre, ces \dfns se simplifient: il suffit pour l'\id ($\fa$ ou $\ff_0$) de demander qu'il admette un pgcd fort. En effet sur un anneau intègre, un \itf qui admet $1$ pour pgcd fort est de \prof $\geq 2$ (lemme \ref{lemGCDFORT}).
\eoe

\hum{En fait vu le lemme \ref{lemGCDFORT}, la \dfn peut être simplifiée
chaque fois que l'anneau $\gA$ satisfait la \prt suivante: tout \id fidèle
contient un \elt \ndz.}  

\begin{theorem} \label{thMacRae} \emph{(Principe \lgb pour les \mMRs)}
\begin{enumerate}
\item 
\begin{enumerate}
\item Un \itf $\fa$ est de MacRae \ssi il existe un \elt \ndz $e$ et une suite $(\sn)$ de profondeur~\hbox{$\geq 2$}, telle qu'après \lon en chacun des $s_i$, l'\id $\fa$ est  égal à~$\gen{e}$.
\item Un module \pf $E$ est de MacRae et d'invariant $\fG(E)=\gen{e}$ \ssi il existe 
une suite $(\sn)$ de profondeur~\hbox{$\geq 2$}, telle qu'après \lon en chacun des $s_i$, le module
devient \elr \hbox{avec $\cF_0(E)=\gen{e}$}.
\end{enumerate}
\item 
\begin{enumerate}
\item Soit $\fa$ un \itf qui admet un 
pgcd fort, et $(\sn)$ 
une suite de profondeur~\hbox{$\geq 2$}. Alors l'\id est de MacRae \ssi il  est de MacRae après \lon en chacun des~$s_i$. 
\item Soit un \mpf $E$ dont le premier \idf admet un 
pgcd fort, et $(\sn)$ 
une suite de profondeur~\hbox{$\geq 2$}. Alors le module est de MacRae \ssi il  est de MacRae après \lon en chacun des $s_i$. 
\end{enumerate}
\end{enumerate}

\end{theorem}
%
\begin{proof} 
\emph{1a.} \emph{La condition est suffisante.} Après \lon en chacun des $s_i$, l'\id~$\fa$ admet le pgcd fort $e$. Par
 la proposition~\ref{propProf2div}, $\fa$ admet le
pgcd fort~$e$. Enfin l'\id $\fa/e$ est de profondeur $\geq 2$ après 
\lon en chacun des~$s_i$ (en fait il devient égal à $\gen{1}$), et l'on conclut par le \plgref{plccProfondeur} pour la profondeur.

\emph{1a.} \emph{La condition est \ncr.} Si $\fa=(\an)$ et si $e$ est un pgcd fort de~$\fa$, on pose $s_i=a_i/e$. Alors la suite $(\sn)$ est de profondeur~$2$, et \hbox{sur $\gA[\fraC1 {s_i}]$}, on a $\fa=e\gen{\sn}=\gen{e}$.

\emph{1b.} \emph{La condition est suffisante.} C'est le point \emph{1a} appliqué à l'\idf~$\ff_0$.

\emph{1b.} \emph{La condition est \ncr.} 
Soit $A\in\MM_{r,m}(\gA)$ une \mpn de $E$ et notons~\hbox{$\mu_\alpha=a_\alpha e$} un mineur d'ordre maximum. Notons $A_\alpha$ la matrice carrée correspondante. Il suffit de montrer que sur l'anneau $\gA[1/a_\alpha]$, la matrice $A_\alpha$ est une \mpn de $E$ (pour le même \sgr). Autrement dit
que chaque colonne de $A$ est une \coli des colonnes de $A_\alpha$. Or,
puisque tous les mineurs maximaux de $A$ sont multiples de $e$ et que $e$ est \ndz, cela résulte de la formule de Cramer usuelle qui donne une \coli \gnq entre les colonnes d'une matrice de format $r\times (r+1)$.

\emph{2a.} Le raisonnement au point \emph{1a} pour montrer que la condition est suffisante fonctionne ici aussi. 

Le reste est clair.  
\end{proof}

Le \plg précédent peut être comparé au \plg pour les \mptfs:
le point \emph{1b} avec les modules \elrs correspond alors pour les \mptfs à la possibilité d'obtenir des \lons toutes libres.

\hum{Peut-être $e$ n'a pas besoin d'être donné dans le point \emph{1a}?

Je n'y crois guère. Un contre-exemple serait bienvenu pour éclairer
la situation!

Le \pb est le suivant. On a un \id $\fa$ de profondeur $\geq 1$ qui devient principal, engendré par $e_i$ après \lon en $s_i$, et la suite des $s_i$
est de profondeur $\geq 2$. Le choix des $e_i$ est un peu arbitraire, car il est à une unité près de $\gA_i=\gA[1/s_i]$.
On aimerait fabriquer un $e\in\gA$ tel que
$\gen{e}=_{\gA_i}\gen{e_i}$ pour chaque $i$.
On a $\gen{e_i}=\gen{e_j}$ sur l'anneau $\gA_{ij}=\gA[1/(s_is_j)]$. On a donc un
\elt $u_{ij}\in\gA_{ij}\eti$ avec $e_i=_{\gA_{ij}}u_{ij}e_j$.
Pour recoller les $e_i$, vue la profondeur $2$ des $s_i$, il semble qu'il suffit 
d'avoir $u_{ij}u_{jk}=u_{ik}$ dans $\gA_{ijk}$ pour tous $i,j,k$.
On peut définir d'abord les $u_{ij}$ pour $i<j$, puis $u_{ji}=u_{ij}^{-1}$.
Mais je ne vois pas comment assurer $u_{ij}u_{jk}=u_{ik}$ si l'anneau n'est pas intègre.
}

\subsec{Suites exactes de \mMRs} 

Voici une petite précision apportée au point \emph{3} de la proposition \ref{FFRpropPfSex}.
\begin{lemma} \label{lemPresfinsex} \emph{(Matrices de \pn: sous-module et quotient)}\\
Soit $E\subseteq  F$ et $ G= F/E$ des \mpfs\\
Soit $(\ux)=(\xm)$ un \sgr de $F$ et $r\in\lrbm$ tel que
la suite $(\xr)$
est un \sgr de $E$. On note $(\uz)=(z_1,\dots,z_{m-r})$ la suite 
$(\ov{x_{r+1}},\dots,\ov{x_m})$ (dans $G$).  
\begin{enumerate}
\item $(\uz)$ est un \sgr de~$G$.
\item Si $
C=\blocs{1.5}{0}{.6}{.7}{$A_1$}{}{$A_2$}{}
\;\in\MM_{m,n}(\gA)
$ est une \mpn de $F$ pour $(\ux)$ avec $A_1\in\MM_{r,n}(\gA)$,
alors $A_2$ est une \mpn de $G$ pour $(\uz)$. 

\item Si $C_E\in\MM_{r,q}(\gA)$ est une \mpn de $E$ pour $(\xr)$
et $C_G\in\MM_{s,q'}(\gA)$ une \mpn de $G$ pour $(\uz)$ alors il
existe une matrice $B$ telle que la matrice 
$C_F=\blocs{1}{.9}{.6}{.8}{$C_E$}{$B$}{$0$}{$C_G$}$
est une \mpn de $F$ pour $(\ux)$. 
\end{enumerate}  
\end{lemma}
%
\begin{proof}\emph{1.} Clair.

\emph{2.} Il est clair que chaque colonne de $A_2$ est une \syzy \hbox{pour $(\ov{x_{k+1}},\dots,\ov{x_m})$}. 
Réciproquement on voit que toute \syzy \hbox{pour $(\ov{x_{k+1}},\dots,\ov{x_m})$}
fournit une \syzy \hbox{pour $(\xm)$}  (avec les mêmes \coes  pour $({x_{k+1}},\dots,{x_m})$), et donc est une \coli des colonnes de $C$. 
Ainsi la matrice~$A_2$
est une \mpn \hbox{pour $(\ov{x_{k+1}},\dots,\ov{x_m})$}.

\emph{3.} Voir le point~\emph{3} de la proposition \ref{FFRpropPfSex}.
\end{proof}
%

\begin{theorem} \label{thMacRae2} \emph{(Suites exactes courtes de \mMRs)}\\
Soit $0\to E\lora F\lora G\to 0$ une \seco de \mpfs
\begin{enumerate}
\item  Si $E$ et $G$ sont de MacRae, $F$ est de MacRae.
\item  Dans ce cas, on a
$\;\;\fG(E)\,\fG(G)=\fG(F)\;\;(*)$
\item  Si $F$ et $G$ sont de MacRae, $E$ est de MacRae. 
\end{enumerate}
\end{theorem}
%
\begin{proof}
\emph{1} et \emph{2.} Puisque le produit de deux \ids de profondeur $\geq 2$
est de profondeur $\geq 2$, après \lon en les \elts d'une suite $(\an)$ de profondeur $\geq 2$, les modules $E$ et $G$ son présentés par des matrices carrées~$C_E$ et $C_G$ avec $\gen{\det(C_E)}=\fG(E)$ et  $\gen{\det(C_G)}=\fG(G)$. Par 
le lemme \ref{lemPresfinsex}, point~\emph{3}, sur chaque anneau $\gA[1/a_i]$, le module  $F$ est \elr, présenté par une matrice~$C_F$ selon le format suivant
$$
C_F=\blocs{1}{.8}{1}{.8}{$C_E$}{$B$}{$0$}{$C_G$}\,,\,\hbox{ avec } \gen{\det(C_F)}=\fG(F)=\fG(E)\,\fG(G).
$$
Ceci montre que $F$ est de MacRae, et comme l'\egt des \idps $(*)$
est vérifiée dans chaque \lon, elle est vraie globalement (proposition~\ref{propProf2div}).

\emph{3.} On identifie $E$ à son image dans $F$ et $G$ à $F/E$.
On note $f$ un \gtr de $\fG(F)$ et  $g$ un \gtr de $\fG(G)$.
On choisit un \sgr~\hbox{$(\xm)$} de $F$ dont les premiers termes $(\xr)$ 
engendrent~$E$. On localise tout d'abord en les \elts d'une suite de profondeur $\geq 2$ de fa{ç}on à ce que
sur chaque anneau localisé, le module $F$ soit \elr. 
Sur un localisé, on écrit
une  \mpn \hbox{pour $(\xm)$} sous la \hbox{forme
$
C_F=\blocs{1.3}{0}{.6}{.7}{$A_1$}{}{$A_2$}{}
\;\in\MM_m(\gA)
$
},
où les $m-r$ dernières lignes forment la matrice~$A_2$.
La matrice~$A_2$
est une \mpn \hbox{pour $(\ov{x_{r+1}},\dots,\ov{x_m})$} (lemme \ref{lemPresfinsex}, point~\emph{2}).

 Soit alors $\alpha\in\cP_{m-r,m}$ et 
$\mu_\alpha=ga_\alpha$ le mineur maximal  de $A_2$ extrait sur les colonnes $\alpha$. Si l'on inverse $a_\alpha$, les colonnes restantes sont \colis des colonnes de la matrice extraite sur les colonnes $\alpha$ (voir la \dem du \thref{thMacRae}). Une manipulation \elr des colonnes de~$C_F$ sur l'anneau localisé ramène donc cette matrice à la forme
$$C'_F=\blocs{.6}{.7}{.6}{.7}{$U$}{$B$}{$0$}{$C$}\,,
$$
avec $C$ \mpn de $G$.
Montrons  que $U$ (sur cet anneau localisé) est bien une \mpn pour $E$, autrement dit que toute \coli de colonnes de~$C_F$ dont les dernières \coos sont nulles est une \coli des colonnes de $\blocs{.6}{0}{.6}{.7}{$U$}{}{$0$}{}$. En effet, dans une telle \coli, les \coes des $m-r$ dernières colonnes
sont \ncrt nuls, parce que $\det(C)$ est \ndz.
Ainsi dans ce localisé, $E$ est \elr et $\fG(E)\fG(G)=\fG(F)$.
\\
On conclut que $g$ divise $f$ par la proposition~\ref{propProf2div}. On note
$e=f/g$. 
\\
En fin de compte, on aura une suite  de profondeur $\geq 2$ telle que dans chaque anneau localisé, le module $E$ sera devenu \elr avec $\fG(E)=\gen{e}$. Ceci prouve que $E$ est un \mMR.     
\end{proof}
%

%
%
%
%
%
%
	
\medskip On obtient alors l'analogue multiplicatif suivant du \tho qui dit que la \cEP
d'un complexe exact de modules libres est nulle.
\begin{theorem} \label{thMacRae3} \emph{(Suites exactes longues de \mMRs)}\\
On considère une \sex  de \mMRs 

\snic{E_{n-1}\vvers{u_{n-1}}E_{n-2} \cdots\cdots \vvers{u_{1}} E_0\lora 0\,.}

\snii
Alors le noyau de chaque~$u_i$ est un \mMR. 
\\
En notant $E_n=\Ker u_{n-1}$, on obtient l'\egt
$$
\prod\nolimits_{i:0\leq 2i\leq n}\fG(E_{2i})=\prod\nolimits_{i:0\leq 2i+1\leq n}\fG(E_{2i+1}). 
$$
\end{theorem}
%
\begin{proof}
Tout d'abord le module $\Ker u_1=\Im u_2$ est un \mMR par le point \emph{3} du \tho précédent en raison de la \seco

\snic{0\lora \Ker u_1\vvers{j_1} E_1\vvers{u_1} E_0\lora 0,}

où $j_1$ est l'injection canonique.\\ 
Ensuite  $\Ker u_2=\Im u_3$ est un \mMR en raison de la \seco

\snic{0\lora \Ker u_2\vvers{j_2} E_2\vvers{u_2} \Im u_2\lora 0.}

On peut décomposer ainsi de proche en proche la \sex en \secos de \mMRs. 
\\
En notant $\;u_n:E_n=\Ker u_{n-1}\to E_{n-1}\;$ l'injection canonique, on obtient une
\sex longue de \mMRs

\smallskip 
\centerline{\fbox{$0\lora E_n\vvers{u_n}E_{n-1}\vvers{u_{n-1}} \cdots\cdots \vvers{u_{1}} E_0\lora 0$}.}

\smallskip 
 On conclut en utilisant le \tho précédent pour ces \secos.
\end{proof}


\bonbreak
\section{Algèbre extérieure d'un module libre}\label{secCayleyAlgExt}

\subsec{Dualité canonique}

L'\emph{algèbre extérieure} d'un \Amo $M$ est une \Alg associative,
que l'on \hbox{note  $\Vi M$}, équipée d'une \Ali $M\vers\varphi\Vi M$,
qui résout le \pb 
\uvl correspondant au diagramme suivant

\vspace{-1em}
\pnv{M}{\varphi}{\psi}{\Vi M}{\theta}{\gB}{}{\Alis avec $\psi(x)^2=0$ pour tout $x$}{morphismes d'\Algs associatives}

L'\alg $\Vi M$ peut être réalisée sous la forme
de la somme directe des modules puissances extérieures de $M$: 

\snic{\Vi M=\bigoplus_{k\in\NN}\Al k M,\quad \hbox{ avec }\Al0 M =\gA\hbox{ et }\Al1M=M.}

\snii
Dans le cas où $M$ est de type fini, cette somme directe est finie car $\Al r M=0$ dès que $r$ est strictement plus grand que le cardinal d'un \sgr
(en fait, dès que $\cF_r(M)=0$).

\hum{Il serait peut-être intéressant de donner aussi $\Vi M$ comme solution d'un \pb \uvl concernant certaines \algs graduées (associatives alternées? quelle est la terminologie?).

Dans ce cas on précise seulement que le morphisme $\psi$ est de degré $0$
avec $M$ de degré $1$. Il n'y a plus besoin de préciser que $\psi(x)^{2}=0$ pour tout~$x$. 

Sans doute on peut \gnr en rempla{ç}ant $M$ par un \Amo gradué en degrés $>0$ arbitraire?}

La dualité canonique entre $M$ et  $M\sta$ (le dual de $M$), donnée par

\snic{(\alpha,u)\longmapsto \scp{\alpha} {u}=\alpha(u),}

\snii
se prolonge en une dualité naturelle entre $\Al k M$ et $\Al k M\sta$ 
définie par l'\egt suivante
$$
\scp{ \alpha_1\vi \cdots \vi \alpha_k} {u_1\vi \cdots \vi u_k}=\det\big((\alpha_i(u_j))_{i,j\in\lrbk}\big).
$$ 
Cette dualité s'étend à son tour en une dualité naturelle entre
$\Vi M$ \hbox{et $\Vi M\sta$} pour laquelle \hbox{on a $\scP{\,\Al k
M\sta}{\Al j M}=0$\,} lorsque $k\neq j$.

\subsubsec{Cas où $M=\Ae n$}

Dans le cas où $M=\Ae n$, notons $(e_1,\dots,e_n)$ la base canonique. On peut décrire directement l'\alg extérieure $\Vi M$ comme suit: $\Vi M$ est un module libre ayant pour base les  $e_I$, où $I$ décrit l'ensemble $\cP_n$ des suites finies strictement croissantes dans $\lrbn$. On identifie $e_\emptyset$ avec $1_\gA$, $e_{\so i}$
avec~$e_i$ (de sorte que $\gA$ et $ M$ sont des sous-modules de $\Vi M$), et le produit~\hbox{$e_I\vi e_J$} est défini par\label{casouM=An}

\snic{e_I\vi e_J=\formule{ 0 \hbox{ si } I\cap J\neq \emptyset\\[1mm]
\vep_{I,J} \,e_K\hbox{ sinon }}}

\snii où $K$ est la liste qui énumère en ordre croissant $I\cup J$,
et $\vep_{I,J}=(-1)^r$ où $r$ est le nombre de couples 
$(i,j)\in I\times J$
tels que $i>j$.

Dans le cas où $M=\Ae n$ on identifie une fois pour toutes $M$ et $M\sta$
en identifiant la base duale de la base canonique $(e_1,\dots,e_n)$ à la base canonique elle-même. 

La dualité naturelle entre $\Vi M$ et $\Vi M\sta$ définie précédemment devient alors un \emph{produit scalaire sur  $\Vi M$} qui admet la base des $e_I$
comme base orthonormée. 
Ce produit scalaire établit donc un \iso entre 
 $\Vi M$ et $\Vi M\sta$ (\iso naturel dans la mesure où la base naturelle de $M$ est fixée), avec en particulier un \iso naturel entre
$\Al k M$ et $\Al k M\sta$ pour chaque entier $k$.

\rdb
Nous introduisons la notation suivante: pour une matrice $U\in \Ae{n\times k}$, 
dont les vecteurs colonnes sont $u_1,$ \dots, $u_k\in \Ae n$  on note

\fnic{\pex U \eqdefi u_1\vi\dots\vi u_k\in \Vi^{k}\Ae n,}\label{NOTA[[U]]}

\snii et l'on a alors la formule compacte suivante pour le produit
scalaire:

\snic{\scP{\,\pex U} {\pex V\,}\, =\, \det(\tra U\,V) \;\;\hbox{ pour } U \hbox{ et } V\in \Ae{n\times k}.}

\subsec{Un \iso de Hodge}\label{IsoHodge}

Pour $\bmx\in \Vi^{n}\Ae n$  on définit
\gui{le \deter} $\dt\bmx\in\gA$ par la formule
 
\snic{\bmx\,=\, \dt\bmx \;e_1\vi\dots\vi e_n.}

\snii
On définit l'\emph{\iso de Hodge à droite} 
comme suit:%
\index{isomorphisme de Hodge}
$$
\formule{\Vi^{p}\Ae n\vvers\star \Vi^{q}\Ae n,\\[1mm]
\bmx\longmapsto\bmx\sta=\som_{J\in\cP_{q,n}} \dt{\bmx\vi e_J}\, e_J} 
\quad \fbox{pour  $p+q=n$}. 
$$
On utilisera \egmt, au lieu de $\bmx\sta$, la notation $\Hd(\bmx)$ plus
précise car elle mentionne le côté droit (l'isomorphisme de Hodge droit
est lié au produit intérieur droit $\intd$, cf. plus loin \paref{prodintdroit}). Cette notation
$\Hd(\bmx)$ peut aussi être utilisée comme notation de secours
typographique.

Par \dfn, on a donc
$$
\framebox [1.1\width][c]{$\scp {\bmx\sta}{e_I} = \dt{\bmx \vi e_I}$}
$$
En particulier pour $\bmx = e_J$ ($J$ de cardinal $p$), en notant 
$\ov J=\lrbn\setminus J$, on voit que $e_J\sta = \pm e_{\ov J}$
et comme $\scp {e_J\sta}{e_{\ov J}} = \dt{e_J \vi e_{\ov J}}$, on a
$$
e_J\sta = \dt{e_J \vi e_{\ov J}}\, e_{\ov J} = \vep_{J,\ov J}\,e_{\ov J}
\quad\hbox {et} \quad
e_J\vi e_J\sta = e_{\lrbn}.
$$
On retiendra donc que l'on peut compléter $e_J$ à droite avec $e_J\sta$
pour obtenir le $n$-vecteur $e_{\lrbn}$:

\snic {
\framebox [1.1\width][c]{$e_J\vi e_J\sta = e_{\lrbn}$}.
}

\snii 
Cette dernière formule, qui, avec la formule $e_I\vi e_J\sta = 0$ si $I\neq
J$, aurait pu servir aussi de \dfn pour l'\iso de Hodge, permet de comprendre
pourquoi nous l'avons qualifié d'\iso \gui{à droite} (une version \gui{à
  gauche} aurait utilisé la formule \smq obtenue en permutant~$e_J$ et
${e_J}\sta$). On voit aussi que cet \iso n'est pas spécifiquement attaché
à la base naturelle de $\Ae n$, mais seulement à une base \gui{naturelle}
de $\Vi^{n}\Ae n$.

\begin{proposition} \label{propDualiteHodge}
Avec $p+q=n$,   $\bmu_1$, $\bmu_2\in\Vi^{p}\Ae n$ et $\bmv\in\Vi^{q}\Ae n$, on a les formules de dualité suivantes:
\begin{enumerate}
\item \qquad \qquad $\scp {{\bmu_1}\sta} \bmv    =     \dt{\bmu_1\vi\bmv},$
\item \qquad \qquad $\scp {\bmu_1} {\bmu_2}    =   \scp {{\bmu_1}\sta}{{\bmu_2}\sta},$
\item \qquad \qquad $  \scp {\bmu_1} {\bmu_2}    = \dt{\bmu_1\vi{\bmu_2}\sta}.$
\end{enumerate}
En outre les formules 1. et 3. caractérisent l'\iso de Hodge.
\end{proposition}
%
\begin{proof}\emph{1.} Par \dfn, $\scp{\bmu\sta}{e_I}=\dt{\bmu\vi e_I}$
lorsque $I\in\cP_{q,n}$.
\\
\emph{2.} Vérification \imde sur les bases.
\\
\emph{3.} Résulte des points \emph{1} et~\emph{2.}
\end{proof}
%

\begin{corollary} \label{corpropDualiteHodge}
Avec $p+q=n$ et $A\in\Ae{q\times p}$, on a la formule de dualité
$$
\pexmat {\I_p\cr A}\sta = \pexmat{-\tra A\cr \I_q}.
$$
Plus \gnlt, avec $U_1\in\gA^{p\times p}$,  $U_2\in\gA^{q\times p}$ et $d=\det(U_1)$
$$
d^{q-1}\pexmat {U_1\cr U_2}\sta = \pexmat{-\tra{\big(U_2\wi{U_1}\big)}  \cr d\I_q} .
$$
 
\end{corollary}
%
Notez dans la première formule que les colonnes de la deuxième matrice forment une base de l'\ort de 
l'espace des colonnes de la première.
\begin{proof} Voyons la première formule. Notons  $\bmu\sta$ et $\bmv$ les $q$-vecteurs du premier et du second membre.
On doit démontrer pour un $q$-vecteur arbitraire $\bmy$ que 
$\dt{\bmu\vi\bmy}=\scp{\bmv}{\bmy} $.  
On considère une matrice arbitraire $Y\in\Ae{n\times q}$ telle que
$\pex Y=\bmy$. On écrit $Y=\cmatrix{Y_1\cr Y_2}$ avec $Y_1\in\Ae{p\times q}$
et $Y_2\in\Ae{q\times q}$. Alors on obtient
$$
\dt{\bmu\vi\bmy}=\dmatrix{\I_p&Y_1\cr A& Y_2},\;\; \hbox{ et } \scp{\bmv}{\bmy}=
\abs{\big[\,{-A\;\I_p}\,\big]\,\cmatrix {Y_1\cr Y_2}}=\abs{Y_2-A\, Y_1}.
$$
Le premier \deter se ramène au dernier par une \mlr 
par blocs de la matrice correspondante.\\
Enfin  la formule \gnle s'obtient (avec une matrice générique)
en se ramenant au cas particulier en multipliant $\cmatrix{U_1\cr U_2}$ par $U_1^{-1}$.  
\end{proof}
%
\subsec{\Tho de proportionnalité}

On garde la notation \fbox{$p+q=n$}.
On rappelle que deux vecteurs $\bmx$ et $\bmz$ dans \hbox{un \Amo} libre  sont dits \emph{proportionnels} lorsque $\bmx\vi\bmz=0$.

\begin{theorem} \label{thOrtho} \emph{(\Tho de proportionnalité)}\\
Soit $(u_1,\dots,u_p)$ et $(v_1,\dots,v_q)$ deux suites dans $\Ae n$,
$U$ et $V$ les matrices correspondantes dans $\Ae{n\times p}$ et $\Ae{n\times q}$, 
$\bmu=\pex U$ et $\bmv=\pex V$.
\begin{enumerate}
\item Les $q$-vecteurs $\bmu\sta$ et $\bmv$ sont proportionnels modulo
$\cD_1(\!{\tra U}\,V)$.  
\item Si $\scp{u_i}{v_j}=0$ pour tous $i$, $j$, alors les vecteurs $\bmu\sta$
et $\bmv$ sont proportionnels.
\end{enumerate}
\end{theorem}
%
\begin{proof} Notons que le point \emph{2} résulte du \emph{1}, et que le \emph{1} résulte du \emph{2} si ce dernier est démontré dans le cas \gnq où les \coes des matrices sont des \idtrs soumises aux seules contraintes $\tra U \,V =0$

\emph{\Demo du point 2. dans le cas où $\Gr\big(\cD_1(\bmu)\big)\geq 1$.}  (Mais on ne sait pas si cette hypothèse est vérifiée dans le cas \gnq évoqué ci-dessus). On considère un mineur maximal $\mu_1$ extrait de $U$, par exemple le premier,
et on se situe sur l'anneau $\gA[1/\mu_1]$.
Soit $U_1$ la matrice extraite de~$U$ sur les $p$ premières lignes,
de sorte que $\mu_1=\det(U_1)$.
En notant $U'=UU_1^{-1}=\cmatrix {\I_p\cr A}$, on obtient $\pex {U'}=\fraC 1{\mu_1}\pex {U}$. 
\'Ecrivons $V=\cmatrix{V_2\cr V_1}$ avec $V_1\in\gA^{q\times q}$.
\\
On a
$0_{p,q}=\tra U \,V=\tra {U_1}\,\tra {U'}\,V$, \hbox{donc $0_{p,q}=\tra {U'} V=\tra A V_1+V_2$}. \\
Compte tenu
du corolaire \ref{corpropDualiteHodge},
on doit montrer que les $q$-vecteurs $\pexmat{-\tra A\cr \I_q}$ et \hbox{$\pex V$} sont proportionnels. Or $\cmatrix{-\tra A\cr \I_q}V_1 =\cmatrix{-{\tra A}\, V_1\cr V_1}=\cmatrix{V_2\cr V_1}$.\\
Donc $ 
\det(V_1)\pexmat{-\tra A\cr \I_q}=\pex V.$\\
On a obtenu $\det(V_1)\bmu\sta=\det(U_1)\bmv$ sur l'anneau $\gA[1/\mu_1]$.
Si l'on note $\mu_{\beta}$ \hbox{et $\nu_\beta$} les \coos d'indice $\beta$ de $\bmu\sta$
et $\bmv$  on a donc pour tout $\beta$ l'\egt~\hbox{$\mu_1\nu_\beta=\mu_\beta \nu_1$}, 
et donc aussi 
$\mu_1 \mu_\beta \nu_\gamma=\nu_1\mu_\beta \mu_\gamma=\mu_1 \mu_\gamma \nu_\beta$, \hbox{puis $\bmu\sta \vi \bmv=0$} (toujours sur  $\gA[1/\mu_1]$). 
\\
Enfin on a le même résultat sur chaque anneau $\gA[1/\mu_\alpha]$
pour chaque \coo $\mu_\alpha$ de $\bmu\sta$, qui est de profondeur $\geq 1$. Donc $\bmu\sta \vi \bmv=0$ sur
$\gA$.

\smallskip \emph{Nous n'utiliserons le \tho de proportionnalité que dans le cas simple envisagé précédemment. La \dem dans le cas \gnl 
est cependant intéressante. Elle nécessite l'introduction de nouvelles notions
en algèbre extérieure et nous la renvoyons en annexe à la fin du chapitre.}
\end{proof}

\rem 
D'un point de vue \gmq le \tho de proportionnalité est intuitivement
évident. On se place dans un espace euclidien orienté. Les colonnes des
matrices $U$ et $V$ engendrent deux sous-espaces \orts de dimensions attendues
$p$ et $q$.  Le $p$-vecteur représentant l'espace engendré par $\Im(U)$
doit donc être proportionnel au $q$-vecteur représentant l'espace
engendré par $\Im(V)$ une fois qu'on les a ramenés dans un même espace
au moyen d'une dualité canonique basée sur le \deter et le produit
scalaire euclidien.\\
Notons aussi que pour établir le \tho de proportionnalité dans le cas \gnl, on pourrait partir du cas des corps et utiliser le \nst formel,  il suffirait alors de savoir que l'anneau \gnq
défini par l'\egt~\hbox{$\tra UV=0$} est un anneau réduit. 
\eoe

\section[Déterminant de Cayley pour certains complexes]{Déterminant de Cayley pour certains complexes}\label{secDetCayley}

\hum{si les entiers $r_1$, \dots, $r_m$ ne sont pas tous $>0$, cela semble sans intérêt réel.

si l'on prend l'option inverse (autoriser des $r_k$ nuls) 
les premiers \ids \caras sont tous égaux à $\gen{1}$ jusqu'au dernier
indice tel que $r_\ell=0$.

Seul le dernier complexe, des indices $\ell-1$ à $0$ compte.}
Dans cette section on considère,
une suite  $(r_0,r_1,\dots,r_m,r_{m+1})$ dans $\NN$, avec \fbox{$m\geq 2$ et  $r_{m+1}=0$}, et un complexe descendant de modules libres

\smallskip 
\centerline{\fbox{$L\ibu:\quad \quad 0 \to L_m \vvers{A_m}  L_{m-1} \vvvers{A_{m-1}}\;  \cdots \cdots \; \vvers{A_2}  L_1 \vvers{A_1} L_0$}, 
}

\smallskip 
avec \fbox{$L_k= \gA^{r_{k+1}+r_k}$},  
$k\in\lrb{0..m}$,  
(en particulier $L_m= \gA^{r_m}$ et~\hbox{$\chi(L\ibu)=r_0$}).
\\
Dans la suite, on considère les \idcas \fbox{$\fD_k=\cD_{r_k}(A_k)$}
et l'on fait les hypothèses suivantes.
\Grandcadre{$\Gr(\fD_1)\geq 1$, et $\Gr(\fD_k)\geq 2$ pour $k\in\lrb{2..m}$.}

\begin{definition} \label{defiCompCay}
Sous ces hypothèses, on dira que le complexe est un \emph{complexe de Cayley}.%
\index{Cayley!complexe de ---}\index{complexe!de Cayley}
 \end{definition}

Dans le \thref{thdetCay} on va définir les \ids de \fcn du complexe, qui permettent de donner une description plus précise des \ids $\fD_\ell$. 

Si en outre $\chi(L\ibu)=0$, on
définira le \deter de Cayley de $L\ibu$, \hbox{noté $\ffg(L\ibu)$}.
Dans ce cas le module $M=\Coker(A_1)$ est un \mMR et $\fG(M)=\gen{\ffg(L\ibu)}$.

\begin{fact} \label{factCay1}
Dans un complexe de Cayley, chaque matrice~$A_k$ est de rang stable $r_k$ (même si le complexe n'est pas exact). 
\end{fact}
%
\begin{proof}
L'\id $\cD_{r_k}(A_k)$ est fidèle par hypothèse. Montrons
\hbox{que $\rg(A_k)\leq r_k$}. La \dem fonctionne 
à peu près comme pour le cas d'un complexe exact. \\
Tout d'abord
$\rg(A_m) \leq  r_m$ est évident.
On remarque que
si l'on inverse un mineur maximal de $A_m$
on peut faire une modification \elr du complexe
qui le raccourcit d'un cran
et ne change pas les \ids \caras.
Donc $\rg(A_{m-1}) \leq  r_{m-1}$ après inversion d'\ecr
(les mineurs maximaux de $A_m$),
\hbox{donc $\rg(A_{m-1}) \leq  r_{m-1}$}.
Et ainsi de suite.
\end{proof}

Rappelons à titre de motivation que selon le \thref{thRangStProfRLF}, si le complexe est exact, chaque~$A_k$ est stable de rang stable $r_k$
et l'on \hbox{a $\Gr(\fD_k)\geq k$} pour tout $k\in\lrbm$. Donc la \rsn libre est un complexe de Cayley.
\\
Dans le cas où $\gA$ est un domaine de Bezout, on n'apprend
rien de nouveau concernant les modules de rang nul, mais c'est assez naturel puisque tout \mpf admet une \rsn libre de longueur $1$, et tout \itf \ndz est engendré par son pgcd fort. 
\\
Par contre, pour un anneau arbitraire, même les \rsns libres
de longueur $2$ des modules monogènes présentent quelques mystères, faisant l'objet du \tho de Hilbert-Burch, que nous présentons plus loin comme un cas particulier des
résultats établis dans cette section.  
 
\smallskip 
Dans la suite pour un vecteur $x$ d'un module libre (par exemple une matrice)
on note $\cD_1(x)$ l'\id engendré par ses \coos sur une base, et $\Gr(x)$
pour $\Gr(\cD_1(x))$.
\\
Dans cette section tous les modules et leurs puissances extérieures sont munis de leurs bases naturelles.
\\
Par exemple dans le lemme qui suit, à rapprocher du lemme \ref{lemSuitExStable}, il faut bien prêter attention aux notations. L'\ali $\Vi^{q}U$ est vue comme une matrice sur les bases naturelles de
$\Vi^{q} F$ et $\Vi^{q} E$, et les $q$-vecteurs~$\bmz$ 
et~$\bmx$ sont vus comme 
des vecteurs colonnes sur les bases naturelles. Ainsi l'hypothèse affirme
que la matrice~$\Vi^{q}U$ est égale au produit du vecteur colonne $\bmz$ par le vecteur ligne~$\tra \bmx$. 

\begin{lemma} \label{lempqrs}
On considère un complexe 
${\gA^{p+q}\vvers U  \gA^{q+r}\vvers V \gA^{r+s}.}$

On note $E=\gA^{p+q}$, $F=\gA^{q+r}$, $G=\gA^{r+s}$ et l'on suppose que  
\begin{enumerate}
%
%
\item  $\Vi^{q}U =\bmz \, \tra \bmx $, où $\,\bmz \in\Vi^{q} F$ et $\,\bmx \in\Vi^{q}E$,
\item  $\Gr (\bmz ) \geq 2$ et $\Gr (\bmx ) \geq 1$.
\end{enumerate}
On a alors les résultats suivants.
\begin{enumerate}\setcounter{enumi}{2}
\item On peut écrire $\Vi^{r}V=\bmy\, \tra {\,\bmz \sta}$, où $\bmy\in\Vi^{r} G$.
\item En outre $\Gr (\bmy) \geq\Gr\big(\cD_r(V)\big)$.
\end{enumerate}  
\end{lemma}
%
\begin{proof} \emph{4.} \'Evident puisque l'\egt du \emph{3} montre que $ \cD_r(V) \subseteq  \cD_1(\bmy)$.

\emph{3.} On regarde $U $ comme une liste de $p+q$ vecteurs $u_i$ de $F$
et $\tra V$ comme une liste de $r+s$ vecteurs $v_j$ de $F$. \\
Puisque $VU =0$,
les produits scalaires $\scp{u_i}{v_j}$ sont nuls. 
\\
Soient $I\in \cP_{q,p+q}$ et  $J\in \cP_{r,r+s}$. Si $U_I$ est la matrice
extraite de $U$ sur les colonnes indexées dans~$I$, et si 
$V_J$ est la matrice
extraite de $\tra V$ sur les colonnes indexées dans~$J$,
le \thref{thOrtho} s'applique et dit que ${\pex {U_I}}\sta$ et $\pex {V_J}$ sont proportionnels. Or $\pex {U_I}=x_I \bmz $, où $x_I$ est la \coo du vecteur ligne $\tra \bmx $ correspondant à la colonne de $\Vi^{q}U$ pour le vecteur 
de base~\hbox{$e_I\in \Vi^{q} E$}.
\\
Ainsi tous les $x_I {\,\bmz \sta}$ sont proportionnels à tous les $\pex {V_J}$. Comme  $\bmx $ est \ndz,~$ {\,\bmz \sta}$ est proportionnel à tous les $\pex {V_J}$. Comme $\Gr(\bmz )=\Gr( {\,\bmz \sta})\geq 2$, cela implique
pour chaque $J$ que l'on peut écrire $\pex {V_J}=y_J{\,\bmz \sta}$, ce qui donne l'écriture voulue de $\Vi^{r}V$.
\end{proof}

Comme corolaire on obtient la \dfn du \deter de Cayley 
du complexe $L\ibu$ (mais à l'époque de Cayley,
tous les anneaux étaient intègres et les \pbs de proportionnalité pouvaient sembler plus simples).

\begin{thdef} \label{thdetCay} \emph{(Complexe de Cayley: \ids de \fcn, \deter de Cayley)}~
On considère un complexe de Cayley $L\ibu$ (\dfn \ref{defiCompCay}). On note $M=\Coker(A_1)$ et l'on pose $\bmu_m=[\,1\,]\in\Ae{1\times 1}$.
 On a les résultats suivants.
\begin{enumerate}
\item Il existe un unique \sys de vecteurs colonnes%
\footnote{Comme expliqué avant le lemme \ref{lempqrs}, on regarde les $\Vi^{\ell}L_k$ comme munis de leurs bases naturelles et les  $\Vi^{r_{k}}A_k$ comme
des matrices sur les bases naturelles, et donc $\bmu_{k-1}$ est un vecteur colonne de taille $r_{k}+r_{k-1}\choose r_{k}$, tandis que $\tra {\,({\bmu_{k}}\sta})$ est un vecteur ligne de taille $r_{k+1}+r_{k}\choose r_{k}$.
La matrice $\Vi^{r_{k}}A_k$ est de rang stable $1$.}

\snic{\bmu_k\in\Vi^{r_{k+1}}L_{k},\;k\in\lrb{0..m-1}} 

tel que  pour tout $k\in\lrbm$

\centerline{\quad \fbox{$\Vi^{r_k}A_k=\bmu_{k-1} \tra {\,({\bmu_{k}}\sta})$}\,.}

On note \fbox{$\fB_k=\cD_1(\bmu_k)$} pour $k\in\lrb{0..m}$. \\ 
On obtient 
\fbox{$\fD_k=\fB_k\fB_{k-1}$} pour $k\in\lrbm$. Ces \ids $\fB_k$ sont appelés les \emph{\ids de \fcn du complexe $L\ibu$}.\index{complexe de Cayley!ideal@\ids de \fcn d'un ---}\index{ideaux de fac@\ids de \fcn!d'un complexe de Cayley} 
\item On obtient en particulier $\Gr(\fB_0)\geq 1$,  $\Gr(\fB_k)\geq 2$ pour $k\geq 1$, et 
$$\preskip.2em \postskip-.2em
\prod_{ k\in\lrb{0.. m}}\fB_k\;=\;\;{\fB_0\prod_{k:1<2k\leq m}\fD_{2k}\;=\prod_{k:1\leq 2k+1\leq m}\fD_{2k+1}}
$$

\item Lorsque $\chi(L\ibu)=0$, \cad lorsque $L_0=\Ae{r_1}$, on a ${\bmu_0}\in\Ae{1\times 1}$, son \coe, noté $\ffg(L\ibu)$ ou $\ffg$, est appelé le \emph{\deter de Cayley du complexe~$L\ibu$}. On notera $\fG$ ou $\fG(L\ibu)$
l'\id $\fB_0=\gen{\ffg}$.\index{Cayley!determin@\deter de ---}%
\index{determin@\deter!de Cayley} 
%
\begin{enumerate}
\item On a  $\Vi^{r_1}A_1=\bmu_0\, \tra {\,({\bmu_{1}}\sta})=\fG\, \tra {\,({\bmu_{1}}\sta})$.
\item Le \deter de Cayley $\ffg$ est un \elt \ndz. 
C'est le pgcd fort de l'\id $\fD_1=\cF_0(M)$, et l'on a $\Gr(\fD_1/\ffg)\geq 2$. 
\\
Autrement dit,  $M$   est un module de MacRae et son invariant de MacRae~$\fG(M)$
est égal à $\fG(L\ibu)$.
\end{enumerate}
\item Lorsque $\chi(L\ibu)>0$, si $\fB_0=\gen{g}$, le \gtr $g$ est \ndz, il est pgcd fort de l'\id $\fD_1=\cF_{r_0}(M)$, et l'on a $\Gr(\fD_1/g)\geq 2$.
\item Si l'on opère des changements de base sur les modules $L_k$ du complexe, et si les matrices sont transformées corrélativement, 
les \ids~$\fD_k$ et~$\fB_k$ ne changent pas, les vecteurs $\bmu_k$ sont prémultipliés par des matrices inversibles (dans le cas où $\chi(L\ibu)=0$,  $\ffg$ est donc multiplié par une unité).
\item Si l'on fait subir au complexe des modifications \elrs pour les~$A_k$ (voir le lemme
\ref{lemModifComplexe}), les \ids $\fD_k$ et~$\fB_k$ ne changent pas non plus.
\end{enumerate}
\end{thdef}
%
\begin{proof}
\emph{1} et \emph{2.} Application directe récurrente du lemme \ref{lempqrs}.

\emph{3a.} Cas particulier de $\emph{1.}$ 

\emph{3b.}  Résulte du point $\emph{2}$:  $\ffg$ est \ndz car $\Gr(\bmu_0)\geq 1$, 
par ailleurs $\Gr(\bmu_1)\geq 2$ implique que $1$ est pgcd fort de $\bmu_1$ 
(point \emph{4} du lemme \ref{lemGCDFORT}), ce qui implique que $\ffg$
est pgcd fort \hbox{de $\ffg\,\bmu_1$} (point \emph{2} du même lemme).

\emph{4.} Ce point peut être vu comme une \gnn du point \emph{3} Cela résulte de l'\egt $\fD_1=\fB_1\fB_0$ avec $\Gr(\fB_1)\geq 2$ et $\Gr(\fB_0)\geq 1$.

\emph{5.} Cela résulte du fait qu'un changement de bases avec matrices $P$ et $Q$ (à la source et au but) pour des modules libres  remplace la matrice $A$ d'une \ali $\varphi$
 par la matrice $Q^{-1}AP$ et la matrice $\Vi^{r}A$ de l'\ali $\Vi^{r}\varphi$ par  $(\Vi^{r}Q)^{-1}(\Vi^{r}A)(\Vi^{r}P)$. 
 Si les changements de base étaient tous de \deter $1$, dans le cas $r_0=0$
 le \deter de Cayley serait inchangé (pas seulement l'\id qu'il engendre).

\emph{6.} On sait déjà qu'une modification \elr ne change pas les \ids 
\caras $\fD_k$ (proposition \ref{propdefidecars}).\\
Il suffit de traiter la question dans la modification \elr décrite au point \emph{1}
du lemme \ref{lemModifComplexe}. Matriciellement c'est visualisé
comme suit:
$$
\wi{u_{k}} =\blocs{1.1}{.6}{.8}{0}{$u_{k}$}{$0$}{}{}\;,\;\;\;
\wi{u_{k+1}} =\blocs{1.3}{.6}{1.1}{.6}{$u_{k+1}$}{$0$}{$0$}{$\I_E$}\;, \;\;\;
\wi{u_{k+2}} =\blocs{0}{1.5}{1.3}{.6}{}{$u_{k+2}$}{}{$0$}\;.
$$
Si $E$ est libre de rang $s$, nous avons $\wi{r_{k+1}}=r_{k+1}+s$ et les autres $r_\ell$ sont inchangés.\\
Dans la nouvelle \rsn les $\wi{\bmu_\ell}$ sont inchangés pour $\ell>k+2$.
On constate ensuite que l'on obtient $\wi{\bmu_{k+2}}$ à partir de $\bmu_{k+2}$ en rajoutant des coordonnées toutes nulles sur les nouveaux vecteurs de base, car il en va de même pour la matrice $\Vi^{r_{k+2}}\wi{A_{k+2}}$.
La même constatation s'applique pour $\wi{\bmu_{k+1}}$. En effet, dans la matrice $\Vi^{r_{k+1}+s}\wi{A_{k+1}}$ toutes les nouvelles coordonnées sont nulles
(il faut ici tenir compte du fait \ref{factCay1}).  
Enfin, on peut prendre $\wi{\bmu_{k}}=\bmu_{k}$, car dans la matrice
$\Vi^{r_{k}}\wi{A_{k}}$ toutes les nouvelles coordonnées sont nulles.
Ensuite, les $\bmu_\ell$ pour $\ell<k$ sont manifestement inchangés.\\
NB: les adaptations \ncrs aux cas où $k=m$, ou $k=1$ sont faciles.
\end{proof}
\rems ~\\ 
1) On peut comprendre la convention $\bmu_{m}=[\,1\,]\in\Ae{1\times 1}$  en référence  au  \deter de la matrice vide, qui est égal à $1$, avec $\Ae{1\times 1}=\Vi^{0}L_m$ (ici, un \elt de $\Vi^{0}L_m$ doit être exprimé sous la forme d'un vecteur colonne).\\
2) Si l'on autorise  $m=1$, le \tho reste juste mais sans intérêt.\\ 
Par exemple dans le  cas où~$M$ est le quotient d'un module libre par un sous-module libre, on a pour la \rsn libre correspondante
$$\preskip.4em \postskip.4em 
\bmu_1=1,\; \bmu_0=\Vi^{r_1}\!A_1 \,\hbox{ et  }\,\fD_1=\cF_{r_0}(M)=\fB_0,
$$
mais aucun résultat distinct de l'hypothèse.
\eoe

\sibook{\exls
\hum{à remplir ...

ne pas donner uniquement des \rlfs de modules de rang $0$}
\eoe
}

\begin{proposition} \label{propidefacteds} \emph{(Idéaux de \fcn et \eds)}
\\
On se place dans le cadre du \thref{thdetCay} et l'on considère un changement d'anneau de base $\rho:\gA\to\gB$ où $\gB$ est plat sur $\gA$. Alors le complexe reste un complexe de Cayley, et les vecteurs $\bmu_\ell$
et \ids $\fB_\ell$
sont transformés en leurs images par $\rho$.
\end{proposition}

\section{Structure multiplicative des \rlfs}\label{secStrMulRsl}

\begin{thdef} \label{thdetCay2}~
\begin{enumerate}
\item Si  un \Amo  $M$ est \lrsb, le \thref{thdetCay} s'applique
à tout complexe de modules libres qui le résout.
En outre les \ids de \fcn $\fB_k$ ne dépendent que du module $M$.
On les appellera les \emph{\ids de \fcn du module $M$}, et on les notera
$\fB_k(M)$.

\end{enumerate}
En particulier, on obtient les résultats suivants. 
\begin{enumerate} \setcounter{enumi}{1}
\item Un module $M$  \lrsb de rang $0$ est un \mMR. Son invariant
de MacRae $\fG(M)$ est égal à $\fB_{0}(M)$.  
\item Si un \itf $\fa$ est \lrsb et  de rang $1$
(rappelons qu'il est nul ou de rang $1$ par le \tho de Vasconcelos~\ref{corVascon}), 
alors il admet un pgcd fort $g$ qui est \ndz, \hbox{et $\Gr_\gA(\fa/g)\geq 2$}.
\end{enumerate}
\end{thdef}
%
\begin{proof} Il suffit de montrer le point \emph{1.}\\
La première affirmation résulte clairement du \thref{thRangStProfRLF} (qui minore la \prof des \ids $\fD_k$) et du \thref{thdetCay}. \\
Pour la deuxième affirmation, on commence par le cas où l'anneau est local \dcd. Alors on peut ramener toute \rlf à une \rsn minimale par
des modifications \elrs du complexe, lesquelles ne changent pas les \ids de \fcn. On conclut en notant que la \rsn minimale est unique à \iso près.\\
On traite ensuite le cas \gnl. On veut montrer que deux \itfs sont égaux.
On peut utiliser le \plg de base. Et l'on applique la machinerie \lgbe à \ideps.  Ceci est légitime car les \ids de \fcn se comportent bien par \lon.
\end{proof}
Comme corolaire on obtient le \tho important qui suit.
\begin{theorem} \label{corthdetCay2} \emph{(Idéaux à deux \gtrs \lrsbs)}
\begin{enumerate}
\item Soit $\fa=\gen{a_1,a_2}$ non nul dans $\gA$. \Propeq
\begin{enumerate}
\item L'\id $\fa$ admet un pgcd $g$ \ndz et $\Gr(\fa/g)\geq 2$.
\item L'\id $\fa$ est \lrsb.
\item  L'\id $\fa$  est fidèle et l'on a une \sex

\snic{0\lora\gA\vvvvers{\tra[\,b_1\;b_2\,]} \Ae2\vvvvers{[\,a_1\;a_2\,]~}
\gen{a_1,a_2}\lora 0.} 
\end{enumerate}

\item Pour un anneau $\gA$ \propeq
\begin{enumerate}
\item Tout \id à deux \gtrs   est \lrsb.
\item L'anneau est un anneau à pgcd intègre.
\end{enumerate}
\end{enumerate}  
\end{theorem}
\rems 1)
Quand le point \emph{1} est satisfait, 
\begin{itemize}
\item l'\elt $g$ est un pgcd fort de $\fa$,
\item on peut prendre $(b_1,b_2)=(-a_2/g,a_1/g)$, 
\item  $\Gr(\fa)\geq 2$ \ssi $g\in\Ati$.
\end{itemize}

\smallskip 
2) Dans le point \emph{1}, l'hypothèse que $\fa$ est non nul est légèrement désagréable.
Si on la supprime, il faut reformuler le point \emph{b} en rajoutant
que si $\fa=0$ alors l'anneau est trivial.

3) Le point \emph{2} n'exclut pas le cas de l'anneau trivial. 
\eoe
\hum{Il n'y a pas de résultat du même style  pour les \ids à $3$ \gtrs. En particulier il y a des \ids à $3$ \gtrs qui admettent des \rlfs de longueur arbitrairment grande ne pouvant pas être raccourcies. On devrait en parler. Cf l'email de Vasconcelos citant un \tho de Bruns}

\begin{theorem} \label{thMcRaeMultSexa} \emph{(Voir [FFR, \tho 12 page 226])}\\
Soit une \sex $0\to E\lora F\lora G\to0$, où
$E$, $F$, $G$ sont des modules  \lrsbs de rang $0$. 
Alors $ \fG(G)\,\fG(E) = \fG(F)$.  
\end{theorem}
%
\begin{proof} 
Les modules sont de MacRae. On applique le \thref{thMacRae2}.
\end{proof}

Le  lemme suivant donne des éclaircissements concernant
ce  que signifie la factorisation d'une matrice $\Vi^{r}A$ lorsque $A$
est de rang $\leq r$. L'explication est que l'on peut unifier toutes les informations 
données par les \idts de Cramer \gnees
concernant les \rdls entre~\hbox{$r+1$} colonnes de la matrice.

\begin{lemma} \label{lemFacMat} \emph{(Ce que signifie la factorisation de la matrice $\Vi^{r}A$)}\\
Soit $A\in\MM_{m,n}(\gA)$ une matrice de rang  $\leq r<n$. On
sait que $\Vi^{r}A$ est de rang $\leq 1$.
Supposons
que $\Vi^{r}A=\bmu\bmv$ où $\bmu$ est un vecteur colonne et $\bmv$ est un vecteur ligne. Considérons pour simplifier les $r+1$ 
premières colonnes de~$A$, que nous notons $C_1$, \dots, 
$C_{r+1}$. Pour $k\in\lrb{1..r+1}$, le vecteur colonne\footnote{Le chapeau signifie que l'on omet le vecteur.} 

\snic{C_1\vi\cdots\vi\wh {C_k}\vi\cdots\vi C_{r+1}}

 correspond à une colonne $J_k$ de $\Vi^{r}A$  
et est donc égal à $\gamma_k\bmu$, où $\gamma_k$ est la \coo $J_k$ de $\bmv$. Supposons aussi que $\bmu$ est de profondeur $\geq 1$. Alors on a la \rdl
$$
\som_{k=1}^{r+1}(-1)^{k}\gamma_k C_k=0.
$$
On aura naturellement des \egts analogues pour tout autre choix de $r+1$ colonnes de $A$.
\end{lemma}
%
\begin{proof}
Considérons $r$ lignes arbitraires de  $A$, correspondant
à une \coo~$u_I$ de $\bmu$ (avec $I\in\cP_{r,m}$).
Comme $A$ est de rang $\leq r$, on a une relation de Cramer \gnee (voir 
  \cref{section II-5, formule (9)}).
$$
\som_{k=1}^{r+1}(-1)^{k}\,u_I\gamma_k\, C_k=0.
$$
Puisque les $u_I$ sont \cor, on obtient $\som_{k=1}^{r+1}(-1)^{k}\gamma_k C_k=0$.
\end{proof}

%
\begin{lemma} \label{lemResFinIdFac}
On considère une suite exacte 
$$
L_2\vvers{A_2}L_1\vvers{A_1}L_0,
$$
avec $L_2=\gA^{r_3+r_2}$, $L_1=\gA^{r_2+r_1}$,  $L_0=\gA^{r_1+r_0}$ et $A_1$ stable de rang stable~$r_1$. On suppose que $\Vi^{r_1}A_1=\bmu \bmv$ pour un vecteur colonne $\bmu$
et un vecteur ligne~$\bmv$. Si une \coo de $\bmv$ est \iv,
 $\Ker A_1$ est libre de rang $r_2$, $A_2$ est \lot simple
de rang $r_2$ et $\cD_{r_2}(A_2)=\gen{1}$. 
\end{lemma}
%
\begin{proof}
Supposons que la \coo \iv de $\bmv$ corresponde aux $r_1$ 
dernières colonnes de $A_1$. Le lemme \ref{lemFacMat} 
dit alors que chacune des $r_2$ premières colonnes est \coli
des $r_1$ dernières. 
En notant $A'_1$  la sous-matrice
extraite sur les $r_1$ dernières colonnes, on obtient
une \egt matricielle~\hbox{$A_1=A'_1\,D$} avec des matrices

\snic{A_1=\blocs{.9}{.7}{1.8}{0}{$F$}{$A'_1$}{}{}\,
\;\hbox{  et  }\;
D=\blocs{.9}{.7}{.7}{0}{$G$}{$\I_{r_1}$}{}{}\,.}

Comme $A'_1$ est injective, on a $\Ker A_1=\Ker D$
et l'on voit facilement \hbox{que  $\Ker D= \Im H$} avec
$H=\blocs{0}{.9}{.9}{.7}{}{$\I_{r_2}$}{}{$-G$}\,$. 
\\
Ainsi,
$\Ker A_1$ est libre de rang $r_2$, en somme directe avec le sous-module 
engendré par les $r_1$ derniers vecteurs de base.
Comme $\Ker A_1=\Im A_2$ on en déduit que $A_2$ est \lot simple,
de rang~$r_2$ avec $\cD_{r_2}(A_2)=\gen{1}$.
\end{proof}
%

\begin{theorem} \label{thResFinIdFac}  \emph{(Voir [FFR, Chap. 7, Th. 8]}\\
Soit  $M$ un  \Amo \lrsb.
On reprend les notations du \thref{thdetCay2} et l'on rappelle que $\fD_{k+1}=\fB_k\fB_{k+1}$.
\begin{enumerate}
\item Si $\fB_k=\gen{1}$ avec $k\geq 1$, alors  $\fB_{k+r}=\gen{1}$ pour tout $r>0$.
\item De manière \gnle, pour $k\geq 1$ 

\snic{\fB_k\subseteq \DA(\fB_{k+1})\;\hbox{ et }\;\DA(\fB_k)= \DA(\fD_{k+1}).}
\end{enumerate}
\end{theorem}
%
\begin{proof} Montrons l'inclusion $\fB_k\subseteq \DA(\fD_{k+1})$.
Il suffit de traiter le cas $k=1$, car un début de \rsn libre résout
le module conoyau de la dernière matrice (qui remplace alors le module $M$).
\\
Soit $\mu$ un \gtr de
$\fB_1$ (une \coo de $\bmu_1$). Pour montrer qu'une puissance de $\mu$
est dans $\fD_2$, il suffit
de montrer que sur l'anneau $\gA[1/\mu]$, l'\id $\fD_2$ contient $1$. 
C'est donné par le lemme \ref{lemResFinIdFac}.\\
Ainsi $\DA(\fB_k)\subseteq \DA(\fD_{k+1})=\DA(\fB_k\fB_{k+1})= \DA(\fB_{k})
\cap \DA(\fB_{k+1})$,
 ce qui implique les points~\emph{1} et \emph{2.}
\end{proof}
%

\section{Le \tho de Hilbert-Burch}\label{secHilbBu}

Nous donnons dans cette section le \tho de Hilbert-Burch sous sa forme \gnle (épurée de toute hypothèse \noee notamment) donnée par Northcott.

Le \tho de Hilbert-Burch 
examine sous quelles conditions le module des \syzys pour un \sys
de $n$ \gtrs d'un \id $\fa$ est un module libre de rang $n-1$.
On a dans ce cas une \sex
$$
    0\lora \Ae {n-1} \vvers \varphi  \Ae n \vvers \alpha \gA\vvers \pi \gA\sur\fa\lora 0 
$$
ce qui est un exemple simple de \rlf (ici, \rsn
du module $\gA/\fa$).

Pour une matrice
$A = (a_{ij}) \in \MM_{n,n-1}(\gA)$, on définit 
\begin{equation}\label{eqn1nxn-1}
\dmatrix {
X_1    & a_{11} & \cdots & a_{1\,n-1} \cr
X_2    & a_{21} & \cdots & a_{2\,n-1} \cr
\vdots & \vdots &        & \vdots \cr
X_n    & a_{n1} & \cdots & a_{n\,n-1} \cr} =
\Delta_1 X_1 + \cdots + \Delta_n X_n,\;\;\Delta=[\,\Delta_1\;\dots\;\Delta_n\,].
\end{equation}
Ici les $X_i$ sont des \idtrs et les $\Delta_i$ sont les mineurs d'ordre $n-1$, convenablement signés, de
la matrice $A$.

\begin{theorem} \label{thHilBur1}\emph{(Matrice $n\times (n-1)$)}.\\
Soit $A  \in \MM_{n,n-1}(\gA)$, on note $\Delta=[\,\Delta_1\;\dots\;\Delta_n\,]$  comme dans~\pref{eqn1nxn-1}. On remarque que $\cD_{n-1}(A)=\gen{\Delta_1,\dots,\Delta_n}$. 
\begin{enumerate}
\item On a  un complexe
\begin{equation}\label{eqnHilBur}
0 \lora \Ae {n-1} \vvers {A} \Ae n \vvers {\Delta} \gA, \end{equation}

\item Ce complexe 
est exact \ssi   
$\Gr({\Delta_1,\dots,\Delta_n})\geq 2$.
\item Si $\Gr({\Delta_1,\dots,\Delta_n})\geq 3$ alors l'\id $\gen{\Delta_1,\dots,\Delta_n}$ contient $1$.
\end{enumerate}
\end{theorem}
%
\begin{proof}
\emph{1.} Le fait que $\Delta\,A=[\,0\;\dots\;0\,]$ est une formule de Cramer.
Cela résulte de ce que dans \pref{eqn1nxn-1}, le \deter est nul si l'on remplace 
les $X_i$ par une colonne quelconque de la matrice $A$.

\emph{2.} Vu le point \emph{1}, c'est  un cas particulier du \thref{cor3thABH1}.

\emph{3.}
On a une \sex 

\snic{0\to\Ae {n-1} \vvers{A} \Ae{n}\vvers\Delta \gA \lora \gA/{\cD_{n-1}(A)}\to 0,}

\snii et l'on obtient le résultat en appliquant la proposition~\ref{corthABH1}.
\end{proof}

\begin{theorem} \label{thHilBur2}\emph{(\Tho de Hilbert-Burch)}\\
Soit $A  \in \MM_{n,n-1}(\gA)$, on note $\Delta=[\,\Delta_1\;\dots\;\Delta_n\,]$  comme dans~\pref{eqn1nxn-1}. On remarque que $\cD_{n-1}(A)=\gen{\Delta_1,\dots,\Delta_n}$.
 Si l'on a une \sex
\begin{equation}\label{eqnHilBur1}
0 \lora \Ae {n-1} \vvers {A} \Ae n \vvers {\alpha} \gA, \qquad
\alpha = [\,a_1\, \cdots\, a_n\,],
\end{equation}
on obtient les résultats suivants.
\begin{enumerate}
\item  Il existe un (unique) \elt \ndz $a \in \gA$ tel que
$\alpha = a\,\Delta$, et $a$ est un pgcd fort de l'\id $\fa=\gen{\an}$.
\item   $\Gr({\Delta_1,\dots,\Delta_n})\geq 2$.
\item Si $\Gr({\Delta_1,\dots,\Delta_n})\geq 3$ alors l'\id $\gen{\Delta_1,\dots,\Delta_n}$ contient $1$ et $\fa$ est un \idp.
\end{enumerate}
\end{theorem}
%
\begin{proof}
\emph{1} et \emph{2.} Le \thref{thRangStProfRLF} nous dit que

\snic{\Gr(\an)\geq 1\hbox{ et }\Gr({\Delta_1,\dots,\Delta_n})\geq 2.}

\snii
On peut alors appliquer
le \thref{thdetCay} avec 

\snic{m=2$, $r_2=n-1$, $r_1=1$, $r_0=0$, $A_2=A$, $A_1=\alpha.}

\snii
Le vecteur
$\tra {\,\,{\bmu_{2}}\sta}$ est le vecteur $\Delta$ et le \tho affirme que $\Vi^{1}A_1$ (\cad $\alpha$)
s'écrit sous la forme $\bmu_1 \Delta$ (avec ici $\bmu_1=\dt{\bmu_1}\in\gA$)
où $\bmu_1$ est un \elt \ndz, pgcd fort des $a_i$.

\emph{3.} Voir le \tho précédent.  
\end{proof}


\sibook{ 
\section{Le cas des modules gradués}\label{secDetCayModGr}

\vspace{3cm}

}


\section{Annexe: \dem du \tho de proportionnalité}\label{secCayleyAnnexe} 

On définit un \ix{produit intérieur droit} (qui sera \gne plus tard) comme suit:
$$\label{prodintdroit}
\formule{\big(\Vi\bul \Ae n\big) \times \Ae n \;\lora\; \Vi^{\bullet-1} \Ae n, \\[1mm]
\big((v_1\vi\dots\vi v_p),u\big)\;\longmapsto\; 
\formule{(v_1\vi\dots\vi v_p)\intd u=\\[1mm]
\som_{i=1}^{p}(-1)^{{i-1}}\scp {v_i}u \;v_1\vi\dots\vi \wh{v_i} \vi\dots\vi v_p.}}
$$
(la notation $\wh{v_i}$ signifie que l'on a supprimé le terme $v_i$ dans le produit extérieur). On laisse le soin \alec de vérifier que cette \dfn est cohérente.
\\
Comme cas particuliers on obtient 

\snic{a\intd u=0\hbox{ pour }a\in\gA,\;\; \hbox{ et }v\intd u= \scp v u\hbox{ pour }v\in\Ae n.}

\snii
En fait, le 
produit intérieur $\bullet \intd u$ est l'unique \ix{anti\dvn à gauche}
 de $\Vi \Ae n$ qui prolonge la forme \lin $\scp{\,\bullet} u$. 
\\
Par \dfn, une anti\dvn à gauche de l'\alg graduée $\Vi \Ae n$ est une
\Ali~$\partial$ de $\Vi \Ae n$ dans elle-même, de degré $-1$, qui satisfait l'axiome suivant:

\snic{\fbox{$\partial(\bmx\vi\bmz)=\partial(\bmx)\vi\bmz+(-1)^{k}\bmx
\vi \partial(\bmz)$}
\quad \hbox{ si } \bmx\in\Vi^{k}\Ae n \hbox{ et } \bmz \in \Vi^{\ell}  \Ae n.}

\snii La \vfn est laissée \alec.

\medskip 
\rem On s'y perd un peu, le produit intérieur est droit et l'anti\dvn est à gauche: la tradition a fixé les choses ainsi.
\eoe

\begin{lemma} \label{lemDualIntd}
Pour $\bmx\in\Vi^{p}\Ae n$ et $\bmz\in\Vi^{p-1}\Ae n$, on a la jolie formule
\begin{equation} \label {eqlemDualIntd}
\scp{\,\bmx\intd u}{\bmz\,}=\scp{\,\bmx}{u\vi \bmz\,}
\end{equation}
\end{lemma}
%
\begin{proof}
Notez que $\bmx$, $u$ et $\bmz$ ne changent pas de place, on remplace seulement~\hbox{\gui{$\intd \bullet \mid$}} par \gui{$\mid\bullet \,\vi$}.
On vérifie la \prt sur les bases, \cad  
$$
{\scp{\,e_I\intd e_i}{e_J\,}=\scp{\,e_I}{e_i\vi e_J\,} \hbox{ avec } I\in\cP_{p,n} \hbox{ et }J\in\cP_{p-1,n}.}
$$
On obtient $0$ si $i\notin I$, si $i\in J$ ou si $J\not\subset I $. 
Sinon, avec $I=\so{i_1,\dots,i_p}$ \hbox{et $i=i_k$}, on peut supposer $J=\so{{i_1,\dots,i_{k-1},i_{k+1},\dots,i_p}}$ et l'on obtient dans les deux membres $(-1)^{k-1}$. 
\end{proof}
%

\begin{proposition} \label{propDualiteHodge2}
Avec $p+q=n$,  $x\in\Ae n$, $\bmu$, $\bmw\in\Vi^{p}\Ae n$ et $\bmv\in\Vi^{q}\Ae n$, on a les formules de dualité suivantes:
\begin{enumerate}
\item \qquad \qquad $\scp {\bmu\sta} \bmv    =     \dt{\bmu\vi\bmv},$
\item \qquad \qquad $\scp {\bmu\sta}{\bmw\sta}    =   \scp \bmu \bmw,$
\item \qquad \qquad $\dt{\bmu\vi\bmw\sta}    =   \scp \bmu \bmw,$
\item \qquad \qquad $x\sta    =    e_{\lrbn} \intd x.$
\end{enumerate}
\end{proposition}
%
\begin{proof}\emph{1},
\emph{2} et
\emph{3.} Rappel de la proposition \ref{propDualiteHodge}.
\\
\emph{4.}  Pour un $\bmz$ arbitraire, on a d'une part $\scp{\bmx\sta}{\bmz}= \dt{\bmx\vi\bmz}$, et d'autre \hbox{part $\scp{e_{\lrbn} \intd x}{\bmz}=\scp{e_{\lrbn}}{\bmx\vi\bmz}$}, qui est égal 
à $\dt{\bmx\vi\bmz}$.
\end{proof}


\begin{lemma} \label{lemthOrtho1}  \emph{(Formule Sylvester-Pl\"ucker)}\\
Soit $\bmx=(\xn)$ et $\bmz=(z_1,\dots, z_p)$ deux suites dans $\Ae n$.  
On a alors:
\begin{equation} \label {eqlemthOrtho1}\preskip.4em \postskip.4em
\dt{x_1\vi\dots\vi x_n}\,z_1\vi\dots\vi z_p
=\som_{K\in\cP_{p,n}}\,\dt{ \bmx \MA{\leftarrow}\limits_K \bmz}
\,\Vi_{k\in K}x_k
\end{equation}
où $\bmx \MA{\leftarrow}\limits_K \bmz$ désigne, pour
$K = \{k_1 < \ldots < k_p\}$, 
la suite de $n$ vecteurs de $\Ae n$ obtenue en rempla\c{c}ant dans
$(x_1, \ldots, x_n)$, chaque vecteur $x_{k_i}$ par le vecteur~$z_i$. Par exemple, le membre droit de la formule pour
$\dt{x_1\vi x_2\vi x_3}\ z_1 \vi z_2$ est:

\snac{\dt{z_1\vi z_2\vi x_3}\ x_1 \vi x_2 + \dt{z_1\vi x_2\vi z_3}\ x_1 \vi x_3 +
\dt{x_1\vi z_2\vi z_3}\ x_2 \vi x_3.}
\end{lemma}

\medskip \rem Pour $p=1$, le premier membre  est
$\dt{x_1\vi \dots\vi x_n}\,z$ et le second membre s'écrit

\snac {
\dt{z\vi x_2\vi\dots\vi x_n}\, x_1 \ +\ \dt{x_1\vi z\vi\dots\vi x_n}\, x_2 
\ +\ \cdots \ +\ \dt{x_1\vi x_2\vi\dots\vi z}\, x_n.
}

L'\egt n'est rien d'autre qu'une formule habituelle de Cramer.

Notons aussi que lorsque les $x_i$ forment une base, le second membre donne les \coos du premier membre sur la base correspondante de  $\Al p \!{\Ae n}$ 
\eoe

\begin{proof}
L'\egt à démontrer est \lin en chacun des $x_i$ et $z_j$. 

\emph{Premier cas}.
Supposons que $(\xn)$ est une  base.  Par linéarité il suffit de traiter le cas où $(z_1,\dots,z_p) = (x_{k_1},\dots,x_{k_p})$. 
\\
Si deux des $k_i$ sont égaux, les deux membres sont nuls: en effet, dans le second membre chaque scalaire \halfsmashbot{$[\bmx \MA{\leftarrow}\limits_K \bmz]$} est nul. 
\\ 
Si les $k_i$ sont tous distincts,  posons $K' =\so{k_1, \ldots, k_p}$. 
Alors, les termes de la somme intervenant dans le second membre
 sont tous nuls, sauf celui correspondant à $K = K'$, qui vaut
$$\preskip.3em \postskip.2em 
\dt{ \bmx \MA{\leftarrow}\limits_{K'} \bmz}\, x_{k_1}\vi\dots\vi x_{k_p}, 
$$
\cad le premier membre: notez que \smash{$(\bmx \MA{\leftarrow}\limits_K \bmz)=(\xn)$} quel que soit l'ordre des $k_i$.

\emph{Cas \gnl.}
Le cas \gnl se ramène au cas précédent comme suit. On considère le localisé de Nagata $\gB=\gA(T)$. On note $(e_1,\dots,e_n)$ la base canonique de $\gB^n$. On considère les vecteurs $y_i=Te_i+x_i$. Leur \deter est un \pol en $T$, unitaire de degré $n$, donc \iv, et les $y_i$ forment une base. On applique le cas précédent et l'on obtient l'\egt voulue avec les $y_i$ à la place des $x_i$. On fait $T=0$ et l'on obtient l'\egt pour les $x_i$ et les $z_j$ vus dans $\gB^n$. On conclut en rappelant que le morphisme $\gA\to\gB$ est injectif.  
\end{proof}
%

\begin{lemma} \label{lemthOrtho2}
Soient $u$, $v_1$, \ldots, $v_q \in \Ae n$, avec chaque $v_j$ \ort  à
$u$. Alors $v_1 \vi\dots\vi v_q$ est \ort au sous-module
$u \vi\, \Vi^{q-1}(\Ae n)$.
\end{lemma}
%
\begin{proof}
On prouve d'abord que $(v_1\vi\dots\vi v_q) \intd u = 0$ par \recu sur $q$.
\\
Pour $q=1$, c'est immédiat car $v_1\intd u = \scp{v_1}{u}$.  
\\
Pour $q \ge 2$,
posons $\bmv = v_2\vi\dots\vi v_q$. Puisque $\bullet\intd u$ est une
anti-dérivation à gauche, on a 

\snic {
(v_1\vi \bmv) \intd u = \scp{v_1}{u}\,\bmv - v_1\vi(\bmv\intd u) 
\buildrel {\rm ici} \over = v_1\vi(\bmv\intd u)
}

Comme $\bmv\intd u = 0$ (\recu sur $q$), on a le résultat annoncé.

Ensuite, soit $\bmw \in \Vi^{q-1}(\Ae n)$. La formule de
dualité \pref{eqlemDualIntd} (lemme~\ref{lemDualIntd}) donne 

\snic {
\scp{v_1\vi\dots\vi v_q}{u\vi\bmw} = 
\scp{(v_1\vi\dots\vi v_q) \intd u}{\bmw},  
}

\snii
d'où le résultat d'orthogonalité de l'énoncé.
\end{proof}
%


%
\begin{Proof}{Démonstration du \tho de proportionnalité \ref{thOrtho}.}\\
Il s'agit de démontrer que les deux $q$-vecteurs $\bmu\sta$ et $\bmv$ sont
proportionnels, i.e. que $\scp{\bmu\sta}{e_I}\,\scp{\bmv}{e_J} =
\scp{\bmu\sta}{e_J}\,\scp{\bmv}{e_I}$ pour tous $I$, $J\in\cP_{q,n}$. Ou
encore, d'après la \dfn de l'isomorphisme de Hodge choisi:

\snic {
\dt{\bmu\vi e_I}\,\scp{\bmv}{e_J} = \dt{\bmu\vi e_J}\,\scp{\bmv}{e_I}.
}

\snii
En posant 

\snic {
\bmw_1 = \dt{\bmu\vi e_I}\,e_J,\quad \bmw_2 = \dt{\bmu\vi e_J}\,e_I \quad\hbox{et}\quad 
\bmw = \bmw_1 - \bmw_2,
}

il s'agit de prouver que le $q$-vecteur $\bmw$ est \ort à
$\bmv$. 
\\ 
D'après la formule de Sylvester-Pl\"ucker \pref{eqlemthOrtho1}, le $q$-vecteur $\bmw_1 = \dt{\bmu\vi e_I}\,e_J$ est une
somme de termes de la forme $a\, y_1 \vi\dots\vi y_q$ avec $a\in \gA$, les
$y_k$ étant pris dans $ [u_1, \ldots, u_p]  \cup  [e_i,\ i\in I] $.

Dans chacun des termes $a\, y_1 \vi\dots\vi y_q$, il y a un $y_k$ égal à
l'un des $u_j$, sauf pour celui correspondant dans la formule \pref{eqlemthOrtho1} à $K = I$, ce terme valant $\dt{\bmu\vi e_J}\, e_I =
\bmw_2$. Ainsi, $\bmw$ est une somme de $q$-vecteurs de la forme $a\, y_1
\vi\dots\vi y_q$ et dans chacun d'entre eux, l'un des $y_k$ appartient 
\hbox{à $[u_1,\ldots, u_p]$}. D'après le lemme \ref{lemthOrtho2}, un tel $q$-vecteur est
orthogonal à $\bmv$. Donc~$\bmw$ est \ort à $\bmv$.
\end{Proof}

\Exercices

\begin{exercise} \label{exoidMacRae}
{(Idéal déterminantiel de MacRae, comparer avec le \thref{thMacRae})}\\
{\rm  
Soit $M$ un module avec $\cF_0(A)=\gen{e}$ et $e$ \ndz.
Alors il existe des \eco $(\sn)$ tels que la module $M$ est \elr après \lon en chacun des  $s_i$.
}

\end{exercise}

\begin{exercise} \label{exoHodgeIsos}
{(Isomorphismes de Hodge)}\\
{\rm  
On a défini dans le cours un \iso de Hodge droit $\Hd$ lié au produit
intérieur droit ($\Hd(x)=e_{\lrbn} \intd x$); on étudie ici une version
gauche, appelée \iso de Hodge gauche.

\snii\emph {1.}
Pour $p+q = n$, on note $\Hd_{pq} : \Vi^p\Ae n \to \Vi^q\Ae n$
la restriction de $\Hd$. Que vaut $\Hd_{pq} \circ \Hd_{qp}$?
Que vaut $\det(\Hd_{pq})$ dans les bases canoniques?

\snii\emph {2.}
Définir un \iso de Hodge gauche $\Hg : \Vi\Ae n \to \Vi\Ae n$
analogue à celui de droite et en donner quelques
\prts.

\snii\emph {3.}
Montrer que $\Hd$ et $\Hg$ sont inverses l'un de l'autre.

}

\end{exercise}

\begin{exercise} \label{exoResultant2VarsAsCayleyDet}
{(Le résultant de deux \pols \hmgs en deux variables comme \deter de
Cayley d'un complexe)}\\
{\rm  
Soient $P, Q \in \gA[X,Y]$ deux \pols \hmgs de degré $p$ et $q$:

\snic {
P = \sum_{i=0}^p a_i X^i Y^{p-i}, \qquad
Q = \sum_{j=0}^q b_j X^j Y^{q-j}
}

\snii
On fixe $d \ge p+q-1$, et on pose $a = d+1-(p+q)$, $b=d+1$. 
Puis on considère les \Amos:

\snic {
\begin {array}{l}
E = \gA[X,Y]_{d-(p+q)} \simeq \gA^a \\
F = \gA[X,Y]_{d-p} \times \gA[X,Y]_{d-q}\simeq \gA^{a+b} \\
G = \gA[X,Y]_{d} \simeq \gA^b \\
\end {array}
}

\snii
et les applications \lins:

\snic {
E \vers {\varphi} F,\quad  W \mapsto (WQ, -WP), \qquad
F \vers {\psi} G,\quad  (U,V) \mapsto UP + VQ
}

\snii
Il est clair que $0\to E\vers {\varphi} F\vers {\psi} G\to 0$ est un complexe.
On veut fournir une condition suffisante pour que ce complexe soit 
de Cayley et montrer, dans ce cas, que son \deter est $\pm \Res(P,Q)$.
On notera le cas particulier de la plus petite valeur possible $d_0$ de
$d$ i.e. $d_0 = p+q-1$ (qui fournit $a = 0$) pour laquelle 
$\psi$ est l'application de Sylvester habituelle. La quantité $a = d-d_0$
doit être vue comme la distance de $d$ à $d_0$.

\snii\emph {1.}
De manière \gnle, comment calculer le \deter d'un complexe de Cayley $0\to\gA^a
\vers {A}\gA^{a+b}\vers{B}\gA^b \to 0$ (avec $a = 0$ comme cas particulier).

\snii\emph {2.}
Le résultat de cette question pourra être utilisé dans la question
suivante. Soit $(\ux) = (\xn)$ une suite d'\elts de $\gA$ que l'on voit comme
un \gui {motif}; on définit, pour tout $r \ge 0$, la matrice $M_r(\ux) \in
\MM_{r,n+r-1}(\gA)$ en faisant glisser le motif comme ci-dessous:

\snic {
M_3(\ux) = \cmatrix {
x_1 & x_2 & \cdots & x_n    & 0      & 0 \cr
0   & x_1 & x_2    & \cdots & x_n    & 0\cr
0   &   0 & x_1    & x_2    & \cdots & x_n \cr}
}

Quel est l'\idd $\cD_r(M_r(\ux))$? 

Question non utilisée dans la suite: si $x_1, \ldots, x_n$ sont
$n$ \idtrs sur $\gA$, que peut-on penser des ${r+n-1 \choose n-1}$
mineurs de $M_r(\ux)$ vis-à-vis de la composante \hmg
$\gA[\ux]_r$?

\snii\emph {3.}
Identifier la matrice $K$ de $\varphi$ dans des bases adéquates 
et déterminer quelques mineurs maximaux (d'ordre $a$) de $K$
simples à calculer. Quel est l'\idd $\cD_a(K)$?

\snii\emph {4.}
Identifier la matrice $S$ de $\psi$ dans des bases adéquates 
(matrice de Sylvester \gui {généralisée}) et déterminer
un mineur maximal (d'ordre $b$) simple à calculer.
Montrer que $\cD_b(S) = \Res(P,Q)(\rc(P)+\rc(Q))^a$.

\snii\emph {5.}
En déduire que le complexe est de Cayley \ssi $\Gr(\rc(P) + \rc(Q)) \ge 2$
et $\Res(P,Q)$ est \ndz; dans ce cas, son \deter de Cayley est $\pm\Res(P,Q)$.

\snii\emph {6.}
On a $X^iY^j \Res(P,Q) \in \gen {P,Q}$ pour $i+j \ge p+q-1$. Montrer une
réciproque en montrant l'inclusion d'\ids de $\gA$:

\snic {
\gA \cap \bigl(\gen{P,Q} : \gen{X,Y}^\infty\bigr) \subseteq 
\sqrt {\gen{\Res(P,Q)}}
}

\sni
De manière précise, en notant $\fa_d=\gA\cap\bigl(\gen{P,Q}:\gen{X,Y}^d\bigr)$,
on a l'inclusion:

\snic {
\fa_d^{d+1} \subseteq \Res(P,Q)(\rc(P)+\rc(Q))^a \subseteq \gen{\Res(P,Q)}
}

En particulier, les deux \ids de $\gA$, 
$\gA \cap \bigl(\gen{P,Q} : \gen{X,Y}^\infty\bigr)$ et $\gen{\Res(P,Q)}$
ont même racine.

}

\entrenous {
Si $P,Q$ sont des \pols \gnqs, on a en fait l'\egt:

\snic {
\gA \cap \bigl(\gen{P,Q} : \gen{X,Y}^\infty\bigr) = \gen{\Res(P,Q)}
}

Confirmé par Laurent Busé. C'est ce que l'on appelle le
\tho principal de l'élimination, cf Jouanolou, in ``Le
formalisme du résultant''.
}

\end{exercise}

\begin{exercise} \label{exoHilbertBurchXY}
{(Hilbert-Burch et trois \moms en $x,y$)}\\
{\rm  
Soient $\gA = \gk[x,y]$ un anneau de \pols, 
quatre entiers $a$, $b$, $c$, $d \in \NN$ et les trois \moms $x^a$, $x^b y^c$, $y^d$.
On définit la forme \lin $\mu : \gA^3 \to \gA$, $\mu = [\,x^a \;
x^by^c \; y^d\,]$. Montrer que $\Ker\mu$ est un \Amo libre de rang $2$.
}
\end{exercise}

\begin{exercise} \label{exoHilbertBurchGeneric} {(Hilbert-Burch générique)}
\\
{\rm  
\noindent
Soit, sur un anneau quelconque $\gk$, un anneau de \pols $\gA$ à
$n(n-1)$ \idtrs $x_{ij}$, $1 \le i \le n$, $1 \le j \le n-1$ et $A$ la matrice
$(x_{ij})$ dont on note les colonnes $A_1, \ldots, A_{n-1}$.  Comme dans le
cours, on pose pour $v \in \gA^n$:

\snic {
\Delta(v) = \det(v, A_1, \ldots, A_{n-1}),  \qquad
\Delta_i = \Delta(e_i)  \hbox { pour $1 \le i \le n$}
}

\snii
En examinant, pour un certain ordre monomial, les monômes dominants des
$\Delta_i$, montrer que la suite ci-dessous est exacte

\snic {
0 \to \gA^{n-1} \vers {A}\gA^n \vers {\Delta} \gA 
\to \gA/\Im\Delta \to 0
}

}
\end{exercise}



\sol

\exer{exoidMacRae} C'est une légère variation sur la \dem du point \emph{1b} du \thref{thMacRae}. 
Soient $(\mu_i)_{i\in\lrbn}$ la famille des mineurs d'ordre $r$
d'une \mpn $A\in\MM_{r,m}$ du module $M$. On a des $s_i$ tels que $\mu_i=s_ie$ et des $x_i$ tels que $e=\sum_ix_i\mu_i=\sum_ix_is_ie$.
Puisque $e$ est \ndz, les $s_i$ sont \com.

\exer{exoHodgeIsos} 

\snii\emph {1.}
On a $\Hd_{pq} \circ \Hd_{qp} = (-1)^{pq}\I_{\wedge^q \Ae n}$. 
En ce qui concerne le \deter, l'auteur de l'exercice
ignore sa valeur.

\snii\emph {2.}
On définit l'\iso de Hodge à gauche comme suit:
$$
\formule{\Vi^{p}\Ae n\vvers\Hg \Vi^{q}\Ae n,\\[1mm]
\bmx\longmapsto\som_{J\in\cP_{q,n}} \dt{e_J \vi\bmx}\, e_J} 
\quad \hbox{pour } p+q=n. 
$$

Par \dfn, on a donc $\scp {e_J}{\Hg(\bmx)} = \dt{e_J \vi\bmx}$.
On laisse le soin au lecteur de vérifier que:

\snic {
\Hg(e_J) = \vep_{\ov J,J}\,e_{\ov J},
\quad\qquad \hbox {et} \qquad\quad
\framebox [1.1\width][c]{$\Hg(e_J) \vi e_J = e_{\lrbn}$}
}

L'\egt encadrée mérite d'être retenue:
on peut compléter $e_J$ à gauche avec $\Hg(e_J)$
pour obtenir le $n$-vecteur $e_{\lrbn}$.

\snii\emph {3.}
On a $\Hd(e_I) = \vep_{I,\ov I}e_{\ov I}$ et $\Hg(e_J) = \vep_{\ov J,J}e_{\ov J}$.
D'où $(\Hg \circ \Hd)(e_I) = \vep_{I,\ov I} \vep_{I,\ov I} e_I = e_I$.
Idem dans l'autre sens.


\exer{exoResultant2VarsAsCayleyDet} 

\snii\emph {1.}
On note $H = \Hd$ l'isomorphisme de Hodge droit. On écrit symboliquement
le déterminant de Cayley $\Delta$ sous la forme:
$$
\Delta =
{\Vi^b(B) \over \bigl( \Vi^a(\tra{A}) \bigr)\sta} =
{\sum_{\#J = a} \delta_{J}(B)\, e_J \over \sum_{\#I=a} \delta_I(\tra{A})\, H(e_I)}
$$
Explication: $\delta_J(B)$ désigne le mineur maximal de $B$ 
sur les colonnes d'indices $J$. Idem pour $\delta_I(\tra{A})$
mineur maximal de $A$ sur ses lignes d'indices $I$.
Si les conditions $\Gr(\cD_a(A)) \ge 2$ et $\Gr(\cD_b(B)) \ge 1$ sont 
satisfaites, alors il y a un scalaire $\Delta$ bien déterminé
tel que:

\snic {
\sum_{\#J = a} \delta_{J}(B)\, e_J = \Delta \cdot \sum_{\#I=a} \delta_I(\tra{A})\, H(e_I)
}

Vu notre choix de l'isomorphisme de Hodge (droit), $H(e_I) = \varepsilon_{I,\ov I}
e_{\ov I}$, on pourra également écrire symboliquement
(avec $\ov I \leftrightarrow J$):
$$
\Delta = {\sum_{\#J=a} \delta_{J}(B)\, e_J \over 
\sum_{\#J=a} \delta_{\ov J}(\tra{A})\, \varepsilon_{\ov J,J} e_J}
$$
Prenons l'exemple simple de $\gA \vvers {\bigl[{a\atop b}\bigr]} \gA^2
\vvers {[c\ d]} \gA$. Puisque $e_1 \vi e_1\sta = e_2\vi e_2\sta = e_{12}$, on a 
$e_1\sta = e_2$, $e_2\sta = -e_1$, donc:

\snic {
\Delta = \frac{ce_1 + de_2}{ae_1\sta + be_2\sta} =
\frac{ce_1 + de_2}{ae_2 - be_1} = \frac{ce_1 + de_2}{-be_1 + ae_2}
}

A noter que le \tho de proportionnalité \ref{thOrtho} donne $c/(-b) = d/a$,
ce qui immédiat ici puisque $ac + bd=0$.

\snii\emph {2.}
Il est facile de voir que $\cD_r(M_r(\ux)) = \gen {\ux}^r$. En ce qui concerne
la question bonus, il est commode de poser $x_i = 0$ pour $i<0$ ou $i>n$.
Pour fixer les idées, je suppose $r = 3$. Le mineur extrait sur les colonnes
$i < j < k$ est

\snic {
\dmatrix {
x_i      & x_j     & x_k \cr
x_{i-1}  & x_{j-1} & x_{k-1} \cr
x_{i-2}  & x_{j-2} & x_{k-2} \cr}
\qquad \hbox {de transposé} \qquad
\dmatrix {
x_i  & x_{i-1} & x_{i-2} \cr
x_j  & x_{j-1} & x_{j-2} \cr
x_k  & x_{k-1} & x_{k-2} \cr}
}

Pour ceux qui connaissent: on reconna\^it à droite un \deter dit de
Dedekind, soit noté $D_{ijk}$ (indexation le long de la première colonne
avec $i<j<k$) ou bien $\Delta_{i,j-1,k-2}$ (indexation le long de la diagonale
avec $i \le j-1 \le k-2$). La famille des ${n+r-1 \choose n-1}$ mineurs
maximaux de $M_r(\ux)$ est ainsi indexée soit par les suites croissantes de
$\lrb{1..n}$ de longueur $r$ soit par les suites strictement croissantes de
$\lrb{1..n+r-1}$ de longueur $r$. Si les $x_i$ sont des \idtrs, un résultat
d\^u à Dedekind dit que ces ${n+r-1 \choose n-1}$ \deters forment une
$\gA$-base de la composante \hmg $\gA[\ux]_r$.

\smallskip
Dans la suite, on fait le choix suivant des bases. \\
Pour $E$: les $X^k
Y^{d-(p+q)-k}$ ($k$ décroissants), pour $F$ les $(X^iY^{d-p-i},0)$, ($i$ décroissants), et les $(0,X^jY^{d-q-j})$ ($j$ décroissants). Et enfin, pour $G$,
les $X^kY^{d-k}$ ($k$ décroissants).

\snii\emph {3.}
La matrice $K$ possède $a$ colonnes et $(d+1-p) + (d+1-q)$ lignes.
Si $a=0$, ${K}$ est une matrice ayant $0$ colonne. Si $a>0$ les
$d+1-p$ premières lignes sont pour les \coes de $Q$ et les suivantes pour 
ceux de $-P$.  Par exemple, pour $p=1$, $q=2$, $d=4$, on a $a=2$, $d+1-p=4$,
$d+1-q=3$ et la matrice transposée $K$ est:

\snic {\tra{K} =
\cmatrix {
b_2  & b_1 & b_0 &  .&-a_1 &-a_0&.\cr
.& b_2  & b_1 & b_0 &  .&-a_1 &-a_0}}

En \gnl,  
$\tra{K} = M_a(\ux)$ où $\ux = (b_q,\ldots, b_0,0,\ldots,0, -a_p, \ldots,-a_0)$ 
avec $a-1$ zéros au milieu.

 On a donc $\cD_a(K)
=(\rc(P) + \rc(Q))^a$.

Mineurs maximaux de $K$ faciles à calculer:
ceux qui sont égaux, au signe près, aux \gtrs naturels de $(\rc(P) + \rc(Q))^a$.

\snii\emph {4.}
La matrice $S$ a $d+1$ lignes; elle possède $d+1-p$ colonnes de
$P$ puis $d+1-q$ colonnes de $Q$.  La voici (à gauche) pour $p=1$, $q=2$,
$d=4$, à comparer avec la matrice de Sylvester habituelle (à droite):

\snic {
\cmatrix {
a_1&   .&   .&   .& b_2&   .&  .\cr
a_0& a_1&   .&   .& b_1& b_2&  .\cr
  .& a_0& a_1&   .& b_0& b_1& b_2\cr
  .&   .& a_0& a_1&   .& b_0& b_1\cr
  .&   .&  .&  a_0&   .&   .& b_0\cr}
\qquad\qquad
\cmatrix {
a_1&   .& b_2\cr
a_0& a_1& b_1\cr
  .& a_0& b_0\cr}
}

Pour $J = \lrb{a+1..a+b}$ de cardinal $b$, montrons que $\delta_J(S) =
(-1)^{aq}b_q^a \Res(P,Q)$. On procède par \recu sur $d$ en notant ($S_d,
J_d$) les objets correspondants ($a$, $b$ dépendent de $d$). On vérifie
que $\delta_{J_{d}}(S_{d}) = (-1)^q b_q \delta_{J_{d-1}}(S_{d-1})$: dans
l'exemple ci-dessus ($d = 4$), il faut calculer le mineur maximal obtenu en
supprimant les $a=2$ premières colonnes, et le développer par rapport à
la première ligne pour faire appara\^itre $(-1)^2 b_2$ et la matrice $S_3$.

Pour $d = d_0 = p+q-1$, on a $\delta_{J_{d_0}}(S_{d_0}) = \Res(P,Q)$ et donc
pour tout $k \ge 0$, $\delta_{J_{d_0+k}}(S_{d_0+k}) = (-1)^{kq}\Res(P,Q)$,
d'où le résultat avec $k=a$ puisque $d_0+a = d$.

Pour $J = \lrb{a+1..a+b}$, on a $\ov J = \lrb{1..a}$, $\varepsilon_{\ov J,J}= 1$
et $\delta_{\ov J}(\tra{K}) = b_q^a$. Avec quelques abus de notations (divisions), le
\tho de proportionnalité \ref{thOrtho} s'écrit:

\snic {
\frac{(-1)^{aq}b_q^a \Res(P,Q)}{b_q^a} =
\frac{\delta_J(S)} {\vep_{\ov J,J}\,\delta_{\ov J}(\tra{K})}
\qquad \forall\, J \hbox { $\#J=b$}
}

En supposant $b_q$ \ndz et en posant $\Delta = (-1)^{aq}\Res(P,Q)$, on
a donc:

\snic {
\delta_J(S) = \Delta \cdot \vep_{\ov J,J}\,\delta_{\ov J}(\tra{K})
\qquad \forall\, J \hbox { $\#J=b$}
}

Mais comme il s'agit d'une \idt \agq, elle est toujours vraie que $b_q$ soit
\ndz ou pas. On en déduit que 

\centerline{\fbox{$\cD_b(S) = \Res(P,Q)\cD_a(K)= \Res(P,Q)(\rc(P) + \rc(Q))^a$}}

\emph {5.}
Si $a>0$, pour que $\Gr(\cD_a(K)) \ge 2$,
il faut et il suffit que $\Gr(\rc(P) + \rc(Q)) \ge 2$. C'est le cas
si $P$ et $Q$ sont des \pols \gnqs.

On a $\cD_b(S) = \Res(P,Q)\cD_a(K)$ et si l'on suppose de plus que $\Res(P,Q)$
est \ndz, alors $\Gr(\cD_b(S)) \ge 1$. Dans ce cas, le complexe est de
Cayley, $0\to E\to F\to G$ est exact  et le \deter de Cayley est 
$\Delta = (-1)^{aq}\Res(P,Q)$.

\snii\emph {6.}
Soient $u_0, u_1, \ldots, u_d \in \fa_d$.
Si $d < p+q-1$, on a $u_k = 0$. On peut donc supposer $d \ge p+q-1$;
les appartenances $u_kX^iY^j \in \gen{P,Q}$ pour $i+j=d$ s'écrivent
dans la base $(X^d, X^{d-1}Y, \ldots, Y^d)$:

\snic {
\Diag(u_0, u_1, \ldots, u_d) = S U  \qquad \hbox {avec $U \in \MM_{a+b,d+1}(\gA)$}
}

En prenant le \deter et en appliquant Binet-Cauchy, on obtient
$u_0u_1\ldots u_d \in \cD_{d+1}(S)$, ce qui donne le résultat si
l'on se souvient que $d+1$, c'est $b$.


\exer{exoHilbertBurchXY} 
On considère la matrice  $\varphi$ ci-dessous dont les colonnes sont
des \syzys pour $(x^a, x^b y^c, y^d)$:

\snic {
\begin {array} {ccc}
\cmatrix {-x^{b-a}y^c & -y^d\cr 1 &0\cr 0 &x^a\cr},\quad &
\cmatrix {0 & y^d\cr -1 &0\cr x^by^{c-d} & -x^a\cr},\quad &
\cmatrix {-y^c & 0\cr x^{a-b} &-y^{d-c}\cr 0 &x^b\cr}. 
\\
\hbox{si }b \ge a &\hbox{si } c \ge d &\hbox{si } a \ge b \hbox { et } d \ge c
\end {array}
}

\snii
Notons $\Delta_1$, $\Delta_2$, $\Delta_3$ les mineurs (signés) d'ordre $2$
de $\varphi$. Dans les trois cas on obtient

\snic {
\Delta_1 = x^a, \quad \Delta_2 = x^by^c,\;\; \hbox{ et}\quad \Delta_3 = y^d.
}

\snii 
De plus, $\Gr_\gA(x^a, x^by^c, y^d) \ge 2$ car $x^a, y^d$
est une suite \ndze. D'après \gui{la réciproque du \tho d'Hilbert-Burch}
(i.e. le point \emph{2} du \thref{thHilBur1}),
la suite $0 \to \gA^2 \vvers {\varphi} \gA^3 \vvers {\mu} \gA$ est
exacte, et les deux colonnes de $\varphi$ forment une base de $\Ker\mu$.

\exer{exoHilbertBurchGeneric} \\
On considère l'ordre lexicographique ligne à ligne par exemple.
Alors les monômes dominants de~$\Delta_1$ et~$\Delta_n$ sont
le produit sur la \gui {sous-diagonale} et celui sur la \gui {diagonale}
i.e.

\snic {
\prod_{i=1}^{n-1} x_{i+1,i}, \qquad 
\prod_{i=1}^{n-1} x_{i,i}
}

\snii
Ils sont premiers entre eux. D'après  le \pb \ref{exoPolynomialSyzygies}
la suite~$(\Delta_1, \Delta_n)$ est \ndze. D'après le point \emph{2} du \thref{thHilBur1}, la suite
donnée est exacte.

\snii
Note: pour~$i < j$ et~$(i,j) \ne (1,n)$, les monômes dominants de
$\Delta_i$,~$\Delta_j$ ne sont pas premiers entre eux.
Mais il est fort probable que la suite~$(\Delta_i, \Delta_j)$
est tout de même \ndze.



\Biblio

\newpage \thispagestyle{empty}
\incrementeexosetprob


\chapter
{Résolutions projectives finies}\label{chapRProjF}
\perso{compilé le \today}

\minitoc

\sibook{

\Intro

Ce chapitre

\smallskip La section \ref{secPrelimRPF}  donne 

\smallskip La section \ref{secRSFSECO}  donne

\smallskip La section \ref{secCompareRSF}  donne 

\smallskip La section \ref{secCPROexact}  donne 

\smallskip La section \ref{secRSFRLF}  donne 

\smallskip La section \ref{secStrMult}  donne 

}


\section{Préliminaires} \label{secPrelimRPF}

\subsec{Dimension projective d'un produit fini}
Nous démontrons d'abord le résultat \gui{naturel} suivant de manière aussi directe que possible.

\begin{proposition} \label{PdimDirectSum} \label{corthcorlemGlaz1}
\label{factEE+P}
Pour $E$, $E'$ \mtfs, on a pour tout $n\geq -1$ les \eqvcs
\[ 
\begin{array}{ccc} 
\Pd(E) \le n  \hbox { et }\Pd(E') \le n  & \iff  &   
\Pd(E\oplus E') \le n\\[1mm] 
\Ld(E)\geq n\hbox{  et  }\Ld(E')\geq n  & \iff  &  
\Ld(E\oplus E')\geq n
\end{array}
\]
En abrégé:

\snic{\Pd(E\oplus E')=\sup\big(\Pd(E),\Pd(E')\big)$ et $\Ld(E\oplus G)= \inf\big(\Ld(E),\Ld(G)\big).}

\begin{enumerate}
\item Plus \prmt pour l'in\egt $\Pd(E)\leq \Pd(E\oplus E')$: si l'on a une \rsf de longueur $n \ge 1$:

\snic {
0\to P_n\vvers{u_n} P_{n-1}\vvvers{u_{n-1}} \cdots\cdots \vvers{u_1} P_0\vvers{u_0} 
E\oplus E' \to 0,
}

alors les modules $u_0^{-1}(E)$ et $u_0^{-1}(E')$ ont des \rsfs de longueur $n-1$, donc $E$ et $E'$ des \rsns de longueur $\leq n$.\\
En outre si les $P_i$ sont \stls, les \rsns obtenues le sont avec des
modules \stls.

\item Pour l'in\egt $\Ld(E)\geq \Ld(E\oplus E')$: si l'on a une $n$-\pn (avec \hbox{les $P_i$} libres finis)

\snic {
 P_n\vvers{u_n} P_{n-1}\vvvers{u_{n-1}} \cdots\cdots \vvers{u_1} P_0\vvers{u_0} 
E\oplus E' \to 0,
}

alors les modules $u_0^{-1}(E)$ et $u_0^{-1}(E')$ ont une $(n-1)$-\pn, 
donc~$E$ et~$E'$ des $n$-\pns.

\end{enumerate}
\end{proposition}

\rem Les cas $n=-1$ et $n=0$ sont clairs.\\
Les deux in\egts suivantes sont \imdes:

\snic{\Pd(E\oplus E')\leq \sup\big(\Pd(E),\Pd(E')\big)$
et $\Ld(E\oplus G)\geq  \inf\big(\Ld(E),\Ld(G)\big).}

Pour les in\egts opposées, nous aurons besoin du lemme suivant.
\eoe

\begin{lemma} \label{lemRaccourcirRsf}
On considère un module  $M$ avec une \rsf $(n\geq 1)$

\snic{0 \;\to\; P_{n}\; \lora \; \cdots \; \vvers{u_2} \; P_1 \; \vvers{u_1} \; P_0
\; \vvers\pi \; M\; \to\; 0
}

 et une suite exacte
$\;{ Q \vers v M \to 0}\;$
avec~$Q$ \ptf. 
\\
Alors $\Pd(\Ker v)\leq n-1$. Plus \prmt, le module $\Ker v$ admet une \rsf du type

\snic{0 \;\to\; P_{n}\; \lora \; \cdots \;\lora P_2\; \lora \; Q_1 \; \lora \;  \; \Ker v\; \to\; 0
}

En outre, si les $P_i$ et $Q$ sont \stls, $Q_1$
est \stl.
  
\end{lemma}
%
\begin{proof} Puisque $Q$ est projectif et $\pi$ surjective, on relève
$v$ en $v_0$ selon le schéma 
{\small $$\preskip.4em \postskip-.4em 
\xymatrix {
                                         & P_0\ar@{>>}[d]^{\pi} \\
Q\ar@{-->}[ur]^{v_0} \ar[r]_{v} & M \\
} 
$$
}
 
On considère alors le module $Q_1=\Ker(Q\times  P_1 
\vvvvers{\lst{v_0 \;u_1}~} P_0)$ et la suite
 
 \snic{
P_2 \; \vvers{\alpha} \; Q_1 \; \vvers{\pi_Q} \; Q
\; \vers {v} \; M\; \to\; 0\quad \quad (*)
}

où $\alpha(x)=\big(0,u_2(x)\big)\in Q_1$ (on a bien $\alpha(x)\in \Ker\lst{v_0 \;u_1}$), et $\pi_Q(y,z)=y$.

On note que $\Ker \alpha=\Ker u_2$.
\\
Montrons que $\lst{v_0 \;u_1}$ est surjective. Pour $y\in P_0$ on a un $x\in Q$ tel \hbox{que $v(x)=\pi(y)$}, d'où $u_0(y)=u_0\big(v_0(x)\big)$.
Donc $y-v_0(x)\in\Ker u_0=\Im u_1$ \hbox{et $y\in \Im v_0+\Im u_1$}.
\\
Comme $Q\times  P_1$ et $P_0$ sont \ptfs, et 
$\lst{v_0 \;u_1}$ est surjective,  le module $Q_1$ est \ptf.
\\
La variante \stl est \egmt claire.
\\
Il suffit donc de montrer que la suite $(*)$ est exacte.
\\
Un $(y,z)\in Q_1$ est dans
$\Ker\pi_Q$ \ssi $y=0$. 
\\
Dans ce cas $u_1(z)=0$, donc $z\in\Im u_2$
et si $z=u_2(x)$, on a $\alpha(x)=(y,z)$.
\\
Inversement tout $\alpha(x)$ est bien dans $\Ker\pi_Q$.
\\ 
Enfin, les \elts de $\Im\pi_Q$ sont les $y\in Q$ tels qu'il existe un
$z\in P_1$ vérifiant~\hbox{$v_0(y)=u_1(z)$}, \cad les $y\in Q$ tels que $v_0(y)\in\Ker\pi$, \cad les $y\in \Ker (\pi\circ v_0)$: on a bien $\Im\pi_Q=\Ker v$.  
\end{proof}
%

\begin{Proof}{\Demo de la proposition \ref{PdimDirectSum}. }
Les implications directes sont \imdes. 
On montre les implications réciproques
par \recu sur $n$. 
\\
On commence avec les dimensions projectives. Le cas $n = 0$ est immédiat. Supposons $n \ge 1$
et notons $\pi$, $\pi'$ les projections de $E\oplus E'$ sur $E$ et $E'$.
On considère l'\ali 

\snic {
v : \formule{P_0 \oplus P_0 &\lora& E\oplus E'\\[1mm]
  x \oplus y &\longmapsto& \pi\big(u_0(x)\big) \oplus \pi'\big(u_0(y)\big).}
}

On a $\Ker v=u_0^{-1}(E') \oplus u_0^{-1}(E)$.
%
%
Donc $u_0^{-1}(E')\oplus u_0^{-1}(E)$ admet une \rsf de longueur $n-1$ (lemme \ref{lemRaccourcirRsf}). 
Elle est formée de modules \stls si la \rsn de $E\oplus E'$ est formée de modules \stls. 
\\
Par \hdr,
$u_0^{-1}(E')$ et $u_0^{-1}(E)$ ont une \rsf de longueur $n-1$, et donc $E$ et
$E'$ ont une \rsf de longueur $n$.
\\
L'argument fonctionne à l'identique pour l'existence de
présentations de longueur $\geq n$ formées avec des \mptfs
plutôt que des modules libres finis. 
On conclut avec le lemme \ref{lempnptfpnlibre} qui nous dit que  $\Ld(E)\geq n$  \ssi le module $E$
admet une $n$-\pn projective \tf.
\end{Proof}

NB: on peut traiter tous les résultats concernant les $n$-\pns parallèlement aux résultats concernant les \rsfs de longueur $n$ en utilisant
l'argument qui vient d'être utilisé.
%

\subsec{Une \gnn du lemme de Schanuel}

Cette \gnn était déjà l'objet de l'exercice \Cref{V-4}.
Notez que, comme dans le lemme de Schanuel, les modules \pros $P_i$
et~$P'_i$ ne sont pas \ncrt \tf.

\CMnewtheorem{lemSchaGen}{Lemme de Schanuel \gne}{\itshape}
\begin{lemSchaGen}
\label{lemSchanuelVariation}
 On considère deux \sexs:%
\index{lemme de Schanuel!\gnn du ---}

\vspace{-1mm}
\[
\begin{array}{ccccccccccccccc}
0 &\to& K& \to & P_{n-1}& \to & \cdots & \to & P_1 & \vers{u} & P_0
& \to & M& \to& 0
\\
0 &\to& K'& \to & P'_{n-1}& \to & \cdots & \to & P'_1 & \vers{u'} & P'_0
& \to & M&  \to& 0
\end {array}
\]
avec les modules $P_i$ et $P'_i$ \pros. Alors on obtient un \iso:
\[
K \;\oplus \bigoplus_{k\in\lrb{0..n-1}\atop k \equiv n-1 \mod 2} P'_k 
\;\oplus
\bigoplus_{\ell\in\lrb{0..n-1}\atop \ell \equiv n \mod 2} P_\ell \;\simeq\;
K'\; \oplus \bigoplus_{k\in\lrb{0..n-1}\atop  k \equiv n-1 \mod 2} P_k 
\;\oplus
\bigoplus_{\ell\in\lrb{0..n-1}\atop \ell \equiv n \mod 2} P'_\ell
\]

\end{lemSchaGen}
%
\begin{proof}
Par \recu sur $n$, le cas $n=1$ étant exactement le lemme de 
Schanuel~\ref{FFRlemSchanuel}. A partir de chaque \sex, on en
construit une autre strictement plus courte

\vspace{-.3em}
\snac {\!\!
\begin{array}{ccccccccccccc}
0 &\rightarrow& K& \to & P_{n-1}& \to & \cdots & \to &
P_1 \oplus P'_0 &   \vvvvers{u \oplus \Id_{{P_0}'}}
& \Im u \oplus P'_0 & \to & 0
\\
0 &\rightarrow& K'& \to & P'_{n-1}& \to & \cdots & \to &
P_1' \oplus P_0 &  \vvvvers{u' \oplus \Id_{P_0}}
& \Im u' \oplus P_0 & \to & 0
\end {array}
}

\snii
Mais on a $\Im u \oplus P'_0 \simeq \Im u' \oplus P_0$ d'après le lemme de
Schanuel appliqué aux deux \secos:

\snic {
\begin{array}{ccccccccc}
0 &\to& \Im u & \to & P_0 & \to & M & \to & 0, \\[1mm]
0 &\to& \Im u' & \to & P'_0 & \to & M & \to & 0.
\end {array}
}

\snii
On peut donc appliquer la \recu (aux deux  \sexs de
longueur un de moins), ce qui fournit le résultat demandé.
\end{proof}
%

\subsec{Permanence}

\entrenous{{ ce qui suit semble surpassé par la proposition \ref{PdimDirectSum}. \`A moins qu'on ait besoin du point \emph{1}? }

\rm On sait que si $P$ est \ptf alors $E$ est \ptf \ssi $E\oplus P$
est \ptf. Une première \gnn est donnée dans le point \emph{2} du lemme suivant.

%
{\bf Lemme.} {\it 
\begin{enumerate}
\item On considère une suite exacte

\snic{P_1\vers{u_1} P_0\vers{\pi} E\oplus P\to 0}

 avec $P$ \pro. 

On peut écrire $P_0=Q_0\oplus P' $ où   $Q_0=\pi^{-1}(E)$ et $\pi\frt{P'}$ réalise un \iso de $P'$ sur $P$. En outre, puisque $\Im u_1\subseteq Q_0$,
en notant $v_1$ la restriction de~$u_1$  
à  $P_1$ et $Q_0$, et $\pi_0$ la restriction de $\pi$ à $Q_0$ et $E$, la suite

\snic{P_1\vers{v_1} Q_0\vers{\pi_0} E\to 0}

est exacte, avec $\Ker v_1=\Ker u_1$.
\item 
Si $P$ est un \mptf et $k\geq 0$,
on a $\Pd(E\oplus P)\leq k$
\ssi $\Pd(E)\leq k$. En abrégé: $\Pd(E\oplus P)=\Pd(E)$ pour tout module $E\neq 0$.

\item Si $P$ est un module libre fini et $k\geq 0$,
le module~\hbox{$E\oplus P$} admet une \rsn \stl de longueur $k$
\ssi $E$ admet une \rsn \stl de longueur $k$.
\end{enumerate}
}
%
\begin{proof}
\emph{1.} Puisque $P$ est \pro, l'application surjective composée 

\snic{P_0\vers{\pi} E\oplus P\to P}

\snii
admet une section $\sigma$. On prend $P'=\sigma(P)$. Les \vfns sont laissées \alec.

\emph{2.} Le \gui{si} est clair. Pour le \gui{seulement si}, à partir d'une \rsf de longueur $k$ de $E\oplus P$ par des \mptfs $P_i$, on obtient une \rsn de $E$ en modifiant seulement $u_1$, $P_0$ et $\pi$ comme indiqué dans le point~\emph{1}. Et $Q_0$ est \ptf car facteur direct dans $P_0$.

\emph{3.} On raisonne comme au point \emph{2} Cette fois-ci le module
$Q_0$ qui remplace le module libre fini $P_0$ est \sul dans $P_0$ de $\sigma(P)$, qui est libre fini. Donc $Q_0$ est \stl. 
\end{proof}
}

\begin{corollary} \label{corlemSchanuelVariation} \emph{(\Tho de permanence, 1)}\\
On reprend les hypothèses du lemme de Schanuel \gne \ref{lemSchanuelVariation}.
\begin{enumerate}
\item Supposons que les $P_i$ et $P'_i$ sont \ptfs.
\begin{enumerate}
\item  $K$ est \tf, \pf ou plat \ssi $K'$  a la même \prt.
\item  $\Pd(K)\leq n$  \ssi $\Pd(K')\leq n$.
\end{enumerate}
\item Si les $P_i$ et $P'_i$ sont libres finis, $K$ est \stl \ssi $K'$ est \stl.
\end{enumerate}
 
\end{corollary}
%
\begin{proof}
Dans le premier cas,
on a $K\oplus P\simeq K'\oplus P'$ avec $P$ et $P'$ \ptfs, dans le second
cas, on a un \iso similaire avec $P$ et $P'$ libres.
Pour \emph{1b}, on applique  la proposition~\ref{factEE+P}.
\\
NB: il faut noter que dans  le point \emph{2}, si $K$ est libre fini,
on obtient seulement que $K'$ est \stl, donc qu'il admet une \rlf
de longueur~$\leq 1$. 
\end{proof}

On sait que si un \Amo $M$ est \pf, les \syzys pour n'importe quel \sgr de $M$ forment un \mtf (\Cref{section IV-1}).
On retrouve ce résultat pour les $n$-\pns.
C'est le point \emph{1} du \tho suivant. Les points \emph{2} et \emph{3}
concernent des résultats analogues 
pour les \rsfs et pour les \rlfs  

\begin{theorem} \label{corLScha} \emph{(\Tho de permanence, 2)}\\
Soient $r$, $s\in\NN$, $m=r+s+1$   et $M$ un \Amo. 
\begin{enumerate}
\item Si $\Ld_\gA(M)\geq m$,  toute $r$-\pn
de~$M$ 

\snic{ L_{r} \vers{u_r}  L_{r-1} \lora  \cdots\cdots \lora    L_0
 \lora  M \to 0}

peut être prolongée en une  $m$-\pn. 
\\
Autrement dit, en notant~$K=\Ker(u_r)$, on \hbox{a $\Ld_\gA(K)\geq s$}.
\item Si $\Pd_\gA(M)\leq m$, 
tout début de \rsp de $M$ de longueur~$r$

\snic{ P_{r} \vers{v_r}  P_{r-1} \lora  \cdots\cdots \lora    P_0
 \lora  M \to 0}

avec les $P_i$ \tf peut être prolongée en une  \rsf de longueur $m$. \\
Autrement dit, en notant~$K=\Ker(v_r)$, on \hbox{a $\Pd_\gA(K)\leq s$}.
\item Si $M$ admet une \rsn \stl de longueur $m$, tout début de \rsn \stl de $M$ de longueur~$r$ peut être prolongé en une  \rsn \stl de longueur $m$.\\
En particulier, si $M$ admet une \rlf  de longueur $m$, toute $r$-\pn de $M$ 
peut être prolongée en une  \rlf de longueur $m$ ou $m+1$. 
\end{enumerate}
 
\end{theorem}
%
\begin{proof}
\emph{2.} On considère les deux débuts de \rsns projectives de longueur $r$. Le lemme de Schanuel \gne
\ref{lemSchanuelVariation} s'applique et l'on obtient un \iso $K'\oplus P\simeq K\oplus P'$ avec $P$ et $P'$ \ptfs Si $K$ admet une \rsp
 de longueur $s$, il en va de même pour $K\oplus P'$, isomorphe à
 $K'\oplus P$,  donc aussi pour $K'$
par la proposition~\ref{factEE+P}.

\emph{1.} Même raisonnement.

\emph{3.}  Même chose. 
Pour la dernière remarque, on fait appel à 
la proposition~\ref{cor2lemModifComplexe} 
pour transformer une \rsn  \stl finie en une \rlf.\\
NB. Normalement on obtient une \rlf de longueur $m+1$. Dans certains cas,
par exemple si tous les modules \stls sont libres, on se ramène à une \rlf de longueur $\leq m$.
\end{proof}
%

%
%
%
%
%
%
%
%

\subsect{Changement d'anneau de base par extensions plates et \fptes}{Extensions plates et \fptes}

\begin{theorem} \label{propRLFRLF} Soit $E$ un \Amo et $\rho:\gA\to\gB$ une \Alg plate.  
\begin{enumerate}
\item     
\begin{enumerate}
\item Si  $E$ admet une \rsf, elle donne par \eds à $\gB$ une \rsf
du \Bmo $\rho\ist(E)$. En particulier, $\Pd_\gB(\rho\ist(E))\leq \Pd_\gA(E)$.
\item Si  $E$ admet une $n$-\pn, elle donne par \eds à $\gB$ une \rsf
du \Bmo $\rho\ist(E)$.  En particulier, $\Ld_\gB(\rho\ist(E))\geq \Ld_\gA(E)$.
\end{enumerate}
\item Supposons $\gB$ \fpte. 
\begin{enumerate}
\item Si  $E$ admet une \rsf après \eds à $\gB$, alors il
admet  une \rsf de même longueur comme \Amo. 
En particulier, $\Pd_\gB(\rho\ist(E))= \Pd_\gA(E)$.
\item Si  $E$ admet une $n$-\pn après \eds à~$\gB$, alors il
admet  une $n$-\pn comme \Amo. \\
En particulier,  $\Ld_\gB(\rho\ist(E))= \Ld_\gA(E)$.
\end{enumerate} 
\end{enumerate}
\end{theorem}
%
\begin{proof} \emph{1.} Le changement d'anneau de base conserve les suites exactes, il transforme un \mptf (resp. un \mlrf) en un \mptf  (resp. un \mlrf).
 
%
%

\emph{2b.} On connait déjà le résultat pour $n=0$ et $n=1$,
et l'on fait une \recu sur $n$. Supposons 
que $\rho\ist(E)$ est $(n+1)$-\pfb avec $n\geq 1$. Par \hdr, $E$ est
$n$-\pfb. Considérons dans la suite exacte correspondante l'\Ali $L_n\vers{u_n}L_{n-1}$. D'après le \tho de permanence, on obtient pour cette $n$-\pn, après \eds  à $\gB$, le fait que $\Ker\big(\rho\ist(v_n)\big)$
est un \mtf. Or $\Ker\big(\rho\ist(v_n)\big)\simeq\rho\ist\big(\Ker(v_n)\big)$ parce que $\gB$ est plate,  et puisque $\gB$ est \fpte, $\Ker(v_n)$ est aussi \tf. 

\emph{2a.} D'après \emph{2b}, on sait que $E$ est $(n-1)$-\pfb.
On considère le noyau~$K$ de la première \ali de la suite exacte correspondante. Par le \tho de permanence $\rho\ist(K)$ est \ptf.
Puisque $\gB$ est \fpte, $K$ est aussi \ptf.        
\end{proof}
%

\subsec{Principe \lgb}

\begin{plcc} \label{plcc.resf}  
\emph{(Modules \lonrsbs, modules $n$-\pfbs)}
Soient $E$ un \Amo  et $S_1$, \dots, $S_n$ des \moco.
On note $\gA_i=\gA_{S_i}$ et $M_i=M_{S_i}$.
\begin{enumerate}
\item Le module $E$ est \lonrsb \ssi il est \lonrsb  après \eds à chacun des
$\gA_i$. 
\\
En notation abrégée: $\Pd_\gA(E)=\sup_{i\in\lrbn}\big(\Pd_{\gA_i}(E_i)\big)$.
\item  Le module $E$ est $n$-\pfb \ssi il est $n$-\pfb  après \eds à chacun des $\gA_i$. 
\\
En notation abrégée: $\Ld_\gA(E)=\inf_{i\in\lrbn}\big(\Ld_{\gA_i}(E_i)\big)$.
\end{enumerate}
\end{plcc}
%
\begin{proof}
Le point \emph{1} du \thref{propRLFRLF} donne comme cas particulier:
si $E$ est \lonrsb (resp. $n$-\pfb), $E_i$ est \lonrsb (resp. $n$-\pfb) comme $\gA_i$-module.

Pour les réciproques, on ne peut pas invoquer directement le point \emph{2}
 du \thref{propRLFRLF}, mais on peut \gui{recopier} les \dems.
 Voici par exemple ce que donne la réécriture du \emph{2b} pour le cas
 des $n$-\pns. \\
On connait déjà le résultat pour $n=0$ et $n=1$,
et l'on fait une \recu sur $n$. Supposons 
que chaque $E_i$ est $(n+1)$-\pfb avec $n\geq 1$. Par \hdr, $E$ est
$n$-\pfb. Considérons dans la suite exacte correspondante l'\Ali $L_n\vers{u_n}L_{n-1}$. D'après le \tho de permanence, on obtient pour cette $n$-\pn, après \eds  à $\gA_i$, le fait que $\Ker\big({\rho_i}\ist(v_n)\big)$
est un \mtf. Or $\Ker\big({\rho_i}\ist(v_n)\big)\simeq{\rho_i}\ist\big(\Ker(v_n)\big)$ parce que $\gA_i$ est plate.  Et puisque chaque ${\rho_i}\ist\big(\Ker(v_n)\big)$ est \tf, $\Ker(v_n)$ est aussi \tf (\plg pour les \mtfs). 
\end{proof}

On rappelle par ailleurs que le lemme \ref{lemlorsbloclresb} 
offre \gui{réciproque forte} pour le \plg pour les modules \lorsbs.

\subsec{Classe dans $\KO\gA$  d'un \Amo \lorsb}

Rappelons qu'un \elt $\rho$ de $\HO\gA$ s'écrit $\sum_{j=1}^{m}\rho_j[e_j]$
avec $\rho_1<\dots<\rho_m$ dans $\ZZ\setminus\so0$ et des \idms $e_j$ deux à deux \orts. Et un tel \gui{rang \gne} $\rho$ est dit $\geq 0$ lorsque les $\rho_j$ sont $>0$,
ce qui signifie que~$\rho$ est le rang du module quasi libre $\bigoplus_{j=1}^{m} (e_j\gA)^{\rho_j}$.

Si un module \pro $M$ est de rang $\rho$, on a alors $M=\bigoplus_{j=1}^{m}e_jM$
et chaque $e_jM$ est de rang constant $\rho_j$ sur $\gA[1/e_j]$.
En outre $e_0=1-\sum_{j=1}^{m}e_j$ est l'\idm annulateur de $M$.
 Si $F$ est une \mprn dont l'image est isomorphe à $M$,
on a l'\egt $\det(\rI+XF)=\sum_{j=0}^{m}e_j (1+X)^{\rho_j}$ (en posant\hbox{ $\rho_0=0$}).

Comme conséquence  du lemme de Schanuel \gne \ref{lemSchanuelVariation}, on obtient le \tho suivant.

\begin{thdef} \label{thSchanuelVariation} 
On considère une \rsf de longueur $n$ pour un \Amo \lorsb $M$ 
$$
0 \to P_n  \lora  \cdots \cdots  \lora  P_1 \lora P_0 \lora M
\to 0.
$$
\begin{enumerate}
\item Pour toute suite exacte où les $P'_i$ sont \ptfs

\snic {
0 \to Q \lora  P'_{n-1} \lora  \cdots \cdots  \lora  P'_1  \lora  P'_0 \lora M
\to 0,
}

\snii
le module $Q$ est \egmt \ptf. 

\item Pour toute  autre \rsf de $M$

\snic {0   \to   P'_{m}  \vvers{u'_m}    \cdots   \cdots  
\lora   P'_1   \vvers{u'_1}   P'_0   \vvers{u'_0}   M   \to 0,}

\snii on obtient  

$$\!\!\!{
 \bigoplus_{k\in\lrb{0..m}\atop k \equiv 0 \mod 2} P'_k 
\;\oplus
\bigoplus_{\ell\in\lrb{0..n}\atop \ell \equiv 1 \mod 2} P_\ell \;\simeq\;
 \bigoplus_{k\in\lrb{0..n}\atop k \equiv 0 \mod 2} P_k 
\;\oplus
\bigoplus_{\ell\in\lrb{0..m}\atop \ell \equiv 1 \mod 2} P'_\ell
}.
$$
On peut alors définir sans ambig\"uité la \emph{classe $[M]$ dans $\KO(\gA)$ du module $M$} en posant 
  $$
  {[M]=[P_0]-[P_1]+[P_2]-\cdots}
  $$
ainsi que le \emph{rang de $M$} en posant $\rg(M)\eqdefi \rg([M])\in\HO(\gA)$, ce qui donne
$$\rg(M)=\rg(P_0)-\rg(P_1)+\rg(P_2)-\dots.$$
\item Par une \eds plate, le module $M$ reste \lorsb. En outre la classe
$[M]$ et le rang $\rg(M)$ sont transformés par la même \eds.
\item Si $\rg(M)=\rho$, en reprenant les notations avant le \tho,
on obtient 

\snic{M=\bigoplus_{j=1}^{m}e_jM}

\snii
et chaque $e_jM$ est de rang $\rho_j\in\NN$ sur l'anneau 
$\gA[1/e_j]$\footnote{Cela généralise donc
ce qui se passe avec un \mptf, mais avec la notion de rang étendue aux modules \lorsbs.}.
\end{enumerate}
\end{thdef}
NB. Lorsque $M$ est \lrsb on retrouve le rang de $M$ défini en~\ref{thStrLocResFin},
et  $[M]$ est la classe du module libre de rang $\rg(M)$.\\
De même si $M$ est \ptf, les \dfns de $[M]$  et  $\rg(M)$ sont non ambig\"ues. 
\eoe

\hum{Est-ce que $\Ann(M)=e_0\gA$?}
\begin{proof}
\emph{1.} Cas particulier du corolaire \ref{corlemSchanuelVariation}.

\emph{2.} Si $n<m$ on prolonge le complexe $P\ibu$ 
 par des $0$ à gauche. 
 On peut donc supposer $m=n$. On applique alors le lemme \ref{lemSchanuelVariation}.
 
 \emph{3.} Par \eds plate la \rsf reste une \rsf.

  \emph{4.} Vu le point \emph{3}, cela résulte du cas où le module est \lrsb, car après \lon en des \eco convenables, la \rsn est formée de \mlrfs.
\end{proof}

 Un module $M=P_1/ P_2$, avec $P_1$ et $P_2$ \pros de même rang,  est de rang nul.

Concernant les classes $[M]$, notez que l'on a  $[M\oplus M']\simeq[M]+ [M']$ dans~$\KO(\gA)$ et \hbox{ $\rg(M\oplus M')=\rg(M)+\rg(M')$} dans~$\HO(\gA)$
lorsque $[M]$ et~$[M']$ sont définis. Plus \gnlt on  le résultat suivant.

\begin{theorem} \label{lemPMP/M0}
Soit $E$ un sous-\Amo du module $F$. 
\\
Si deux des modules $F$, $E$ et $F/E$ admettent une \rsf, il en va de même pour le troisième.
En outre  $[F]=[E]+[F/E]$ dans $\KO(\gA)$    et  
$\rg(F)=\rg(E)+\rg(F/E)$ dans $\HO(\gA)$.
\end{theorem}
\begin{proof}
La première affirmation est donnée par le \thref{lemPMP/M}, la deuxième résulte facilement du point \emph{a} du même \tho.
\end{proof}

\section{Résolutions et suites exactes courtes}\label{secRSFSECO}

\subsec{Deux lemmes utiles}

\begin{lemma} \label{lemGlaz1}
On considère cinq \Amos $E$, $F$, $G$,  $E_0$ et  $G_0$ qui s'insèrent dans le diagramme ci-dessous, où la suite horizontale est exacte 
$$
\xymatrix@C=3em@R=1.6em {
&E_0\ar[d]_{e_0}  
& &G_0\ar[d]_{g_0}
\\
&E \ar[r]^{\iota} &F\ar[r]^{\pi } &G\ar[r]&0
\\
}
$$
Si le module $G_0$ est \pro le diagramme précédent peut être complété
en un diagramme commutatif (pour les flèches en traits pleins) ci-dessous où la nouvelle suite horizontale est exacte: 
$$
\xymatrix@C=3em@R=1.6em {
0\ar[r]&E_0\ar[d]_{e_0} \ar[r]^(.4){\iota_0} 
&E_0\times G_0\ar[d]_{f_0}\ar[r]^(.6){\pi_0} &
G_0\ar@{..>}[ld]_{\sigma_0}\ar[d]_{g_0}\ar[r]&0
\\
&E \ar[r]^{\iota} &F\ar[r]^{\pi } &G\ar[r]&0
\\
}
$$
Ici $\iota_0$ et $\pi_0$ sont les applications canoniques,
et 
${f_0(x,y)=\iota\big(e_0(x)\big)+\sigma_0(y) ,}$
où $\sigma_0:G_0\to F$
est une \ali qui satisfait $\pi \circ \sigma_0=g_0$
(elle existe parce que $G_0$ est \pro et $\pi $ est surjective,
mais elle n'est pas déterminée de manière unique).\\ 
Enfin, si $e_0$ et $g_0$ sont surjectives, il en va de même pour $f_0$.
\end{lemma}
%
\begin{proof}
Tout est expliqué dans l'énoncé. Les \vfns concernant la commutativité et les \sexs sont laissées \alec.
\end{proof}
%

\begin{lemma} \label{lemGlaz2} 
On considère six \Amos $E$, $F$, $G$,  $E_0$,  $F_0$ et  $G_0$ qui s'insèrent dans le diagramme ci-dessous, où les suites horizontales et les suites verticales sont exactes 
$$
\xymatrix@C=3em@R=1.6em {
0\ar[r]&E_0\ar[d]_{\pi_E} \ar[r]^(.4){\alpha_0} 
&F_0\ar[d]_{\pi_F}\ar[r]^(.6){\beta_0} &
G_0\ar[d]_{\pi_G}\ar[r]&0
\\
0\ar[r]&E\ar[d] \ar[r]^{\alpha} &F\ar[d]\ar[r]^{\beta } &G\ar[d]\ar[r]&0
\\
&0  &  0  & 0
\\
}
$$
Le diagramme précédent peut être complété
en un diagramme commutatif  où les suites horizontales et  verticales sont exactes 
$$
\xymatrix@C=3em@R=1.6em {
&0\ar[d]  &  0\ar[d]  & 0\ar[d]
\\
0\ar[r]&E'_0\ar[d]_{\iota_E} \ar[r]^{\alpha'_0} 
&F'_0\ar[d]_{\iota_F}\ar[r]^{\beta'_0} &
G'_0\ar[d]_{\iota_G}\ar[r]&0
\\
0\ar[r]&E_0\ar[d]_{\pi_E} \ar[r]^{\alpha_0} 
&F_0\ar[d]_{\pi_F}\ar[r]^{\beta_0} &
G_0\ar[d]_{\pi_G}\ar[r]&0
\\
0\ar[r]&E\ar[d] \ar[r]^{\alpha} &F\ar[d]\ar[r]^{\beta } &G\ar[d]\ar[r]&0
\\
&0  &  0  & 0
\\
}
$$
\end{lemma}
\facile
NB: on peut prendre $E'_0\subseteq E_0$, $F'_0\subseteq F_0$ et $G'_0\subseteq G_0$, et les flèches~$\alpha'_0$ et~$\beta'_0$ sont déterminées de manière unique. 

\bonbreak
\subsec{Suites exactes courtes de modules $n$-présentables}

\begin{theorem} \label{corlemGlaz1}\label{thSecoRlf0}
On considère une \seco de \Amos 
$${0\lora E\vvers{\iota}F\vvers{\pi}G\lora 0.}
$$
\begin{enumerate}
\item 
Si $E$ et $G$ sont $n$-\pfbs, il en va de même de $F$.
En outre on peut s'arranger pour que les trois~$n$-\pns s'insèrent dans un diagramme commutatif avec une \seco
de complexes comme dans le dessin ci-après.\\
Ici, les \sexs $0\to E_k\to F_k\to G_k\to 0$ sont scindées\footnote{Plus précisément, on peut prendre $F_k=E_k\oplus G_k$ avec les applications canoniques \hbox{pour $E_k\to E_k\oplus G_k$} et $E_k\oplus G_k\to G_k$.}, les modules $E_k$ et $G_k$ sont \lrfs, et les $n$-\pns de~$E$ et $G$ sont celles données au départ. 
\\
Notons $E'_k$, $F'_k$ et $G'_k$ les noyaux des flèches verticales 
de sources~\hbox{$E_k$, $F_k$} et $G_k$. Alors chaque suite
$0\to E'_k\to F'_k\to G'_k\to 0$ est exacte.
 
$$
\xymatrix@C=3em@R=1.6em {
0\ar[r]&E_n\ar[d] \ar[r]^{\iota_n} 
&F_n\ar[d]\ar[r]^{\pi_n} &
G_n
\ar[d]
\ar[r]&0
\\
0\ar[r]&E_{n-1} \ar[r]^{\iota_{n-1}}\ar@{--}[d] 
&F_{n-1}\ar[r]^{\pi_{n-1}}\ar@{--}[d] &
G_{n-1}\ar[r]\ar@{--}[d]&0
\\
0\ar[r]&E_1\ar[d] \ar[r]^{\iota_1} 
&F_1\ar[d]\ar[r]^{\pi_1} &
G_1
\ar[d]
\ar[r]&0
\\
0\ar[r]&E_0\ar[d] \ar[r]^{\iota_0} 
&F_0\ar[d]\ar[r]^{\pi_0} &
G_0
\ar[d]
\ar[r]&0
\\
0\ar[r]&E\ar[d] \ar[r]^{\iota} &F \ar[d]\ar[r]^{\pi} &G\ar[d]\ar[r]&0
\\
&0  &  0  & 0
\\
}
$$
\item En particulier,  si $E$ et $G$ sont 
\lnrsbs, il en va de même pour $F$,
et les trois \rsns peuvent être insérées dans un diagramme commutatif où les \rsns de $E$ et $G$ sont données.
\item Les résultats analogues s'appliquent avec des \gui{$n$-\pns} de $E$ et $G$ formées par des \mptfs
En particulier,  si~$E$ et~$G$ sont \lot $n$-\rsbs, il en va de même pour $F$, et les trois \rsns peuvent être insérées dans un diagramme commutatif  où les \rsns de $E$ et $G$ sont données.

\end{enumerate} 
\end{theorem}
%
\begin{proof} \emph{1.}
Le lemme \ref{lemGlaz1} traite le cas $n=0$.\\
Le lemme \ref{lemGlaz2}  dit que la suite $0\to E'_0\to F'_0\to G'_0\to 0$ est exacte.
 On se retrouve alors dans la situation du lemme \ref{lemGlaz1} avec le diagramme suivant
$$
\xymatrix@C=3em@R=1.6em {
&E_1\ar[d]  
& &G_1\ar[d]
\\
0\ar[r]&E'_0 \ar[r] &F'_0\ar[r] &G'_0\ar[r]&0,
\\
}
$$
où les flèches verticales sont surjectives,
ce qui permet de monter l'étage 1.\\
On continue de la même manière jusqu'à l'étage $n$.

\emph{2.} 
On applique le point \emph{1} du \tho mais avec
l'entier $n+1$
et les modules $E_{n+1}=G_{n+1}=0$.

\emph{3.} On remplace partout \gui{\lrf} par \gui{\ptf}. 
\end{proof}

On en déduit le \tho suivant.

\begin{theorem} \label{thcorlemGlaz1} 
Soit $n\in\NN$. On considère une \seco de \Amos 
$$
0\lora E\lora F\lora G\lora 0.
$$
Alors on a les implications suivantes pour tout $n\in\NN$.
\begin{enumerate}
\item $\big(\Ld(F)\geq n$ et $\Ld(G)\geq n+1\big)$ $\Longrightarrow$ $\Ld(E)\geq n$.
\item $\big(\Ld(E)\geq n$ et $\Ld(G)\geq n\big)$ $\Longrightarrow$ $\Ld(F)\geq n$.
\item $\big(\Ld(E)\geq n$ et $\Ld(F)\geq n+1\big)$ $\Longrightarrow$ $\Ld(G)\geq n+1$.
\end{enumerate}
En notation abrégée avec $\lambda_E=\Ld(E)$,  $\lambda_F=\Ld(F)$,  $\lambda_G=\Ld(G)$:

\fnic{\lambda_E\geq \inf (\lambda_F,\lambda_G-1), \;
\lambda_F\geq \inf (\lambda_E,\lambda_G), \;
\lambda_G\geq \inf (\lambda_E+1,\lambda_F)\,.}

\end{theorem}
%
\begin{proof}
Le point \emph{2} est donné par le \thref{corlemGlaz1}.
Les points \emph{1}  et  \emph{3} se montrent par \recu sur $n$ comme suit.

\emph{3.} Pour $n=0$: le quotient d'un \mpf par un sous-\mtf est \pf
(\ref{FFRpropPfSex} point~\emph{1}). Supposons l'implication vraie pour $n-1\geq 0$,
et supposons que $\Ld(E)\geq n$ et $\Ld(F)\geq n+1$.
Par \hdr on  sait déjà que $\Ld(G)\geq n$. Le \thref{corlemGlaz1}
nous donne trois $(n-1)$-\pns qui forment une \seco de complexes.
On obtient à l'étage $n-1$ une \sex avec les noyaux des flèches 
verticales qui forment une \seco  

\snic{0\to E'_{n-1}\lora F'_{n-1}\lora G'_{n-1}\to 0.}

\snii
Par le  \thref{corLScha}, $E'_{n-1}$ est \tf et $F'_{n-1}$
est \pf. Le cas $n=0$ nous donne que $G'_{n-1}$
est \pf.

\emph{1.} Pour $n=0$: c'est la proposition \ref{FFRpropPfSex}, point~\emph{2}.
La suite du raisonnement est identique à celle du point \emph{3}. 
\end{proof}
%

\begin{corollary} \label{cor0thcorlemGlaz1}  On considère une \seco de \Amos 

\snic{0\lora E\lora F\lora G\lora 0.}

\snii
Si deux des trois modules admettent une \pn de longueur infinie,
il en va de même pour le troisième, et l'on peut organiser les trois \pns en une \seco de complexes comme dans le  \thref{corlemGlaz1}.
\end{corollary}

%
%
%
%

\subsec{Suites exactes courtes de modules \lorsbs}

On commence par un lemme avec des \pdis petites.
Il est de toute manière utile pour comprendre la situation \gnle.

\begin{lemma} \label{lemEFG} On considère une \seco $0\to E\to F \to G\to 0$, et $r\in\NN$.
\begin{enumerate}
\item Si $\Pd(F)\leq 0$, alors $\Pd(G)\leq r+1 \iff \Pd(E)\leq r$.\\
Cas particuliers avec précisions.
\begin{enumerate}
\item Si $F$ est libre et $E$ admet une \rlf de longueur~$n$, alors 
 $G$ admet une \rlf de longueur $n+1$.
\item Si $F$ est libre et $G$ admet une \rsn \stl de longueur $n+1$, alors 
 $F$ admet une \rsn \stl de longueur $n$. 
\end{enumerate}

\item Si $\Pd(E)\leq r$ et $\Pd(G)\leq r$,  alors $\Pd(F)\leq r$.\\
Cas particulier avec précisions.
Si $E$ et $G$ admettent une \rlf (resp. une \rsn \stl finie) de longueur $n$, il en va de même pour $F$.
\item Si $\Pd(E)\leq 0$ et $\Pd(F)\leq 1$, alors $\Pd(G)\leq 1$.\\
Cas particulier avec précisions.
 Si $E$ est libre et $F$ admet une \rsn \stl de longueur $1$, alors 
 $G$ admet une \rsn \stl de longueur $1$.

\end{enumerate} 
\end{lemma}
%
\begin{proof}
\emph{1.} Comme $F$ est \ptf,  une \rsf de $E$  de longueur $r$
fournit  une \rsf de $G$  de longueur~\hbox{$r+1$}. Dans l'autre sens, on applique le point \emph{2} ou \emph{3} du \thref{corLScha}.

\emph{2.} 
Voir le \thref{corlemGlaz1}.

\emph{3.} Tout d'abord le cas \gnl avec des \mptfs.
Par hypothèse $F\simeq F_0/F_1$ avec $F_0$ et $F_1$ \ptfs. 
Le sous-module~$E$ de $F$ s'écrit $E_1/F_1$ où $E_1$ est un sous-module de $F_0$ qui contient~$F_1$.
La \seco $0\to F_1\to E_1\to E\to 0$ implique d'après le point~\emph{2} 
que~$E_1$ est \ptf. 
Donc $\Pd(G)\leq 1$ puisque $G\simeq F_0/E_1$.\\
Voyons ensuite le cas particulier. Ici $E$ et $F_0$ sont libres, et $F_1$
\stl. On en déduit que $E_1$ est \stl. Donc  $G$ admet une \rsn \stl de longueur $1$ puisque $G\simeq F_0/E_1$.
\end{proof}
%


Voici un analogue du \thref{corlemGlaz1}.
\begin{theorem} \label{lemPMP/M} \emph{(Suite exacte courte de modules \lorsbs)}\\
Soit $E$ un sous-\Amo d'un module $F$. Si deux parmi les trois modules~\hbox{$F$, $E$} \hbox{et $G=F/E$} sont \lrsbs (resp. \lorsbs), il en va de même pour le troisième et l'on a les résultats suivants.
\begin{enumerate}
\item [a.]  Pour un entier $m$ convenable, on a une suite exacte courte de complexes acycliques, avec des modules libres finis (resp. des \mptfs) $E_i$, $F_i$, $G_i$, du type suivant: 

\vspace{1.5mm}

\snac{ \!\!\!\!
\begin{array}{ccccccccccccccccccccc} 
 && 0 &&0 && 0 &&0 && 0 \\[1mm] 
 && \downarrow && \downarrow &&\downarrow && \downarrow&& \downarrow \\[1mm] 
0 &\to& E_ m &\lora& E_{ m-1} &  \cdots   \cdots & E_ 1 &\lora& E_ 0 &\lora&E& \to &0
\\[1.5mm] 
 && \downarrow && \downarrow &&\downarrow && \downarrow&& \downarrow \\[1.5mm] 
0 &\to& F_ m&\lora& F_{ m-1} &  \cdots   \cdots & F_ 1 &\lora& F_ 0 &\lora &F& \to &0
\\[1.5mm] 
 && \downarrow&& \downarrow && \downarrow &&\downarrow &&
 \downarrow \\[.5mm] 
0 &\to& G_ m&\lora& G_{ m-1}&    \cdots   \cdots & G_ 1& \lora& G_ 0 &\lora& G&
\to& 0\\[1mm] 
 & & \downarrow && \downarrow &&\downarrow && \downarrow&& \downarrow \\[1mm] 
 & & 0 && 0 &&0 && 0&& 0   
\end{array} 
}

\sni Ici, les  \rsns libres (resp. projectives) de $E$ et $G$ peuvent être imposées.
Mais bien que chaque suite exacte
$0\to E_i\lora F_i\lora G_i\to 0$ soit scindée\footnote{Plus précisément, on peut prendre $F_i=E_i\oplus G_i$ avec les applications canoniques \hbox{pour $E_i\to E_i\oplus G_i$} et $E_i\oplus G_i\to G_i$.},
le complexe  du milieu n'est pas en \gnl la simple somme directe
des deux autres complexes.
\hum{Cette remarque devrait être sortie du \tho et faire l'objet d'un exemple, spontanément elle a l'air fausse: à cause de la commutativité des diagrammes et du scindage, on dirait qu'il n'y a pas le choix pour le complexe du milieu???}

\item  [b.] \emph{(Dimensions projectives)} \\
En notant $p_E=\Pd_\gA(E)$, $p_F=\Pd_\gA(F)$, $p_G=\Pd_\gA(F/E)$, on obtient
les inégalités suivantes:\label{ineqpEFG} 

\fnic{p_E\leq \sup(p_F,p_G-1), \quad p_F\leq \sup(p_E,p_G) ,\quad 
p_G\leq \sup(1+p_E,p_F).}

\snii Autrement dit on a les implications suivantes pour $n\in\NN$:
\begin{itemize}
\item $\big(\Pd_\gA(F)\leq n$ et $\Pd_\gA(G)\leq n+1\big) \;\Longrightarrow\;\Pd_\gA(E)\leq n$,
\item $\big(\Pd_\gA(E)\leq n$ et $\Pd_\gA(G)\leq n\big) \;\Longrightarrow\;\Pd_\gA(F)\leq n$,
\item  $\big(\Pd_\gA(E)\leq n$ et $\Pd_\gA(F)\leq n+1\big) \;\Longrightarrow\;\Pd_\gA(G)\leq n+1$.
\end{itemize}
\end{enumerate}
\end{theorem}
\rem
 On notera que    $p_E$, $p_F$ et $p_G$ sont des entiers bien définis 
 si~$\gA$ est \fdi et non trivial et si les trois modules admettent une \rsf.
Dans ce cas, on peut formuler les trois in\egts précédentes
sous l'une des trois formes \gui{ultramétriques} \eqves suivantes:
\begin{itemize}
\item  $p_E\leq \sup(p_F,p_G-1)$  avec \egt si $p_F\neq p_G$,
\item $p_F\leq \sup(p_E,p_G)$ avec \egt si $1+p_E\neq p_G$, 
\item $p_G\leq \sup(1+p_E,p_F)$ avec \egt si $p_E\neq p_F$. 
\end{itemize}
Voici par exemple le tableau des valeurs autorisées pour $p_F$ 
et~$p_G$ lorsque l'on fixe $p_E$, par exemple $p_E=2$, en faisant cro\^{\i}tre~$p_G$:
\[ 
\begin{array}{cccccccccccccccccccccccccccccccccccccccc} 
p_E & | &  2 & 2   & 2  & |&  2  &  2 &  2  &  2 & |&  2  & 2  & 2   & 2& \dots  \\[1mm] 
p_F & | &  2 & 2   & 2  & |&  0  & 1  & 2   & 3  & |&  4  &  5 &  6  & 7&  \dots    \\[1mm] 
p_G & | &  0 & 1   & 2  & |& 3   & 3  & 3   & 3  & |&  4  & 5  & 6   & 7&  \dots    \\[1mm] 
& | &  p_G & \leq    & p_E  & |& p_G & =    & p_E   & +\;1  & |&  p_G & >    & p_E   &  +\;1 &\dots    \\[1mm] 
 \end{array}
\]
Ce tableau aidera \llec à se convaincre que les différentes
formes sont bien \eqves  lorsque les \pdis sont des entiers bien définis.
\eoe 

\begin{proof}  Pour le cas des \rsfs on note d'abord que la proposition \ref{cor2lemModifComplexe} implique qu'une \rsf donne une \pn de longueur 
infinie (parce qu'un \mptf en a une). Pour  le  cas des \rlfs, on note 
que toute \rsn \stl finie donne lieu à une \rlf (dont la longueur a augmenté au plus d'un unité). 
\\
Le corolaire \ref{cor0thcorlemGlaz1}
nous dit que les trois modules admettent des \pns infinies qui 
peuvent être organisées en une \seco de complexes.
\\
En outre, comme dans le \thref{corlemGlaz1}, si nous notons $E'_k$, $F'_k$ et $G'_k$ les noyaux des flèches 
$E_k\to E_{k-1}$, $F_k\to F_{k-1}$ et $G_k\to G_{k-1}$,  on obtient une \seco
$0\to E'_k\lora F'_k\lora G'_k\to 0$.

Le point \emph{a} et l'in\egt $p_F\leq \sup(p_E,p_G)$  résultent du \thref{thSecoRlf0} dès que l'on sait que $\Pd(E)<\infty$ et $\Pd(G)<\infty$.

Pour terminer la \dem du point \emph{a}, il suffit de voir que l'on a

\snic{\Pd(F)\leq m \,\Longrightarrow\, \big( \Pd(E)<\infty\iff \Pd(G)<\infty\big).}

\snii
Si $m=0$, $F$ est \ptf et le lemme \ref{lemEFG} donne le résultat.
 \\
Si $m\geq 1$    
on écrit une \seco 

\snic{0\to K\lora L\lora F\to 0 \qquad(\alpha)}

\snii
avec $L$ \lrf, $K\subseteq L$
et $\Pd(K)\leq m-1$. On considère \hbox{que $F=L/K$}, puis que $E$ est un sous-module de $F$, et l'on écrit~\hbox{$E=\wi{E}/K$} avec $K\subseteq \wi{E}\subseteq L$ ce qui donne deux \secos 

\snic{0\to K\lora \wi{E}\lora E\to 0 \qquad(\beta), }

\snic{ 0\to \wi{E}\lora L\lora G\to 0 \qquad(\gamma).}

($G\simeq F/E\simeq(L/K)/(\wi{E}/K)\simeq L/\wi{E} $). 
\\
Supposons tout d'abord $\Pd(E)<\infty$. Alors $(\beta)$ donne $\Pd(\wi{E})<\infty$
puis on obtient $\Pd(G)<\infty$ par $(\gamma)$ en utilisant le cas $m=0$.
\\
Supposons ensuite $\Pd(G)<\infty$, alors $\Pd(\wi{E})<\infty$ par $(\gamma)$ en utilisant le \hbox{cas $m=0$}. Et $\Pd(E)<\infty$ par $(\beta)$ en utilisant le cas précédent. 

\emph{b.}
Il reste à vérifier les implications dans le premier et le troisième cas.
\\
Supposons 
que  $\Pd_\gA(F)\leq n$ et $\Pd_\gA(G)\leq n+1$.
On a donc $\Pd(F'_n)\leq 0$ \hbox{et $\Pd(G'_n)\leq 1$}. Le lemme \ref{lemEFG}
nous dit que $\Pd(E'_n)\leq 0$.
\\
Supposons 
que  $\Pd_\gA(E)\leq n$  et $\Pd_\gA(F)\leq n+1$.
On a donc \hbox{$\Pd(E'_n)\leq 0$} \hbox{et $\Pd(F'_n)\leq 1$}. Le lemme \ref{lemEFG}
nous dit que $\Pd(G'_n)\leq 1$.
\end{proof}
%
%

\penalty-2500 
Il est instructif de comparer les trois types d'in\egts obtenues pour les
entiers $\Ld(\bullet)$ (\thref{thcorlemGlaz1}), $\Pd(\bullet)$ (\thref{ineqpEFG}), et $\Gr(\bullet)$ (\thref{thSESPrf}) avec une \seco $0\to E\to F\to G\to0$.

\medskip\centerline{ 
{\setlength{\unitlength}{.0833\textwidth}
\tabcolsep0pt\renewcommand{\arraystretch}{0}%
\begin{tabular}{|c|c|c|}
\hline
\Boite{.8}{3.5}{$\lambda_E\geq \inf (\lambda_F,\lambda_G-1)$}&
\Boite{.8}{3.2}{$\lambda_F\geq \inf (\lambda_E,\lambda_G)$}&
\Boite{.8}{3.5}{$\lambda_G\geq \inf (\lambda_E+1,\lambda_F)$}\\
\hline
\Boite{.8}{3.5}{$p_E\leq \sup(p_F,p_G-1)$}&
\Boite{.8}{3.2}{$p_F\leq \sup(p_E,p_G)$}&
\Boite{.8}{3.5}{$p_G\leq \sup(p_E+1,p_F)$}\\
\hline
\Boite{.8}{3.5}{$g_E\geq \inf(g_F,g_G+1) $}&
\Boite{.8}{3.2}{$g_F\geq \inf(g_E,g_G) $}&
\Boite{.8}{3.5}{$g_G\geq \inf(g_E-1,g_F)$}\\
\hline 
\end{tabular}
}
}

\section[Comparaison de \rsfs]{Comparaison de  deux \rsfs d'un même module}
\label{secCompareRSF}

\begin {lemma}\label{lemCompareRSFS}
Soit $u_{-1} : E \to E'$ un morphisme de \Amos, $(P\ibu,\partial)$ un
complexe descendant se terminant par $E$ avec les $P_i$ \pros, et
$(P'_\bullet,\partial')$ une \rsn de $E'$:

\snic {
\xymatrix {
\cdots\ \ar[r] &P_2\ar@{-->}[d]_{u_2}\ar[r]^{\partial_2} 
       &P_1\ar@{-->}[d]_{u_1}\ar[r]^{\partial_1} 
&P_0\ar@{-->}[d]_{u_0}\ar[r]^{\partial_0}   & E\ar[d]_{u_{-1}}
\\
\cdots\ \ar[r] &P'_2\ar[r]_{\partial'_2} &P'_1\ar[r]_{\partial'_1} 
&P'_0\ar[r]_{\partial'_0} & E'\ar[r] & 0 
}
}

\snii 
Alors, on peut prolonger $u_{-1}$ en un morphisme de $P\ibu$ dans $P_\bullet'$.
\\
De plus si~$u$ et~$v : P\ibu \to P_\bullet'$ sont deux prolongements de $u_{-1}$,
alors $u$ et~$v$ sont reliés par une homotopie \lin.
\end {lemma}
NB. On ne fait aucune hypothèse concernant les modules $P'_i$.
Et les $P_i$ ne sont pas \ncrt \tf.

\begin{proof}
L'\ali $u_0$ s'obtient en relevant $u_{-1} \circ \partial_0 : P_0
\to E'$ à $P'_0$. On a utilisé le fait que $P_0$ est projectif et que
$\partial'_0$ est surjectif. 
\\
Ensuite, remarquons que $\Im(u_0 \circ
\partial_1) \subseteq \Im\partial'_1=\Ker \partial'_0$. \\
En effet, $\partial'_0 \circ u_0
\circ \partial_1 = u \circ \partial_0 \circ \partial_1 = 0$, donc
$\partial'_0 \circ (u_0 \circ \partial_1) = 0$. 
\\
Comme $u_0 \circ\partial_1$ est à valeurs dans $\Im\partial'_1$, on
peut  relever $u_0 \circ \partial_1$ en $u_1 : P_1 \to P'_1$ (on 
utilise le fait que $P_1$ est projectif et l'exactitude de la ligne du bas en $P'_0$).  On continue ainsi de proche en proche: on relève $u_1 \circ
\partial_2$ en $u_2 : P_2 \to P'_2$ et ainsi de suite.

Soient $u$, $v : P \to P'$ deux prolongements de $u_{-1}$.
On veut montrer qu'ils sont homotopes.  En rempla\c{c}ant $u$
par $u-v$, on se ramène à montrer qu'un morphisme $u : P \to P'$
qui prolonge $0$ ($u_{-1} = 0$) est homotope à zéro.

\snic {
\xymatrix {
\cdots\ \ar[r] &P_2\ar[d]_{u_2}\ar[r]^{\partial_2} 
       &P_1\ar@{-->}[dl]|-{h_1}\ar[d]_{u_1}\ar[r]^{\partial_1} 
&P_0\ar@{-->}[dl]|-{h_0}\ar[d]_{u_0}\ar[r]^{\partial_0}   & E\ar[d]_{0}\ar[r] & 0
\\
\cdots\ \ar[r] &P'_2\ar[r]_{\partial'_2} &P'_1\ar[r]_{\partial'_1} 
&P'_0\ar[r]_{\partial'_0} & E'\ar[r] & 0 
}
}

Comme $u_0$ induit $0 : E\to E'$, on a $\Im u_0 \subseteq \Ker\partial'_0 =
\Im\partial'_1$ et comme $P_0$ est projectif, on peut relever $u_0$
en $h_0 : P_0 \to P'_1$, i.e. $u_0 = \partial'_1 \circ h_0$. 
\\
Ensuite on considère l'\egt
$$
\partial'_1 \circ (u_1 - h_0\circ\partial_1) =
\partial'_1\circ u_1 - u_0\circ\partial_1= 0,
$$
donc $\Im(u_1 - h_0\circ\partial_1) \subseteq \Ker\partial'_1 =
\Im\partial'_2$, et comme $P_1$ est projectif, on peut relever
$u_1 - h_0\circ\partial_1$ en $h_1 : P_1 \to P'_2$, i.e.

\snic{u_1 - h_0\circ\partial_1 = \partial'_2\circ h_1.}

\snii
De proche en proche,
on construit des \alis $h_i : P_i \to P'_{i+1}$ vérifiant

\snic{u_i = h_{i-1}\circ\partial_i + \partial'_{i+1} \circ h_i.}

\snii
Autrement dit, $u$ est homotope à $0$ via $h$.
\end{proof}

Les points \emph{2} et \emph{3}  du \tho qui suit  constituent une \gnn du \thref{thBetti} (consacré au \rsns libres minimales).
Le point \emph{1} est un corolaire immédiat du lemme précédent.

\begin {theorem} \label{thcomparaisonRSFS}
Soient  $P_\bullet$ et $P_\bullet'$ deux  \rsns d'un même \Amo~$E$
par des \mptfs:
\vspace{-2pt}
$$
\begin{array}{cccccccccccccc}
 \cdots  &  P_n& \vers{\partial_n} & P_{n-1}&  \cdots \cdots  & P_1 
  &\vers{\partial_1} & P_0 & \vers\pi & E& \to& 0
\\[1mm]
 \cdots  & P'_n& \vers{\partial'_n} & P'_{n-1}&  \cdots \cdots & P'_1 
  & \vers{\partial'_1} & P'_0
& \vers{\pi'} & E&  \to& 0
\end {array}
$$
\begin{enumerate}
\item Les deux complexes sont homotopiquement \eqvs
\item On suppose que $\Im \partial \subseteq \Rad(\gA)P$ (on dit dans ce cas que la \rsn est minimale).  Alors on a un \iso de complexes $P_\bullet' \simeq P_\bullet \oplus K_\bullet$ où $K_\bullet$ est une \rsn du module nul
par des \mptfs

\snic{\cdots   \;K_n\; \lora \;\cdots\; \lora\; K_1 \;\lora\; K_0\; \lora \;0.}

\snii
En fait, $K_\bullet$ est isomorphe
à une somme directe de complexes triviaux 

\snic{\cdots \;\lora\; 0\; \vvvvers{v_{k+1}}\; Q_k \;\vvvvers {v_k=\Id_{Q_k}}\;
Q_k  \;\vvvvers{v_{k-1}}\; 0 \;\lora\; \cdots}

\snii
avec des  \mptfs $Q_k$ ($k\geq 1$). \\
Enfin si les $P_i$ et $P'_i$ sont \stls ou de rang constant, il en va de même pour les $Q_k$.
\item Si les deux \rsns sont minimales, les deux complexes sont isomorphes (avec $\Id$ comme flèche de $E$ dans $E$).
\end{enumerate}
\end{theorem}
NB: on n'a pas supposé que les \rsns sont de longueur finie.
\begin{proof} 
D'après le lemme \ref{lemCompareRSFS}, il existe un morphisme $u_\bullet : {P_\bullet} \to P_\bullet'$ qui
prolonge~$\Id_E$ et un morphisme $u'_\bullet : P_\bullet' \to P_\bullet$ qui prolonge~$\Id_E$. 
\\
Par
conséquent, $u_\bullet \circ u_\bullet' : P'_\bullet \to P'_\bullet$ est un morphisme qui prolonge $\Id_E$ et est donc
homotope à $\Id_{P'_\bullet}$.
De même  $u_\bullet' \circ u_\bullet : P_\bullet \to P_\bullet$ est un morphisme 
homotope à $\Id_{P_\bullet}$.
\\
Ceci démontre le point \emph{1} sans supposer que $P\ibu$ est une \rsn minimale.

Poursuivons.
Il existe une morphisme $h_\bullet : P_\bullet \to P_\bullet$ de degré
$+1$ tel que (on supprime ci-après les $\bullet$ en indice) 

\snic{\Id_{P} - u' \circ u = \partial \circ h + h \circ \partial,}

\snii
d'où
$\Im(\Id_{P} - u' \circ u) \subseteq \Rad(\gA)P$. On en déduit que pour chaque $n\geq 0$, on a $\det(u'_n \circ
u_n) \in 1 + \Rad(\gA)$, donc $\det(u'_n \circ u_n)$ est \iv et $u' \circ u$ est
un \auto de $P$.

On a ainsi démontré les points \emph{1} et \emph{3.}  

Voyons le point \emph{2.}
Tout d'abord, 

\snic{P' = \Im u \oplus \Ker u'\simeq P\oplus \Ker u'}

\snii parce que
$u' \circ u$ est
un \auto de $P$(\footnote{Si $w=(u' \circ u)^{-1}$, $w\circ u'$ est une surjection scindée, de section $u$.}).
\\
Ainsi $P' \simeq P \oplus K$ avec $K = \Ker u'\simeq \Coker u$.

Examinons enfin le complexe exact

\snic{K_\bullet : \cdots \cdots\; \lora\; K_2 \;\vers {\delta_2}\; K_1 \;\vers {\delta_1}\; 
 K_0\; \to 0.}

\snii
Puisque les $K_i$ sont \ptfs,
$K_\bullet$ est une
somme directe de complexes triviaux 
(fait~\ref{factCompExPro}).
\end{proof}
%

\section[Idéaux \caras]{Idéaux \caras}\label{secCPROexact}

Un grand nombre de \prts concernant les \rsfs se déduisent 
de celles concernant les \rlfs en appliquant le lemme \ref{lemlorsbloclresb} 
  et le \plgref{plcc.resf}.
La fin de ce chapitre consiste essentiellement 
à passer en revue un certain nombre de ces \prts.

\medskip 
Pour que les énoncés soient plus agréables on a intérêt à étendre
la plupart des \dfns où interviennent des entiers dans le cas libre
(rang, \pdi, profondeur),
en remplaçant ces entiers par des \elts arbitraires de $\HO\gA$.

Rappelons que pour une \ali $\varphi:P\to Q$ entre deux \Amos \ptfs et pour $r\in\HO(\gA)$
on sait définir \emph{l'\idd d'ordre~$r$ de~$\varphi$}, noté $\cD_r(\varphi)$ (voir les exercices \Cref{X-21 et X-22}).
La \prt \cara la plus importante est que ces \idds se comportent bien par \eds
et qu'ils co\"{\i}ncident avec les \idds usuels lorsqu'une
\lon rend les modules libres (après \lon en le \mo $S$, l'\elt~$r$ devient alors égal à un entier
de $\HO\gA_S$).

\subsec{Rang stable}
La \dfn suivante généralise la \dfn \ref{defiRangStable}.
\begin{definition} \label{defiRangstableHO} \emph{(Rang stable)} \, 
Soit $r \in\HO\gA$.
\begin{enumerate}
\item Une \Ali $\varphi$ entre \mptfs est dite de \emph{rang stable supérieur ou égal à $ r $} si l'\idd $\cD_r (\varphi)$ est fidèle. On écrit $\rgst_\gA(\varphi)\geq r $ ou $\rgst(\varphi)\geq r $. 
\item L'\ali $\varphi$ est dite de \emph{rang stable $r $},  ce que l'on écrit $\rgst_\gA(\varphi)= r $, si en outre $\cD_{\gA,r +1}(\varphi)=\gen{0}$, \cad si $\rg_\gA(\varphi)\leq r $. On dit dans ce cas que l'\ali $\varphi$ est \emph{stable}. 
\item Un \Amo \pf $M$ est dit de \emph{rang stable $r $},  et l'on \hbox{écrit $\rgst_\gA(M)= r $}, si $\cF_{\gA,r }(M)$ est fidèle
et $\cF_{\gA,r -1}(M)=0$. Ainsi, lorsque l'on a une \pn de $M$ sous la forme
$$
\gA^{m +r }\vers{A}\gA^{n+r }\vers\pi M\to 0  \;\;\hbox{où }\;
m ,\,r  \hbox{ et } n \in\NN
$$
on a $\rgst(M)=r $ \ssi $\rgst(A)=n $.
\end{enumerate}%
\index{rang stable!d'une \ali entre \mptfs}\index{rang stable!d'une \ali entre \mptfs}\index{rang stable!d'un \mpf}
\end{definition}

Le \plgref{plcc.RangStable}, les propositions \ref{prop.rgsta.rgconstant} et \ref{propSecoStab}
et le lemme du rang stable \ref{lemSuitExStable} restent
valables en prenant les rangs dans $\HOp\gA$: il faut éventuellement remplacer les matrices
par des \alis entre \mptfs. 

\subsec{Complexes de \mptfs \gnqt exacts}

\begin{definition} \label{defiidcars} \emph{(Caractéristique d'Euler-Poincaré, \idcas)}\\
On considère un complexe borné de \mptfs 

\snic{P\ibu:\quad 0 \lora P_n \vvers{u_n}  P_{n-1} \vvvers{u_{n-1}}\;  \cdots \cdots \; \vvers{u_2}  P_1 \vvers{u_1} P_0 \;}

On note pour $k\in\lrbn$, $p_k=\rg(P_k)\in\HO(\gA)$, et l'on définit
\begin{itemize}
\item $r_{n+j}=0$ pour $j>0$,
\item  $r_n=p_n$,   puis de proche en proche,
\item   $r_k\in\HO(\gA)$ 
vérifiant $r_k=p_k-r_{k+1}$ pour $k=n-1$, \dots, $0$.
\end{itemize}
\begin{enumerate}
\item Le rang $r_0=p_0-p_{1}+p_{2}-\dots$ s'appelle la \emph{\cEP du complexe~$P\ibu$}, on le note $\chi(P\ibu)$.%
\index{caracteristique@caractéristique d'Euler-Poincaré!d'un complexe borné de modules projectifs de type fini}
\item Pour $k\geq 1$, le rang $r_k$ s'appelle le \emph{rang stable attendu de 
l'\ali~$u_k$}.%
\index{rang stable attendu!d'une \ali dans un complexe de \mptfs}
\item On suppose maintenant que $r_k\in\HOp(\gA)$ pour tout $k$.
\\
On note alors \fbox{$\fD_k=\cD_{r_k}(u_k)$} et \fbox{$\fD_{k,\ell}=\cD_{r_k-\ell}(u_k)$} pour $k\in\NN$ \hbox{et $\ell\in\ZZ$}. On appelle ces \ids les \emph{\idcas
du complexe~$P\ibu$}.%
\index{ideaux cara@\idcas!d'un complexe de modules projectifs de type fini.}%
\end{enumerate}

\end{definition}

\begin{definition} \label{defignqexact}
Un complexe est dit \emph{\gnqt exact} s'il devient 
exact   après 
\lon en des \ecr.\index{complexe!generiquement@\gnqt exact}%
\index{generiquement@\gnqt exact!complexe} 
\end{definition}

\begin{theorem} \label{lemcEPcompPTF} \emph{(Structure locale d'un complexe \gnqt exact)}\\
On considère un complexe borné de \mptfs 

\snic{P\ibu:\quad 0 \lora P_n \vvers{u_n}  P_{n-1} \vvvers{u_{n-1}}\;  \cdots \cdots \; \vvers{u_2}  P_1 \vvers{u_1} P_0 \;}

On suppose que ce complexe est \gnqt exact. 
\begin{enumerate}
\item Les \elts $r_k$ (définis en \ref{defiidcars}) sont tous $\geq 0$ dans $\HO(\gA)$.
En particulier si l'un des~$r_k$ est un entier $<0$, l'anneau est trivial.
\item Chaque \ali $u_k$ est stable ($k\geq 1$), de rang stable $r_k$,
autrement dit $\fD_k(P\ibu)$ est fidèle et $\fD_{k,-1}(P\ibu)=0$.
\\
En particulier le module $M=\Coker(u_1)$  est stable, de rang stable $r_0$, 
\hbox{i.e. l'\id $\cF_{r_0}(M)=\fD_1(P\ibu)$} est fidèle et $\cF_{r_0-1}(M)=
0$. Notons \hbox{que $r_0$} est aussi le rang de $M$ défini en \ref{thSchanuelVariation}.
\item Il existe des \ecr $s_1$, \dots, $s_m$ tels que, après \lon en un  
$s_i$ arbitraire, $P\ibu$ devient un complexe complètement trivial
de modules libres (comme dans le lemme \ref{lemSuitExStable2}) avec chaque $u_j$ de rang $r_j$ (qui est
un entier dans $\HOp(\gA[1/s_i])$).
\end{enumerate} 
\end{theorem}
%
\begin{proof} 
Cela résulte du cas où les $P_i$ sont \lrfs
et où le complexe est exact
(cas traité dans les \thos~\ref{thStrLocResFin} et~\ref{thRangStRLF2}). 
En effet, après \lon en des \ecr, le complexe devient un complexe de \mlrfs. Et les résultats sont assurés dès qu'ils le sont après \lon en des \ecr.  
\end{proof}
%
 
\begin{thdef} \label{thidecarnrsf}
Soit $M$  un \Amo \lorsb, et
 $P\ibu$ un complexe exact de \mptfs qui résout~$M$.
\\
 Les \idcas du complexe, $\fD_{k}(P\ibu)$ et  $\fD_{k,\ell}(P\ibu)$,
 définis en~\ref{lemcEPcompPTF}, ne dépendent que du module $M$.
 On les note $\fD_{\gA,k}(M)$ et  $\fD_{\gA,k,\ell}(M)$
et on les appelle les \emph{\idcas du \mlorsb~$M$}.%
\index{ideaux cara@\idcas!d'un module localement \rsb}%
\end{thdef}
%
\begin{proof}
Cela résulte du cas où le module est \lrsb, car après \lon en des \eco, les \rsns sont formées de \mlrfs: on applique alors le \thref{thidecarnresol}.
\end{proof}

\subsec{\Thos de Vasconcelos}
Les \thos de Vasconcelos \ref{corVascon} et \ref{corcorVascon} donnent  les résultats suivants.

\begin{theorem} \label{thVascorsf} \emph{(Annulateur et rang)}\\
Soit $M$ un \mlorsb avec $\rg(M)=r\in\HOp\gA$. 
\begin{enumerate}
\item $M$ est stable, de rang stable $r$.
\item Si $r=[e]$ pour un \idm $e$ (i.e. si $0\leq r\leq 1$),  $\Ann(M)=\gen{1-e}$.
\item Si $r\geq 1$, $M$ est fidèle.
\end{enumerate}
En particulier si l'anneau est connexe,  $\rg(M)=0\Longleftrightarrow \Ann(M)$ est fidèle, \hbox{et  $\rg(M)>0 \Longleftrightarrow M$} est fidèle.
\end{theorem}

\begin{corollary} \label{corVasconrsf} \emph{(Vasconcelos)}\\
Soit $\fa$ un \id \lorsb et $r=\rg(\fa)$.
On a $ r=[e]$ pour un \idm $e$, et $\Ann(\fa)=\gen{1-e}$.%
\begin{enumerate}
\item Si $r=0$, $\fa=0$.
\item Si $r=1$, $\fa$ est fidèle.
\item Un \idp \lorsb est \pro.
\end{enumerate}
\end{corollary}

\begin{theorem} \label{corcorVasconrsf} 
Si dans l'anneau $\gA$ tout \idp admet une \rsf, alors $\gA$ est \qi.
\end{theorem}

La réciproque est évidente.

\subsec{Ce qui rend exact un complexe de \mptfs}
\begin{theorem} \label{thCPROexact} \\
On considère un complexe de \mptfs

\smallskip 
\centerline{\fbox{$P\ibu:\quad 0 \to P_n \vvers{u_n}  P_{n-1} \vvvers{u_{n-1}}\;  \cdots \cdots \; \vvers{u_2}  P_1 \vvers{u_1} P_0\;$}}

\smallskip 
 Le complexe est exact \ssi $\Gr(\fD_\ell(P\ibu))\geq \ell$ pour tout $\ell$.
\end{theorem}
%
%
\begin{proof}
Cela résulte du cas où le module est \lrsb (cf. le \thref{cor3thABH1}), car après \lon en des \eco, la \rsn est formée de \mlrfs.
Il faut noter que l'exactitude d'un complexe et les in\egts 
$\Gr(\fa)\geq  r$ sont des \prts \lgbes. 
\end{proof}
%

\bonbreak
\section{Profondeur et \pdi}
\label{secRSFRLF} 
\label{secAusBuResProj}

\subsec{Dimension projective et \ddk}

La proposition \ref{propfDk=1} est inchangée, à ceci près qu'elle
s'applique maintenant aux modules \lorsbs, et que l'on peut mettre une \eqvc
dans le point \emph{1.}

\begin{proposition} \label{propfDk=1rsf} Soit $k\geq 1$ et $M$
un \Amo \lorsb.
\begin{enumerate}
\item On a $\fD_{k}(M)=\gen{1}$ \ssi $\Pd_\gA(M)\leq k-1$. \\
Dans ce cas $\fD_{k+s}(M)=\gen{1}$ pour tout $s\geq 1$.
\item Pour tous $\ell\geq 1$, on a $\fD_{\ell}(M)\subseteq \DA\big(\fD_{\ell+1}(M)\big)$.
\end{enumerate}
\end{proposition}

On en déduit une version (à peine) \gnee du 
corolaire~\ref{corthDimKrullGr}\sibook{ et de la proposition \ref{propDimAri}}. 

\siarticle{
\begin{theorem} \label{corthDimKrullGrrsf}
Soit $\gA$ un anneau de \ddk  $\leq n$. 
Si un \Amo $M$ est \lorsb, il  est \lonrsb, \cad
$\Pd(M)\leq n$. 
\end{theorem}}
\sibook{
\begin{theorem} \label{corthDimKrullGrrsf}
Soit $\gA$ un anneau de dimension \ari  $\leq n$ (a fortiori s'il est de \ddk $\leq n$). 
Si un \Amo $M$ est \lorsb, il  est \lonrsb, \cad
$\Pd(M)\leq n$. 
\end{theorem}}

\subsec{Dimensions et profondeur vues dans $\HO\gA$}

Soit $E$ un \Amo \lorsb et $\fa$ un \itf.

Soit  $\rho$ un \elt arbitraire de $\HOp(\gA)$, 
 $\rho=\sum_{j=0}^{m}\rho_j[e_j]$
avec  un \sfio $(e_j)_{j\in\lrb{0..m}}$  et $\rho_0<\rho_1<\dots<\rho_m$ dans $\NN$. 
\\
Notons $\gA_j=\gA[1/e_j]$, et rappelons que $E[1/e_j]$ s'identifie à $e_jE$.
\\
On rappelle que  $\Pd_{\gA}(E)=-1$ signifie $E=0$.

\begin{definition} \label{defiPrGrHO} On utilise les notations ci-dessus.
\begin{enumerate}
\item On dit que $\Kdim {\gA}\leq \rho-1$ si pour $j\in\lrb{0..m}$,  $\Kdim {\gA_j}\leq \rho_j-1$(\footnote{Cette \dfn généralise le \plg \ref{FFRthDdkLoc}.}). 
\item On dit que $\Pd_{\gA}(E)\leq \rho-1$ si pour $j\in\lrb{0..m}$,  $\Pd_{\gA_j}(E)\leq \rho_j-1$(\footnote{Cette \dfn à mettre en rapport avec le \plgref{plcc.resf}, point~\emph{1}}).
\item On dit que $\Gr_{\gA}(\fa,E)\geq \rho$  si pour $j\in\lrb{1..m}$, $\Gr_{\gA_j}(\fa,E)\geq \rho_j$(\footnote{Cette \dfn à mettre en rapport avec le \plgref{plccProfondeur}.}).
\item On définit  $\fD_{\gA,\rho}(E)=\sum_{j=1}^{m}\fD_{\gA_j,\rho_j}(E)$. 
\end{enumerate}
\end{definition}

Avec ces \dfns on peut remplacer dans la proposition \ref{propfDk=1rsf}
les entiers $k$ et $\ell$,  et dans le \thref{corthDimKrullGrrsf}
l'entier $n$, par des \elts  de $\HOp(\gA)$.

\subsec{Le \tho d'\ABH}

Le \thref{thABH1} donne le \tho suivant.

\begin{theorem} \label{thABH1rsf} \emph{(\ABH, direct, bis)}\\
Soient $\kappa$ et $\mu\in\HOp\gA$, $\fa$ un \itf de $\gA$, et $E$ un \Amo \lorsb.
Si $\Pd(E)\leq \mu$ et $\Gr(\fa)\geq \kappa+\mu$
alors  $\Gr(\fa,E)\geq \kappa$.
\end{theorem}

Voici une \gnn du \tho réciproque \ref{thABH2}.

%
\begin{theorem} \label{thABH2rsf}  \emph{(\ABH, réciproque, bis)}\\
Soit $\kappa\in\HOp\gA$ et $\fa$ un \itf de $\gA$.  On considère un \Amo $E$ 
qui admet une \rsf de longueur $m\geq 1$
$$ 
    0\to P_m\vvers{u_m} P_{m-1}\lora \cdots\cdots\cdots \vvers{u_1} P_0\vvers \pi E\to 0
$$
avec  $\rg(P_j)\geq 1$ dans $\HO\gA$ pour chaque~$j$. \\
\hbox{Si  $u_m(P_m)\subseteq \fa P_{m-1}$}, et \hbox{si $\Gr(\fa,E)\geq \kappa$}, \hbox{alors
$\Gr(\fa)\geq \kappa+m$}. 
\end{theorem}

\sibook{Concernant cette réciproque, la version la plus 
\gnle avec  $m$ dans $\HOp\gA$ est un peu rébarbative. Voir l'exercice~\ref{exoPdimHO}.}


\entrenous{
Nombres de Betti locaux?

En \clama on doit avoir pour un anneau arbitraire que les nombres de Betti de $\gA_\fp$
sont des fonctions localement constantes sur le spectre de Zariski
(muni de la topologie constructible?).

Une version \cov de ce résultat (éventuellement modifié convenablement)
doit être la suivante:

\emph{Soit $M$ un \Amo qui admet une \rsf. Il existe une famille finie $(a_i)_{i\in\lrbm}$
dans $\gA$ telle que, en notant, pour $J\in\cP_m$,  $J'=\lrbm\setminus J$, après \lon en chacun des $\cS(J;J')$ le module admet une \rsn libre minimale.}

La question est de savoir 

-- primo, si c'est bien vrai, cela en a bien l'air, il suffit de suivre la \dem
que sur un \alo \dcd toute \rlf peut être raccourcie en une \rsn minimale

-- secundo, si c'est intéressant, \cad quels résultats plus concrets on peut
en déduire.

De toute manière il me semble plus pertinent de se demander ce qu'il
reste dans le cas d'un anneau arbitraire du \tho qui dit que
dans le cas local, pour deux \rsns
libres minimales d'un même module, les matrices sont deux à deux \eqves.
Cela a l'air nettement plus fort que l'invariance des \idcas.
}

\section{Structures multiplicatives}\label{secStrMult}

Dans cette section on donne la version \gui{locale} 
de \dfns et résultats obtenus dans le chapitre \ref{chapCayley}.

\subsec{Modules localement de MacRay}

\begin{lemma} \label{lem0MacRaersf} 
Si un \id s'écrit sous la forme $\fa \ffg$ avec $\fa$ \tf, de profondeur $\geq 2$
 et $\ffg$ est \pro de rang $1$, cette écriture est unique. 
\end{lemma}
%
\begin{proof}
Après \lon en des \eco $(s_1,\dots,s_n)$ l'\id $\ffg$ devient libre de rang $1$, \cad principal engendré par un \elt \ndz~$g_i$. Puisque $\fa$ est de profondeur
$\geq 2$, il admet $1$ pour pgcd fort, et \hbox{l'\id $g_i\fa$} admet $g_i$ comme pgcd fort,
ce qui implique que $g_i$ et $\fa$ sont déterminés par
l'\id $\fa$ vu dans $\gA[1/s_i]$. Puisque $\ffg$  et $\fa$ sont  connus localement, il sont connus globalement. 
\end{proof}
%
\begin{defi} \label{defiloMacRae} 
 Un \Amo $E$ est appelé un \emph{\mlMR} s'il est \pf et si l'\id $\ff_0=\cF_0(E)$ peut s'écrire sous la forme $\fa \ffg$ avec $\fa$ \tf,  $\Gr(\fa)\geq 2$
 et $\ffg$  \pro de rang $1$.
\index{module!de MacRae}\index{MacRae!module localement de ---}\index{MacRae!invariant de ---}\index{invariant de MacRae}
\\
Dans ce cas l'\id $\ffg$ (qui ne dépend que de $E$) est appelé
\emph{l'invariant de MacRae du module $E$}, et il est noté $\fG(E)$.
\end{defi}

Le rapport avec les \mMRs définis en section  \ref{secMacRae} est donné par le lemme et le \plg qui suivent.

\begin{lemma} \label{lemMacRaersfloc} \emph{(Le cas local)}
Sur un \alo un module est \lMR \ssi il est \MR.  
\end{lemma}

\begin{plcc}\label{plcc.MRmsf} \emph{(Pour les \mlMRs)}
Soient $(S_1,\dots,S_n)$ des \moco et $E$ un \Amo.
\begin{enumerate}
\item \label{i1plcc.MRmsf} Le module $E$ est \lMR \ssi
chacun des modules obtenus après \lon en $S_i$ est \lMR.
\item \label{i2plcc.MRmsf} Si $E$ est \lMR, il existe des \eco tels qu'après \lon en chacun
de ces \elts, le module est \MR.
\end{enumerate}
\end{plcc}

On en déduit que les \thos \ref{thMacRae2} et \ref{thMacRae3}
concernant le comportement des invariants de MacRae dans les suites exactes
de \mMRs peuvent se recopier sans modification pour les  suites exactes
de \mlMRs.

\subsec{Complexes localement de Cayley}
Soient  $(r_0,r_1,\dots,r_m,r_{m+1})$ dans $\HOp\gA$, avec \fbox{$m\geq 2$ et  $r_{m+1}=0$}, et un complexe descendant de \mptfs

\smallskip 
\centerline{\fbox{$P\ibu:\quad \quad 0 \to P_m \vvers{u_m}  P_{m-1} \vvvers{u_{m-1}}\;  \cdots \cdots \; \vvers{u_2}  P_1 \vvers{u_1} P_0$}, 
}

\smallskip 
avec \fbox{$\rg(P_k)= {r_{k+1}+r_k}$}, $k\in\lrb{0..m}$,  
(donc $\rg(P_m)={r_m}$ et~\hbox{$\chi(P\ibu)=r_0$}).
\\
Dans la suite, on considère les \idcas \fbox{$\fD_k(P\ibu)=\cD_{r_k}(u_k)$}
et l'on fait les hypothèses suivantes.
\Grandcadre{$\Gr\big(\fD_1(P\ibu)\big)\geq 1$, et $\Gr\big(\fD_k(P\ibu)\big)\geq 2$ pour $k\in\lrb{2..m}$.}

\begin{definition} \label{defiCompCayrsf}
Sous ces hypothèses, on dira que le complexe est un \emph{complexe localement de Cayley}.%
\index{Cayley!complexe localement de ---}\index{complexe!localement de Cayley}
 \end{definition}

Le rapport avec les complexes de Cayley définis en section  \ref{secDetCayley} est donné par le lemme et le \plg qui suivent.

\begin{lemma} \label{lemCayleyrsfloc} \emph{(Le cas local)}
Un complexe borné de \mlrfs est \lot de Cayley
\ssi il est de Cayley.  
\end{lemma}

\begin{plcc}\label{plcc.Cayleyrsf} \emph{(Pour les complexes \lot de~Cayley)}
Soient $(S_1,\dots,S_n)$ des \moco et $P\ibu$ un complexe borné de \mptfs
\begin{enumerate}
\item \label{i1plcc.Cayleyrsf} Le complexe $P\ibu$ est \lot de Cayley \ssi
chacun des complexes obtenus après \lon en $S_i$ est \lot de Cayley.
\item \label{i2plcc.Cayleyrsf} Si $P\ibu$ est \lot de Cayley, il existe des \eco tels qu'après \lon en chacun
de ces \elts, le complexe est de Cayley.
%
%
\end{enumerate}
\end{plcc}

Notez que dans le point \emph{\ref{i2plcc.Cayleyrsf}}, il faut et suffit que chaque $P_k$ devienne libre après chaque \lon.

On obtiendra donc les résultats qui suivent par application
de \plgcs divers et variés. 

\begin{fact} \label{factCay1rsf}
Dans un complexe de Cayley, chaque \ali~$u_k$ est de rang stable $r_k$. 
\end{fact}

Pour le \tho fondamental \ref{thdetCay}, voici la version \gnee que nous proposons faute de mieux.

\begin{thdef} \label{thdetCayrsf} \emph{(Complexe \lot de Cayley: \ids de \fcn, \deter de Cayley)}\\
On considère un complexe \lot de Cayley $P\ibu$ (\dfn \ref{defiCompCayrsf}). 
 On a les résultats suivants.
\begin{enumerate}
\item Il existe un unique \sys d'\ids $\fB_k$  pour $k\in\lrbm$,
avec $1\in\fB_m$, satisfaisant la \prt suivante:
\begin{enumerate}
\item [] Si le complexe devient un complexe de modules libres après \lon en 
un \elt $s$, alors l'\id $\fB_k$ est égal à 
$\fB_{\gA[1/s],k}(P\ibu)$ dans~$\gA[1/s]$.
\end{enumerate}
Ces \ids $\fB_k$ sont appelés les \emph{\ids de \fcn du complexe $P\ibu$}.
On les note plus \prmt $\fB_{\gA,k}(P\ibu)$.%
\index{complexe localement de Cayley!ideaux de fac@\ids de \fcn d'un ---}\index{ideaux de fac@\ids de \fcn!d'un complexe localement de Cayley} 
\item On a  \fbox{$\fD_k(P\ibu)=\fB_k(P\ibu)\fB_{k-1}(P\ibu)$} pour $k\in\lrbm$.
\item On a donc $\Gr(\fB_0)\geq 1$,  $\Gr(\fB_k)\geq 2$ pour $k\geq 1$, et 
$$
\prod_{ k\in\lrb{0.. m}}\fB_k\;=\;\;{\fB_0\prod_{k:1<2k\leq m}\fD_{2k}\;=\prod_{k:1\leq 2k+1\leq m}\fD_{2k+1}}
$$
\item Lorsque $\chi(P\ibu)=0$, \cad lorsque $r_0=0$ et $\rg(P_0)={r_1}$, le module~\hbox{$M=\Coker u_1$} est un \mlMR, autrement dit  $\Gr(\fB_1)\geq 2$ 
et $\fB_0$
est un \id \pro de rang $1$ (c'est l'invariant \MR de $M$). 
On note $\fG$ ou $\fG(P\ibu)$
cet \id $\fG(M)$.\index{Cayley!determin@\deter de ---}%
\index{determin@\deter!de Cayley} 
\item Considérons un changement d'anneau de base $\rho:\gA\to\gB$ 
où $\gB$ est plat sur $\gA$. 
Alors le complexe reste un complexe de Cayley, et les  \ids $\fB_\ell$
sont transformés en leurs images par $\rho$.

%
\end{enumerate}
\end{thdef}

\subsec{Structure multiplicative des \rsfs}

Les \thos \ref{thdetCay2} et \ref{thResFinIdFac} donnent le \tho suivant.

\begin{thdef} \label{thdetCay2rsf}~
\begin{enumerate}
\item Si  un \Amo  $M$ est \lorsb, le \thref{thdetCayrsf} s'applique
à tout complexe de \mptfs qui le résout.
En outre les \ids de \fcn $\fB_k$ ne dépendent que du module $M$.
\\
On les appellera les \emph{\ids de \fcn du module $M$}, et on les notera
$\fB_k(M)$.

\end{enumerate}
En particulier, on obtient les résultats suivants. 
\begin{enumerate} \setcounter{enumi}{1}
\item Un module $M$  \lorsb de rang $0$ est un \mlMR. Son invariant
de MacRae $\fG(M)$ est égal à $\fB_{0}(M)$.  
\item Si un \itf $\fa$ est \lorsb et  de rang $1$, 
alors il s'écrit de manière unique sous forme $\fa\ffg$ avec \hbox{$\Gr_\gA(\fa)\geq 2$} et $\ffg$ \pro de rang $1$.
\item Pour $k\geq 1$ on a 
$
{\fB_k\subseteq \DA(\fB_{k+1})\;\hbox{ et }\;\DA(\fB_k)= \DA(\fD_{k+1}).}
$
\\
Enfin: $\;\;\fB_k=\gen{1}\;\Longleftrightarrow\;\fD_k=\gen{1}\;\Longleftrightarrow\;
\Pd(M)<k$.
\end{enumerate}
\end{thdef}

Le \thref{corthdetCay2} donne le résultat suivant. 

\begin{theorem} \label{corthdetCay2rsf} \emph{(Idéaux à deux \gtrs \lorsbs)}
\begin{enumerate}
\item On considère un \id $\fa=\gen{a_1,a_2}$ fidèle dans $\gA$. \Propeq
\begin{enumerate}
\item  Après \lon en des \eco $s_i$ on a: l'\id $\fa$ admet un pgcd $g_i$ \ndz et $\Gr(\fa/g_i)\geq 2$.
\item L'\id $\fa$ est \lorsb.
\item  Après \lon en des \eco $s_i$ on a des \sexs
$$\preskip-.4em \postskip.2em 
0\lora\gA_i\vvvvers{[\,b_{i,1}\;b_{i,2}\,]} \gA_i^2\vvvvers{\tra[\,a_1\;a_2\,]~}
\gen{a_1,a_2}\lora 0. 
$$
\item Le module $\gA/\fa$ est \lMR.
\end{enumerate}

\item Pour un anneau $\gA$ \propeq
\begin{enumerate}
\item Tout \id à deux \gtrs   est \lorsb.
\item L'anneau est \qi et l'intersection de deux \idps
arbitraires est un \id \lop.
\item L'anneau est \qi et l'intersection de deux \id \lops
arbitraires est un \id \lop.
\end{enumerate}
\end{enumerate}  
\end{theorem}

\newpage	
\Exercices

\begin{exercise}
\label{exoPdimHO} (Autres résulats concernant la \pdi vue dans $\HO(\gA)$)
{\rm  Soit $M$ un \Amo \lorsb.
\begin{enumerate}
\item \ABH réciproque.
\item Cas où $\Pd(M)$ est \gui{bien défini}.
\item Dimensions projectives dans une suite exacte courte. Les in\egts
dans le point \emph{2} du \thref{lemPMP/M} restent valables
avec les dimensions prises dans~$\HO(\gA)$.   
\end{enumerate}

}
\end{exercise}

\begin{exercise} \label{exoPdimQuotient}
{(Dimension projective d'un quotient)}\\
{\rm  
Soient $n$, $d\in\NN$ avec $d \le n+1$.
Donner un exemple de deux modules $F \subseteq E$
tels \hbox{que $\Pdim E = \Pdim F = n$} et $\Pdim E/F = d$.
}

\end{exercise}


\sol
\exer{exoPdimQuotient}
On sait que $\Pdim E/F \leq 1+\max(\Pdim E,\Pdim F)$.
L'idée est d'avoir~$E/F \simeq E'/F'$ avec~$\Pdim E' = \Pdim F' = d-1$,
en espérant~\hbox{$\Pdim E'/F' = d$}.

On note~$\gA = \gk[x_1, \ldots, x_{n+1}]$ avec~$\gk$ un \cdi. 
\\
Considérons
les \ids $F = x_{n+1} \gen {x_1, \ldots, x_{n+1}}$ et

\snic {
E = \gen {x_1, \ldots, x_d} + F =
\gen {x_1, \ldots, x_d,\ x_{n+1}x_{d+1}, \ldots, x_{n+1}x_{n+1}}
}

\snii
On a~$F \simeq \gen {x_1, \ldots, x_{n+1}}$ et comme~$(x_1, \ldots, x_{n+1})$
est une \srg, son complexe de Koszul descendant fournit une
résolution libre graduée minimale de longueur~$n$ de l'\id $\gen {x_1,
  \ldots, x_{n+1}}$. On a donc~$\Pdim F = n$.

Montrons que l'on a \egmt $\Pdim E = n$.
\entrenous{TROUVER l'ARGUMENT ET L'\'ECRIRE} 

Passons maintenant au quotient. On a:

\snic {
E/F \simeq \gen {x_1, \ldots, x_d}/(\gen {x_1, \ldots, x_d} \cap F) =
\gen {x_1, \ldots, x_d}/x_{n+1}\gen {x_1, \ldots, x_d}
}

On voit donc que~$E/F \simeq M/aM$, avec~$M = \gen {x_1, \ldots, x_d}$, et
l'\elt~$a =
x_{n+1}$ est $M$-\ndz. De manière \gnle, considérons un \Amo $M$, une
\rsn \emph {quelconque} de~$M$ de longueur~$m$

\snic {
0\to L_m\vvers{u_m} L_{m-1}\vvers{u_{m-1}} \cdots\cdots \vvers{u_1} L_0\vvers{u_0} 
M \to 0,
}

et un \elt $a\in \gA$ $M$-\ndz. La \rsn de~$M/aM$ de longueur
$m+1$ obtenue par la construction mapping-cône de la multiplication
par~$a$ est la suivante: 
$$
0\to L_m\vers{v_{m+1}} L_{m-1}\oplus L_m \vers{v_m} 
\cdots   L_1\oplus L_2 \vers{v_2}  L_0\oplus L_1 \vers{v_1} L_0
\vers {v_0} M/aM \to 0.
$$
On a~$v_0 : x_0 \mapsto u_0(x_0) \bmod aM$, 
$v_1 : x_0 \oplus x_1 \mapsto u_1(x_1) + ax_0$ et
pour~$i \ge 2$ (en convenant de~$L_i = 0$ pour~$i > m$), on a:

\snic {
v_i = \cmatrix {-u_{i-1} & 0\cr a & u_i}, \qquad
x_{i-1} \oplus x_i \mapsto -u_{i-1}(x_{i-1}) \oplus (u_i(x_i) + ax_{i-1}).
}

Le lecteur peut d'ailleurs vérifier dans ce cas particulier qu'il s'agit
bien d'une résolution de~$M/aM$ (il est impératif d'utiliser le fait que
$a$ est~$M$-\ndz).

Dans un contexte gradué, si~$M$ est gradué,~$a$ \hmg et la résolution
$(u_i)$ de~$M$ est graduée, alors la résolution~$(v_j)$ de~$M/aM$
est graduée.

Ici, on a~$M = \gen {x_1, \ldots, x_d}$ et le complexe de Koszul descendant de
$x_1, \ldots, x_d$ fournit une résolution libre graduée minimale de
longueur~$d-1$ de l'\id $\gen {x_1, \ldots, x_d}$.  On obtient alors une
résolution libre graduée minimale de longueur~$d$ du quotient~$\gen {x_1,
  \ldots, x_d}/x_{n+1}\gen {x_1, \ldots, x_d}$ et on en déduit que~$\Pdim
E/F = d$.




\Biblio

\newpage \thispagestyle{empty}
\incrementeexosetprob

\junk{

\begin{proof}
On sait déjà qu'après \lon en des \eco on a une \rlf pour $M$ 
(proposition \ref{propRSFRLF}). On peut donc supposer que
$M$ possède une \rlf et que son rang est un entier $r$.
On considère une \rlf  
$$
0 \to L_n \vvers{u_n}  L_{n-1} \vvers{u_{n-1}}\;  \cdots \cdots \; \vvers{u_2}  L_1 \vvers{u_1} L_0 \vers{\pi} M
\to 0,
$$
On note $p_i$ le rang de $L_i$. Si $n=0$ le résultat est clair. Supposons $n>0$. Puisque $u_n$ est une matrice injective,
les mineurs d'ordre  $p_n$ de $u_n$ sont \cor. Si on localise en un de ces \elts, d'après la proposition \ref{corlemModifComplexe}
on peut remplacer $p_n$ et $p_{n-1}$ par $0$ et $p_{n-1}-p_n$,
sauf si $p_{n-1}<p_n$, auquel cas $\gA=0$. 
On a raccourci la \rsn de au moins une unité.

\noindent En répétant le processus on obtient qu'après des \lons en des \ecr on a une \rsn de longueur $0$ (et le module est libre), 
ou on a découvert que $\gA=0$ (ce qui se produit si $p_{n-1}<p_n$, ou $p_{n-2}<p_{n-1}-p_n$ etc.)

\noindent 
Les affirmations du point \emph{2} résultent alors de ce que les \idfs
se comportent bien par \lon et du \plgref{plcc.regularite} (appliqué avec $E=\gA$).

\noindent 
Concernant le point \emph{3} on a des \lons en des \eco qui
rendent les \rsfs de $M$, $N$ et $N/M$ libres.
Il suffit donc de montrer le résultat lorsque les 3 \rsns sont libres. Les
rangs entiers, disons $r_1=\rg(N)$, $r_2=\rg(M)$, $r_3=\rg(N/M)$,  se lisent alors après avoir localisé en des \ecr $(a_i)_{i\in I}$ comme rangs de modules libres
avec $M\subseteq N$. Si $r_1\neq r_2+r_3$ alors on a $1=0$ dans tous les $\gA[1/a_i]$, et puisque les $a_i$ sont \cor cela donne $1=0$ dans $\gA$
et les trois rangs sont égaux dans $\HO(\gA)$.
\end{proof}

Comme conséquence \imde on obtient le \tho suivant.
}


\let\showchapter\oldshowchapter
\let\showsection\oldshowsection

\pagestyle{CMpreheadings}
\pagestyle{CMExercicesheadings}

\small 
\bibliographystyle{plain}
\bibliography{BibACMC.bib} 


\cleardoublepage \thispagestyle{CMcadreseul}

\thispagestyle{CMcadreseul}
\cleardoublepage 
\rdb

\renewcommand\indexname{Index des termes}

\addtocontents{toc}{\vskip-0.8em}
\addcontentsline{toc}{chapterbis}{Index des termes}
\markboth{Index des termes}{Index des termes}

\printindex


\end{document}